\documentclass[12pt,twoside,reqno]{amsart}

\usepackage{amsmath, amssymb, amsthm, enumitem}
\usepackage[small,nohug,heads=vee]{diagrams}

\usepackage[hyperindex]{hyperref}
\hypersetup{bookmarksdepth=3}

\newcommand{\lb}{\linebreak[1]}
\newcommand{\wh}[1]{\widehat{#1}}
\renewcommand{\t}{\mathfrak{t}}
\newcommand{\RR}{\mathcal{R}}
\renewcommand{\SS}{\mathcal{S}}
\newcommand{\C}{\mathcal{C}}
\newcommand{\M}{\mathcal{M}}

\newcommand{\IR}{\mathbb{R}}

\newcommand{\IH}{\mathbb{H}}
\newcommand{\IN}{\mathbb{N}}

\newcommand{\IZ}{\mathbb{Z}}

\newcommand{\eps}{\varepsilon}
\newcommand{\ov}[1]{\overline{#1}}
\newcommand{\td}[1]{\widetilde{#1}}

\DeclareMathOperator{\initial}{initial}
\DeclareMathOperator{\can}{can}
\DeclareMathOperator{\rot}{rot}
\DeclareMathOperator{\RF}{RF}
\DeclareMathOperator{\RD}{RD}
\DeclareMathOperator{\Met}{Met}
\DeclareMathOperator{\confexp}{confexp}
\DeclareMathOperator{\domain}{domain}
\DeclareMathOperator{\met}{Met}
\DeclareMathOperator{\Int}{Int}
\DeclareMathOperator{\Ric}{Ric}
\DeclareMathOperator{\Rm}{Rm}
\DeclareMathOperator{\tr}{tr}
\DeclareMathOperator{\id}{id}

\DeclareMathOperator{\Isom}{Isom}
\DeclareMathOperator{\diam}{diam}

\DeclareMathOperator{\eucl}{eucl}

\DeclareMathOperator{\supp}{supp}

\DeclareMathOperator{\Bry}{Bry}

\DeclareMathOperator{\spann}{span}

\DeclareMathOperator{\proj}{proj}
\DeclareMathOperator{\sh}{sh}
\DeclareMathOperator{\ch}{ch}

\DeclareMathOperator{\ext}{ext}

\DeclareMathOperator{\length}{length}

\renewcommand{\sinh}{\sh}
\renewcommand{\cosh}{\ch}

\newcommand{\EMPTY}[1]{}

\numberwithin{equation}{section}
\theoremstyle{plain}
\newtheorem{Claim}[equation]{Claim}
\newtheorem{theorem}[equation]{Theorem}

\newtheorem{proposition}[equation]{Proposition}
\newtheorem{lemma}[equation]{Lemma}

\newtheorem{corollary}[equation]{Corollary}

\theoremstyle{remark}
\newtheorem{remark}[equation]{Remark}

\theoremstyle{definition}
\newtheorem{definition}[equation]{Definition}

\newtheorem{question}[equation]{Question}

\newcommand{\acts}{\curvearrowright}
\newcommand{\al}{\alpha}
\newcommand{\be}{\beta}
\newcommand{\ben}{\begin{enumerate}}
\newcommand{\een}{\end{enumerate}}
\newcommand{\bit}{\begin{itemize}}
\newcommand{\conf}{\operatorname{Conf}}
\newcommand{\eit}{\end{itemize}}
\newcommand{\D}{\partial}
\newcommand{\de}{\delta}

\newcommand{\Diff}{\operatorname{Diff}}

\newcommand{\injrad}{\operatorname{injrad}}
\newcommand{\isom}{\operatorname{Isom}}

\newcommand{\loc}{\operatorname{loc}}

\newcommand{\ra}{\rightarrow}

\newcommand{\R}{\mathbb{R}}
\newcommand{\si}{\sigma}

\author{Richard H. Bamler}
\address{Department of Mathematics, UC Berkeley, CA 94720, USA}
\email{rbamler@berkeley.edu}
\author{Bruce Kleiner}
\address{Courant Institute of Mathematical Sciences, New York University,  251 M
ercer St., New York, NY 10012}
\email{bkleiner@cims.nyu.edu}

\date{\today}
\thanks{The first author was supported by NSF grants DMS-1611906 and DMS-1906500.  The second author was supported by NSF grant DMS-1711556 and a Simons Collaboration grant}

\begin{document}

 \begin{abstract}
We show that the space of metrics of positive scalar curvature on any 3-manifold is either empty or contractible.
Second, we show that the diffeomorphism group of every 3-dimensional spherical space form deformation retracts to its isometry group.
This proves the Generalized Smale Conjecture.
Our argument is independent of Hatcher's theorem in the $S^3$ case and in particular it gives a new proof of the $S^3$ case.
 \end{abstract}
\title{Ricci flow and contractibility of spaces of metrics}

 \maketitle
 \tableofcontents

\section{Introduction}
\label{sec_introduction}
This paper is a continuation of earlier work on Ricci flow through singularities \cite{Kleiner:2014le,bamler_kleiner_uniqueness_stability,gsc}.  
While the focus in \cite{Kleiner:2014le,bamler_kleiner_uniqueness_stability} was on analytical properties of singular Ricci flows, such as existence and uniqueness, the aim of \cite{gsc} and the present paper is to apply the flow to topological and geometric problems. 
One of the main contributions in this paper is a new geometric/topological method for using families of flows with singularities to produce families of nonsingular deformations.  
This method is new and may be applicable to other settings that are unrelated to Ricci flow.

We now present our main results.
Let $M$ be a connected, orientable, closed, smooth 3-manifold.
We denote by $\met(M)$ and $\Diff(M)$ the space of Riemannian metrics on $M$ and the diffeomorphism group of $M$, respectively; we equip both spaces with the $C^\infty$-topology, and let $\met_{PSC}(M)\subset\met(M)$ denote the subspace of metrics with positive scalar curvature.

Our first result settles a well-known conjecture about the topology of the space of metrics with positive scalar curvature:

\begin{theorem}
\label{thm_psc_contractible}
$\met_{PSC} (M)$ is either empty or contractible.
\end{theorem}

For our second main result consider the subset $\met_{CC}(M)\subset\met(M)$ of metrics that are locally isometric to either the round sphere $S^3$ or the round cylinder $S^2 \times \IR$.
We will show:

\begin{theorem}
\label{thm_cc_contractible}
$\met_{CC} (M)$ is either empty or contractible.
\end{theorem}

By a well-known argument, Theorem~\ref{thm_cc_contractible} implies the following conjecture about the structure of diffeomorphism groups:

\begin{theorem}[Generalized Smale Conjecture] \label{thm_gen_smal}
If $(M,g)$ is an isometric quotient of the round sphere, then the inclusion map $\Isom (M,g) \hookrightarrow \Diff (M)$ is a homotopy equivalence.
\end{theorem}

We now provide a brief historical overview, before discussing other results.

Theorem~\ref{thm_psc_contractible} was inspired by the work of Marques \cite{marques}, who  showed that $\met_{PSC}(M)$ is path connected.  
The analogous statement in dimension $2$ --- the contractibility of $\met_{PSC}(S^2)$ --- can be proven using the uniformization theorem, or by Ricci flow.   Starting with the famous paper of Hitchin \cite{hitchin}, there has been a long history of results based on index theory, which show that $\met_{PSC}(M)$ 
has nontrivial topology when $M$ is high dimensional; we refer the reader to the survey \cite{rosenberg_progress_report} for details.  Theorem~\ref{thm_psc_contractible} provides the first examples of manifolds of dimension $\geq 3$ for which the homotopy type of $\met_{PSC}(M)$ is completely understood.

Regarding Theorem~\ref{thm_gen_smal}, 
Smale made his original 1961 conjecture for the case $M=S^3$; this is the first step toward the larger project of understanding $\Diff(M)$ for other $3$-manifolds, which was already underway in the 70s \cite{hatcher_haken,ivanov_haken}.  
We recommend \cite[Section 1]{rubinstein_et_al} for a nice discussion of the history and other background on diffeomorphism groups. 
Theorem~\ref{thm_gen_smal}  completes the proof of the Generalized Smale Conjecture after prior work by many people.  Cerf proved that the inclusion $\isom(S^3,g)\ra \Diff(S^3)$ induces a bijection on path components \cite{cerf1,cerf2,cerf3,cerf4}, and the full conjecture for $S^3$ was proven by Hatcher \cite{hatcher_smale_conjecture}.  Hatcher used a blend of combinatorial and smooth techniques to show that the space of smoothly embedded $2$-spheres in $\R^3$ is contractible.  This  is equivalent to the assertion that $O(4)\simeq \isom(S^3,g) \ra \Diff(S^3)$ is a homotopy equivalence when $g$ has sectional curvature $1$ (see the appendix in \cite{hatcher_smale_conjecture}).  Other spherical space forms were studied starting in the late 1970s.  Through the work of a number of authors it was shown that the inclusion $\isom(M)\ra \Diff(M)$ induces a bijection on path components for any spherical space form $M$ \cite{asano,rubinstein_klein_bottles,cappell_shaneson,bonahon,rubinstein_birman,boileau_otal}.  The full conjecture was proven for  spherical space forms containing geometrically incompressible one-sided Klein bottles (prism and quaternionic manifolds), and lens spaces other than $\IR P^3$ \cite{ivanov_1,ivanov_2,rubinstein_et_al}. The methods used in these proofs were of topological nature.
In our previous paper \cite{gsc}, we used Ricci flow methods to give a unified proof of the conjecture for all spherical space forms, except for $\IR P^3$. This established the conjecture for the three remaining families of spherical space forms: tetrahedral, octahedral, and icosahedral manifolds.  Although the techniques in \cite{ivanov_1,ivanov_2,rubinstein_et_al} and \cite{gsc} were very different, these results all relied on Hatcher's resolution of the $S^3$ case.
    
It has been a longstanding question whether it is possible to use techniques from geometric analysis to give a new proof of Hatcher's theorem for $S^3$.   There are several well-known variational approaches to studying the topology of the space of $2$-spheres in $\R^3$ (or $S^3$); however, they all break down due to the absence of a Palais-Smale condition, because there are too many critical points, or because the natural gradient flow does not respect embeddedness.  Analogous issues plague other strategies based more directly on diffeomorphisms.  The argument for Theorem~\ref{thm_cc_contractible} is independent of Hatcher's work and applies uniformly to all spherical space forms; in particular it gives a new proof of the Smale Conjecture based on Ricci flow.

We believe that the methods used in this paper may be readily adapted to other situations where a geometric flow produces neck-pinch type singularities, for instance $2$-convex mean curvature flow and (conjecturally) mean curvature flow of $2$-spheres in $\R^3$  \cite{haslhofer_et_al,white_icm,brendle_mcf_genus_zero} or to study the space of metrics with positive isotropic curvature in higher dimensions \cite{Hamilton-PIC, Brendle-2019}.

\bigskip
We now present some further results.

Applying Theorem~\ref{thm_cc_contractible} in the case of manifolds covered by $S^2\times\R$, one obtains the following corollaries:

\begin{theorem} \label{thm_S2S1_diff}
$\Diff (S^2 \times S^1)$ is homotopy equivalent to $O(2) \times O(3) \times \Omega O(3)$, where $\Omega O(3)$ denotes the loop space of $O(3)$.
\end{theorem}

\begin{theorem} \label{thm_RP3RP3}
$\Diff (\IR P^3 \# \IR P^3)$ is homotopy equivalent to $O(1)\times O(3)$.  
\end{theorem}
Theorem~\ref{thm_S2S1_diff} is due to Hatcher \cite{hatcher_s2xs1}.     While Theorem~\ref{thm_RP3RP3} can be deduced directly from Theorem~\ref{thm_cc_contractible}, there is also an alternate approach based on a result of \cite{hatcher_s2xs1}, which reduces it to Theorem~\ref{thm_gen_smal} in the $\IR P^3$ case.

Theorems~\ref{thm_gen_smal}, \ref{thm_S2S1_diff}, and \ref{thm_RP3RP3} describe the structure of $\Diff(M)$ when $M$ has a geometric structure modelled on $S^3$ or $S^2\times S^1$.  In \cite{bamler_kleiner_gsc_ii} we will use Ricci flow methods to study $\Diff(M)$ when $M$ is modelled the Thurston geometry Nil.  Combined with earlier work, this completes the classification of $\Diff(M)$ when $M$ is prime \cite{hatcher_haken,ivanov_haken,ivanov_1,ivanov_2,gabai_smale_conjecture_hyperbolic,rubinstein_et_al,mccullough_soma}.

\bigskip

Theorems~\ref{thm_psc_contractible} and \ref{thm_cc_contractible} are deduced from a single result, which involves fiberwise Riemannian metrics on fiber bundles.  A special case is:

\begin{theorem}[See Theorem~\ref{Thm_main_conf_flat}] \label{Thm_main_general_case} 
Let $M$ be a connected sum of spherical space forms and copies of $S^2 \times S^1$, and $K$ be the geometric realization of a simplicial complex.    Consider a fiber bundle $\pi : E \to K$ with fibers  homeomorphic to $M$ and  structure group $\Diff (M)$.

Suppose that $g^s$ is a Riemannian metric on the fiber $\pi^{-1}(s)$ for every $s\in K$, such that the family $(g^s)_{s \in K}$ varies continuously with respect the fiber bundle structure.   Then there is a  family of  Riemannian metrics $(h^s_t)_{s \in K, t \in [0,1]}$ such that:
\begin{enumerate}[label=(\alph*)]
\item $h^s_t$ is a Riemannian metric on $\pi^{-1}(s)$  for every $s\in K$, $t \in [0,1]$, and the family $(h^s_t)_{s\in K, t\in[0,1]}$ varies continuously with respect to the fiber bundle structure. 
\item $h^s_0 = g^s$ for all $s \in K$.
\item $h^s_1$ is conformally flat and has positive scalar curvature for all $s \in K$.
\item If for some $s\in K$ the manifold $(\pi^{-1} (s), g^s)$ has positive scalar curvature, then so does $(M^s, h^s_t)$ for all $t\in [0,1]$.
\end{enumerate}
\end{theorem}

As a corollary we have:

\begin{corollary}
\label{cor_deform_to_psc_cf}
If $\pi:E\ra K$ is as in Theorem~\ref{Thm_main_general_case}, then there is a continuously varying family of fiberwise Riemannian metrics $(h^s)_{s \in K}$ such that $(\pi^{-1} (s), h^s)$ is  conformally flat and has positive scalar curvature.
\end{corollary}

  Let $\met_{CF}(M)\subset \met(M)$ denote the subspace of conformally flat metrics.  Theorem~\ref{Thm_main_general_case} and Corollary~\ref{cor_deform_to_psc_cf} are indications that the space $\met_{PSC}(M)\cap\met_{CF}(M)$ has simple topology, and suggest the following question in conformal geometry:

\begin{question}
\label{conj_psc_cf_contractible}
Is $\met_{PSC}(M)\cap\met_{CF}(M)$ always empty or contractible?  Equivalently, is the space of conformally flat metrics with positive Yamabe constant always empty or contractible? 
\end{question}

The fundamental work of Schoen-Yau \cite{schoen_yau_conformally_flat} on the geometry of individual metrics  $g\in\met_{PSC}(M)\cap \met_{CF}(M)$ should be helpful in addressing this question.   To our knowledge, the current understanding of the corresponding Teichmuller space of conformally flat structures is rather limited.   The connection between Question~\ref{conj_psc_cf_contractible} and the preceding results is that the contractibility of $\met_{PSC}(M)\cap\met_{CF}(M)$ logically implies  Theorem~\ref{Thm_main_general_case} and Corollary~\ref{cor_deform_to_psc_cf}.

\subsection*{Discussion of the proof}
To give the reader an indication of some of the issues that must be addressed when attempting to apply Ricci flow to a deformation problem, we first recall the outline
of Marques' proof \cite{marques} that $\met_{PSC}(M)/\Diff(M)$ is path-connected, in the special case when $M$ is a spherical space form (see also \cite{haslhofer_et_al}, which was inspired by \cite{marques}).  Starting with a metric $h\in \met_{PSC}(M)$, one applies Perelman's Ricci flow with surgery to obtain a finite sequence $\{(M_j,(g_j(t))_{t\in [t_{j-1},t_j]})\}_{1\leq j\leq N}$ of ordinary Ricci flows where $g_1(0)=h$.
Since $h$ has positive scalar curvature, so does each of the Ricci flows $g_j(t)$.  Marques shows by backward induction on $j$ that the Riemannian metric $g_j(t_{j-1})$ on $M_j$ can be deformed through metrics of positive scalar curvature to a metric of constant sectional curvature.  In the induction step he carefully analyzes Perelman's surgery process, and shows that one may pass from $(M_{j+1},g_{j+1}(t_j))$ to $(M_j,g_j(t_j))$ by means of a geometric connected sum operation, which is compatible with deformations through metrics of positive scalar curvature.  In other words, the Ricci flow with surgery can be seen as a sequence of  of continuous curves in $\Met_{PSC} (M_j)$, whose endpoints are related by a surgery process.
Marques' work was to join these endpoints in order to produce a single continuous curve.

Let us now consider a family of metrics $(h_s)_{s\in K}$ depending continuously on a parameter $s$.  
If one attempted to use the above strategy for such a family, then one would immediately run into the problem that the resulting Ricci flows with surgery starting from the metrics $h_s$ may not be unique or may not depend continuously on the parameter $s$.
Moreover, the locations and times of the surgery operations may change as $s$ varies --- possibly in a discontinuous fashion.
As we vary $s$, the order in which these operations are performed may  change and some surgery operations may even appear or disappear.
In particular, this means that the underlying topology of the flow at some (or most) positive times may not be constant in $s$.
So in summary, every single metric $h_s$ defines a Ricci flow with surgery, which can be turned into a continuous metric deformation.
However, there is little hope of producing a useful topological object based on the collection of all such flows for different $s$.
In addition, since our second goal is to study the structure of diffeomorphism groups, we are faced with the complication that the argument from the previous paragraph only works modulo the diffeomorphism group.

We address these issues using a number of new techniques.
First, we employ the singular Ricci flow (or ``Ricci flow through singularities'') from \cite{Kleiner:2014le} in lieu of the Ricci flow with surgery.
In \cite{bamler_kleiner_uniqueness_stability} we showed that this flow is canonical.
Based on a stability result from the same paper, we show that any continuous family of metrics $(h_s)_{s \in K}$ can be evolved into a \emph{continuous} family of singular Ricci flows.
Here the word ``continuous'' has to be defined carefully, since the flows are not embedded in a larger space.

Our use of singular Ricci flows ameliorates some of the issues raised above, however, the underlying problem still remains.
More specifically, our notion of continuous dependence of singular Ricci flows still allows the possibility that singularities, which are the analogues of the surgery operations from before, may vary in space and in time and may appear or disappear as we vary the parameter $s$.
While these phenomena are now slightly more controlled due to our continuity property, we now have to deal with  singularities that may occur on some complicated set, possibly of fractal dimension.

One of the main conceptual novelties in our proof is a new topological notion called a \emph{partial homotopy}, which can be viewed as a hybrid between a continuous family of singular Ricci flows and a homotopy in the space of metrics.
On the one hand, this notion captures the phenomenon of non-constant topology as the parameter $s$ varies by neglecting the singular part, whose topological structure may be ``messy''.
On the other hand, if a partial homotopy is ``complete'' in a certain sense, then it constitutes a classical homotopy in a space of metrics.
The notion of a partial homotopy is not inherently restricted to to $3d$ Ricci flow. 
It may therefore also have applications in higher dimensions or to other geometric flows.

A large part of our paper will be devoted to the development of a theory that allows us to construct and modify partial homotopies through certain modification moves.  We then combine this with the theory of continuous families of singular Ricci flows. 
Roughly speaking, we will use a continuous family of singular Ricci flows as a blueprint to carry out the modification moves of a partial homotopy, with the goal of improving it towards a complete one.

In order to relate partial homotopies to continuous families of singular Ricci flows, we have to do some pre-processing.
More specifically, we will study the continuous dependence of the singular set of singular Ricci flows and equip a neighborhood with a continuous family of ``$\RR$-structures''.
These $\RR$-structures limit the dependence of the singular set on the parameter and are used to relate a partial homotopy to a family of singular Ricci flows.
In our pre-processing step we also produce a ``rounded'' family of metrics, which are part of these $\RR$-structures.
Taking a broader perspective, our partial homotopy machinery will ultimately enable us to ``weave'' these metrics together to produce the desired homotopy in the space of metrics.

Our theory of partial homotopies brings together and generalizes a number of technical ingredients that have existed in the fields of topology and geometric analysis.
Most notable among these are: a surgery technique generalizing connected sums of conformally flat structures to arbitrary metrics and a notion of positivity of the Yamabe constant relative boundary.

\subsection*{Organization of the paper}
In Section~\ref{sec_preliminaries}, we briefly recapitulate the most important definitions and results related to singular Ricci flows. 
In Section~\ref{sec_families_srfs} we formalize the idea of a continuous family of singular Ricci flows $(\M^s)_{s\in X}$, where the parameter lies in some topological space $X$;  using results from \cite{Kleiner:2014le,bamler_kleiner_uniqueness_stability}, we prove the existence of a unique continuous family of singular Ricci flows with a prescribed family of initial conditions.   
In Section~\ref{sec_rounding_process},  we implement a ``rounding'' procedure to construct $\RR$-structures, which characterize the geometry and topology of the singular part and provides a family of metrics, whose high curvature part is precisely symmetric.   
In Sections~\ref{sec_partial_homotopy} and \ref{sec_partial_homotopy},  we set up and develop the theory of partial homotopies.
In Section~\ref{sec_deforming_families_metrics}, we apply this theory to our families of $\RR$-structures from Section~\ref{sec_rounding_process}.
More specifically, given a continuous family $(g_{s,0})_{s\in K}$ of Riemannian metrics parametrized by a simplicial complex $K$, we will construct a continuous metric deformation $(g_{s,t})_{s\in K,t\in[0,1]}$, where $g_{s,1}$ is conformally flat for all $s$.
Based on this result, we prove the main theorems of this paper in Section~\ref{sec_proofs_main_theorems}.

\subsection*{Acknowledgements}
The first named author would like to thank Boris Botvinnik for bringing up the question on the contractibility of spaces of positive scalar curvature and for many inspiring discussions.

\section{Conventions}
Unless otherwise noted, all manifolds are assumed to be smooth, orientable and 3-dimensional.
Whenever we refer to a continuous family of functions or maps, we will always work in the smooth topology $C^\infty$, or $C^\infty_{\loc}$ in the non-compact case.
So continuity means that all partial derivatives vary continuously in the $C^0$-sense.
The same applies to the terminology of ``transverse continuity'' (compare with the discussion after Definition~\ref{def_continuity_maps_between_families}, Definitions~\ref{def_continuity_smooth_objects}, \ref{Def_transverse_O3_action}), which will always mean ``transverse continuity in the smooth topology''.

If $(g_t)_{t \in I}$ is a family of Riemannian metrics on a manifold $M$, then we denote by $B(x,t,r)$ the distance ball with respect to the metric $g_t$.

If $(M,g)$ is a Riemannian manifold and $X \subset M$ is a measurable subset, then we denote by $V(X, g)$ its volume with respect to the Riemannian measure $d \mu_g$.

We will denote by $B^n (r), D^n (r) \subset \IR^n$ the open and closed balls of radius $r$ around the origin.
Moreover, we will set $B^n := B^n (1)$, $D^n := D^n (1)$ and denote by $A^n (r_1, r_2) := B^n (r_2) \setminus D^n (r_1)$ the annulus of radii $r_1, r_2$.

We will say that a Riemannian metric $g$ is \emph{conformally flat} if $e^{2\phi} g$ is flat for some smooth function $\phi$.
In dimension 3 this condition is equivalent to the Cotton-York condition, see Remark~\ref{rmk_Cotton_York} and the discussion in Subsection~\ref{subsec_conf_exp}.
 
\bigskip\bigskip

\section{Preliminaries} \label{sec_preliminaries}
In this section we collect some preliminary material, including background required for Sections~\ref{sec_families_srfs} and \ref{sec_rounding_process}.
\subsection{$\kappa$-solutions}
The singularity formation in 3-dimensional Ricci flows is usually understood via singularity models called $\kappa$-solutions (see \cite[Sec. 11]{Perelman1}).
The definition of a $\kappa$-solution consists of a list of properties that are known to be true for $3$-dimensional singularity models.

\begin{definition}[$\kappa$-solution] \label{def_kappa_solution}
An ancient Ricci flow $(M, (g_t)_{t \leq 0} )$ on a $3$-dimensional manifold $M$ is called a \textbf{(3-dimensional) $\kappa$-solution}, for $\kappa > 0$, if the following holds:
\begin{enumerate}[label=(\arabic*)]
\item $(M, g_t)$ is complete for all $t \in (- \infty, 0]$,
\item $|{\Rm}|$ is bounded on $M \times I$ for all compact $I \subset ( - \infty, 0]$,
\item $\sec_{g_t} \geq 0$ on $M$ for all $t \in (- \infty, 0]$,
\item $R > 0$ on $M \times (- \infty, 0]$,
\item $(M, g_t)$ is $\kappa$-noncollapsed at all scales for all $t \in (- \infty, 0]$

(This means that for any $(x,t) \in M \times (- \infty, 0]$ and any $r > 0$ if $|{\Rm}| \leq r^{-2}$ on the time-$t$ ball $B(x,t,r)$, then we have $|B(x,t,r)| \geq \kappa r^n$ for its volume.)
\end{enumerate}
\end{definition}

Important examples of $\kappa$-solutions are the \textbf{round shrinking cylinder}
\[ \big( S^2 \times \IR, (g^{S^2 \times \IR}_t := (1-2t) g_{S^2} + g_{\IR} )_{t \leq 0} \big), \]
the \textbf{round shrinking sphere}
\[ \big( S^3, (g^{S^3}_t := (1- 4t) g_{S^3})_{t \leq 0} \big) \]
and the \textbf{Bryant soliton} \cite{Bryant2005}
\[ \big( M_{\Bry}, (g_{\Bry, t})_{t \leq 0} \big). \]
We recall that $M_{\Bry} \approx \IR^3$ and the flow $(g_{\Bry, t})_{t \leq 0}$ is invariant under the standard $O(3)$-action.
We will denote the fixed point of this action by $x_{\Bry} \in M_{\Bry}$.
By parabolic rescaling we may assume in this paper that $R(x_{\Bry}) = 1$.
The flow $(g_{\Bry, t})_{t \leq 0}$ is moreover a steady gradient soliton.
For the purpose of this paper, it is helpful to remember that the pointed $(M_{\Bry}, g_{\Bry,t}, x_{\Bry})$ are isometric for all $t \leq 0$.
We will set $g_{\Bry} := g_{\Bry, 0}$.
For more details see the discussion in \cite[Appendix B]{bamler_kleiner_uniqueness_stability}.

The following theorem states that these examples provide an almost complete list of $\kappa$-solutions.

\begin{theorem}[Classification of $\kappa$-solutions] \label{Thm_kappa_sol_classification}
There is a constant $\kappa_0 > 0$ such for any $\kappa$-solution $(M, (g_t)_{t \leq 0} )$ one of the following is true:
\begin{enumerate}[label=(\alph*)]
\item \label{ass_kappa_sol_classification_a} $(M, (g_t)_{t \leq 0} )$ is homothetic to the round shrinking cylinder
or its $\IZ_2$-quotient.
\item \label{ass_kappa_sol_classification_b} $(M, (g_t)_{t \leq 0} )$ is homothetic to an isometric quotient of the round shrinking sphere.
\item \label{ass_kappa_sol_classification_c} $(M, (g_t)_{t \leq 0]} )$ is homothetic to the Bryant soliton.
\item \label{ass_kappa_sol_classification_d} $M \approx S^3$ or $\IR P^3$ and $(M, (g_t)_{t \leq 0} )$ is rotationally symmetric, i.e. the flow is invariant under the standard $O(3)$-action whose principal orbits are 2-spheres.
Moreover, for every $x \in M$ the limit of $(M, R(x,t) g_t, x)$ as $t \searrow - \infty$ exists and is isometric to a pointed round cylinder or Bryant soliton. 
\end{enumerate}
Moreover, in cases (a)--(c) the solution is even a $\kappa_0$-solution.
\end{theorem}

\begin{proof}
See \cite{Brendle2018,  BamlerKleiner2019, Brendle2019}.
\end{proof}

We also recall the following compactness result for $\kappa$-solutions, which is independent of Theorem~\ref{Thm_kappa_sol_classification}.

\begin{theorem} \label{Thm_kappa_compactness_theory}
If $(M^i, (g^i_t)_{t \leq 0}, x^i)$ is a sequence of pointed $\kappa_i$-solutions for some $\kappa_i$.
Suppose that $\lim_{i \to \infty} R(x^i) > 0$ exists.
Then, after passing to a subsequence, either all flows $(M^i, (g^i_t)_{t \leq 0})$ are homothetic to quotients of the round shrinking sphere or we have convergence to a pointed $\kappa_\infty$-solution $(M^\infty, (g^\infty_t)_{t \leq 0}, x^\infty)$ in the sense of Hamilton \cite{Hamilton1995}.
\end{theorem}

\begin{proof}
See \cite[Sec. 11]{Perelman1}.
\end{proof}

The following result will be used in Section~\ref{sec_rounding_process}.

\begin{lemma} \label{lem_kappa_identity_1}
For any $A < \infty$ there is a constant $D = D(A) < \infty$ such that the following holds.
If $(M,(g_t)_{\leq 0})$ is a compact, simply-connected $\kappa$-solution and
\[  \int_M R(\cdot, 0) d\mu_{g_0}  < A V^{1/3} ( M, g_0), \]
then $\diam (M, g_0) < D R^{-1/2} (x,0)$ for all $x \in M$.
\end{lemma}

\begin{proof}
Assume that the lemma was wrong for some fixed $A$.
Then we can find a sequence of counterexamples $(M^i, (g^i_t)_{t \leq 0}, x^i)$ with 
\begin{equation} \label{eq_diam_to_infty}
\diam (M^i, g^i_0) R^{1/2} (x^i, 0) \to \infty.
\end{equation}
It follows that $(M^i, (g^i_t)_{t \leq 0})$ is not homothetic to a round shrinking sphere for large $i$.

\begin{Claim}
There are constants $C, I < \infty$ such that for all $i \geq I$ and $y \in M^i$ we have
\begin{equation} \label{eq_radius_int_R}
 R^{-1/2} (y,0) \leq  \int_{B(y,0,C R^{-1/2} (y,0))} R(\cdot, 0) d\mu_{g^i_0} .
\end{equation}
\end{Claim}

\begin{proof}
Assume that, after passing to a subsequence, (\ref{eq_radius_int_R}) was violated for some $y^i \in M^i$ and $C^i \to \infty$.
Since the lemma and (\ref{eq_radius_int_R}) are invariant under parabolic rescaling, we can assume without loss of generality that $R(y^i,0) = 1$.
Apply Theorem~\ref{Thm_kappa_compactness_theory} to extract a limit $(M^\infty, (g^\infty_t)_{t \leq 0}, y^\infty)$ with $\int_{M^\infty} R(\cdot, 0) d\mu_{g^\infty_0} \leq 1$, which would have to be compact, in contradiction to our assumption (\ref{eq_diam_to_infty}).
\end{proof}

Fix some $i \geq I$ for a moment.
By Vitali's covering theorem we can find points $y_1, \ldots, y_N \in M^i$ such that the balls $B(y_j, 0, C R^{-1/2} (y_j, 0))$ are pairwise disjoint and the balls $B(y_j, 0, 3C R^{-1/2} (y_j, 0))$ cover $M^i$.
It follows that
\[ \diam (M^i, g^i_0) \leq \sum_{j=1}^N 6C R^{-1/2} (y_j, 0) \leq 6C \int_{M^i} R(\cdot, 0) d\mu_{g^i} <6C A V^{1/3}( M^i, g^i_0). \]
By volume comparison this implies that there is a uniform constant $c > 0$ such that $V ( B(z,0,r) , g_0^i) \geq c r^3$ for all $r < \diam (M^i, g^i)$.
After parabolic rescaling to normalize $R(x^i, 0)$, the pointed solutions $(M^i, (g^i_t)_{t \leq 0}, x^i)$ subsequentially converge to a non-compact $\kappa$-solution satisfying the same volume bound.
This, however, contradicts the fact that non-compact $\kappa$-solutions have vanishing asymptotic volume ratio (see either \cite[Sec. 11]{Perelman1} or Theorem~\ref{Thm_kappa_sol_classification}).
\end{proof}

Lastly, we recall:

\begin{lemma} \label{Lem_Bry_R_Hessian_positive}
On $(M_{\Bry}, g_{\Bry})$ the scalar curvature $R$ attains a unique global maximum at $x_{\Bry}$ and the Hessian of $R$ is negative definite at $x_{\Bry}$.
\end{lemma}

\begin{proof}
This follows from the soliton equations $\Ric + \nabla^2 f = 0$ and $R + |\nabla f|^2 = R(x_{\Bry})$ (see \cite[Appendix B]{bamler_kleiner_uniqueness_stability} for more details.).
If the $\nabla^2 R$ at $x_{\Bry}$ was not strictly negative definite, then it would vanish due to symmetry.
This would imply 
\[ 0 = \nabla^2 |\nabla f|^2 = 2 |\nabla^2 f|^2 + 2 \nabla^3 f \cdot \nabla f = 2 |\nabla^2 f|^2 = 2 |{\Ric}|^2, \]
in contradiction to the positivity of the scalar curvature on $(M_{\Bry}, g_{\Bry})$.
\end{proof}

\subsection{Singular Ricci flows --- Definition}
In the following we recall terminology related to singular Ricci flows.
In order to keep this subsection concise, our discussion has been simplified to fit the needs of this paper.
For more details, see \cite[Sec~5]{bamler_kleiner_uniqueness_stability}.

Singular Ricci flows were introduced by Lott and the second author in \cite{Kleiner:2014le}.
In the same paper the existence of a singular Ricci flow starting from compact initial condition was established.
Subsequently, uniqueness was shown by the authors in \cite{bamler_kleiner_uniqueness_stability}.
The definition of a singular Ricci flow provided in this paper differs slightly from the original definition in \cite{Kleiner:2014le}.
It is a priori more general, however, the uniqueness result in \cite{bamler_kleiner_uniqueness_stability} implies that both definitions are equivalent.

We first introduce a broader class of Ricci flow spacetimes.
A singular Ricci flow will be defined as a Ricci flow spacetime that satisfies certain conditions.

\begin{definition}[Ricci flow spacetimes] \label{def_RF_spacetime}
A {\bf Ricci flow spacetime}  is a tuple $(\M, \lb \mathfrak{t}, \lb \partial_{\mathfrak{t}}, \lb g)$ with the following properties:
\begin{enumerate}[label=(\arabic*)]
\item $\M$ is a smooth $4$-manifold with (smooth) boundary $\partial \M$.
\item $\mathfrak{t} : \M \to [0, \infty)$ is a smooth function without critical points (called {\bf time function}).
For any $t \geq 0$ we denote by $\M_t := \mathfrak{t}^{-1} (t) \subset \M$ the {\bf time-$t$-slice} of $\M$.
\item $\M_0 = \mathfrak{t}^{-1} (0) = \partial \M$, i.e. the initial time-slice is equal to the boundary of $\M$.
\item $\partial_{\mathfrak{t}}$ is a smooth vector field (the {\bf time vector field}) on $\M$ that satisfies $\partial_{\mathfrak{t}} \mathfrak{t} \equiv 1$.
\item $g$ is a smooth inner product on the spatial subbundle $\ker (d \mathfrak{t} ) \subset T \M$.
For any $t \geq 0$ we denote by $g_t$ the restriction of $g$ to the time-$t$-slice $\M_t$ (note that $g_t$ is a Riemannian metric on $\M_t$).
\item $g$ satisfies the Ricci flow equation: $\mathcal{L}_{\partial_\mathfrak{t}} g = - 2 \Ric (g)$.
Here $\Ric (g)$ denotes the symmetric $(0,2)$-tensor on $\ker (d \mathfrak{t} )$ that restricts to the Ricci tensor of $(\M_t, g_t)$ for all $t \geq 0$.
\end{enumerate}
Curvature quantities on $\M$, such as the Riemannian curvature tensor $\Rm$, the Ricci curvature $\Ric$, or the scalar curvature $R$ will refer to the corresponding quantities with respect to the metric $g_t$ on each time-slice $\M_t$.
Tensorial quantities will be imbedded using the splitting $T\M = \ker (d\mathfrak{t} ) \oplus \langle \partial_{\mathfrak{t}} \rangle$.

When there is no chance of confusion, we will sometimes abbreviate the tuple $(\M, \mathfrak{t}, \partial_{\mathfrak{t}}, g)$ by $\M$.
\end{definition}

We emphasize that, while a Ricci flow spacetime may have singularities --- in fact the sole purpose of our definition is to understand flows with singularities --- such singularities are not directly captured by a Ricci flow spacetime, because ``singular points'' are not contained in the spacetime manifold $\M$.
Instead, the idea behind the definition of a Ricci flow spacetime is to understand a possibly singular flow by analyzing its asymptotic behavior on its regular part. 
This will always be sufficient for our applications.

Any (classical) Ricci flow of the form $(g_t)_{t \in [0,T)}$, $0 < T \leq \infty$, on a $3$-manifold $M$ can be converted into a Ricci flow spacetime  by setting $\M = M \times [0,T)$, letting $\mathfrak{t}$ be the projection to the second factor and letting $\partial_{\mathfrak{t}}$ correspond to the unit vector field on $[0,T)$.
Vice versa, if $(\M, \mathfrak{t}, \partial_{\mathfrak{t}}, g)$ is a Ricci flow spacetime with $\mathfrak{t}(\M) = [0, T)$ for some $0 < T \leq \infty$ and the property that every trajectory of $\partial_{\mathfrak{t}}$ is defined on the entire time-interval $[0,T)$, then $\M$ comes from such a classical Ricci flow.

We now generalize some basic geometric notions to Ricci flow spacetimes.

\begin{definition}[Length, distance and metric balls in Ricci flow spacetimes]
Let $(\M, \mathfrak{t}, \partial_{\mathfrak{t}}, g)$ be a Ricci flow spacetime.
For any two points $x, y \in \M_t$ in the same time-slice of $\M$ we denote by $d(x,y)$ or $d_t (x,y)$ the {\bf distance} between $x, y$ within $(\M_t, g_t)$.
The distance between points in different time-slices is not defined.

Similarly, we define the {\bf length} $\length (\gamma)$ or $\length_t (\gamma)$ of a path $\gamma : [0,1] \to \M_t$ whose image lies in a single time-slice to be the length of this path when viewed as a path inside the Riemannian manifold $(\M_t, g_t)$.

For any $x \in \M_t$ and $r \geq 0$ we denote by $B(x,r) \subset \M_t$ the {\bf $r$-ball} around $x$ with respect to the Riemannian metric $g_t$.
\end{definition}

Our next goal is to characterize the (microscopic) geometry of a Ricci flow spacetime near a singularity or at an almost singular point.
For this purpose, we will introduce a \textbf{(curvature) scale function} $\rho : \M \to (0, \infty]$ with the property that
\begin{equation} \label{eq_rho_equivalent_Rm}
 C^{-1} \rho^{-2} \leq |{\Rm}| \leq C \rho^{-2} 
\end{equation}
for some universal constant $C < \infty$.
We will write $\rho (p) = \infty$ if $\Rm_p = 0$ at some point $p$.
Note that $\rho$ has the dimension of length.
In Subsection~\ref{subsec_curvature_scale} we will make a specific choice for $\rho$, which will turn out to be suitable for our needs.
The notions introduced in the remainder of this subsection will, however, be independent of the precise choice of $\rho$ or the constant $C$.

We now define what we mean by completeness for Ricci flow spacetimes.
Intuitively, a Ricci flow spacetime is called complete if its time-slices can be completed by adding countably many ``singular points'' and if no component ``appears'' or ``disappears'' suddenly without the formation of a singularity.

\begin{definition}[Completeness of Ricci flow spacetimes] \label{def_completeness}
We say that a Ricci flow spacetime $(\M,\mathfrak{t}, \partial_{\mathfrak{t}}, g)$ is {\bf complete}  if the following holds:
Consider a path $\gamma : [0, s_0) \to \M$ such that $\inf_{s \in [0,s_0)}  \rho (\gamma(s)) > 0$ and such that:
\begin{enumerate}
\item The image $\gamma ([0,s_0))$ lies in a time-slice $\M_t$ and the time-$t$ length of $\gamma$ is finite or
\item $\gamma$ is a trajectory of $\partial_{\mathfrak{t}}$ or of $- \partial_{\mathfrak{t}}$.
\end{enumerate}
Then the limit $\lim_{s \nearrow s_0} \gamma (s)$ exists.
\end{definition}

Lastly, we need to characterize the asymptotic geometry of a Ricci flow spacetime near its singularities.
The idea is to impose the same asymptotic behavior near singular points in Ricci flow spacetimes as is encountered in the singularity formation of a classical (smooth) 3-dimensional Ricci flow.
This is done by comparing the geometry to the geometry of $\kappa$-solutions using the following concept of pointed closeness.

\begin{definition}[Geometric closeness] \label{def_geometric_closeness_time_slice}
We say that a pointed Riemannian manifold $(M, g, x)$ is \textbf{$\eps$-close} to another pointed Riemannian manifold $(\ov{M}, \ov{g}, \ov{x})$ \textbf{at scale $\lambda > 0$} if there is a diffeomorphism onto its image
\[ \psi : B^{\ov{M}} (\ov{x}, \eps^{-1} ) \longrightarrow M \]
such that $\psi (\ov{x}) = x$ and
\[ \big\Vert \lambda^{-2} \psi^* g - \ov{g} \big\Vert_{C^{[\eps^{-1}]}(B^{\ov{M}} (\ov{x}, \eps^{-1} ))} < \eps. \]
Here the $C^{[\eps^{-1}]}$-norm  of a tensor $h$ is defined to be the sum of the $C^0$-norms of the tensors $h$, $\nabla^{\ov{g}} h$, $\nabla^{\ov{g},2} h$, \ldots, $\nabla^{\ov{g}, [\eps^{-1}]} h$ with respect to the metric $\ov{g}$.
\end{definition}

We can now define the canonical neighborhood assumption.
The main statement of this assumption is that regions of small scale (i.e. high curvature) are geometrically close to regions of $\kappa$-solutions.

\begin{definition}[Canonical neighborhood assumption] \label{def_canonical_nbhd_asspt}
Let $(M, g)$ be a (possibly incomplete) Riemannian manifold.
We say that $(M, g)$ satisfies the {\bf $\eps$-canonical neighborhood assumption} at some point $x \in M$ if there is a $\kappa > 0$, a $\kappa$-solution $(\overline{M}, \linebreak[1] (\overline{g}_t)_{t \leq 0})$ and a point $\ov{x} \in \ov{M}$ such that $\rho (\overline{x}, 0) = 1$ and such that $(M, g, x)$ is $\eps$-close to $(\ov{M}, \ov{g}_0, \ov{x})$ at some (unspecified) scale $\lambda > 0$.

We say that $(M,g)$ {\bf satisfies the $\eps$-canonical neighborhood assumption below scale $r_0$,} for some $r_0 > 0$, if every point $x \in M$ with $\rho(x) < r_0$ satisfies the $\eps$-canonical neighborhood assumption.

We say that a Ricci flow spacetime $(\M, \mathfrak{t}, \partial_{\mathfrak{t}}, g)$ satisfies the \textbf{$\eps$-ca\-non\-i\-cal neighborhood assumption} at a point $x \in \M$ if the same is true at $x$ in the time-slice $(\M_{\mathfrak{t}(x)}, g_{\mathfrak{t}(x)})$.
Moreover, we say that $(\M, \mathfrak{t}, \partial_{\mathfrak{t}}, g)$ satisfies the {\bf $\eps$-canonical neighborhood assumption below scale $r_0$ (on $[0,T]$)} if the same is true for each time-slice $(\M_t, g_t)$ (if $t \in [0,T]$)
\end{definition}

Using this terminology, we can finally define a singular Ricci flow.

\begin{definition}[Singular Ricci flow] \label{Def_sing_RF}
A {\bf singular Ricci flow} is a Ricci flow spacetime $(\M,\mathfrak{t}, \partial_{\mathfrak{t}}, g)$ that has the following properties:
\begin{enumerate}
\item \label{Prop_sing_RF_1} $\M_0$ is compact.
\item \label{Prop_sing_RF_2} $(\M,\mathfrak{t}, \partial_{\mathfrak{t}}, g)$ is complete.
\item \label{Prop_sing_RF_3} For every $\eps > 0$ and $0 \leq T < \infty$ there is a constant $r_{\eps, T} > 0$ such that $(\M,\mathfrak{t}, \partial_{\mathfrak{t}}, g)$ satisfies the $\eps$-canonical neighborhood assumption below scale $r_{\eps, T}$ on $[0,T]$.
\end{enumerate}
We say that $(\M,\mathfrak{t}, \partial_{\mathfrak{t}}, g)$ is {\bf extinct at time $t \geq 0$} if $\M_t = \emptyset$.
\end{definition}

We remark that we have added Property~\ref{Prop_sing_RF_1} in Definition~\ref{Def_sing_RF} in order to make it equivalent to the definition in \cite{Kleiner:2014le}.
All flows encountered in this paper will have compact and non-singular initial data.
The property could potentially be dropped or replaced by requiring $(\M_0, g_0)$ to be complete and possibly have bounded curvature.
In addition, it can be shown that there is a universal constant $\eps_{\can} > 0$ such that Property~\ref{Prop_sing_RF_3} can be replaced by one of the following properties (due to the results in \cite{bamler_kleiner_uniqueness_stability, Bamler-finite-surg-0}):
\begin{enumerate}[label=(\arabic*$'$), start=3]
\item[($3'$)] For every $0 \leq T < \infty$ there is a constant $r_{T} > 0$ such that $(\M,\mathfrak{t}, \partial_{\mathfrak{t}}, g)$ satisfies the $\eps_{\can}$-canonical neighborhood assumption below scale $r_{T}$ on $[0,T]$.
\item[($3''$)] $(\M,\mathfrak{t}, \partial_{\mathfrak{t}}, g)$ satisfies the $\eps_{\can}$-canonical neighborhood assumption below some scale $r_0 > 0$.
\end{enumerate}
This aspect is, however, inessential for this paper.

\subsection{Singular Ricci flows --- Existence and Uniqueness} \label{subsec_sing_RF_exist_unique}
The following results establish the existence of a unique (or canonical) singular Ricci flow starting from any compact Riemannian 3-manifold $(M,g)$.

The existence result is from \cite[Theorem~1.1]{Kleiner:2014le}.

\begin{theorem}[Existence] \label{Thm_sing_RF_existence}
For every compact, orientable, Riemannian 3-man\-i\-fold $(M,g)$ there is a singular Ricci flow $(\M,\mathfrak{t}, \partial_{\mathfrak{t}}, g)$ with the property that $(\M_0, g_0)$ is isometric to $(M,g)$.
\end{theorem}

The uniqueness result is from \cite[Theorem~1.3]{bamler_kleiner_uniqueness_stability}.

\begin{theorem}[Uniqueness] \label{Thm_sing_RF_uniqueness}
Any singular Ricci flow $(\M,\mathfrak{t}, \partial_{\mathfrak{t}}, g)$ is uniquely determined by its initial time-slice $(\M_0, g_0)$ up to isometry in the following sense:
If $(\M,\mathfrak{t}, \partial_{\mathfrak{t}}, g)$ and $(\M',\mathfrak{t}', \partial'_{\mathfrak{t}}, g')$ are two singular Ricci flows and $\phi : (\M, g_0) \to (\M',g'_0)$ is an isometry, then there is a diffeomorphism $\td\phi : \M \to \M'$  that is an isometry of Ricci flow spacetimes:
\[ \td\phi |_{\M_0} = \phi, \quad  \t = \t' \circ \td\phi , \quad \partial_{\mathfrak{t}} = \td\phi^* \partial'_{\mathfrak{t}}, \quad g = \td\phi^* g'. \]
\end{theorem}

In this paper we will often identify the given initial condition $(M,g)$ with the initial time-slice $(\M_0,g_0)$ if there is no chance of confusion.
We will view $\M$ as the ``unique'' singular Ricci flow with initial time-slice $(\M_0, g_0) = (M,g)$.

\subsection{The curvature scale}
\label{subsec_curvature_scale}
We will now define a curvature scale function $\rho$ that satisfies (\ref{eq_rho_equivalent_Rm}).
This subsection can be skipped upon a first reading.
For most applications we could simply take $\rho := |{\Rm}|^{-1/2}$, however, it will turn out to be convenient at certain points in our proof to work with a slightly different definition.
More specifically, our main objective will be to ensure that $\rho = R^{-1/2}$ wherever $\eps$-canonical neighborhood assumption holds for a small enough $\eps$.

To achieve this, observe that there is a constant $c_0 > 0$ such that the following holds:
Whenever $\Rm$ is an algebraic curvature tensor with the property that its scalar curvature $R$ is positive and all its sectional curvatures are bounded from below by $-\frac1{10} R$, then $c_0 |{\Rm}| \leq  R$.
We will fix $c_0$ for the remainder of this paper. 

\begin{definition}[Curvature scale] \label{def_curvature_scale}
Let $(M, g)$ be a 3-dimensional Riemannian manifold and $x \in M$ a point.
We define the {\bf (curvature) scale} at $x$ to be 
\begin{equation} \label{eq_def_curvature_scale}
 \rho (x) = \min \big\{ R_+^{-1/2} (x), \big( c_0 |{\Rm}| (x) \big)^{-1/2}\big\}. 
\end{equation}
Here $R_+(x) := \max \{ R(x), 0 \}$ and we use the convention $0^{-1/2} = \infty$.

If $(\M, \mathfrak{t}, \partial_{\mathfrak{t}}, g)$ is a Ricci flow spacetime, then we define $\rho: \M \to \R$ such that it restricts to the corresponding scale functions on the time-slices. 
\end{definition}

The following lemma summarizes the important properties of the curvature scale.

\begin{lemma} \label{lem_rho_Rm_R}
There is a universal constant $C < \infty$ such that
\begin{equation} \label{eq_equivalence_bound_rho_Rm}
 C^{-1} \rho^{-2} (x) \leq |{\Rm}|(x) \leq C \rho^{-2} (x). 
\end{equation}
Moreover, there is a universal constant $\eps_0 > 0$ such that if $x$ satisfies the $\eps$-canonical neighborhood assumption for some $\eps \leq \eps_0$, then $R(x) = \rho^{-2} (x)$.
\end{lemma}

\begin{proof}
The bound (\ref{eq_equivalence_bound_rho_Rm}) is obvious.
For the second part of the lemma observe that for sufficiently small $\eps$ we have $R(x) > 0$ and $\sec \geq - \frac1{10} R(x)$ at $x$.
So $R_+^{-1/2}(x) \leq (c_0 |{\Rm}| (x) )^{-1/2}$.
\end{proof}

In all our future references to the $\eps$-canonical neighborhood assumption, we will assume that $\eps \leq \eps_0$, such that $R = \rho^{-2}$ is guaranteed.

Note that in \cite{bamler_kleiner_uniqueness_stability} a factor of $\frac13$ was used in front of $R_+$ in (\ref{eq_def_curvature_scale}).
We have omitted this factor for convenience, as it is inessential for the purpose of this paper.

\subsection{Singular Ricci flows --- further definitions and properties}
The following definitions and results, which are of a more technical nature,  will be used in this paper.

Let us first discuss the concept of parabolic rescaling for Ricci flow spacetimes.
For this purpose, recall that if $(g_t)_{t \in (t_1, t_2)}$ is a conventional Ricci flow and $a > 0$, then $(a^2 g_{a^{-2} t} )_{t \in (a^2 t_1, a^2 t_2)}$ satisfies the Ricci flow equation as well and we refer to this flow as the parabolically rescaled Ricci flow.
Similarly, if $(\M, \mathfrak{t}, \partial_{\mathfrak{t}}, g)$ is a Ricci flow spacetime, then so is  $(\M, a^2 \t, a^{-2}  \partial_{\mathfrak{t}},a^2 g)$, which we will refer to as the {\bf parabolically rescaled Ricci flow spacetime}.
If $(\M, \mathfrak{t}, \partial_{\mathfrak{t}}, g)$ is a singular Ricci flow, then so is $(\M, a^2 \t, a^{-2}  \partial_{\mathfrak{t}},a^2 g)$.
Moreover, if $(\M, \mathfrak{t}, \partial_{\mathfrak{t}}, g)$ (locally or globally) corresponds to a conventional Ricci flow $(g_t)_{t \in (t_1, t_2)}$, as discussed after Definition~\ref{def_RF_spacetime}, then both notions of parabolic rescaling are the same.

Next, we introduce some more useful terminology, which helps us characterize the local geometry of singular Ricci flows.
Let in the following $(\M, \mathfrak{t}, \partial_{\mathfrak{t}}, g)$, or simply $\M$, be a Ricci flow spacetime; for the purpose of this paper, we may also take $\M$ to be singular Ricci flow.

\begin{definition}[Points in Ricci flow spacetimes] \label{def_points_in_RF_spacetimes}
Let $x \in \M$ be a point and set $t := \mathfrak{t} (x)$.
Consider the maximal trajectory $\gamma_x : I \to \M$, $I \subset [0, \infty)$, of the time-vector field $\partial_{\mathfrak{t}}$ such that $\gamma_x (t) = x$.
Note that then $\mathfrak{t} (\gamma_x(t')) = t'$ for all $t' \in I$.
For any $t' \in I$ we say that $x$ \textbf{survives until time $t'$} and we write 
\[ x(t') := \gamma_x (t'). \]
Similarly, if $X \subset \M_t$ is a subset in the time-$t$ time-slice, then we say that $X$ \textbf{survives until time $t'$} if this is true for every $x \in X$ and we set $X(t') := \{ x(t') \;\; : \;\; x \in X \}$.
\end{definition}

\begin{definition}[Product domain]
\label{def_product_domain}
We call a subset $X \subset \M$ a \emph{product domain} if there is an interval $I \subset [0, \infty)$ such that for any $t \in I$ any point $x \in X$ survives until time $t$ and $x(t) \in X$.
\end{definition}

Note that a product domain $X$ can be identified with the product $(X \cap \M_{t_0}) \times I$ for an arbitrary $t_0 \in I$.
If $X \cap \M_{t_0}$ is sufficiently regular (e.g. open or a domain with smooth boundary in $\M_{t_0}$), then the metric $g$ induces a classical Ricci flow $(g_t)_{t \in I}$ on $X \cap \M_{t_0}$.
We will often use the metric $g$ and the Ricci flow $(g_t)_{t \in I}$ synonymously when our analysis is restricted to a product domain.

\begin{definition}[Parabolic neighborhood]
For any $y \in \M$ let $I_y \subset [0, \infty)$ be the set of all times until which $y$ survives.
Now consider a point $x \in \M$ and two numbers $a \geq 0$, $b \in \R$.
Set $t := \mathfrak{t} (x)$.
Then we define the \textbf{parabolic neighborhood} $P(x, a, b) \subset \M$ as follows:
\[ P(x,a,b) := \bigcup_{y \in B(x,a)} \bigcup_{t' \in [t, t+b] \cap I_y} y(t'). \]
If $b < 0$, then we replace $[t,t+b]$ by $[t+b, t]$.
We call $P(x,a,b)$ \textbf{unscathed} if $B(x,a)$ is relatively compact in $\M_t$ and if $I_y \supset [t, t+b]$ or $I_y\supset [t +b, t] \cap [0, \infty)$ for all $y \in B(x,a)$.
Lastly, for any $r > 0$ we introduce the simplified notation
\[ P(x,r) := P(x,r,-r^2) \]
for the \textbf{(backward) parabolic ball} with center $x$ and radius $r$.
\end{definition}

Note that if $P(x,a,b)$ is unscathed, then it is a product domain of the form $B(x,a)  \times I_y$ for any $y \in B(x,a)$.

Borrowing from Definition~\ref{def_geometric_closeness_time_slice}, we will introduce the notion of a $\delta$-neck. 
\begin{definition}[$\delta$-neck]
Let $(M,g)$ be a Riemannian manifold and $U \subset M$ an open subset.
We say that $U$ is a {\bf $\delta$-neck at scale $\lambda > 0$} if there is a diffeomorphism
\[ \psi : S^2 \times \big( {- \delta^{-1}, \delta^{-1} }\big) \longrightarrow U \]
such that
\[ \big\Vert \lambda^{-2} \psi^* g - \big( 2 g_{S^2} + g_{\R} \big) \big\Vert_{C^{[\delta^{-1}]}(S^2 \times (- \delta^{-1}, \delta^{-1}))} < \delta. \]
We call the image $\psi ( S^2 \times \{ 0 \})$ a {\bf central 2-sphere of $U$} and every point on a central $2$-sphere a {\bf center of $U$}.
\end{definition}

Note that by our convention (see Definition~\ref{def_curvature_scale}) we have $\rho \equiv 1$ on $(S^2 \times \R,  2 g_{S^2} + g_{\R})$.
So on a $\delta$-neck at scale $\lambda$ we have $\rho \approx \lambda$, where the accuracy depends on the smallness of $\delta$.
We also remark that a $\delta$-neck $U$ has infinitely many central $2$-spheres, as we may perturb $\psi$ slightly.
This is why we speak of \emph{a} central 2-sphere of $U$, as opposed to \emph{the} central 2-sphere of $U$.
Similarly, the centers of $U$ are not unique, but form an open subset of $U$. 

Lastly, we define the initial condition scale.

\begin{definition}[Initial condition scale] \label{Def_r_initial}
For any closed 3-manifold $(M,g)$ define the {\bf initial condition scale $r_{\initial}(M,g)$} as follows:
\[ r_{\initial} (M,g) := \min \big\{ \inf_M |{\Rm_g}|^{-1/2} , \inf_M |\nabla{\Rm_g}|^{-1/3} , \injrad (M,g) \big\} \]
\end{definition}

So $|{\Rm}| \leq r^{-2}_{\initial} (M,g)$, $|{\nabla\Rm}| \leq r^{-3}_{\initial} (M,g)$ and $\injrad (M,g) \geq r_{\initial} (M,g)$.
Moreover, the map $g \mapsto  r_{\injrad}(M,g)$ is continuous on $\Met (M)$.
\bigskip

Let us now state further results.
The first result offers more quantitative geometric control on any singular Ricci flow $\M$, depending only on the initial condition scale $r_{\initial} (\M_0, g_0)$ of the initial time-slice.

\begin{lemma} \label{lem_rcan_control}
For any $\eps > 0$ there is a smooth function $r_{\can, \eps} : \IR_+ \times [0, \infty) \to \IR_+$ such that the following holds for any $T \geq 0$ and any singular Ricci flow $\M$:
\begin{enumerate}[label=(\alph*)]
\item \label{ass_rcan_a} $\M$ satisfies the $\eps$-canonical neighborhood assumption below scale $r_{\can, \eps} \lb (r_{\initial} (\M_0, \lb g_0), \lb T)$
on $[0,T]$.
\item \label{ass_rcan_b} If $x \in \M$, $\t (x) \leq T$ and $\rho (x) \leq r_{\can, \eps} (r_{\initial} (\M_0, g_0), T)$, then the parabolic neighborhood $P := P(x, \eps^{-1} \rho(x))$ is unscathed and after parabolic rescaling by $\rho^{-2}(x)$ the flow on $P$ is  $\eps$-close to the flow on a $\kappa$-solution.
\item \label{ass_rcan_c} $\rho \geq \eps^{-1} r_{\can, \eps} (r_{\initial} (\M_0, g_0), T)$ on $\M_0$.
\item \label{ass_rcan_d} For any $a, r_0 > 0$ we have $r_{\can, \eps} (a  r_0, a^2 T) = a \cdot r_{\can, \eps} (r_0, T)$.
\item $|\partial_T^m r_{\can, \eps}| \leq \eps  r_{\can, \eps}^{1-2m}$ for $ m = 0, \ldots, [\eps^{-1}]$.
\item $r_{\can, \eps} (r_0, T)$ is decreasing in $T$ for any $r_0 > 0$.
\end{enumerate}
\end{lemma}

\begin{proof}
We will set 
\[ r_{\can,\eps} (r_0, T) := r_0 \cdot r'_{\can, \eps} (T r_0^{-2}) \]
for some smooth decreasing function $r'_{\can, \eps} : [0, \infty) \to \IR_+$, which we will determine in the following.
Then Assertion~\ref{ass_rcan_d} holds and due to invariance of all other assertions under parabolic rescaling, it suffices to assume in the following that $r_{\initial} (\M_0, g_0) = 1$.

By \cite[Thm. 1.3]{Kleiner:2014le}, \cite[Thm. 1.3]{bamler_kleiner_uniqueness_stability}, it follows that $\M$ is a limit of Ricci flows with surgery $\M^{(\delta_i)}$ with the same initial condition and performed at surgery scales $\delta_i \to 0$, .
By Perelman's construction \cite{Perelman2} of the flows $\M^{(\delta_i)}$, we know that there is a continuous, decreasing function $r'_{\can, \eps} : [0, \infty) \to \IR_+$ that is independent of $(\M_0, g_0)$ and $i$ such that for any $T > 0$ and $r < r'_{can, \eps} (T)$ and $i \geq \underline{i} (r, T)$ the flows $\M^{(\delta_i)}$ satisfy the hypothesis of Assertion~\ref{ass_rcan_b} with $\eps$ replaced by $\eps/2$ if $\rho(x) \in (r, 2r'_{\can, \eps} (T))$.
So Assertion~\ref{ass_rcan_b} also holds for $\M$ if $\rho(x) \leq r'_{\can, \eps} (T)$.
Assertion~\ref{ass_rcan_a} is a direct consequence of Assertion~\ref{ass_rcan_b}.

In order to prove the remaining assertions, we claim that there is a smooth decreasing function $r''_{\can, \eps} : [0, \infty) \to \IR_+$ such that $r''_{\can, \eps} (t) < r'_{\can, \eps}(t)$, $r''_{\can, \eps} (0) < 10^{-3}$ and $|\partial_t^m r''_{\can, \eps}| \leq \eps  (r''_{\can, \eps})^{1-2m}$ for $m = 0, \ldots, [\eps^{-1}]$.
By convolution with a smooth kernel, we can find a smooth, decreasing function $r'''_{\can, \eps} : [1, \infty) \to \IR_+$ such that 
\[ |\partial_t^m r'''_{\can, \eps} (t)| \leq C_m \sup_{[t-1,t+1]} r'_{\can, \eps} \leq C_m r'_{\can, \eps} (t-1).  \]
for some universal constants $C_m < \infty$.
So $r''_{\can, \eps} (t) := a_\eps r'''_{\can, \eps}(t+1)$, for sufficiently small $a_\eps > 0$, has the desired properties.
\end{proof}

The next result concerns the preservation of the positive scalar curvature condition.

\begin{theorem} \label{Thm_PSC_preservation}
If $\M$ is a singular Ricci flow such that $R > 0$ (resp. $R \geq 0$) on $\M_0$, then the same is true on all of $\M$.
\end{theorem}

\begin{proof}
This follows from the corresponding fact for Ricci flows with surgery since $\M$ is a limit of Ricci flows with surgery, as discussed in the proof of Lemma~\ref{lem_rcan_control}.
\end{proof}

The next result concerns the extinction of singular Ricci flows.

\begin{theorem}
If a singular Ricci flow $\M$ is extinct at time $t$ (i.e. $\M_t = \emptyset$), then it is also extinct at all later times $t' \geq t$.
\end{theorem}

\begin{proof}
This is a direct consequence of \cite[Theorem 1.11]{Kleiner:2014le}.
It also follows using Lemma~\ref{lem_bryant_increasing_scale} below.
\end{proof}

The next result gives a uniform bound on the extinction time of a singular Ricci flow starting from a connected sum of spherical space forms and copies of $S^2 \times S^1$.

\begin{theorem} \label{Thm_extinction_time}
Let $M$ be a connected sum of spherical space forms and copies of $S^2 \times S^1$.
Consider a compact subset $K \subset  \Met (M)$ of Riemannian metrics on $M$.
Then there is a time $T_{\ext} < \infty$ such that any singular Ricci flow $(\M,\mathfrak{t}, \partial_{\mathfrak{t}}, g)$ with the property that $(\M_0, g_0)$ is isometric to $(M, h)$ for some $h \in K$ is extinct at time $T_{\ext}$.
\end{theorem}

\begin{proof}
See \cite[Thm. 2.16(b)]{gsc}.
\end{proof}

The next result implies that the function $\t + \rho^{-2}$ on any singular Ricci flow is proper.

\begin{theorem} \label{Thm_rho_proper_sing_RF}
For any singular Ricci flow $\M$ and any $r, T > 0$ the subset $\{ \rho \geq r, \t \leq T \} \subset \M$ is compact.
\end{theorem}

\begin{proof}
This is one of the properties of a singular Ricci flow according to the definition in \cite{Kleiner:2014le}.
By Theorem~\ref{Thm_sing_RF_uniqueness} this definition is equivalent to Definition~\ref{Def_sing_RF}.
Alternatively, the theorem can be shown directly using Lemma~\ref{lem_rcan_control}(b).
\end{proof}

The last result essentially states that at points that satisfy the canonical neighborhood assumption we have $\partial_t \rho < 0$, unless the geometry is sufficiently closely modeled on the tip of a Bryant soliton.
See also \cite[Lemma~8.40]{bamler_kleiner_uniqueness_stability}, which is more general.

\begin{lemma}[Non-decreasing scale  implies Bryant-like geometry]
\label{lem_bryant_increasing_scale}
For any $\delta > 0$ the following holds if $\eps \leq \ov\eps (\delta)$.

Let $\M$ be a singular Ricci flow, $x \in \M_t$ a point and assume that $r:=\rho (x) < r_{\can, \eps} (r_{\initial} (\M_0, g_0), t)$.
Then $x$ survives until time $\max \{ t - r^2, 0 \}$.
Assume that $\rho (x(t')) \leq \rho (x)$ for some $t' \in [\max \{ t - r^2, 0 \}, t)$.
Then the pointed Riemannian manifold $(\M_t, g_t, x)$ is $\delta$-close to the pointed Bryant soliton $(M_{\Bry}, g_{\Bry}, x_{\Bry})$ at scale $\rho(x)$.
\end{lemma}

Note that here our choice of $\rho$, such that $\rho^{-2} = R$ at points that satisfy a precise enough canonical neighborhood assumption, is important.
Had we chosen $\rho$ differently, then we would have had to use a more complicated wording of Lemma~\ref{lem_bryant_increasing_scale}.

\begin{proof}
Assume that the lemma was false for some fixed $\delta > 0$ and choose a sequence of counterexamples $\M^i, x^i, t^i, t^{\prime, i}$ for $\eps^i \to 0$.
By parabolic rescaling we may assume that $\rho (x^i) = 1$.
By Lemma~\ref{lem_rcan_control}\ref{ass_rcan_b}, \ref{ass_rcan_c} we may pass to a subsequence and assume that the flows restricted to the universal covers of the parabolic neighborhoods $P(x^i, (\eps^i)^{-1} )$ and pointed at lifts of $x^i$ converge to a pointed $\kappa$-solution $(M^\infty, (g^{\infty}_t)_{t \leq 0}, x^\infty)$ with $R(x^\infty, 0) = \rho^{-2} (x^\infty, 0) = 1$.
By assumption $(M^\infty, g^{\infty}_0, x^\infty)$ cannot be isometric to $(M_{\Bry}, g_{\Bry}, x_{\Bry})$.
Therefore, by \cite[Proposition C.3]{bamler_kleiner_uniqueness_stability} and Definition~\ref{def_kappa_solution} we have $\partial_t R (x^\infty, 0) > 0$ and $\partial_t R(x^\infty, t) \geq 0$ for all $t \leq 0$.
Since the functions $t'' \mapsto R(x^i ( t + t''))$ smoothly converge to $t'' \mapsto R(x^\infty, t'')$ on $[-1,0]$, we obtain a contradiction to the fact that $R(x^i (t^{\prime,i})) = \rho^{-2} (x^i(t^{\prime,i})) \geq 1$ and $t^{\prime,i} < t^i$.
\end{proof}

\subsection{Singular Ricci flows --- Stability} \label{subsec_sing_RF_stability}
Next we formalize the Stability Theorem \cite[Theorem~1.5]{bamler_kleiner_uniqueness_stability} in a way that will fit our needs.
In short, this theorem states that two singular Ricci flows are geometrically close on the set $\{ \rho \geq \eps \} \cap \{ \t \leq \eps^{-1} \}$ if their initial time-slices are close enough.
This fact will be key to our understanding of continuous dependence of singular Ricci flows on their initial data and the construction of continuous families of singular Ricci flows in Section~\ref{sec_families_srfs}.

Let first $(M, g)$, $(M',g')$ be two Riemannian manifolds.

\begin{definition}[$\eps$-isometry between Riemannian manifolds]
An {\bf $\eps$-isometry from $M$ to $M'$} is a diffeomorphism $\phi:M \to M'$  such that 
\[
\big\Vert \phi^*g'-g \Vert_{C^{[\eps^{-1}]} (M)} <\eps .
\]
\end{definition}

Next, let $\M, \M'$ be two Ricci flow spacetimes.
For our purposes, we may take $\M, \M'$ to be singular Ricci flows.

\begin{definition}[$\eps$-isometry between Ricci flow spacetimes]
\label{def_eps_isometry_rf_spacetime}
An {\bf $\eps$-isometry from $\M$ to $\M'$} is a diffeomorphism $\phi:\M\supset U\ra U'\subset \M'$ where:
\ben
\item \label{prop_eps_isometry_rf_spacetime_1} $U$, $U'$ are open subsets such that $\rho \leq \eps$ on the subsets
\[
 (\M\setminus U)\cap\{\t\leq \eps^{-1}\}\,,\quad  (\M'\setminus U')\cap\{\t'\leq \eps^{-1}\}\,.
\]
\item \label{prop_eps_isometry_rf_spacetime_2} $\t'\circ\phi=\t$.
\item \label{prop_eps_isometry_rf_spacetime_3} For every $m_1, m_2 = 0, \ldots, [\eps^{-1}]$  we have
\[ | \nabla^{m_1} \partial_{\t}^{m_2} (\phi^* g' - g) | \leq \eps, \qquad | \nabla^{m_1} \partial_{\t}^{m_2} (\phi^* \partial'_\t - \partial_\t) | \leq \eps. \]
on $U$.
Here $\nabla^{m_1}$ denotes the $m_1$-fold covariant derivative with respect to the Riemannian metrics $g_t$ on each time-slice $\M_t \cap U$ and $\partial^{m_2}$ denotes the $m_2$-fold Lie derivative $\mathcal{L}^{m_2}_{\partial_\t}$.
\een
\end{definition}

We can now state our main stability result.
For a more general result, which also holds for more general Ricci flow spacetimes, see \cite[Theorem~1.7]{bamler_kleiner_uniqueness_stability}.

\begin{lemma}[Stability of singular Ricci flows]
\label{lem_stability}
Let $\M$ be a singular Ricci flow.  Then for every $\eps>0$ there is a $\de = \delta (\M, \eps)>0$ such that if $\M'$ is a singular Ricci flow and $\phi:\M_0\ra\M'_0$ is a $\de$-isometry, then there is an $\eps$-isometry $\wh\phi: \M\supset U\ra U'\subset\M'$ extending $\phi$, meaning that $\M_0 \subset U$, $\M'_0 \subset U'$ and $\wh\phi |_{\M_0} = \phi$.
\end{lemma}

\begin{proof}
Fix $\M$ and $\eps > 0$, set $T := \eps^{-1}$ and let $\eps_{\can}, \delta > 0$ be some constants, which we will determine in the course of the proof.
By Lemma~\ref{lem_rcan_control}, and assuming that $\delta$ is sufficiently small, $\M$ and $\M'$ both satisfy the $\eps_{can}$-canonical neighborhood assumption below scales $r_0$ on $[0,T]$ for some scale $r_0 (\M, T, \eps_{\can}) > 0$.
By \cite[Theorem~1.5]{bamler_kleiner_uniqueness_stability}, and assuming $\eps_{\can}$ and $\delta$ to be sufficiently small depending on $\eps$, $T$, $r_0$ and $\M$, we can extend $\phi$ to a map $\wh\phi : \M \supset U \to U' \subset \M'$ such that Properties~\ref{prop_eps_isometry_rf_spacetime_1}, \ref{prop_eps_isometry_rf_spacetime_2} of Definition~\ref{def_eps_isometry_rf_spacetime} hold.
The bounds from Property~\ref{prop_eps_isometry_rf_spacetime_1} follow from the Addendum to \cite[Theorem~1.5]{bamler_kleiner_uniqueness_stability} and the fact that $\wh\phi^* \partial_{\t'} -\partial_\t = \sum_{i=1}^3 \nabla^g_{e_i} e_i - \nabla^{\wh\phi^* g'}_{e_i} e_i$ for any local orthonormal frame $\{ e_i \}_{i=1}^3$, after adjusting $\delta$.
\end{proof}

The following corollary illustrates the statement of Theorem~\ref{lem_stability}.

\begin{corollary}
Let $\{\M^i\}$, $\M^\infty$ be singular Ricci flows and suppose that we have convergence $(\M^i_0 , g^i_0) \to (\M^\infty, g^\infty_0)$ in the sense that there is a sequence of $\delta_i$-isometries $\phi_i : \M^\infty \to \M^i$ with $\delta_i \to 0$.
Then we have convergence $\M^i \to \M^\infty$ in the sense that there is a sequence of $\eps_i$-isometries $\wh\phi_i$ between $\M^\infty, \M^i$ with $\eps_i \to 0$, which extend $\phi_i$.
\end{corollary}

\section{Families of singular Ricci flows}
\label{sec_families_srfs}
The purpose of this section is to distill the results about existence, uniqueness, and continuous dependence of singular Ricci flows into an  object that efficiently encodes the properties needed in the remainder of the proof.  
To motivate this, we recall that by \cite{Kleiner:2014le,bamler_kleiner_uniqueness_stability} (see also Subsection~\ref{subsec_sing_RF_exist_unique}), for every Riemannian manifold $(M,g)$ there exists a singular Ricci flow $\M^{(M,g)}$ with initial condition isometric to $(M,g)$, which is unique up to isometry.  
Our main result will be to formalize the stability property from \cite{bamler_kleiner_uniqueness_stability} (see also Subsection~\ref{subsec_sing_RF_stability}) as a continuous dependence of $\M^{(M,g)}$ on $(M,g)$.
More specifically, we will state that any ``continuous family'' of Riemannian manifolds yields a ``continuous family of singular Ricci flows''.  

Our starting point is a family of Riemannian manifolds $(M^s, g^s)_{s \in X}$, which depends continuously on a parameter $s$ in a certain sense.
One special case is the case in which $M^s$ is constant and $g^s$ depends continuously on $s$ in the smooth topology.
More generally, we may also consider a fiber bundle over $X$ whose fibers are smooth 3-dimensional manifolds $M^s$ equipped with a continuous family of  Riemannian metrics $g^s$.
For each $s \in X$ we consider the singular Ricci flow $\M^s := \M^{(M^s, g^s)}$.
Our first step will be to define a topology on the disjoint union $\sqcup_{s\in X}\M^s$ such that the natural projection $\pi:\sqcup_{s\in X}\M^s\ra X$ given by  $\pi(\M^s)=s$ is a topological submersion.
Secondly, we will endow this space with a lamination structure\footnote{Laminations have arisen in various contexts, including foliation theory, dynamical systems, complex geometry, $3$-manifolds,  geometric group theory, and minimal surfaces; see for instance \cite{sullivan,candel,mosher_oertel,gabai,colding_minicozzi}.  
The laminations appearing in this paper are particularly tame due to the existence of the compatible topological submersion structure, which prohibits nontrivial leaf dynamics -- a central phenomenon in other contexts. 
} whose leaves are the singular Ricci flows $\M^s$ and that satisfies certain compatibility conditions with the submersion structure.
The lamination structure will determine whether a family of maps from or to $\sqcup_{s\in X}\M^s$ is ``transversely continuous'' in the smooth topology.
In fact, the objects $\t^s$, $\partial_\t^s$, $g^s$ associated to each singular Ricci flow $\M^s$ will be transversely continuous in the smooth topology.  A collection of singular Ricci flows $\{ \M^s \}_{x \in X}$ together with the topology on the disjoint union and the lamination structure, as described above, will be called a ``continuous family of singular Ricci flows'' (see Definition~\ref{def_family_RF_spacetime} for further details).
This notion is an instance of a general notion of a continuous family of differential geometric structures, which may be useful in other contexts, and appears to be new.

Let us now state our main results.
We refer to Subsection~\ref{subsec_laminations} for the precise definitions of the terminology used.
For now, we mention that a continuous family of Riemannian manifolds $(M^s, g^s)_{s \in X}$ is essentially given by a fiber bundle over $X$ with smooth fibers and equipped with a continuous family of Riemannian metrics $g^s$ (see Corollary~\ref{Cor_family_compact_fiber_bundle} below).
An important special case is given by a family of Riemannian metrics $g^s$ on a fixed manifold $M$, which depend continuously in $s$ in the smooth topology.

For the remainder of this section, $X$ will denote an arbitrary topological space.

We first address the existence of a ``continuous family of singular Ricci flows''.

\begin{theorem}[Existence]
\label{thm_existence_family_k}
For any continuous family of closed Riemannian 3-manifolds $(M^s, g^s)_{s \in X}$ there is a continuous family of singular Ricci flows $(\M^s)_{s\in X}$ whose continuous family of time-$0$-slices $(\M^s_0, g^s_0)_{s\in X}$ is isometric to $(M^s, g^s)_{s \in X}$.
\end{theorem}

Although not needed in the remainder of this paper, we will also show that this family is unique.

\begin{theorem}[Uniqueness]
\label{thm_uniqueness_family_k}
The continuous family of singular Ricci flows from Theorem~\ref{thm_existence_family_k} is unique in the following sense.
Consider two such families of singular Ricci flows $(\M^{i,s})_{s\in X}$, $i = 1,2$, and isometries $\phi_i : \cup_{s \in X} \M^{i,s}_0 \to \cup_{s \in X} M^s$.
Then there is an isometry $\Psi : \cup_{s \in X} \M^{1,s} \to \cup_{s \in X} \M^{2,s}$ of continuous families of singular Ricci flows with the property that $\phi_1  = \phi_2 \circ \Psi$ on $\cup_{s \in X} \M^{1,s}_0$
\end{theorem}

We will also show the following properness property:

\begin{theorem} \label{Thm_properness_fam_sing_RF}
Let $(\M^s)_{s \in X}$ be a continuous family of singular Ricci flows and consider the projection $\pi : \cup_{s \in X} \M^s \to X$.
For any $s_0 \in X$, $r > 0$ and $t \geq 0$ there is a family chart $(U, \phi, V)$ with $s_0 \in \pi (U)$ and a compact subset $K \subset V$ such that
\[ \phi \big( \cup_{s \in \pi (U)} \{ \rho_{g^{s}} \geq r, \t^{s} \leq t \} \big) \subset K \times \pi (U). \]
In particular, the following projection is proper:
\[ \pi : \cup_{s \in X} \M^s \cap \{ \rho_{g^s} \geq r, \t^s \leq t \} \longrightarrow X \]
\end{theorem}

We remark that Theorems~\ref{thm_existence_family_k} and \ref{Thm_properness_fam_sing_RF} imply Theorem~\ref{lem_stability}.

\subsection{Continuous families of smooth objects}
\label{subsec_laminations}
In this subsection we formalize concepts relating to the terminologies of ``continuous family of Riemannian manifolds'' and ``continuous family of singular Ricci flows''.

\begin{definition}[Continuous family of $n$-manifolds]
\label{def_continuous_family_manifolds}
A  {\bf continuous family of (smooth) $n$-manifolds (with boundary, over $X$)} is given by a topological space $Y$, a continuous map $\pi:Y\ra X$ and a maximal collection of tuples $\{(U_i,\phi_i,V_i)\}_{i\in I}$ (called {\bf family charts}) such that:
\begin{enumerate}[label=(\arabic*)]
\item \label{prop_continuous_family_manifolds_1} For all $i \in I$, $U_i$ is an open subset of $Y$, $V_i$ is a smooth $n$-manifold (with boundary), and $\phi_i:U_i\ra V_i \times \pi(U_i)$ is a homeomorphism.
\item \label{prop_continuous_family_manifolds_2} $\cup_{i \in I} U_i=Y$.
\item \label{prop_continuous_family_manifolds_3} (Local trivialization) For every $i\in I$ the map $\phi_i$ induces an equivalence of the restriction $\pi|_{U_i}:U_i\ra \pi(U_i)$ to the projection $V_i \times \pi(U_i) \ra \pi(U_i)$, i.e. the following diagram commutes:  
\begin{diagram}
U_i       & \rTo^{\phi_i}        &V_i \times \pi(U_i)\\
\dTo^\pi  & \ldTo_{\proj_{\pi(U_i)}} &\\
\pi(U_i)      &   & 
\end{diagram}

\item \label{prop_continuous_family_manifolds_4} (Compatibility)  
For any $i,j\in I$ the transition homeomorphism 
\[
\phi_{ij}:=\phi_j\circ\phi_i^{-1}: V_i \times \pi(U_i)\supset\phi_i(U_i\cap U_j) \longrightarrow \phi_j(U_i\cap U_j)\subset V_j \times \pi(U_j)
\] 
 has the form $\phi_{ij}(v,s)=(\be(v,s),s)$
where $\be : \phi_i(U_i\cap U_j)\ra V_j$  locally defines a family of smooth maps $s \mapsto \be(\cdot, s)$ that depend continuously on $s$ in the $C^\infty_{\loc}$-topology. 
\end{enumerate}
Note that Properties~\ref{prop_continuous_family_manifolds_3} and \ref{prop_continuous_family_manifolds_4} imply that the family charts induce the structure of a smooth $n$-manifold with boundary on each fiber $Y^s :=\pi^{-1}(s)\subset Y$, for every $s\in X$.  
In order to use a more suggestive notation, we will sometimes  denote a continuous family of smooth manifolds by $( Y^s )_{s\in X}$, or by the map $\pi:Y\ra X$,  suppressing the collection of family charts. 
\end{definition}

\begin{remark} \label{rmk_set_theory_issue}
For convenience we have suppressed a set theoretic issue relating to the maximality property in Definition~\ref{def_continuous_family_manifolds}.
This issue can be remedied easily by requiring the following weaker version of maximality: If $\{ (U_i, \phi_i, V_i) \}_{i \in I}$ can be enlarged by a tuple $(U, \phi, V)$, while maintaining the validity of Properties~\ref{prop_continuous_family_manifolds_1}--\ref{prop_continuous_family_manifolds_4}, then $(U, \phi, V)$ is conjugate to some $(U_i, \phi_i, V_i)$, $i \in I$, in the sense that $U = U_i$ and $\phi_i = ( \psi, \id_{\pi(U)})  \circ \phi$ for some diffeomorphism $\psi : V \to V_i$.
In this case, we say that $(U, \phi, V)$ is a ``family chart'' if it is conjugate to some $(U_i, \phi_i, V_i)$.
\end{remark}

Relating to this we have the following:

\begin{lemma} \label{lem_cont_fam_no_max}
If the maximality property in Definition~\ref{def_continuous_family_manifolds} is dropped, then there is a unique maximal extension of $\{ (U_i, \phi_i, V_i ) \}_{i \in I}$ in the sense of Remark~\ref{rmk_set_theory_issue}, whose elements are unique up to conjugation.
\end{lemma}

\begin{proof}
It can be checked easily that if $(U, \phi, V)$ and $(U', \phi', V')$ can each be added to $\{ (U_i, \phi_i, V_i) \}_{i \in I}$ while maintaining the validity of Properties~\ref{prop_continuous_family_manifolds_1}--\ref{prop_continuous_family_manifolds_4} of Definition~\ref{def_continuous_family_manifolds}, then both triples can be added at the same time.
We can therefore add all such triples $(U, \phi, V)$, with the extra assumption that $V$ is a smooth manifold structure defined on a fixed set of cardinality $\aleph_1$.
\end{proof}

\begin{remark}
Definition~\ref{def_continuous_family_manifolds} combines two standard notions --- smooth laminations, and topological submersions: Property~\ref{prop_continuous_family_manifolds_3}  asserts that $\pi$ is a topological submersion, while Property~\ref{prop_continuous_family_manifolds_4} implies that the collection of family charts defines the structure of a lamination of smooth $n$-manifolds.
\end{remark}

\begin{remark} \label{rmk_fiber_bundle_construction}
If $M$ is a smooth manifold with boundary and $\pi : Y \to X$ is a fiber bundle with fiber $M$ and structure group $\Diff (M)$, then $\pi : Y \to X$ can also be viewed as continuous family of smooth manifolds with boundary.
To see this, consider all local trivializations $\phi_i : \pi_i^{-1} (W_i) \to M \times W_i$, where $W_i \subset X$ are open, and form the triples $(U_i := \pi^{-1} (W_i), \phi_i, V_i := M)$.
The set of these triples satisfy all properties of Definition~\ref{def_continuous_family_manifolds} except for the maximality property.
Due to Lemma~\ref{lem_cont_fam_no_max} these triples define a unique structure of a continuous family of manifolds.

A special case of this construction is the case in which $\pi : Y \to X$ is a trivial fiber bundle $Y = M \times X$.
In this case the associated continuous family of manifolds can be denoted by $(Y^s = M \times \{ s \})_{s \in X}$.

Conversely, every continuous family $(Y^s)_{s \in X}$ of compact manifolds over a connected space $X$ is given by a fiber bundle.
This fact will follow from Corollary~\ref{Cor_family_compact_fiber_bundle} below.
\end{remark}

\begin{remark}
If $( Y^s )_{s \in X}$ is a continuous family of $n$-manifolds (with boundary) and $W \subset \cup_{s \in X} Y^s$ is open, then $( Y^s \cap W )_{s \in X}$ carries a natural structure of a continuous family of $n$-manifolds (with boundary).
The family charts of this family is the subset of all family charts $(U_i, \phi_i, V_i)$ of $( Y^s )_{s \in X}$ with the property that $U_i \subset W$.

So for example, if $W \subset \IR^2$ is open, then the projection onto the second factor restricted to $W$, $\pi : W \to \IR$, defines a continuous family of 1-manifolds $(Y^s := W \cap \IR \times \{ s \})_{s \in \IR}$.
This is example shows that the topology of the fibers $Y^s$ is not necessarily constant.
\end{remark}

Next, we characterize maps between continuous families of manifolds.
\begin{definition}[Continuity of maps between continuous families]
\label{def_continuity_maps_between_families}
If $\pi_i:Y_i\ra X_i$, $i = 1,2$, are continuous families of $n$-manifolds with boundary, then a {\bf (continuous) family of smooth maps} is a pair $(F, f)$ where $F:Y_1 \ra Y_2$, $ f:X_1\ra X_2$ are continuous maps such that $\pi_2 \circ F= f\circ\pi_1$, and for every $y \in Y_1$ there are family charts $(U_i, \phi_i, V_i)$ of $\pi_i :Y_i \ra X_i$ such that $y \in U_1$, $F (U_1) \subset U_2$ and such that
\[ V_1 \times \pi_1 (U_1) \xrightarrow{\phi_1^{-1}} U_1 \xrightarrow{\;\; F \;\;} U_2 \xrightarrow{\; \phi_2 \;} V_2 \times \pi_2 (U_2) \longrightarrow V_2 \]
describes a family $(\beta_s : V_1 \to V_2)_{s \in \pi_1 (U_1)}$ of smooth maps  that is continuous in the $C^\infty_{\loc}$-topology.
\end{definition}

If we express the two continuous families $\pi_i : Y_i \to X_i$ as $( Y^s_i )_{s \in X_i}$, then we will sometimes also express $(F, f)$ as $( F_s : Y^s_1 \to Y^{f(s)}_2)_{s \in X_1}$ and we will say that this family of smooth maps is {\bf transversely continuous (in the smooth topology)}.

Two special cases of this terminology will be particularly important for us.
First, consider the case in which $X_2$ consists of a single point and $(Y^s_2 = M)_{s \in X_2}$ consists of a single smooth manifold with boundary $M$.
Then Definition~\ref{def_continuity_maps_between_families} expresses transverse continuity of a family of maps $(F^s : Y^s_1 \to M)_{s \in X_1}$, or equivalently, of a map $F : \cup_{s \in X_1} Y^s_1 \to M$; in the case in which $M = \IR$ this yields a definition of transverse continuity for scalar maps $F : \cup_{s \in X_1} Y^s_1 \to \IR$.
Second, consider the case in which $X_1 = X_2$, $f = \id_{X_1}$ and $(Y_1^s := M \times \{ s \})_{s \in X_1}$ is the trivial family.
In this case Definition~\ref{def_continuity_maps_between_families} introduces the notion of transverse continuity for families of smooth maps $(F^s : M \to Y_2^s)$. 

\begin{definition}[Isomorphism of continuous families]
A continuous family of smooth maps $(F, f)$ is called an {\bf isomorphism} if $F, f$ are invertible and if $(F^{-1}, f^{-1})$ constitutes a continuous family of smooth maps as well.
If $X_1 = X_2$, then we also call map $F : Y_1 \to Y_2$ an isomorphism if $(F, \id_{X_1})$ is an isomorphism.
\end{definition}

Due to the implicit function theorem we have:

\begin{lemma}
A continuous family $(F,f)$ is an isomorphism if and only if $f$ is a homeomorphism and all maps $F^s : Y_1^s \to Y_2^{f(s)}$, $s \in X_1$, are diffeomorphisms.
In particular, if $(Y_i^s)_{s \in X}$, $i=1,2$, are continuous families of smooth manifolds over the same space $X$, then a continuous family of maps $(F^s : Y_1^s \to Y_2^s)_{s \in X}$ constitutes an isomorphism if and only if all maps are diffeomorphisms.
\end{lemma}

Next, we define the notion of transverse continuity for tensor fields on a continuous family of smooth manifolds.

\begin{definition}[Transverse continuity of tensor fields]
\label{def_continuity_smooth_objects}
Let $(Y^s)_{s \in X}$ be a continuous family of smooth $n$-manifolds with boundary and for every $s\in X$, let $\xi^s$ be a tensor field on $Y^s$.  
We say that the family $(\xi^s)_{s\in X}$ is {\bf transversely continuous (in the smooth topology)} if for every family chart $\phi:U\ra  V \times \pi(U)$ of $(Y^s)_{s \in X}$ and all $s \in \pi (U)$ the push forwards of $\xi^s$ under \[
U \cap Y^s \xrightarrow{\quad \phi \quad} V \times  \{s\} \longrightarrow V
\]
are smooth tensor fields on $V$ that depend continuously on $s \in \pi(U)$ in the $C^\infty_{\loc}$-topology.  
\end{definition}

\begin{remark}
In the case of $(0,0)$-tensor fields, this notion offers another definition of transverse continuity of scalar functions, which is equivalent with the one derived from Definition~\ref{def_continuity_maps_between_families}.
\end{remark}

Adapting elementary results to families we have:
\begin{lemma}
\label{lem_generalizations_to_families}
\begin{enumerate}[label=(\alph*)]
\item  Let $(F_s : Y_1^s \to Y_2^{f(s)})_{s \in X_1}$ be a continuous family of diffeomorphisms between two continuous families of smooth manifolds with boundary $(Y_i^s)_{s \in X_i}$, $i = 1,2$ and let $(\xi_s)_{s \in X_2}$ be a transversely continuous family of tensor fields on $(Y_2^s)_{s \in X_2}$.
Then the pullbacks $((F^s)^* \xi_{f(s)})_{s \in X_1}$ are also transversely continuous.
\item Consider  is a continuous family  $(Y^s)_{s \in X}$ of smooth $n$-manifolds with boundary and a transversely continuous family of smooth maps $(F^s : Y^s \to Z)$ into a $k$-dimensional manifold with boundary.
Assume that $F^s (\partial Y^s) \subset \partial Z$ for all $s \in X$.
Let $z \in Z$ be a regular value of $F_s$ for all $s \in X$.
Then the collection  $( \td{Y}^s := (F^s)^{-1}(z)\subset Y^s)_{s\in X}$ of submanifolds inherits the structure of a continuous family of smooth $(n-k)$-manifolds with boundary.
\end{enumerate}
\end{lemma}

Using Definition~\ref{def_continuity_smooth_objects}, we can make the following definition:

\begin{definition}[Continuous family of Riemannian manifolds] \label{Def_cont_fam_RM}
A {\bf (continuous) family of Riemannian manifolds $( M^s, g^s)_{s \in X}$} consists of a continuous family $( M^s)_{s \in X}$ of smooth manifolds with boundary and a transversely continuous family of Riemannian metrics $(g^s)_{s \in X}$.

An isometry between two continuous families of Riemannian manifolds $(M_i^s, \lb g_i^s)_{s \in X_1}$, $i=1,2$, is an isomorphism $(F, f)$ between the associated continuous families of manifolds with boundary $(M_i^s)_{s \in X_1}$ with the property that $(F^s)^* g_2^{f(s)} = g_1^s$ for all $s \in X_1$.
\end{definition}

\begin{remark}
If $(M^s)_{s \in X}$ is given by a fiber bundle $\pi : Y \to X$ with fiber $M \approx M^s$ and structure group $\Diff (M)$ (see Remark~\ref{rmk_fiber_bundle_construction}), then any continuous family of Riemannian metrics $(g^s)_{s \in X}$, turning $(M^s, g^s)_{s \in X}$ into a continuous family of Riemannian manifolds in the sense of Definition~\ref{Def_cont_fam_RM}, is given by a fiberwise family of Riemannian metrics $(g^s)_{s \in X}$.
\end{remark}

\begin{remark} \label{rmk_universal_family_RM}
For any smooth manifold $M$ with boundary consider the space $\Met (M)$ of Riemannian metrics equipped with the $C^\infty_{\loc}$-topology.
Then $(M \times \{ g \}, g )_{g \in \Met (M)}$ is a continuous family of Riemannian manifolds.
If $M$ is closed and 3-dimensional, then Theorem~\ref{thm_existence_family_k} applied to this family asserts the existence of a family of singular Ricci flows $(\M^g)_{g \in \Met (M)}$ such that  $\M^g_0 = (M,g)$ for all $g \in \Met (M)$.
\end{remark}

Lastly, we define continuous families of Ricci flow spacetimes and continuous families of singular Ricci flows.

\begin{definition} \label{def_family_RF_spacetime}
A  {\bf continuous family of Ricci flow spacetimes} 
\[ ( \M^s, \t^s,  \partial_\t^s, g^s )_{s \in X}, \]
or in short $(\M^s)_{s \in X}$, consists of:
\begin{itemize}
\item a continuous family $(\M^s)_{s \in X}$ of smooth $4$-manifolds with boundary,
\item a transversely continuous family of smooth scalar functions $\t^s$ on $\M^s$; we will often write $\t : \cup_{s \in X} \M^s \to [0, \infty)$,
\item a transversely continuous family of smooth vector fields $\partial_\t^s$ on $\M^s$ such that $\partial^s_\t \t^s \equiv 1$,
\item a transversely continuous family of smooth inner products $g^s$ on the subbundle $\ker d\t^s \subset T\M^s$;
here we use the splitting $T\M^s = \ker d\t^s \oplus \spann \{ \partial_\t^s \}$ to view $g^s$ as a $(0,2)$-tensor field on $\M^s$ and define transverse continuity as in Definition~\ref{def_continuity_smooth_objects}.
\end{itemize}
We assume that for all $s \in X$ the tuple $( \M^s, \t^s,  \partial_\t^s, g^s )$ is a Ricci flow spacetime in the sense of Definition~\ref{def_RF_spacetime}.
If $\M^s$ is even a singular Ricci flow, in the sense of Definition~\ref{Def_sing_RF}, for all $s \in X$, then we call $(\M^s)_{s \in X}$ a {\bf continuous family of singular Ricci flows}.
 \end{definition} 
 
By Lemma~\ref{lem_generalizations_to_families} we have:
 
\begin{lemma} \label{lem_cont_fam_time_slice_inherit}
If $(\M^s)_{s \in X}$ is a continuous family of Ricci flow spacetimes, then for any $t \geq 0$, the set of time-$t$-slices $(\M^s_T = (\t^s)^{-1} (t), g^s_t)_{s \in X}$ inherits a structure of a continuous family of Riemannian manifolds.
\end{lemma}

\begin{remark}
In Definition~\ref{def_family_RF_spacetime} we don't require that $(\M^s)_{s \in X}$ comes from a fiber bundle, as explained in Remark~\ref{rmk_fiber_bundle_construction}.
In fact, even if $X$ is connected --- such as in the case $X = \Met (M)$ of Remark~\ref{rmk_universal_family_RM} --- the topology of the spacetimes $\M^s$ may depend on $s$.
Similarly, in the context of Lemma~\ref{lem_cont_fam_time_slice_inherit}, the time $t$ may be a singular time for some but not all parameters $s \in X$.
In this case, the topology of the time-$t$-slices may depend on $s$, even if the topology of $\M^s$ doesn't.
\end{remark}

\subsection{Existence of family charts}
The following result will be used frequently throughout this paper.

\begin{lemma}
\label{lem_chart_near_compact_subset}
If $(M^s)_{s \in X}$ is a continuous family of smooth manifolds with boundary, $s_0 \in X$ and $K \subset M^{s_0}$ is a compact subset, then there is a family chart $(U,\phi, V)$ such that $K \subset U$.
\end{lemma}

\begin{remark}
We note that Lemma~\ref{lem_chart_near_compact_subset} holds more generally for laminations, when one has a compact subset $K$ of a leaf with trivial holonomy.
\end{remark}

A consequence of Lemma~\ref{lem_chart_near_compact_subset} is:

\begin{corollary} \label{Cor_family_compact_fiber_bundle}
If $(M^s)_{s \in X}$ is a continuous family of smooth and compact manifolds with boundary and if $X$ is connected, then $(M^s)_{s \in X}$ comes from a fiber bundle with fiber $M \approx M^s$ and structure group $\Diff (M)$, in the sense of Remark~\ref{rmk_fiber_bundle_construction}.
\end{corollary}

\begin{proof}[Proof of Lemma~\ref{lem_chart_near_compact_subset}]
Recall the standard projection $\pi : \cup_{s \in X} M^s \to X$.

\begin{Claim}  There is an open subset $K \subset U\subset \cup_{s \in X} M^s$ and a retraction $\pi_V :U\ra  U \cap M^{s_0} =: V$ that is transversely continuous in the smooth topology.
\end{Claim}

\begin{proof}
Let $\{ (U_i, \phi_i, V_i) \}_{i\in I' \subset I}$ be a finite collection of family charts  such that $K\subset \cup_{i \in I'} U_i$.   
For each $i\in I'$ define $\pi_i:U_i\ra U_i\cap M^{s_0}$ to be the composition 
\[
U_i \xrightarrow{ \; \phi_i \;} V_i \times \pi (U_i) \longrightarrow V_i \longrightarrow V_i \times \{s_0 \} \xrightarrow{\phi_i^{-1}} M^{s_0} \cap U_i.
\]
We will obtain the desired map $\pi_V$ by gluing the retractions $\pi_i$, $i\in I'$, with a partition of unity.  

Let $Z:=\cup_{i \in I'} M^{s_0} \cap U_i$ and fix a partition of unity $\{\al_i:Z\ra [0,1]\}_{i\in I'}$ subordinate to the cover $\{ M^{s_0} \cap U_i \}_{i\in I}$ of $Z$.
We may assume without loss of generality that $Z$ is an embedded submanifold of $\R^N_{\geq 0} = \IR^{N-1} \times [0, \infty)$, for some large $N$, such that $Z \cap \partial \IR^N_{\geq 0} = \partial Z$ and $Z$ meets the boundary of $\R^N_{\geq 0}$ orthogonally.
Let $Z' \Subset Z$ be a relatively compact neighborhood of $K$.
Choose $\de>0$ such that the nearest point retraction $r_{Z}: N_{\de}(Z')\ra Z'$ is well-defined and smooth, where $N_{\de}(Z')$ is the exponential image of the vectors of length $< \delta$ in the normal bundle $\nu Z'$ of $Z'$ in $\IR^N_{\geq 0}$.  
Consider the convex combination $s_{Z'} := \sum_{i \in I'} (\al_i \circ \pi_i) \pi_i:\cup_{i \in I'} U_i\ra \R^N_{\geq 0}$.
Note that $s_{Z'} |_{Z'} = \id_{Z'}$.
Then $U := s_{Z'}^{-1} ( N_\delta (Z')) \subset \cup_{i \in I'} U_i$ and $\pi_V := r_{Z'} \circ s_{Z'}$ have the desired properties.
\end{proof}

Let $\pi_V:U\ra   V$ be as in the claim.  
After shrinking $U$ and $V$ we may assume without loss of generality that $\pi_V |_{U\cap M^s}:  M^s \cap U \to V$ is a local diffeomorphism for every $s \in X$.
Consider the map $\phi:=(\pi_V, \pi): U\ra  V \times X$.  
Since $\pi_V$ is a fiberwise local diffeomorphism, it follows from the inverse function theorem (for families of maps) that $\phi$ is a local homeomorphism. 
Therefore, by compactness of $K$ and the fact that $\pi_V |_V=\id_V$, we find an open neighborhood $U' \subset U$ of $K$ such that $\pi |_{U'}$ is injective.
Hence, after shrinking $U$ and $V$ we may assume that $\phi$ is a homeomorphism onto its image $\phi(V)\subset V \times X$.   Shrinking $U$ and $V$ further we can arrange for this image to be a product region in $V \times X$.
Finally, we replace the target with $\phi(U)$, so that $\phi$ is a homeomorphism.
\end{proof}

\subsection{Proof of Theorem~\ref{thm_existence_family_k}}
\label{subsec_construction_um}
By Theorem~\ref{subsec_sing_RF_exist_unique}, for every $s \in X$ we may fix a singular Ricci flow $(\M^s, \t^s, \partial^s_\t, g^s)$ such that $\M^s_0$ is isometric to $(M^s,g^s)$.  
In the following, we will identify $(\M^s_0, g^s_0)$ with $(M^s, g^s)$.
Let $Y := \sqcup_{s \in X} \M^s$ be the disjoint union and $\pi : Y \to X$ the natural projection.

Before proceeding further, we first indicate the idea of the proof.
Our goal will be to construct a suitable topology, as well as family charts on $Y$.
In order to achieve this, we will first introduce the notion of ``families of almost-isometries'' (see Definition~\ref{Def_fam_alm_isometries} below), which will serve as a guide to decide whether a subset of $Y$ is open or a map defined on a subset of $Y$ is a family chart.
Roughly speaking, a family of almost-isometries near some parameter $s \in X$ consists of maps $\psi_{s'}$ between a large subset of $\M^s$ and a large subset $\M^{s'}$, for $s'$ close to $s$, which arise from the stability theory of singular Ricci flows in Lemma~\ref{lem_stability}.
These maps are almost isometric up to arbitrarily high precision as $s' \to s$.
In our construction, the topology and the family charts on $Y$ will be chosen in such a way that a posteriori we have $\psi_{s'} \to \id_{\M^s}$ in the smooth topology as $s' \to s$.

After defining the topology and family charts on $Y$, we need to verify all required properties, such as the property that the domains of all family charts cover $Y$.
We will do this by propagating family charts from the family of time-$0$ slices to all of $Y$ using the flow of the time vector fields $\partial^s_{\t}$ and the exponential map with respect to the metrics $g^s_t$.

Let us now continue with the proof.

\begin{definition}[Families of almost-isometries] \label{Def_fam_alm_isometries}
Consider a neighborhood $S \subset X$ of some $s \in X$.
A family of diffeomorphisms 
\[
\{ \psi_{s'}: \M^s \supset Z_{s'} \longrightarrow Z_{s'}'\subset\M^{s'} \}_{s' \in S},
\]
where $Z_{s'}, Z'_{s'}$ are open, is called {\bf family of almost-isometries near $s$} if the following is true:
\begin{enumerate}[label=(\arabic*)]
\item \label{prop_fam_alm_isometries_1} For every $\eps > 0$ there is a neighborhood $S_\eps \subset S$ of $s$ such that $\psi_{s'}$ is an $\eps$-isometry for all $s' \in S_\eps$, in the sense of Definition~\ref{def_eps_isometry_rf_spacetime}.
\item \label{prop_fam_alm_isometries_2} $\psi_{s'} |_{\M^s_0} \to \id_{\M^s_0}$ in $C^\infty_{\loc}$ as $s' \to s$, where we use the identification $\M^s_0 = M^s$ and interpret $C^\infty_{\loc}$-convergence within the family of 3-manifolds $( M^s )_{s \in X}$.
\end{enumerate}
\end{definition}

Due to Lemmas~\ref{lem_stability} and \ref{lem_chart_near_compact_subset} we have:

\begin{lemma} \label{lem_almost_isometries_existence}
For every $s \in X$ there is a family of almost-isometries near $s$.
\end{lemma}

Next we use families of almost-isometries to define a topology on the total space $Y$.

\begin{definition}[Topology on $Y$] \label{def_topology_on_Y}
We define a subset $U \subset Y$ to be {\bf open} if for every $p \in U$, $s := \pi (p) \in X$ and every family of almost-isometries
\[
\{ \psi_{s'}: \M^s \supset Z_{s'} \longrightarrow Z_{s'}'\subset\M^{s'} \}_{s' \in S}
\]
near $s$ there are neighborhoods $W^* \subset S$ of $s$ and $U^* \subset \M^s$ of $p$ such that $U^* \subset  \psi_{s'}^{-1} (U)$ for all $s' \in W^*$.
\end{definition}

\begin{lemma}
Definition~\ref{def_topology_on_Y} defines a topology on $Y$ and the projection $\pi : Y \to X$ is continuous and open.
\end{lemma}

\begin{proof}
For the first statement the only non-trivial part is showing that the intersection of two open subsets $U_1, U_2 \subset Y$ is open.
Assume that $p \in U_1 \cap U_2$, $s := \pi (p)$ and consider a family of almost-isometries $\{ \psi_{s'} \}_{s' \in S}$ near $s$.
There are neighborhoods $W^*_i \subset W$ of $s$ and $U^*_i \subset \M^s$, of $p$, $i = 1,2$, such that $U^*_i \subset \psi_{s'}^{-1} (U_i )$ for all $s' \in W^*_i$.
It follows that $U^*_1 \cap U^*_2 \subset  \psi_{s'}^{-1} (U_1 \cap U_2)$ for all $s' \in W^*_1 \cap W^*_2$.
This shows that $U_1 \cap U_2$ is open.

The continuity of $\pi$ is a direct consequence of Definition~\ref{def_topology_on_Y} and the openness of $\pi$ follows using Lemma~\ref{lem_almost_isometries_existence}.
\end{proof}

Our next goal will be to find family charts on $Y$, which turn $\{ \M^s \}_{s \in X}$ into a continuous family in the sense of Definition~\ref{def_continuous_family_manifolds}.
Due to technical reasons, which will become apparent later, we will only construct a certain subclass of such family charts.
We will also define a variant of a family chart whose domain is contained in the union of all time-$t$-slices $\M^s_t$ for a fixed $t$.
This notion will be helpful in the statement and proof of Lemma~\ref{lem_propagating_charts}.

\begin{definition} \label{Def_time_pres_fam_chart}
A triple $(U, \phi, V, \t_V)$ is called a {\bf time-preserving family chart} if
\begin{enumerate}[label=(\arabic*)]
\item \label{prop_time_pres_fam_chart_1} $V$ is a smooth 4-manifold with boundary equipped with a smooth map $\t_V : V \to [0, \infty)$ such that $\partial V = \t^{-1} (0)$.
\item \label{prop_time_pres_fam_chart_2} $\phi : U \to V \times \pi (U)$ is a map.
\item \label{prop_time_pres_fam_chart_3} $\t_V \circ \proj_V \circ \phi = \t^s$ for all $s \in \pi (U)$.
\item \label{prop_time_pres_fam_chart_4} $\proj_{\pi (U)} \circ \phi = \pi|_U$, as in Property~\ref{prop_continuous_family_manifolds_3} of Definition~\ref{def_continuous_family_manifolds}.
\item \label{prop_time_pres_fam_chart_5} $\proj_V \circ \phi |_{U \cap \M^s} : U \cap \M^s \to V$ is a diffeomorphism for all $s \in \pi (U)$.
\item \label{prop_time_pres_fam_chart_6} For every $(v,s) \in V \times \pi (U)$ and every family of almost-isometries $\{ \psi_{s'} \}_{s' \in S}$ near $s$ there is a neighborhood $V' \times W' \subset V \times \pi (U)$ such that for all $s' \in W'$ the maps $(\psi_{s'}^{-1} \circ \phi^{-1}) (\cdot, s') : V' \to \M^s$ are well defined on $V'$ and converge to $\phi^{-1} (\cdot, s)$ in $C^\infty_{\loc}$ as $s' \to s$.
\end{enumerate}
We say that $(U, \phi, V)$ is a {\bf family chart at time $t$} if $U \subset \cup_{s \in X} \M^s_t$, $V$ is a smooth 3-manifold and Properties~\ref{prop_time_pres_fam_chart_2}, \ref{prop_time_pres_fam_chart_4}--\ref{prop_time_pres_fam_chart_6} hold with $\M^s$ replaced by $\M^s_t$.
\end{definition}

\begin{lemma}
Assume that $(U, \phi, V, \t_V)$ is a time-preserving family chart.
Then:
\begin{enumerate}[label=(\alph*)]
\item $U \subset Y$ is open and $\phi$ is a homeomorphism.
\item The push forwards of the objects $\partial_\t^s, g^s$ onto $V$ via $(\proj_V \circ \phi )(\cdot, s)$ vary continuously in $s$ in the $C^\infty_{\loc}$-sense.
\item For any $t \geq 0$ the restriction $\phi |_{U \cap \cup_{s \in X} \M^s_t} : U \cap \cup_{s \in X} \M^s_t \to \t_{V}^{-1} (t) \times \pi (U)$ is a family chart at time $t$.
\end{enumerate}
\end{lemma}

\begin{proof}
(a) \quad 
Let $W := \pi (U)$ and consider a non-empty subset $U_0 \subset U$.
We will show that $U_0$ is open in the sense of Definition~\ref{def_topology_on_Y} if and only if $\phi (U_0) \subset V \times \pi (U)$ is open in the product topology.
For this purpose let $p \in U_0$ and set $(v,s) := \phi (p) \in V \times W$.
Choose a family of almost-isometries $\{ \psi_{s'}  \}_{s' \in S}$ near $s$ according to Lemma~\ref{lem_almost_isometries_existence} and choose $V' \subset V$ and $W' \subset W$ as in Property~\ref{prop_time_pres_fam_chart_6} of Definition~\ref{Def_time_pres_fam_chart}.
Set $U' := \phi^{-1} (V' \times W')$.

Recall that 
\begin{equation} \label{eq_psi_phi_C_infty_convergence}
(\psi^{-1}_{s'} \circ \phi^{-1} )(\cdot, s') \xrightarrow{\quad C^\infty_{\loc} \quad} \phi^{-1} (\cdot, s) \qquad \text{as} \quad s' \to s.
\end{equation}

If $U_0$ is open in the sense of Definition~\ref{def_topology_on_Y}, then there are neighborhoods $W^* \subset S$ of $s$ and $U^* \subset U \cap \M^s$ of $p$ such that $U^* \subset \psi_{s'}^{-1} (U_0)$ for all $s' \in W^*$.
Let $V^* \times \{ s \} := \phi (U^*)$.
By (\ref{eq_psi_phi_C_infty_convergence}) there are open neighborhoods $V'' \subset V'$ of $v$ and $W'' \subset W'$ of $s$ such that for all $s' \in W''$
\[ (\psi^{-1}_{s'} \circ \phi^{-1} )(V'' \times \{ s' \}) \subset \phi^{-1} ( V^* \times \{ s \} ) = U^*. \]
It follows that $\phi^{-1} (V'' \times \{ s' \}) \subset \psi_{s'} (U^*) \subset U_0$ and therefore $\phi^{-1} ( V'' \times W'' ) \subset U_0$.
This proves that $\phi (U_0)$ is open.

Conversely assume that $\phi (U_0)$ is open.
Without loss of generality, we may assume that $\phi (U_0 ) = V_0 \times W_0$ for some open subsets $V_0 \subset V$, $W_0 \subset W$, so $(v,s) \in V_0 \times W_0$.
By (\ref{eq_psi_phi_C_infty_convergence}) and the implicit function theorem we can find neighborhoods $V'' \subset  V_0 \cap V'$ of $v$ and $W'' \subset  W_0 \cap W'$ of $s$ such that for all $s' \in W''$ we have
\[  U^* := \phi^{-1} ( V'' \times \{ s \} ) \subset ( \psi^{-1}_{s'} \circ \phi^{-1} )( V_0 \times \{ s' \} ) \subset \psi_{s'}^{-1} ( \phi^{-1} (V_0 \times W_0)) = \psi_{s'}^{-1} (U_0). \]
This proves that $U_0$ is open.

(b) \quad Fix $s \in \pi (U)$ and choose a family of almost-isometries $\{ \psi_{s'}  \}_{s' \in S}$ near $s$.
For any $s' \in \pi (U)$ we have
\[ \big( (\psi_{s'}^{-1} \circ \phi^{-1} ) (\cdot, s') \big)_* (\proj_V \circ \phi  )_* \partial^{s'}_\t = (\psi_{s'}^{-1})_*  \partial^{s'}_\t \xrightarrow{\quad C^\infty_{\loc} \quad} \partial^s_\t \]
as $s' \to s$.
Since $ (\psi_{s'}^{-1} \circ \phi^{-1} ) (\cdot, s') \to \phi^{-1} (\cdot, s)$ in $C^\infty_{\loc}$, we obtain that
\[  (\proj_V \circ \phi  )_* \partial^{s'}_\t  \xrightarrow{\quad C^\infty_{\loc} \quad} (\proj_V \circ \phi  )_* \partial^s_\t, \]
as desired.
The continuity of the push forwards of $g^s$ follows analoguously.

Assertion (c) is a direct consequence of Definition~\ref{Def_time_pres_fam_chart}.
\end{proof}

\begin{lemma} \label{lem_fam_chart_is_fam_chart_time_0}
The family charts at time $0$ and subspace topology induced from Definition~\ref{def_topology_on_Y} define a structure of a continuous family of manifolds on $\{ \M^s_0 \}_{s \in X}$ that agrees with the given structure on $(M^s)_{s \in X}$ if we use the identification $M^s = \M^s_0$.
\end{lemma}

\begin{proof}
This is a direct consequence of Definitions~\ref{def_continuous_family_manifolds}, \ref{def_topology_on_Y} and \ref{Def_time_pres_fam_chart} and Property~\ref{prop_fam_alm_isometries_2} of \ref{Def_fam_alm_isometries}.
\end{proof}

\begin{lemma} \label{lem_compatible_fam_chart}
If $(U_i, \phi_i, V_i, \t_{V_i})$, $i=1,2$, are two time-preserving family charts according to Definition~\ref{Def_time_pres_fam_chart}, then both charts are compatible in the sense of Property~\ref{prop_continuous_family_manifolds_4} of Definition~\ref{def_continuous_family_manifolds}.
In other words, the transition map
\[
\phi_{12}:=\phi_2 \circ\phi_1^{-1}: V_1 \times \pi(U_1)\supset\phi_1(U_1\cap U_2)\ra \phi_2(U_1\cap U_2)\subset V_2 \times \pi(U_2)
\] 
 has the form $\phi_{12}(v,s)=(\be(v,s),s)$, where $\be : \phi_1(U_1\cap U_2)\ra V_2$  locally defines a family of smooth maps $s \mapsto \be(\cdot, s)$ that depend continuously on $s$ in the $C^\infty_{\loc}$-topology. 
\end{lemma}

\begin{proof}
Let $p \in U_1 \cap U_2$ and set $(v_i, s) := \phi_i (p)$.
Fix a family of almost-isometries $\{ \psi_{s'}  \}_{s' \in S}$ near $s$.
Choose neighborhoods $V'_i \times W'_i \subset V_i \times \pi (U_i)$ of $(v_i, s)$ according to Property~\ref{prop_time_pres_fam_chart_6} of Definition~\ref{Def_time_pres_fam_chart}.
After shrinking $V'_1, W'_1$, we may assume without loss of generality that $\phi_1^{-1} ( V'_1 \times W'_1 ) \subset \phi_2^{-1} (V'_2 \times W'_2)$.
So for all $s' \in W'_1 \cap W'_2$
\[ (\psi_{s'}^{-1} \circ \phi_1^{-1}) (V'_1 \times \{ s' \}) \subset (\psi_{s'}^{-1} \circ \phi_2^{-1}) (V'_2 \times \{ s' \}) \]
and therefore for all $v' \in V'_1$
\[ (\beta (v', s')  ,s') = \big( ( \psi_{s'}^{-1} \circ \phi_2^{-1} )^{-1} \circ (\psi_{s'}^{-1} \circ \phi_1^{-1}) \big) (v', s') \]
Since $(\psi_{s'}^{-1} \circ \phi_i^{-1}) (\cdot, s') \to \phi_i^{-1} (\cdot, s)$  in $C^\infty_{\loc}$ as $s' \to s$, this implies that $\beta (\cdot, s') \to \beta (\cdot, s)$ in $C^\infty_{\loc}$ on $V_1$ as $s' \to s$.
\end{proof}

Next we show that the domains of all time-preserving family charts cover $Y$.
For this purpose, we first prove:

\begin{lemma}
\label{lem_propagating_charts}
Suppose that for some $s \in X$, $t\geq 0$ the point $p\in \M^s_t$ lies in the domain $U$ of a family chart $(U, \phi, V)$ at time $t$, in the sense of Definition~\ref{Def_time_pres_fam_chart}.
Consider another point $p'\in \M^s$.
\begin{enumerate}[label=(\alph*)]
\item \label{ass_propagating_charts_a} If $p'=p(t')$ for some $t'  \geq 0$, then $p'$ lies in the domain $U'$ of a time-preserving family chart $(U', \phi', V', \t_{V'})$. See Definition~\ref{def_points_in_RF_spacetimes} for the notation $p(t')$.
\item \label{ass_propagating_charts_b} If $p'\in B(p,r)$ for  $r<\injrad(\M^s_t,p)$, then $p'$ lies in the domain $U'$ of a family chart $(U', \phi', V')$ at time $t$. 
\end{enumerate}
\end{lemma}
\begin{proof}
(a) \quad  Our strategy will be to extend $\phi$ via the flow of $\partial^s_\t$. 
Choose a bounded interval $I \subset [0, \infty)$ that is open in $[0, \infty)$, contains $t, t'$ and for which $p ( \ov{I})$ is well-defined, where $\ov{I}$ denotes the closure of $I$.
Let $V'' \subset V$ be a subset such that:
\begin{enumerate}
\item $V''$ is open and has compact closure $\ov{V}''$ in $V$.
\item $\phi (p) \in V'' \times \{ s \}$.
\item $(\phi^{-1} (\ov{V}'',s))(t'')$ is well-defined for all $t'' \in \ov{I}$.
\end{enumerate}
Let $W' \subset \pi(U)$ be the set of parameters $s' \in \pi (U)$ for which $(\phi^{-1} (\ov{V}'',s'))(t'')$ is well-defined for all $t'' \in \ov{I}$.
Consider the map
\[ \alpha : (V'' \times I) \times W' \longrightarrow Y, \qquad (v',t'',s') \longmapsto (\phi^{-1} (v',s')) (t''). \]
Note that $\alpha$ is injective.
Let $V' := V'' \times I$, $U' := \alpha (V' \times W')$ and $\phi' := \alpha^{-1} : U' \to V' \times W'$.
We claim that $(U', \phi', V', \proj_I)$ is a time-preserving family chart in the sense of Definition~\ref{Def_time_pres_fam_chart}.

We need to verify Property~\ref{prop_time_pres_fam_chart_6} of Definition~\ref{Def_time_pres_fam_chart} and that $W'$ is open; the remaining properties  of Definition~\ref{Def_time_pres_fam_chart} follow directly by construction.
For this purpose consider some $s_0  \in W'$, and let $\{ \psi_{s'}  \}_{s' \in S}$ be a family of almost-isometries near $s_0$.
Since $\ov{V}''$ is compact, we can find neighborhoods $W_0 \subset W \cap S$ of $s_0$ and $V_0 \subset V$ of $\ov{V}''$ such that for all $s' \in W_0$ the maps $(\psi_{s'}^{-1} \circ \phi^{-1}) (\cdot, s')$ are well defined on $V_0$ and converge to $\phi^{-1} (\cdot, s_0)$ in $C^\infty_{\loc}$ as $s' \to s_0$.

For any $s' \in W_0$ and $v' \in V_0$ consider the trajectory of $\partial^{s'}_\t$ through $\phi^{-1} (v', s')$.
The image of this trajectory under the map $\psi^{-1}_{s'}$ is a trajectory of $\psi_{s'}^* \partial^{s'}_\t$ through $(\psi^{-1}_{s'} \circ \phi^{-1} )(v', s')$, wherever defined.
Since $(\psi_{s'}^{-1} \circ \phi^{-1}) (\cdot, s') \to \phi^{-1} (\cdot, s_0)$ and $\psi_{s'}^* \partial^{s'}_\t \to \partial^{s_0}_\t$ in $C^\infty_{\loc}$ as $s' \to s_0$, we can find a neighborhood $W_1 \subset W_0$ of $s_0$ with the property that for any $s' \in W_1$ and $v' \in \ov{V}''$ the trajectory of $\psi_{s'}^* \partial^{s'}_\t$ through $(\psi^{-1}_{s'} \circ \phi^{-1}) (v', s')$ exists for all times of the interval $\ov{I}$.
It follows that $W_1 \subset W'$ and for any $s' \in W_1$ the map $(\psi_{s'}^{-1} \circ \alpha) (\cdot, s')$ restricted to $V''$ is given by the flow of $\psi_{s'}^* \partial^{s'}_\t$ starting from $(\psi^{-1}_{s'} \circ \phi^{-1}) (\cdot, s')$.
Due to the smooth convergence discussed before we have
\begin{equation*} \label{eq_psi_alpha_convergence}
 (\psi_{s'}^{-1} \circ \alpha) (\cdot, s') \xrightarrow{\quad C^\infty_{\loc} \quad} \alpha (\cdot, s_0) \qquad \text{as} \quad s' \to s_0. 
\end{equation*}
This verifies Property~\ref{prop_time_pres_fam_chart_6} of Definition~\ref{Def_time_pres_fam_chart} and shows that $W'$ is open.

\medskip
(b) \quad  After shrinking $V$ and $W := \pi (U)$ if necessary, we may assume that for some $r'>r$ and every $s' \in W$ and $v' \in V$, we have 
\begin{equation}
\label{eqn_r_prime_injrad}
r<r'<\injrad(\M^{s'}_t, \phi^{-1} (v',s'))\,;
\end{equation} 
this follows from a straightforward convergence argument.   
Let $(v,s) := \phi (p)$.
Using Gram-Schmidt orthogonalization we can find a continuous family of linear maps $( \varphi_{s'} : \IR^3 \to T_v V )_{s' \in W}$ that are isometries with respect to the push forward of $g^{s'}_t$ via $\proj_V \circ \phi$.
Denote by $B (0,r) \subset \IR^3$ the $r$-distance ball and define
\[ \alpha : B (0,r) \times W \longrightarrow \cup_{s' \in X} \M^{s'}_t, \quad \alpha (\cdot ,s') := \exp^{g^{s'}_t}_{\phi^{-1} (v, s')} \circ d \big(\phi^{-1} (\cdot, s') \big)_v \circ \varphi_{s'} . \]
Due to (\ref{eqn_r_prime_injrad}) this map is injective.
Let $V' := B (0,r)$, $U' := \alpha (V' \times W )$,  and $\phi := \alpha^{-1} : U' \to V' \times W$.

As in the previous case it remains to show that Property~\ref{prop_time_pres_fam_chart_6} of Definition~\ref{Def_time_pres_fam_chart} holds.
Let $s_0 \in W$ and let $\{ \psi_{s'}  \}_{s' \in S}$ be a family of almost-isometries near $s_0$.
For $s'$ sufficiently close to $s_0$ we have
\[ (\psi^{-1}_{s'} \circ \alpha) (\cdot, s') = \exp^{\psi^*_{s'} g^{s'}_t}_{(\psi^{-1}_{s'} \circ \phi^{-1}) (v,s')} \circ d \big( ( \psi_{s'}^{-1} \circ \phi^{-1}) (\cdot, s') \big)_v \circ \varphi_{s'}. \]
Since
\[ \psi^*_{s'} g^{s'}_t \xrightarrow{\quad C^\infty_{\loc} \quad}  g^{s_0}_t, \qquad  (\psi^{-1}_{s'} \circ \phi^{-1} )(\cdot, s') \xrightarrow{\quad C^\infty_{\loc} \quad}  \phi^{-1} (\cdot, s_0) \]
as $s' \to s_0$, we therefore obtain that
\[ (\psi^{-1}_{s'} \circ \alpha )(\cdot, s') \xrightarrow{\quad C^\infty_{\loc} \quad}  \alpha (\cdot, s_0) \]
as $s' \to s_0$.
This establishes Property~\ref{prop_time_pres_fam_chart_6} of Definition~\ref{Def_time_pres_fam_chart}.
\end{proof}

\begin{corollary}
\label{cor_chart_domains_cover}
Every point $p\in Y$ lies in the domain of a time-preserving family chart.
\end{corollary}

\begin{proof}
Let $C\subset Y$ be the set of points lying in the domain of a time-preserving family chart.
By Lemma~\ref{lem_fam_chart_is_fam_chart_time_0} and Lemma~\ref{lem_propagating_charts}\ref{ass_propagating_charts_a} we have $\M^s_0 \subset C$ for all $s \in X$.
By Lemma~\ref{lem_propagating_charts}, for every $s \in X$, the intersection $C\cap\M^s$ is an open and closed subset of $\M^s$; since it is nonempty,  it follows from \cite[Prop. 5.38]{Kleiner:2014le} that $C\cap\M^s=\M^s$.  
Hence $C=Y$.
\end{proof}

Corollary~\ref{cor_chart_domains_cover} verifies that the collection of all the set of time-preserving family charts satisfy Property~\ref{prop_continuous_family_manifolds_2} of Definition~\ref{def_continuous_family_manifolds} if we drop the fourth entry ``$\t_V$''.
By Lemma~\ref{lem_cont_fam_no_max} this set can be extended to a maximal  collection of family charts.
All other properties of Definition~\ref{def_continuous_family_manifolds} and the fact that induced continuous structure on the set of time-$0$-slices $(\M^s_0)_{s \in X}$ coincides with that on $(M^s)_{s \in X}$ follow directly from our construction and from Lemmas~\ref{lem_fam_chart_is_fam_chart_time_0} and \ref{lem_compatible_fam_chart}.

\subsection{Proof of Theorem~\ref{thm_uniqueness_family_k}}
Consider a continuous family of singular Ricci flows $(\M^s)_{s \in X}$ and the associated continuous family of time-$0$-slices $(M^s, g^s)_{s \in X} := (\M^s_0, g^s_0)_{s \in X}$.
It suffices to show that in the proof of Theorem~\ref{thm_existence_family_k} the topology and the family charts on $Y = \sqcup_{s \in X} \M^s$ are uniquely determined by $(M^s, g^s)_{s \in X}$.
For this purpose we first show:

\begin{lemma} \label{lem_convergence_psi}
Let $\{ \psi_{s'}  \}_{s' \in S}$ be a family of almost-isometries near some $s \in X$.
Then $\psi_{s'} \to \id_{\M^s}$ in $C^\infty_{\loc}$ as $s' \to s$, in the sense that for any family chart $(U, \phi, V)$ we have
\begin{equation} \label{eq_psi_s_prime_to_id}
\proj_V \circ \phi \circ \psi_{s'} \xrightarrow{\quad C^\infty_{\loc} \quad} \proj_V \circ \phi \qquad  \text{as} \quad s' \to s.
\end{equation}
\end{lemma}

\begin{proof}
We say that $\psi_{s'} \to \id_{\M^s}$ near some point $p \in \M^s$ if  (\ref{eq_psi_s_prime_to_id}) holds near $p$ for some (and therefore every) family chart $(U, \phi, V)$ with $p \in U$.
Furthermore, we say that $\psi_{s'} \to \id_{\M^s}$ near some point $p \in \M^s_t$ at time $t$ if (\ref{eq_psi_s_prime_to_id}) holds near $p$ in $\M^s_t$ for any family chart $(U,\phi,V)$, with $p \in U$, of the continuous family of time-$t$-slices $(\M^s_t)_{s \in X}$; compare with Lemma~\ref{lem_cont_fam_time_slice_inherit}.

\begin{Claim}
Assume that $p \in \M^s_t$ has the property that $\psi_{s'} \to \id_{\M^s}$ near $p$ at time $t$ and let $p' \in \M^s$ be another point.
\begin{enumerate}[label=(\alph*)]
\item If $p' = p (t')$ for some $t' \geq 0$, then $\psi_{s'} \to \id_{\M^s}$ near $p'$.
\item If $p' \in B(p,r)$ for $r < \injrad (\M^s_t, p)$, then $\psi_{s'} \to \id_{\M^s}$ near $p'$ at time $t$.
\end{enumerate}
\end{Claim}

\begin{proof}
This is a direct consequence of the fact that $\psi_{s'}^* \partial^{s'}_\t \to \partial^s_\t$ and $\psi_{s'}^* g^{s'} \to g^s$ in $C^\infty_{\loc}$ as $s' \to s$.
Compare also with the proof of Lemma~\ref{lem_propagating_charts}.
\end{proof}

By combining the claim with the proof of Corollary~\ref{cor_chart_domains_cover}, we obtain that $\psi_{s'} \to \id_{\M^s}$ everywhere.
\end{proof}

By Lemma~\ref{lem_convergence_psi} Definition~\ref{def_topology_on_Y} offers the correct description for the topology on $\cup_{s \in X} \M^s$.
Next, if $(U, \phi, V, \t_V)$ is a time-preserving family chart, in the sense of Definition~\ref{Def_time_pres_fam_chart}, then $(U, \phi, V)$ satisfies Properties~\ref{prop_continuous_family_manifolds_1} and \ref{prop_continuous_family_manifolds_3} of Definition~\ref{def_continuous_family_manifolds}.
By Lemma~\ref{lem_convergence_psi} and the proof of Lemma~\ref{lem_compatible_fam_chart}, $(U, \phi, V)$ is moreover compatible with all  family charts of $(\M^s)_{s \in X}$ in the sense of Property~\ref{prop_continuous_family_manifolds_4} of Definition~\ref{def_continuous_family_manifolds}.
Therefore, by maximality $(U, \phi, V)$ is a family chart of $(\M^s)_{s \in X}$.
This shows that the the topology and the family charts on $(\M^s)_{s \in X}$ are uniquely determined by $(M^s, g^s)_{s \in X}$, concluding the proof.

\subsection{Proof of Theorem~\ref{Thm_properness_fam_sing_RF}}
We first show:

\begin{lemma} \label{lem_psi_phi_K_converge}
If $(U, \phi, V)$ is a family chart of $(\M^s)_{s \in X}$, then for every $s \in \pi (U)$, every compact subset $K \subset V$ and every family of almost-isometries $\{ \psi_{s'} \}_{s' \in S}$ near $s$ there is a neighborhood $V' \times W' \subset V \times \pi (U)$ of $K \times \{ s \}$ such that for all $s' \in W'$ the maps $(\psi_{s'}^{-1} \circ \phi^{-1}) (\cdot, s')$ are well defined on $V'$ and converge to $\phi^{-1} (\cdot, s)$ in $C^\infty_{\loc}$ as $s' \to s$.
\end{lemma}

\begin{proof}
Via a covering argument, we can reduce the lemma to the case in which $K = \{ v \}$ consists of a single point.
Choose a time-preserving family chart $(U'', \phi'', V'', \t_{V''})$ with $(\phi'')^{-1} (v'', s) := \phi^{-1} (v,s) \in U''$.
Then by Property~\ref{prop_time_pres_fam_chart_6} of Definition~\ref{Def_time_pres_fam_chart} the assertion of the lemma holds if $(U,\phi,V)$ and $v$ are replaced by $(U'', \phi'', V'')$ and $v''$.
The lemma now follows due to the compatibility of the family charts $(U,\phi,V)$ and $(U'', \phi'', V'')$.
\end{proof}

By Theorem~\ref{Thm_rho_proper_sing_RF} the subset $K_0 := \{ \rho_{g^{s_0}} \geq r/10, \t^{s_0} \leq t + 1 \}   \subset \M^{s_0}$ is compact.
Therefore, by Lemma~\ref{lem_chart_near_compact_subset} there is a family chart $(U_0, \phi_0, V_0)$ of $(\M^s)_{s \in X}$ with $K_0 \subset U_0$.
Set $K \times \{ s_0 \} := \phi_0 ( K_0 )$.
By Lemma~\ref{lem_almost_isometries_existence} there is a family $\{ \psi_{s} : \M^{s_0} \supset Z_{s} \to Z'_{s} \subset \M^{s} \}_{s \in S}$ of almost isometries near $s_0$.
Let $\eps > 0$ be a constant whose value we will determine later.
By shrinking $U_0$, we may assume without loss of generality that $\pi (U_0) \subset S$ and that all maps $\psi_s$, $s \in \pi (U_0)$, are $\eps$-isometries.
If $\eps$ is chosen small enough, then $\rho_{g^s} < r$ on $(\{ \t^s \leq t+1 \} \cap \M^s) \setminus Z'_s$ and $\rho_{g^s} (\psi_s (x)) \leq 2\rho_{g^{s_0}} (x)$  for all $s \in S$ and $x \in Z_s$ with $\t (x) \leq t$.
We will now show that there is a neighborhood $W \subset \pi (U)$ of $s_0$ such that for all $s \in W$ we have
\begin{equation} \label{eq_rho_r2_psi_phi}
 \{ \rho_{g^{s_0}}  \geq r/2, \t^{s_0} \leq t \} \subset \psi_s^{-1} (  \phi_0^{-1} (K \times \{ s \} )). 
\end{equation}
Then $U := \pi^{-1} (W) \cap U_0$, $\phi :=  \phi_0 |_{U}$ and $V$ have the desired properties, since (\ref{eq_rho_r2_psi_phi}) implies that for all $s \in W$ we have
\[ \phi_0 ( \{ \rho_{g^s} \geq r, \t^s \leq t \} ) \subset \phi_0 \big( \psi_s ( \{ \rho_{g^{s_0}} \geq r/2, \t^{s_0} \leq t \}  ) \big)  \subset K \times \{ s \}. \]

To see that (\ref{eq_rho_r2_psi_phi}) is true let $K_1 \times \{ s_0 \} := \phi ( \{ \rho_{g^{s_0}}  \geq r/2, \t^{s_0} \leq t \} )$ and apply Lemma~\ref{lem_psi_phi_K_converge} for $(U_0, \phi_0, V_0)$.
We obtain a neighborhood $V' \times W' \subset V_0 \times \pi (U_0)$ of $K \times \{ s_0 \}$ such that for all $s \in W'$ the maps $(\psi^{-1}_s \circ \phi_0^{-1})(\cdot, s)$ are well defined on $V$ and converge to $\phi_0^{-1} (\cdot, s_0)$ in $C^\infty_{\loc}$ as $s \to s_0$.
So $\phi_0^{-1} (K \times \{s \}) \subset Z'_s$ for all $s \in W'$ and  for $s$ near $s_0$ the set $\phi_0^{-1} (K_1 \times \{ s_0 \})$ lies in the image of the map $(\psi^{-1}_s \circ \phi_0^{-1} )(\cdot, s)$.
This implies (\ref{eq_rho_r2_psi_phi}) for $s$ near $s_0$.

\section{Rounding process}
\label{sec_rounding_process}
\subsection{Introduction}
Consider a continuous family $(\M^s)_{s \in X}$ of singular Ricci flows over some topological space $X$.
By the canonical neighborhood assumption (see Definition~\ref{def_canonical_nbhd_asspt}) we know that regions of every $\M^s$ where the curvature scale $\rho$ is small are modeled on a $\kappa$-solutions, which are rotationally symmetric or have constant curvature.
The goal of this section is to perturb the metric $g^s$ on each $\M^s$ to a \emph{rounded} metric $g^{\prime,s}$, which is locally rotationally symmetric or has constant curvature wherever $\rho$ is small.
Our process will be carried out in such a way that the rounded metrics still depend continuously on $s$.

In addition to the rounded metrics $g^{\prime,s}$, we will also record the spherical fibrations consisting of the orbits of local isometric $O(3)$-actions wherever the metric is rotationally symmetric.
The precise structure that we will attach to each flow $\M^s$ will be called an \emph{$\RR$-structure} and will be defined in Subsections~\ref{subsec_spherical_struct} and \ref{subsec_RR_structure}.
In Subsection~\ref{subsec_main_rounding_statement} we will then state the main result of this section, followed by a proof in the remaining subsections.

\subsection{Spherical structures} \label{subsec_spherical_struct}
We first formalize a structure that is induced by a locally rotationally symmetric metric.
Let $M$ be a smooth manifold of dimension $n \geq 3$; in the sequel we will have $n\in \{3,4\}$.

\begin{definition}[Spherical structure] \label{Def_spherical_structure}
A {\bf spherical structure}  $\mathcal{S}$ on a subset $U \subset M$ of a smooth manifold with boundary $M$ is a smooth fiber bundle structure on an open dense subset $U' \subset U$ whose fibers are diffeomorphic to $S^2$ and equipped with a smooth fiberwise metric of constant curvature $1$ such that the following holds.
For every point $x \in U$ there is a neighborhood $V \subset U$ of $x$ and an $O(3)$-action $\zeta : O(3) \times V \to V$ such that $\zeta |_{V \cap U'}$ preserves all $S^2$-fibers and acts effectively and isometrically on them.
Moreover, all orbits on $V \setminus U'$ are not diffeomorphic to spheres.
Any such local action $\zeta$ is called a {\bf local $O(3)$-action compatible with $\mathcal{S}$}.
We call $U = \domain (\SS)$ the {\bf domain} of $\SS$.
\end{definition}

Consider an action $\zeta$ as in Definition~\ref{Def_spherical_structure}.
For any sequence $x_i \to x_\infty \in \mathcal{O}$ the corresponding sequence of orbits $\mathcal{O}_i$ converges to $\mathcal{O}$ in the Hausdorff sense.
As $U'$ is dense in $U$, this implies that $\mathcal{O}$ is independent of the choice of $\zeta$. 
So $\mathcal{O}$ is determined uniquely by a point $x \in \mathcal{O}$ and the spherical structure $\mathcal{S}$.
We call any such orbit $\mathcal{O} \subset U \setminus U'$ a {\bf singular (spherical) fiber} and any fiber in $U'$ a {\bf regular (spherical) fiber} of $\mathcal{S}$.

By analyzing the quotients of $O(3)$ we get:

\begin{lemma}
Any singular spherical fiber is either a point or is diffeomorphic to $\IR P^2$.
\end{lemma}

\begin{lemma} \label{lem_local_spherical_struct}
If $n = 3$, then $\zeta$ from Definition~\ref{Def_spherical_structure} is locally conjugate to one of the following models equipped with the standard $O(3)$-action:
\[ S^2 \times (-1,1), \quad \big( S^2 \times (-1,1) \big) / \IZ_2, \quad B^3, \quad S^2 \times [0, 1). \]
In the last case $S^2 \times \{ 0 \}$ corresponds to a boundary component of $M$.
In the second case there is a unique spherical structure on the local two-fold cover consisting only of regular fibers that extends the pullback via the covering map of the original spherical structure restricted to the union of the regular fibers.
\end{lemma}

Next we formalize the notion of a rotationally symmetric metric compatible with a given spherical structure.

\begin{definition}[Compatible metric] \label{Def_compatible_metric}
Let $U \subset M$ be an open subset.
A smooth metric $g$ on a subbundle of $TU$ or $TM$ is said to be {\bf compatible with a spherical structure $\mathcal{S}$} if near every point $x \in U$ there is a local $O(3)$-action that is compatible with $\mathcal{S}$ and isometric with respect to $g$.
\end{definition}

If $M = \M$ is a Ricci flow spacetime and $g$ is a smooth metric on the subbundle $\ker d\t$, then Definition~\ref{Def_compatible_metric} still makes sense.

In dimension 3 a metric $g$ compatible with a spherical structure $\SS$ can  locally be written as a warped product of the form $g = a^2(r) g_{S^2} + b^2(r) dr^2$ near the regular fibers.
However, note that a spherical structure does not record the splitting of the tangent space into directions that are tangential and orthogonal to the spherical fibers.
So the metrics $g$ compatible with $\SS$ depend on more data than the  warping functions $a(r), b(r)$.
Consider for example the quotient of the round cylinder $S^2 \times \IR$ by an isometry of the form $(x,r) \mapsto (A x, r+a)$, where $A \in O(3)$, $a > 0$.
The induced spherical structures are equivalent for all choices of $A, a$.

The following lemma classifies the geometry of metrics that are compatible with a spherical structure near regular fibers.

\begin{lemma} \label{lem_compatible_metric_general_form}
Let $u < v$ and consider the standard spherical structure $\SS$ on $M = S^2 \times (u, v)$, i.e. the structure whose fibers are of the form $S^2 \times \{ r \}$, endowed with the standard round metric $g_{S^2}$.
A Riemannian metric $g$ on $M$ is compatible with $\SS$ if and only if it is of the form
\begin{equation} \label{eq_compatible_form_g}
 g = a^2(r) g_{S^2} + b^2(r) dr^2 + \sum_{i=1}^3 c_i (r)  (dr \, \xi_i + \xi_i \, dr), 
\end{equation}
for some functions $a, b, c_1, c_2, c_3 \in C^\infty ((u,v))$ with $a,b > 0$.
Here $\xi_i := * dx^i$ denote the 1-forms that are dual to the standard Killing fields on $S^2 \subset \IR^3$.
\end{lemma}

\begin{proof}
Consider diffeomorphisms $\phi : M \to M$ of the form $\phi (v,r) = (A(r) v, r)$, where $A : (u,v) \to O(3)$ is smooth.
These diffeomorphisms leave $\SS$ invariant and if $g$ is of the form (\ref{eq_compatible_form_g}), then so is $\phi^* g$.
It follows that every metric $g$ of the form (\ref{eq_compatible_form_g}) is compatible with $\SS$.
On the other hand, assume that $g$ is compatible with $\SS$.
Fix some $r_0 \in (u,v)$ and consider the normal exponential map to $S^2 \times \{ r_0 \}$.
Using this map we can construct a diffeomorphism of the form $\phi (v,r) = (A(r) v, r)$ such that $\phi^* g = a^2 (r) g_{S^2} + b^2 (r) dr^2$.
Thus $g$ is of the form (\ref{eq_compatible_form_g}).
\end{proof}

Next we define what we mean by the preservation of a spherical structure by a vector field.

\begin{definition}[Preservation by vector field]
Let $U_1 \subset U_2 \subset M$ be open subsets and consider a vector field $X$ on $U_2$.
A spherical structure $\mathcal{S}$ on $U_1$ is said to be {\bf preserved} by  $X$ if the flow of $X$ preserves the regular and singular fibers as well as the fiberwise metric on the regular fibers.
\end{definition}

Lastly, we consider a continuous family of manifolds $(M^s)_{s \in X}$ of arbitrary dimension, which may for example be taken to be a continuous family of Ricci flow space times $(\M^s)_{s \in X}$.
We will define the notion of transverse continuity for spherical structures.
For this purpose, we need:

\begin{definition}[Transversely continuity for families of local $O(3)$-actions] \label{Def_transverse_O3_action}
Let $(V^s \subset M^s)_{s \in X}$ be a family of open subsets such that $V := \cup_{s \in X} V^s \subset \cup_{s \in X} M^s$ is open and consider a family of $O(3)$-actions $(\zeta^s : O(3) \times V^s \to V^s)_{s \in X}$, which we may also express as $\zeta : O(3) \times V \to V$.
We say that $\zeta^s$ is {\bf transversely continuous in the smooth topology} if for any $A_0 \in O(3)$, $s_0 \in X$, $x_0 \in V^{s_0}$ there are family charts $(U_0, \phi_0, V_0)$, $(U'_0, \phi'_0, V'_0)$ (in the sense of Definition~\ref{def_continuous_family_manifolds}) with $x_0 \in U_0$ and $\zeta (A_0, x_0) \in U'_0$ such that the map $(A, v, s) \mapsto (\proj_{V'_0} \circ \phi'_0 \circ \zeta) (A, \phi_0^{-1} (v,s)) )$ can be viewed as a family of maps in the first two arguments that depend continuously on $s$ in $C^\infty_{\loc}$ near $(A_0, \phi (x_0))$.
\end{definition}

Let now $(\SS^s)_{s \in X}$ be a family of spherical structures defined on a family of open subsets $(U^s := \domain (\SS^s) \subset M^s)_{s \in X}$.

\begin{definition}[Transverse continuity for spherical structures] \label{Def_spherical_struct_transverse_cont}
We say that $(\SS^s)_{s \in X}$ is {\bf transversely continuous} if:
\begin{enumerate}
\item $U := \cup_{s \in X} U^s$ is an open subset in the total space $\cup_{s \in X} M^s$.
\item For every point $x \in U$ there is an open neighborhood $V = \cup_{s \in X} V^s \subset U$ and a transversely continuous family of local $O(3)$-actions $(\zeta^s : O(3) \times V^s \to V^s)_{s \in X}$ that are each compatible with $\SS^s$.
\end{enumerate}
A family of spherical structures $(\SS^s)_{s \in X}$ on a fixed manifold $M$ is called transversely continuous, if it is transversely continuous on the associated continuous family of manifolds $(M \times \{ s \} )_{s \in X}$.
\end{definition}

\subsection{$\RR$-structures} \label{subsec_RR_structure}
Consider a singular Ricci flow $\M$.
We now define the structure that we will construct in this section.

\begin{definition}[$\mathcal{R}$-structure] \label{Def_R_structure}
An {\bf $\mathcal{R}$-structure} on a singular Ricci flow $\M$ is a tuple $\mathcal{R} = ( g', \partial'_\t, U_{S2}, U_{S3}, \mathcal{S})$ consisting of a smooth metric $g'$ on $\ker d\t$, a vector field $\partial'_\t$ on $\M$ with $\partial'_\t \, \t = 1$, open subsets $U_{S2}, U_{S3} \subset \M$ and a spherical structure $\mathcal{S}$ on $U_{S2}$ such that for all $t \geq 0$:
\begin{enumerate}[label=(\arabic*)]
\item \label{prop_def_RR_1} $U_{S3} \setminus U_{S2}$ is open.
\item \label{prop_def_RR_2} $U_{S2} \cap \M_t$ is a union of regular and singular fibers of $\mathcal{S}$.
\item \label{prop_def_RR_3} $\partial'_{\t}$ preserves $\mathcal{S}$.
\item \label{prop_def_RR_4} $g'_t$ is compatible with $\mathcal{S}$.
\item \label{prop_def_RR_5} $U_{S3} \cap \M_t$ is a union of compact components of $\M_t$ on which $g'_t$ has constant curvature.
\item \label{prop_def_RR_6} $U_{S3}$ is invariant under the forward flow of the vector field $\partial'_\t$, i.e. any trajectory of $\partial'_\t$ whose initial condition is located in $U_{S3}$ remains in $U_{S3}$ for all future times.
\item \label{prop_def_RR_7} The flow of $\partial'_\t$ restricted to every component of $U_{S3} \cap \M_t$ is a homothety with respect to $g'$, whenever defined.
\end{enumerate}
We say that the $\RR$-structure is {\bf supported on $U_{S2} \cup U_{S3}$.}
\end{definition}

Note that Property~\ref{prop_def_RR_1} is equivalent to the statement that any component of $U_{S3}$ is either contained in $U_{S2}$ or disjoint from it.

At this point we reiterate the importance of the fact that the spherical structure $\SS$ does not record the splitting of the tangent space into directions that are tangential and orthogonal to the spherical fibers.
Therefore the preservation of $\SS$ by $\partial'_{\t}$ does not guarantee a preservation of this splitting under the flow of $\partial'_{\t}$.
So the metrics $g'_t$ could be isometric to quotients of the round cylinder  $S^2 \times \IR$ by isometries of the form $(x,r) \mapsto (A_t x, r+a_t)$, where $A_t \in O(3)$, $a_t > 0$ may vary smoothly in $t$, and $\SS$ could consist of all cross-sectional 2-spheres.
In this case the flow of $\partial'_{\t}$ may leave the metric tangential to the fibers of $\SS$ invariant while distorting the metric in the orthogonal and mixed directions.

Consider now a continuous family $(\M^s)_{s \in X}$ of singular Ricci flows.
For every $s \in X$ choose an $\RR$-structure
\[ \mathcal{R}^s = ( g^{\prime, s}, \lb \partial^{\prime, s}_\t,  \lb U^s_{S2}, \lb U^s_{S3}, \lb \mathcal{S}^s)  \]
on $\M^s$.
We define the following notion of transverse continuity:

\begin{definition}[Transverse continuity for $\RR$-structures] \label{Def_RR_structure_transverse_cont}
The family of $\RR$-structures $( \mathcal{R}^s  )_{s \in X}$ is called {\bf transversely continuous} if:
\begin{enumerate}
\item $(g^{\prime, s})_{s \in X}, (\partial^{\prime, s}_\t)_{s \in X}$ are transversely continuous in the smooth topology.
\item $U_{S2} := \cup_{s \in X} U_{S2}^s$ and $U_{S3} := \cup_{s \in X} U_{S3}^s$ are open subsets of the total space $\cup_{s \in X} \M^s$.
\item $(\mathcal{S}^s)_{s \in X}$ is transversely continuous in the sense of Definition~\ref{Def_spherical_struct_transverse_cont}.
\end{enumerate}
\end{definition}

\subsection{Statement of the main result} \label{subsec_main_rounding_statement}
The main result of this section, Theorem~\ref{Thm_rounding}, states that for every continuous family of singular Ricci flows there is a transversely continuous family of $\RR$-structures supported in regions where $\rho$ is small.
Moreover the metrics $g^{\prime,s}$ and $g^s$ will be close in some scaling invariant $C^{[\delta]}$-sense and equal where $\rho$ is bounded from below.
The same applies to the vector fields $\partial^{\prime,s}_\t$ and $\partial^s_\t$.

Recall the scale $r_{\initial}$ from Definition~\ref{Def_r_initial}.

\begin{theorem}[Existence of family of $\RR$-structures] \label{Thm_rounding}
For any $\delta > 0$ there is a constant $C = C(\delta) < \infty$ and a continuous, decreasing function $r_{\rot, \delta} : \R_+ \times [0, \infty) \to \IR_+$ such that the following holds.

Consider a continuous family $(\M^s)_{s \in X}$ of singular Ricci flows.
Then there is a transversely continuous family of $\RR$-structures $(\mathcal{R}^s = ( g^{\prime, s}, \lb \partial^{\prime, s}_\t,  \lb U^s_{S2}, \lb U^s_{S3}, \lb \mathcal{S}^s))_{s \in X}$ such that for any $s \in X$:
\begin{enumerate}[label=(\alph*)]
\item \label{ass_thm_rounding_a} $\RR^s$ is supported on \[ \big\{ x \in \M^s \;\; : \;\; \rho_{g^{\prime,s}} (x)< r_{\rot, \delta} (r_{\initial} (\M_0^s,g^s_0), \t(x)) \big\}. \]
\item \label{ass_thm_rounding_b} $g^{\prime, s} = g^s$ and $\partial^{\prime,s}_\t = \partial^s_\t$ on 
\[ \big\{ x \in \M^s \;\; : \;\; \rho_{g^{\prime,s}} (x) > C r_{\rot, \delta} (r_{\initial} (\M_0^s, g^s_0), \t(x)) \big\} \supset \M^s_0. \]
\item \label{ass_thm_rounding_c}  For $m_1, m_2 = 0, \ldots,  [\delta^{-1}]$ we have
\[ | \nabla^{m_1} \partial_{\t}^{m_2} (g^{\prime, s} - g^s) | \leq \delta \rho^{-m_1-2m_2}, \qquad | \nabla^{m_1} \partial_{\t}^{m_2} (\partial^{\prime, s}_\t - \partial^s_\t) | \leq \delta \rho^{1-m_1-2m_2}. \]
\item \label{ass_thm_rounding_d} If $(\M^s_0, g^s_0)$ is homothetic to a quotient of the round sphere or the round cylinder, then $g^{\prime, s} = g^s$ and $\partial^{\prime,s}_\t = \partial^s_\t$ on all of $\M^s$.
\item \label{ass_thm_rounding_e} $r_{\rot, \delta} (a \cdot r_0, a^2 \cdot t) = a \cdot r_{\rot, \delta} (r_0,t)$ for all $a, r_0 > 0$ and $t \geq 0$.
\end{enumerate}
\end{theorem}

The proof of this theorem, which will occupy the remainder of this section, is carried out in several steps:
\begin{enumerate}[label=\arabic*.]
\item In Subsection~\ref{subsec_cutoff}, we first define a cutoff function $\eta : \cup_{s \in X} \M^s \to [0,1]$ whose support is contained in the union of components of time-slices with positive sectional curvature and bounded normalized diameter.
By Hamilton's result \cite{hamilton_positive_ricci}, these components become extinct in finite time and the metric becomes asymptotically round modulo rescaling as we approach the extinction time.
On the other hand, time-slices on which $\eta < 1$ have sufficiently large normalized diameter such that any point that satisfies the canonical neighborhood assumption is either contained in an $\eps$-neck or has a neighborhood that is sufficiently close to a Bryant soliton.
\item In Subsection~\ref{subsec_modify_Bryant}, we modify the metrics $(g^s)_{s \in X}$ at bounded distance to points where the geometry is close to the tip of a Bryant soliton.
The resulting metric will be compatible with a spherical structure $\SS_2$ near these points.
So any point $x \in \{ \eta < 1 \} \setminus \domain (\SS_2)$ of small curvature scale must be a center of an $\eps$-neck.
Our rounding procedure will employ the exponential map based at critical points of the scalar curvature $R$.
\item In Subsection~\ref{subsec_modify_necks}, we modify the metrics from the previous step near centers of $\eps$-necks, using the (canonical) constant mean curvature (CMC) foliation by 2-spheres.
The resulting metrics will be compatible with a spherical structure $\SS_3$ extending $\SS_2$, whose domain contains all points $x \in \{ \eta < 1 \}$ of small curvature scale.
\item In Subsection~\ref{subsec_modify_dt}, we  use an averaging procedure to modify the time vector fields $\partial^s_\t$ so that they preserve the spherical structure $\SS_3$.
\item In Subsection~\ref{subsec_extend_to_almost_round}, we modify the metric on the support of $\eta$ such that it is compatible with an extension of $\SS_3$.
We obtain this new metric by evolving the metric from Step 3 forward in time by the Ricci flow.
Therefore the new metric will remain compatible with a spherical structure up to some time that is close to the extinction time.
After this time the metric is $\delta$-close to a round spherical space form modulo rescaling, for some small $\delta > 0$.
\item In Subsection~\ref{subsec_proof_rounding}, we replace the almost round metrics near the extinction time by canonical metrics of constant curvature.
We also modify the time-vector fields to respect these new metrics.
This concludes the proof of Theorem~\ref{Thm_rounding}.
\end{enumerate}
Throughout this section we will attach indices 1--6 to the objects that are constructed in each step.
For example, the cutoff function from Step 1 will be called $\eta_1$.
The main result in Step 2 will repeat the claim on the existence of this cutoff function, which is then called $\eta_2$, and so on.
In order to avoid confusion we will increase the index of each object in each step, even if it remained unchanged during the construction.

\subsection{The rounding operators} \label{subsec_RD}
We will define operators $\RD^n$ that assign a constant curvature metric to every metric of almost constant curvature in a canonical way.

Consider first an $n$-dimensional Riemannian manifold $(M,g)$ that is $\delta$-close to the standard $n$-dimensional round sphere $(S^n, g_{S^n})$ in the smooth Cheeger-Gromov topology for some small $\delta > 0$, which we will determine later.
Consider the eigenvalues $0 = \lambda_0 \leq \lambda_1 \leq \ldots$ of the Laplacian on $(M,g)$, counted with multiplicity.
Recall that on the standard round $n$-sphere $S^n$ we have $\lambda_0 = 0$ and $\lambda_1 = \ldots = \lambda_{n+1} = n < \lambda_{n+2}$, where the coordinate functions of the standard embedding $S^n \subset \IR^{n+1}$ form an orthogonal basis of the eigenspace corresponding to the eigenvalue $n$.
So if $\delta$ is chosen sufficiently small, then $\lambda_{n+1} < \lambda_{n+2}$ and there is an $L^2$-orthonormal system $x^0_g, x^1_g, \ldots, x^{n+1}_g$ of eigenfunctions such that $x^0 \equiv const$, $\beta_g : (x^1_g, \ldots, x^{n+1}_g) : M \to \IR^{n+1} \setminus \{ 0 \}$ and $\alpha_g := \beta_g / |\beta_g| : M \to S^n$ is a diffeomorphism.
Define 
\[ \td\RD^n ( g ) := \alpha_g^* g_{S^n}. \]
Note that $(M, \td\RD^n ( g ))$ is isometric to $(S^n, g_{S^n})$, in particular $V (M,\td\RD^n ( g )) = V(S^n, g_{S^n})$.

Next, assume that $(M,g)$ is $\delta$-close to $(S^n, g_{S^n})$ modulo rescaling, in the sense that $(M, a^2 \cdot g)$ is $\delta$-close to $(S^n, g_{S^n})$ for some $a > 0$.
If $\delta$ is sufficiently small, then $(M, (V (M,g) / V(S^n, g_{S^n} ))^{-2/n}  g)$ is sufficiently close to $(S^n, g_{S^n})$, such that we can define
\[ \RD^n (g) :=\bigg( \frac{V (M,g)}{V(S^n, g)} \bigg)^{2/n}  \td\RD^n \bigg(  \bigg( \frac{V (M,g)}{V(S^n, g)} \bigg)^{-2/n}  g \bigg). \]
Then
\[ V (M, \RD^n(g)) = V(M,g) \]
and for any diffeomorphism $\phi : M \to M$ we have
\[ \RD^n (\phi^* g) = \phi^* \RD^n (g). \]
So if $\phi$ is an isometry with respect to $g$, then it is also an isometry with respect to $\RD^n (g)$.
Hence if $(M,g)$ is a Riemannian manifold whose universal cover $\pi : (\td{M}, \td{g}) \to (M,g)$ is $\delta$-close to $(S^n, g_{S^n})$ modulo rescaling, then we can define $\RD^n (g)$ as the unique metric with the property that $\pi^* \RD^n (g) = \RD^n (\td{g})$.

We record:

\begin{lemma}
If $\delta \leq \ov\delta$ and if the universal cover of $(M,g)$ is $\delta$-close to $(S^n, g_{S^n})$ modulo rescaling, then:
\begin{enumerate}[label=(\alph*)]
\item $\RD^n (g)$ has constant curvature.
\item $V (M, \RD^n(g)) = V(M,g)$.
\item If $g$ has constant curvature, then $\RD^n (g) = g$.
\item $\RD^n (\phi^* g) = \phi^* \RD^n (g)$ for any diffeomorphism $\phi : M \to M$.
\item If $g_i \to g$ in the smooth topology, then $\RD^n (g_i) \to \RD^n (g)$ in the smooth topology.
\item If $Z$ is a smooth manifold and $(g_z)_{z \in Z}$ is a smooth family of metrics on $M$ such that $(M,g_z)$ is $\delta$-close to $(S^n, g_{S^n})$ modulo rescaling for each $z$, then $(\RD^n (g_z))_{z \in Z}$ is smooth.
Moreover, if $(g^i_z)_{z \in Z} \to (g_z)_{z \in Z}$ locally smoothly, then also $((\RD^n (g^i_z))_{z \in Z} \to ((\RD^n (g_z))_{z \in Z}$.
\end{enumerate}

\end{lemma}

\subsection{Conventions and terminology} \label{subsec_rounding_conventions}
For the remainder of this section let us fix a continuous family $(\M^s)_{s \in X}$ or singular Ricci flows.
It will be clear from our construction that all constants will only depend on the data indicated in each lemma and not on the specific choice of the family $(\M^s)_{s \in X}$.

Consider the initial condition scale $r_{\initial}$ from Definition~\ref{Def_r_initial} and the canonical neighborhood scale $r_{\can, \eps} : \IR_+ \times [0, \infty) \to \IR_+$ from Lemma~\ref{lem_rcan_control}.
Then for any $x \in \M^s_t \subset \cup_{s' \in X} \M^{s'}$ the assertions of Lemma~\ref{lem_rcan_control} hold below scale $\td{r}_{\can, \eps} (x) := r_{\can, \eps} (r_{\initial} (\M^s_0, g^s_0), t)$.
For ease of notation, we will drop the tilde from $\td{r}_{\can, \eps}$ and view $r_{\can, \eps}$ as a function of the form
\[ r_{\can, \eps} :  \cup_{s \in X} \M^{s} \longrightarrow \IR_+. \]
Note that $r_{\can, \eps}$ is smooth on each fiber $\M^s$ and transversely continuous in the smooth topology.

Next, we define a scale function $\wh\rho$ on each $\M^s$, which is comparable to $\rho$ and constant on time-slices.
For this purpose, let $x \in \M^s_t$ and consider the component $\C \subset \M^s_t$ containing $x$.
Let $(\td{\C}, \td{g}^s_t)$ be the universal cover
of $(\C, g^s_t |_\C )$ and set
\[ \wh\rho (x) := V^{1/n} (\td{\C}, \td{g}^s_t) \in (0, \infty]. \]
We will frequently use the following fact, which is a direct consequence of the compactness of $\kappa$-solutions (see Theorem~\ref{Thm_kappa_compactness_theory}):

\begin{lemma}
Assume that $\eps \leq \ov\eps (D)$.
If $\rho (x) < r_{\can, \eps} (x)$ and if the component $\C \subset \M^s_t$ containing $x$ has $\diam \C \leq D \rho (x)$, then
\[ C^{-1} (D) \rho (x) \leq \wh\rho (x) \leq C(D) \rho (x). \]
\end{lemma}

In the following we will successively construct metrics $g^s_{1}, \ldots, g^s_5$ on $\M^s$.
Unless otherwise noted, we will compute the quantities $\rho$ and $\wh\rho$ using the original metric $g^s$.
Otherwise, we indicated the use of another metric by a subscript, such as ``$\rho_{g^{\prime}_i}$'' or ``$\wh\rho_{g^{\prime}_i}$''.

Lastly, let us fix a smooth, non-decreasing cutoff function $\nu : \IR \to [0,1]$ such that
\[ \nu \equiv 0 \quad \text{on} \quad (-\infty, .1], \qquad \nu \equiv 1 \quad \text{on} \quad [.9, \infty). \]

 \subsection{Construction of a cutoff function near almost extinct components} \label{subsec_cutoff}
 Our first goal will be to introduce a cutoff function $\eta_1$ that is supported on components on which the curvature scale is small and the renormalized diameter is bounded.
On these components we have $\sec > 0$ and therefore these components become extinct in finite time.
On the other hand, regions where the curvature scale is small and where $\eta_1 < 1$ are modeled on a Bryant soliton or a cylinder.

\begin{lemma}
 \label{lem_cutoff_almost_extinct}
For every $\delta^*> 0$, $N \in \IN$, and assuming $\alpha \leq \ov\alpha(\delta^*)$, $D \geq \underline{D} (\delta^*), D_0 \geq \underline{D}_0 (\delta^*)$, $C_m \geq \underline{C}_m(\delta^*)$  and $\eps \leq \ov\eps (\delta^*, N)$, there is:
\begin{itemize}
\item a continuous function $\eta_1 : \cup_{s \in X} \M^s \to [0,1]$ that is smooth on each fiber $\M^s$ and transversely continuous in the smooth topology,
\end{itemize}
such that for any $s \in X$ and $t \geq 0$:
\begin{enumerate}[label=(\alph*)]
\item  \label{ass_cutoff_a} $\partial^s_\t \eta_1 \geq 0$.
\item \label{ass_cutoff_b} $\eta_1$ is constant on every connected component of $\M^{s}_t$.
\item  \label{ass_cutoff_c} Any connected component $\C \in \M^{s}_t$ on which $\eta_1  > 0$ is compact there is a time $t_\C > t$ such that:
\begin{enumerate}[label=(c\arabic*)]
\item  \label{ass_cutoff_c1} $\sec > 0$ on $\C(t')$ for all $t' \in [t, t_\C)$.
\item \label{ass_cutoff_c2} $\C$ survives until time $t'$ for all $t' \in [t, t_\C)$ and no point in $\C$ survives until or past time $t_\C$.
\item \label{ass_cutoff_c3}  $\diam \C(t') < D \rho (x)$ for all $x \in \C(t')$ and $t' \in [t, t_\C)$.
\item \label{ass_cutoff_c4}  $\rho < 10^{-1} r_{\can, \eps}$ on $\C (t')$ for all $t' \in [t, t_\C)$.
\item \label{ass_cutoff_c5}  The time-slices $(\C(t'), g^s_t)$ converge, modulo rescaling, to an isometric quotient of the round sphere as $t' \nearrow t_\C$.
\end{enumerate}
\item  \label{ass_cutoff_d} For every point $x \in \M^s_t$ with $\eta_1 (x) < 1$ (at least) one of the following is true:
\begin{enumerate}[label=(d\arabic*)]
\item \label{ass_cutoff_d1} $\rho (x) > \alpha r_{\can, \eps} (x)$.
\item \label{ass_cutoff_d2} $x$ is a center of a $\delta^*$-neck.
\item \label{ass_cutoff_d3} There is an open neighborhood of $x$ in $\M^s_t$ that admits a two-fold cover in which a lift of $x$ is a center of a $\delta^*$-neck.
\item \label{ass_cutoff_d4} There is a point $x' \in \M^s_t$ with $d (x, x') < D_0 \rho (x')$ such that $\rho(x') < 10^{-1} r_{\can, \eps} (x')$, $\nabla R (x') = 0$ and $(\M^s_t, g^s_t, x')$ is $\delta^*$-close to $(M_{\Bry},\lb g_{\Bry},\lb x_{\Bry})$ at some scale.
Moreover, the component $\C \subset \M^s_t$ containing $x$ has diameter $> 100 D_0 \rho (x')$.
\end{enumerate}
\item  \label{ass_cutoff_e} $| \partial_\t^m \eta_1 | \leq C_m \rho^{-2m}$ for  $m = 0, \ldots, N$.
\end{enumerate}
\end{lemma}

Note that only $\eps$ depends on $N$.

The strategy in the following proof is to define $\eta_1$ using two auxiliary functions $\eta^*_1$, $\eta^{**}_1$, which are essentially supported in regions where $\rho$ is small and where the normalized diameter is bounded, respectively.
In order to achieve the monotonicity property \ref{ass_cutoff_a}, we define $\eta_1$ by integrating the product $\eta^*_1 \eta^{**}_1$ in time against a weight.

\begin{proof}
Let $a, \alpha', A, D, D'$ be constants, which will be chosen depending on $\delta^*$ in the course of the proof.

We say that  a component $\C \subset \M^s_t$ is \emph{$D'$-small} if 
\[ \diam \C < D' \rho \quad \text{and} \quad  \rho < 10^{-1} r_{\can, \eps} \qquad \text{on} \quad \C. \]
Let $U_{D'} \subset \cup_{s \in X} \M^s$ be the union of all $D'$-small components.

\begin{Claim} \label{cl_UD_prime}
\begin{enumerate}[label=(\alph*)]
\item \label{ass_UD_prime_a} $U_{D'}$ is open in $\cup_{s \in X} \M^s$.
\item \label{ass_UD_prime_b} Assuming $\alpha' \leq \ov\alpha' (D')$, $D \geq \underline{D} (D')$ and $\eps \leq \ov\eps (D')$, the following is true.
If $\C \subset \M^s_t$ is a $D'$-small component and $\rho (x) \leq \alpha' r_{\can, \eps} (x)$ for some $x \in \C$, then there is a time $t_\C > t$ such that Assertions~\ref{ass_cutoff_c1}--\ref{ass_cutoff_c5} of this lemma hold.
\end{enumerate}
\end{Claim}

\begin{proof}
Assertion~\ref{ass_UD_prime_a} is clear by definition.
If Assertion~\ref{ass_UD_prime_b} was false for some fixed $D'$, then we can find singular Ricci flows $\M^i$, $D'$-small components $\C^i \subset \M^i_{t^i}$ and points $x^i \in \C^i$ contradicting the assertion for sequences $\alpha^{\prime, i}, \eps^i \to 0$, $D^i \to \infty$.
Since the $\eps^i$-canonical neighborhood assumption holds at $x^i$ and $\eps^i \to 0$, we may assume, after passing to a subsequence, that the universal covers of $(\C^i, \rho^{-2} (x^i) g^i_{t^i})$ converge to a compact smooth limit $(\ov{M}, \ov{g})$, which is the final time-slice of a $\kappa$-solution.
Therefore, there is a $c > 0$ such that $\sec > c R$ and $R > c \rho^{-2} (x^i)$ on $\C^i$ for large $i$.
By Hamilton's result \cite{hamilton_positive_ricci}, for large $i$ the flow past $\C^i$ becomes asymptotically round and goes extinct at a finite time $t_{\C^i} \in ( t^i, t^i + C \rho^2 (x^i))$ for some universal constant $C = C(c) < \infty$.
 Since $|\partial_\t r_{\can, \eps^i}| \leq \eps^i r^{-1}_{can, \eps^i}$ by Lemma~\ref{lem_rcan_control} we obtain that $r_{\can, \eps^i} (x^i(t')) \leq 2 r_{\can, \eps^i} (x^i)$ for all $t' \in [t^i, t_{\C^i})$ and large $i$.
 So Assertions~\ref{ass_cutoff_c1}, \ref{ass_cutoff_c2}, \ref{ass_cutoff_c4}, \ref{ass_cutoff_c5} hold for large $i$.
 In particular, for large $i$ the $\eps^i$-canonical neighborhood assumption holds on $\C^i (t')$ for all $t' \in [t^i, t_{\C^i})$.
 Since the pinching $\sec > c R$ is preserved by the flow, another limit argument implies Assertion~\ref{ass_cutoff_c3}.
 \end{proof}

Define functions $\eta^*_1, \eta^*_2 : \cup_{s \in X} \M^s \to [0,1]$ as follows.
For any $x \in \M^{s}_t$ let $\C \subset \M^s_t$ be the component containing $x$ and set
\[ \eta_1^* (x) = \nu \bigg(a \cdot \frac{r_{\can, \eps}(x)}{\wh\rho(x)}  \bigg), \qquad \eta_1^{**} (x) = \nu \bigg( A \cdot \frac{\wh\rho(x)}{ \mathcal{R} (\td\C)}  \bigg), \]
where $\mathcal{R} (\td\C) := \int_{\td\C} R_{g_t} d\mu_{g_t}$ and $\nu$ is the cutoff function from Subsection~\ref{subsec_rounding_conventions}.
Near any component $\C \subset \M^s_t$ with compact universal cover $\td\C$ the functions $\eta^*_1, \eta^*_2$ are smooth on $\M^s_t$ and transversely continuous in the smooth topology.

We now define functions $\eta'_1 : \cup_{s \in X} \M^s \to [0,\infty)$  and $\eta_1 : \cup_{s \in X} \M^s \to [0,1]$ as follows.
Set $\eta'_1 :\equiv 0$ on $\cup_{s \in X} \M^s \setminus U_{D'}$.
For every component $\C \subset \M^s_t$ with $\C \subset U_{D'}$ and for every $x \in \C \subset \M^s_t$  choose $t^*_\C < t$ minimal such that $\C$ survives until all times $t' \in (t^*_\C, t]$ and let
\[ \eta'_1(x) :=  \int_{t^*_\C}^t \frac{\eta^*_1 (x(t')) \eta^{**}_1 (x(t'))}{\wh\rho^2 (x(t'))} dt'. \]
Lastly, set
\[ \eta_1 (x) := \nu \big(A \cdot \eta'_1 (x) \big). \]
Note that the definitions of $\eta'_1, \eta_1$ are invariant under parabolic rescaling, but the definition of $\eta^*_1$ is not invariant under time shifts.
Assertions~\ref{ass_cutoff_a} and \ref{ass_cutoff_b} of this lemma hold wherever $\eta_1$ is differentiable.
Moreover, $\partial_\t \eta'_1$ restricted to $U_{D'}$ is smooth on each fiber $\M^s$ and transversely continuous in the smooth topology, because the same is true for $\eta^*_1, \eta^{**}_1, \wh\rho$.

\begin{Claim} \label{cl_eta1_regularity}
If $D' \geq \underline{D}' (A)$, $a \leq \ov{a} (\alpha', A, D')$, $\eps \leq \ov\eps (\alpha', A,a,D', N)$, then 
\begin{enumerate}[label=(\alph*)]
\item \label{ass_eta1_regularity_a} The closure of the support of $\eta^*_1 \eta^{**}_2 |_{U_{D'}}$ in $\cup_{s \in X} \M^s$ is contained in $U_{D'} \cap \{ \rho \leq \alpha' r_{\can, \eps} \}$.
\item \label{ass_eta1_regularity_b} $\eta_1$, $\eta'_1$ are smooth on every fiber $\M^s$ and transversely continuous in the smooth topology.
\item \label{ass_eta1_regularity_bb} Assertion~\ref{ass_cutoff_c} of this lemma holds.
\item \label{ass_eta1_regularity_d} Assertion~\ref{ass_cutoff_e} of this lemma holds for some constants $C_m = C_m (\alpha', \lb A, \lb a, \lb D') \lb < \infty$.
\end{enumerate}
\end{Claim}

\begin{proof}
Let us first show that Assertions~\ref{ass_eta1_regularity_b} and \ref{ass_eta1_regularity_bb} of this claim follow from Assertion~\ref{ass_eta1_regularity_a}.
For Assertions~\ref{ass_eta1_regularity_b} observe that Assertion~\ref{ass_eta1_regularity_a} implies smoothness and transverse continuity of $\partial_\t \eta'_1$.
For any $D'$-small component $\C$ we have $\C(t') \not\subset U_{D'}$ for $t' \in (t^*_\C, t)$ close enough to $t^*_\C$, because otherwise $\C$ would survive until time $t^*_\C$.
So every neighborhood of a point in $U_{D'}$ can be evolved backwards by the flow of $\partial_\t$ into $\{ \eta'_1 = 0 \}$.
Assertion~\ref{ass_eta1_regularity_b} now follows by integrating  $\partial_\t \eta'_1$ in time.
For Assertion~\ref{ass_eta1_regularity_bb} observe that whenever $\eta_1 (x) > 0$, then there is a $t' \leq t$ such that $x(t') \in U_{D'}$ and $\eta^*_1 (x(t')) \eta^{**}_1 (x(t')) > 0$.
By Assertion~\ref{ass_eta1_regularity_a} we have $\rho (x(t')) \leq \alpha' r_{\can, \eps} (x(t'))$, which implies Assertion~\ref{ass_eta1_regularity_bb} via  Claim~\ref{cl_UD_prime}\ref{ass_UD_prime_b}.

It remains to prove Assertions~\ref{ass_eta1_regularity_a} and \ref{ass_eta1_regularity_d} of this claim.
For this purpose fix $A > 0$ and assume that $D' \geq \underline{D}' (A)$, such that if $(\ov{M}, \ov{g})$ is the final time-slice of a compact, simply-connected $\kappa$-solution and for some $\ov{x} \in \ov{M}$
\[ A \cdot \frac{\wh\rho(\ov{x}, 0)}{\mathcal{R} (\ov{M}, \ov{g}_{0})} \geq \frac1{10}, \]
then $\diam ( \ov{M}, g_0) <  D' \rho (\ov{x})$.
This is possible due to Lemma~\ref{lem_kappa_identity_1}.

Assume now that Assertion~\ref{ass_eta1_regularity_a} was false for fixed $\alpha', A, D'$ .
Choose sequences $a^i,\eps^i \to 0$.
Then we can find singular Ricci flows $\M^i$ and points $x^i \in \M^i_{t^i}$ that are contained in the closure of $U_{D'}$ and the support of $\eta^*_1 \eta^{**}_1$, but $x^i \not\in U_{D'}$ or $\rho (x^i) > \alpha' r_{\can, \eps^i} (x^i)$.
Let $\C^i \subset \M^i_{t^i}$ be the component containing $x^i$.
Since $x^i$ is contained in the closure of $U_{D'}$ we have
\[ \diam \C^i \leq D \rho (x^i), \qquad \rho (x^i) \leq 10^{-1} r_{\can, \eps^i} (x^i). \]
This implies that $x^i$ satisfies the $\eps^i$-canonical neighborhood assumption and therefore, after passing to a subsequence, the universal covers $(\td\C^i, \rho^{-2} (x^i) g^i_{t^i}, x^i)$ converge to a compact smooth pointed limit $(\ov{M}, \ov{g}, \ov{x})$, which is the final time-slice of a $\kappa$-solution.

Next, since $x^i$ is contained in the support of $\eta^*_1 \eta^{**}_1$ we have
\begin{equation} \label{eq_eta_limits}
 \liminf_{i \to \infty} a^i \cdot \frac{r_{\can, \eps^i}(x^i)}{\wh\rho(x^i)} \geq \frac1{10}, \qquad
 \liminf_{i \to \infty} A \cdot \frac{\wh\rho(x^i)}{ \mathcal{R} (\td\C^i)}  \geq \frac1{10}. 
\end{equation}
By our choice of $D'$, the second bound implies $\diam (\ov{M} , \ov{g}_0) < D' \rho (\ov{x})$.
So for large $i$ we have $\diam \C^i < D' \rho (x^i)$.
Since $\lim_{i \to \infty} \wh\rho (x^i) / \rho (x^i) = V^{1/3} (\ov{M}, \ov{g}_0) > 0$, the first bound of (\ref{eq_eta_limits}) implies that $\liminf_{i \to \infty} r_{\can, \eps^i} (x^i) / \rho (x^i) = \infty$.
This implies that $\C^i \subset U^i_{D'}$ and $\rho (x^i) \leq \alpha' r_{\can, \eps^i} (x^i)$  for large $i$, which contradicts our assumptions.

Lastly assume that Assertion~\ref{ass_eta1_regularity_d} of this claim was false.
Then we can find a sequence of counterexamples as before, but this time for fixed $\alpha', A, a, D'$ and $\eps^i \to 0$ such that $\partial^{m_0}_\t \eta'_1 (x^i) \rho^{2m_0} (x^i) \to \infty$ for some fixed $m_0 \geq 1$.
So for large $i$ the point $x^i$ must lie in the support of $\eta^*_1 \eta^{**}_2 |_{U_{D'}}$, because otherwise $\eta'_1$ would be constant near $x^i$.
As before, we can pass to a pointed, compact, simply-connected $\kappa$-solution $(\ov{M}, \ov{g}, \ov{x})$ such that for any $m \geq 0$ we have 
\[ \partial^m_\t \eta^{**}_1 (x^i) \cdot \rho^{2m} (x^i) \longrightarrow \partial^m_t \eta^{**}_1 (\ov{x}), \qquad  \partial^m_\t \wh\rho (x^i) \cdot \rho^{-1+2m} (x^i) \longrightarrow \partial^m_t \wh\rho (\ov{x}). \]
Moreover, by Lemma~\ref{lem_rcan_control} we have for $m \geq 1$
\[ \limsup_{i \to \infty} \big|\partial^m_\t    r_{\can, \eps^i} (x^i) \big| \cdot \rho^{-1+2m} (x^i)  \leq 
\limsup_{i \to \infty}   \eps^i r_{\can, \eps^i}^{1-2m} (x^i) \cdot \rho^{-1+2m} (x^i) = 0, \]
 which implies that for $m \geq 1$
 \[ \partial^m_\t \eta^{*}_1 (x^i) \cdot \rho^{2m} (x^i) \longrightarrow  0. \]
 By combining these estimates, we obtain a contradiction to our assumption that $\partial^{m_0}_\t \eta'_1 (x^i) \rho^{2m_0} (x^i) \to \infty$.
 \end{proof}
 
 \begin{Claim} \label{cl_choice_of_A}
If $\alpha \leq c(\delta^*)  a$, $A \geq \underline{A} (\delta^*)$, $D' \geq \underline{D}' (\delta^*)$, $D_0 \geq \underline{D}_0 (\delta^*)$, $\eps \leq \eps (\delta^*)$, then Assertion~\ref{ass_cutoff_d} of this lemma holds.
\end{Claim}

\begin{proof}
Assume that the claim was false for some fixed $\delta^*$.
Then we can find singular Ricci flows $\M^i$ and points $x^i \in \M^i_{t^i}$ such that $\eta_1 (x^i) < 1$ and all Assertions~\ref{ass_cutoff_d1}--\ref{ass_cutoff_d4} are false, for parameters $a^i, A^i, \alpha^i \leq c^i  a^i$, $D^i, D^{\prime, i}, D^i_0$, $\eps^i$ that satisfy the bounds of Claims~\ref{cl_UD_prime} and \ref{cl_eta1_regularity} and $\alpha^i, c^i, \eps^i \to 0$ and $A^i, D^{\prime, i}, D^i_0 \to \infty$.
Let $\C^i \subset \M^i_{t^i}$ be the component containing $x^i$ and observe that $\eta_1(x^i) < 1$ implies $\eta'_1 (x^i) <  (A^i)^{-1}$.

Since Assertion~\ref{ass_cutoff_d1} is violated we have
\[ \rho (x^i) \leq \alpha^i \cdot r_{\can, \eps^i} (x^i) \leq c^i a^i \cdot r_{\can, \eps^i} (x^i) < r_{\can, \eps^i} (x^i) \]
for large $i$, which implies that the points $x^i$ satisfy the $\eps^i$-canonical neighborhood assumption.
So after passing to a subsequence, we may assume that after parabolic rescaling by $\rho^{-2} (x^i)$ the universal covers of the flows restricted to larger and larger backwards parabolic neighborhoods of $x^i$ converge to a smooth, pointed $\kappa$-solution $(\ov{M}, (\ov{g})_{t \leq 0}, \ov{x})$.

We claim that $\ov{M}$ must be compact.
Assume not.
Then by Theorem~\ref{Thm_kappa_sol_classification}, after passing to another subsequence, $(\C^i, \rho^{-2} (x^i) g^i_{t^i}, x^i)$ would converge to a pointed round cylinder, its $\IZ_2$-quotient or a pointed Bryant soliton, in contradiction to our assumption that Assertions~\ref{ass_cutoff_d3} and \ref{ass_cutoff_d4} are violated.
Note here that the Hessian of $R$ on $(M_{\Bry}, g_{\Bry})$ at $x_{\Bry}$ is non-degenerate (see Lemma~\ref{Lem_Bry_R_Hessian_positive}), so if $(\C^i, x'')$ is sufficiently close to $(M_{\Bry}, g_{\Bry}, x_{\Bry})$, then we can find an $x' \in \C^i$ near $x''$ with $\nabla R (x') = 0$.

So $\ov{M}$ must be compact.
Since Assertion~\ref{ass_cutoff_d1} is violated by assumption, we have for all $t' \in [t^i - \rho^2 (x^i), t^i]$
\[ a^i \cdot \frac{r_{\can, \eps^i} (x^i(t'))}{\wh\rho (x^i(t'))} \geq a^i \cdot \frac{r_{\can, \eps^i} (x^i)}{\wh\rho (x^i(t'))} 
\geq \frac{a^i}{\alpha^i} \cdot \frac{\rho(x^i)}{\wh\rho (x^i(t'))} 
\geq \frac{1}{c^i} \cdot \frac{\rho(x^i)}{\wh\rho (x^i(t'))} 
\]
Due to smooth convergence to a compact $\kappa$-solution, which also holds on larger and larger parabolic neighborhoods, the right-hand side must go to infinity.
Therefore, for large $i$ we have $\eta^*_1 (x^i(t')) = 1$ for all $t' \in [t^i - \rho^2 (x^i), t^i]$.
Similarly, by smooth convergence  to a compact $\kappa$-solution, $A^i \to \infty$ and the fact that $\wh{\rho} / \RR$ is scaling invariant, we obtain that for large $i$ we have $\eta^{**}_1 (x^i(t')) = 1$ for all $t' \in [t^i - \rho^2 (x^i), t^i]$.
So since $D^{\prime, i} \to \infty$ we obtain that for large $i$
\[ (A^i)^{-1} > \eta'_1(x^i) \geq  \int_{t^i - \rho^2 (x^i)}^{t^i} \frac{1}{\wh\rho^2 (x^i(t'))} dt'. \]
This, again, contradicts smooth convergence to a compact $\kappa$-solution.\end{proof}

Lastly, let us summarize the choice of constants.
Given $\delta^*$, we can determine $A$ and $D_0$ using Claim~\ref{cl_choice_of_A}.
Next, we can choose $D'$ using Claims~\ref{cl_eta1_regularity}, \ref{cl_choice_of_A}  and then $\alpha'$, $a$ using Claim~\ref{cl_UD_prime}.
Once $a$ is fixed, we can choose $\alpha$ using Claim~\ref{cl_choice_of_A}.
These constants can be used to determine $D, D_0, C_m$.
Lastly, we can choose $\eps$ depending on all previous constants and $N$.
\end{proof}

 \subsection{Modification in regions that are geometrically close to Bryant solitons} \label{subsec_modify_Bryant}
Our next goal will be to round the metrics $g^s$ in regions that are close to the tip of a Bryant soliton at an appropriately small scale.
The resulting metrics will be called $g^{\prime,s}_2$.

\begin{lemma}\label{lem_Bryant_rounding}
For every $\delta, \delta^* > 0$, and assuming $\alpha \leq \ov\alpha(\delta^*)$, $D \geq \underline{D} (\delta^*)$, $C_m \geq \underline{C}_m$ and $\eps \leq \ov\eps (\delta^*, \delta)$, there are:
\begin{itemize}
\item a  transversely continuous family of smooth metrics $(g^{\prime, s}_{2})_{s \in X}$ on $\ker d\t$, 
\item a continuous function $\eta_2 : \cup_{s \in X} \M^s \to [0,1]$ that is smooth on each fiber $\M^s$ and transversely continuous in the smooth topology,
\item a transversely continuous family of spherical structures $(\SS^s_2)_{s \in X}$ on open subsets of $\M^s$,
\end{itemize}
such that for any $s \in X$ and $t \geq 0$:
\begin{enumerate}[label=(\alph*)]
\item  \label{ass_round_Bry_a} The fibers of $\SS^s_2$ are contained in time-slices of $\M^s$.
\item  \label{ass_round_Bry_b} $g^{\prime, s}_{2}$ is compatible with $\SS^s_2$.
\item  \label{ass_round_Bry_e} $\eta_2$ satisfies all assertions of Lemma~\ref{lem_cutoff_almost_extinct} for the new constants $\alpha$, $D$, $C_m$, $\eps$ and $N := [\delta^{-1}]$ and with Assertion~\ref{ass_cutoff_d} replaced by:
For every point $x \in \M^s_t$ with $\eta_2 (x) < 1$ (at least) one of the following is true:
\begin{enumerate}[label=(c\arabic*)]
\item \label{ass_round_Bry_e1} $\rho (x) > \alpha r_{\can, \eps} (x)$.
\item \label{ass_round_Bry_e2} $x$ is a center of a $\delta^*$-neck with respect to $g^{\prime,s}_2$,
\item \label{ass_round_Bry_e3} There is an open neighborhood of $x$ that admits a two-fold cover in which a lift of $x$ is a center of a $\delta^*$-neck with respect to $g^{\prime,s}_2$.
\item \label{ass_round_Bry_e4} $x \in \domain (\SS^s_2)$.
\end{enumerate}
\item  \label{ass_round_Bry_f} $g^{\prime,s}_2 = g^s$ on $\{ \rho > 10^{-1} r_{\can, \eps} \}$.
\item  \label{ass_round_Bry_g} $|\nabla^{m_1} \partial_\t^{m_2} ( g^{\prime, s}_2 - g^s )| < \delta \rho^{-m_1 - 2m_2}$ for $m_1, m_2 = 0, \ldots, [\delta^{-1}]$.
\item \label{ass_round_Bry_i} If $(\M^s_t, g^s_t)$ is homothetic to the round sphere or a quotient of the round cylinder, then $g^{\prime,s}_{2,t} = g^s_t$.
\item \label{ass_round_Bry_j} For every spherical fiber $\mathcal{O}$ of $\SS^s_2$ there is a family of local spatial vector fields $(Y_\mathcal{O}^{s'})_{s' \in X}$ defined in a neighborhood of $\mathcal{O}$ in $\cup_{s \in X} \M^s$ such that  for all $s' \in X$ the vector field $\partial^{s'}_\t + Y^{s'}_\mathcal{O}$ preserves $\SS^{s'}_2$ and $|\nabla^{m_1} \partial_\t^{m_2} Y^{s'}_\mathcal{O}| < \delta \rho^{1-m_1 - 2m_2}$ for $m_1, m_2 = 0, \ldots, [\delta^{-1}]$.
\end{enumerate}
\end{lemma}

Note that we may choose $\delta \ll \delta^*$.
This will be important for us later when we analyze components where $\eta_2 \in (0,1)$ in the proof of Lemma~\ref{lem_exten_almost_round}.
More specifically, the diameter of these components is bounded by a constant of the form $\underline{D} (\delta^*) \rho (x)$, while $g'_2$ is $\delta$-close to $g$.
So by choosing $\delta \ll \delta^*$, we can guarantee that the Ricci flows starting from both metrics on these components remain arbitrarily close on an arbitrarily large time-interval.

\begin{proof}
In the following we will assume that $\delta^*$ is smaller than some universal constant, which we will determine in the course of the proof.
Apply Lemma~\ref{lem_cutoff_almost_extinct} with $\delta^*$ replaced by $\delta^*/2$ and $N = [\delta^{-1}]$, set 
\[ \eta_2 (x) := \nu ( 2 \eta_1 (x) ) \]
and consider the constants $\alpha \leq \ov\alpha(\delta^*)$, $D \geq \underline{D} (\delta^*), D_0 \geq \underline{D}_0 (\delta^*)$ and $C_m \geq \underline{C}_m (\delta^*)$.
Note that $\{ \eta_2 > 0 \} \subset \{ \eta_1 > 0 \}$ and $\{ \eta_2 < 1 \} \subset \{ \eta_1 < \frac12 \} \subset \{ \eta_1 < 1\}$.
So all assertions of Lemma~\ref{lem_cutoff_almost_extinct} remain true for $\eta_2$ after modifying the constants $C_m \geq \underline{C}_m (\delta^*)$.

In the following we will construct $g^{\prime,s}_2$ and $\SS^s_2$ for all $s \in X$.
The fact that these objects are transversely continuous, as well as Assertion~\ref{ass_round_Bry_j}, will mostly be clear due to our construction.

Let $E^s \subset \M^s$ be the set of points $x' \in \M^s_t$ such that:
\begin{enumerate}[label=(\arabic*)]
\item \label{prop_E_1} $\eta_1 (x') < 1$.
\item \label{prop_E_2} $\rho (x') < 10^{-1}  r_{\can, \eps} (x')$.
\item \label{prop_E_3} $\nabla R (x') = 0$.
\item \label{prop_E_4} $(\M^s_t, g^s_t, x')$ is $\delta^*$-close to $(M_{\Bry}, g_{\Bry}, x_{\Bry})$ at some scale.
\item \label{prop_E_5} The diameter of the component of $\M^s_t$ containing $x'$ is $> 100 D_0 \rho(x')$.
\end{enumerate}

\begin{Claim} \label{cl_def_E}
Assuming $\delta^* \leq \ov\delta^*$, $\eps \leq \ov\eps ( D_0)$ the following is true for any $x' \in E^s \cap \M^s_t$:
\begin{enumerate}[label=(\alph*)]
\item \label{ass_Es_a} The Hessian of $R$ at $x'$ is strictly negative.
\item \label{ass_Es_b} $E^s \subset \M^s$ is an 1-dimensional submanifold.
\item \label{ass_Es_c} $(E^s)_{s \in X}$ is transversely continuous in the following sense: There are neighborhoods $U \subset X$, $I \subset [0, \infty)$ of $x'$ and $t$ and a transversely continuous family of smooth maps $(\wh{x}'_{s'} : I \to \M^{s'})_{s' \in U}$ with $\wh{x}'_{s'} (t') \in E^{s'} \cap \M^{s'}_{t'}$, $\wh{x}'_s (t) = x'$ and $\cup_{s' \in X} \wh{x}'_{s'} (I) \cap V = \cup_{s' \in X} E^{s'} \cap V$ for some neighborhood $V \subset \cup_{s \in X} \M^s$ of $x'$.
Moreover, $(\wh{x}'_{s'})_{s' \in U}$ is locally uniquely determined by $x'$.
\item \label{ass_Es_d} The balls $B(x', 10 D_0 \rho (x'))$, $x' \in E^s \cap \M^s_t$ are pairwise disjoint.
\item \label{ass_Es_e} If $x \in \M^s_t$ with $\eta_1 (x) < 1$, then
\[ x \in \cup_{x' \in E^s \cap \M^s_t} B(x', D_0 \rho(x')) \]
or one of the Assertions~\ref{ass_round_Bry_e1}--\ref{ass_round_Bry_e3} of this lemma hold with $\delta^*$ replaced by $\delta^* / 2$ and $g^{\prime,s}_2$ replaced by $g^s$.
\item \label{ass_Es_f} The injectivity radius at $x'$ is $> 10 D_0 \rho (x')$.
\end{enumerate}
\end{Claim}

\begin{proof}
Assertion~\ref{ass_Es_a} is a consequence of Lemma~\ref{Lem_Bry_R_Hessian_positive} and Property~\ref{prop_E_4} for $\delta^* \leq \ov\delta^*$ and Assertions~\ref{ass_Es_b} and \ref{ass_Es_c} are an immediate consequence of Assertion~\ref{ass_Es_a} due to the implicit function theorem applied to $\nabla R$.
Assertion~\ref{ass_Es_e} is a direct consequence of Assertion~\ref{ass_cutoff_d} from Lemma~\ref{lem_cutoff_almost_extinct}.

For Assertion~\ref{ass_Es_d} it suffices to show that 
\begin{equation} \label{eq_Es_cap_x_prime}
E^s \cap B(x',20 D_0 \rho (x')) = \{ x' \}.
\end{equation}

If Assertions~\ref{ass_Es_d} or \ref{ass_Es_f}  were false, then we could find a sequence of singular Ricci flows $\M^i$ and points $x^{\prime,i} \in \M^i_{t^i}$ that satisfy Properties~\ref{prop_E_1}--\ref{prop_E_5} for $\eps^i \to 0$, but violate (\ref{eq_Es_cap_x_prime}) or Assertion~\ref{ass_Es_f}.
After passing to a subsequence, we may assume that $(\M^i_{t^i}, \rho^{-2} (x^{\prime,i}) g^i_{t^i}, x^{\prime,i})$ either converge to a pointed final time-slice $(\ov{M}, \ov{g}, \ov{x}')$ of a $\kappa$-solution, or their universal covers converge to a round sphere.
The second case can be excluded by Property~\ref{prop_E_4}, assuming $\delta^* \leq \ov\delta^*$, which also implies that $(\ov{M}, \ov{g})$ is not a quotient of the round sphere or the round cylinder.
It follows that $(\ov{M}, \ov{g})$ is rotationally symmetric due to Theorem~\ref{Thm_kappa_sol_classification} and by Property~\ref{prop_E_3} the point $\ov{x}'$ is a center of rotation.
By Property~\ref{prop_E_5} we obtain that $\diam (\ov{M}, \ov{g}) \geq 100 D_0$.
This implies (\ref{eq_Es_cap_x_prime}) for large $i$.
Assertion~\ref{ass_Es_f} follows for large $i$ since the injectivity radius at $\ov{x}'$ is $\geq 100 D_0$.
\end{proof}

Fix some $x' \in E^s \cap \M^s_t$ for the moment and choose a continuous family $(\wh{x}'_{s'})$ near $x'$ as in Assertion~\ref{ass_Es_c} of Claim~\ref{cl_def_E}.
Choose a family of linear isometries $\varphi_{s',t'} : \IR^3 \to T_{\wh{x}'_{s'}(t')} \M^{s'}_{t'}$ such that for every $s'$ the family $t' \mapsto \varphi_{s',t'}$ is parallel along $t' \mapsto x'_{s',t'}$.
Then 
\[ \chi : (s',t', v) \longmapsto \exp_{g^{s'}_{t'}, \wh{x}'_{s'}(t')} (\varphi_{s',t'}(v)) \]
defines a family of exponential coordinates near $x'$, which induce a family  of $O(3)$-actions on $B(\wh{x}'_{s'}(t'), 10 D_0 \rho (\wh{x}'_{s'} (t')))$ that are transversely continuous in the smooth topology.
Let $g''$ be average of $g$ under these local actions.
Note that $g''$ does not depend on the choice of the family $(\varphi_{s',t'})$, so it extends to a smooth and transversely continuous family of metrics on $\cup_{x' \in E^s} B(x', 10 D_0 \rho (x'))$, which is compatible with a unique transversely continuous family of spherical structure $(\SS^{\prime,s})_{s \in X}$.
Next, define $\partial^{s'}_\t + Y^{s'}_{\chi}$ to be the average of $\partial^{s'}_\t$ under the same action on the image of $\chi$.
Then $\partial^{s'}_\t + Y^{s'}_{\chi}$ preserves $\SS^{\prime,s'}$ for $s'$ near $s$.
A standard limit argument as in the proof of Claim~\ref{cl_def_E} (the limit again being a rotationally symmetric $\kappa$-solution) shows that if $\eps \leq \ov\eps (\delta', D_0 (\delta^*))$, then
\begin{equation}\label{eq_gppYchi}
|\nabla^{m_1} \partial_\t^{m_2} ( g^{\prime\prime, s} - g^s )| < \delta' \rho^{-m_1 - 2m_2}, \qquad
|\nabla^{m_1} \partial_\t^{m_2} Y^{s'}_\chi | < \delta' \rho^{1-m_1 - 2m_2}
\end{equation}
for $m_1, m_2 = 0, \ldots, [(\delta')^{-1}]$.

For $x \in B (x', \lb 10 D_0 \rho (x'))$ set
\[ (g^{\prime\prime\prime, s}_t)_x := (g^{\prime\prime,s}_t)_x + \nu \bigg( \frac{d(x, x')}{D_0 \rho (x')} - 2 \bigg) \cdot \big( (g^s_{t})_x - (g^{\prime\prime,s}_t)_x \big) \]
and $(g^{\prime\prime\prime, s}_t)_x := (g^s_{t})_x$ otherwise.
Then $g^{\prime\prime\prime, s}_t$ is smooth on $\M^s_t$ and transversely continuous and $g^{\prime\prime\prime, s}_t = g^{\prime\prime, s}_t$ on $B (x', \lb 2 D_0 \rho (x'))$.
We can now define $g^{\prime,s}_2$ as follows:
\[ (g^{\prime,s}_2)_x := (g^s)_x + \nu \bigg( \frac{ r_{\can, \eps} (x)}{10^2 \rho (x)} \bigg) \cdot \nu \big( 2 - 2 \eta_1 (x) \big) \cdot \big(  (g^{\prime\prime\prime,s})_x - (g^s)_x   \big). \]
Then $(g^{\prime,s}_2)_x = (g^{\prime\prime\prime,s})_x$ whenever $\rho (x) < 10^{-2}  r_{\can, \eps} (x)$ and $\eta_1(x) < \frac12$.
On the other hand,  $(g^{\prime,s}_2)_x = (g^{s})_x$ whenever $\rho (x) >10^{-1} r_{\can, \eps} (x)$ or $\eta_1 (x) = 1$.
This implies Assertion~\ref{ass_round_Bry_f}.
If $\delta' \leq \ov\delta' (\delta, D_0(\delta^*), (C_m (\delta^*)))$ and $\eps \leq \ov\eps (\delta)$, then Assertion~\ref{ass_round_Bry_g} holds due to (\ref{eq_gppYchi}) and Lemma~\ref{lem_rcan_control}; moreover we can assume that $\frac12 \rho_g < \rho_{g'_2} < 2 \rho_g$.

Next let $\SS_2$ be the restriction of $\SS'$ to $\{ \rho_{g'_2}  < \frac12 10^{-2} r_{\can, \eps} \} \cap \{ \eta_1 < \frac12 \}$.
Then Assertion~\ref{ass_round_Bry_e} holds due to Claim~\ref{cl_def_E}\ref{ass_Es_e}, assuming that we have chosen $\delta' \leq \ov\delta' (\delta^*, D_0)$ and $\alpha \leq 10^{-3}$; note we need to ensure that  a center of a $\delta^*/2$-neck with respect to $g^s$ is automatically a center of a $\delta^*$-neck with respect to $g^{\prime,s}_2$.
Assertions~\ref{ass_round_Bry_a}, \ref{ass_round_Bry_b},\ref{ass_round_Bry_i} hold by construction.
Assertion~\ref{ass_round_Bry_j} holds due to (\ref{eq_gppYchi}), assuming that $\delta' \leq \delta$.
\end{proof}

\subsection{Modification in cylindrical regions} \label{subsec_modify_necks}
In the next lemma we construct metrics $g^{\prime,s}_3$ by rounding the metrics $g^{\prime,s}_2$ in the neck-like regions.
This new metric will be rotationally symmetric everywhere, except at points of large scale or on components of bounded normalized diameter.

\begin{lemma}\label{lem_round_cyl}
For every $\delta > 0$, and assuming that $\alpha \leq \ov\alpha, D \geq \underline{D}$, $C_m \geq \underline{C}_m$ and $\eps \leq \ov\eps(\delta)$, there are:
\begin{itemize}
\item a transversely continuous family of smooth metrics $(g^{\prime, s}_{3})_{s \in X}$ on $\ker d\t$, 
\item a continuous function $\eta_3 : \cup_{s \in X} \M^s \to [0,1]$ that is smooth on each fiber $\M^s$ and transversely continuous in the smooth topology,
\item a transversely continuous family of spherical structures $(\SS^s_3)_{s \in X}$ on open subsets of $\M^s$,
\end{itemize}
such that for all $s \in X$:
\begin{enumerate}[label=(\alph*)]
\item  \label{ass_round_Cyl_a} The fibers of $\SS^s_3$ are contained in time-slices of $\M^s$.
\item  \label{ass_round_Cyl_b} $g^{\prime, s}_{3}$ is compatible with $\SS^s_3$.
\item  \label{ass_round_Cyl_c}  $\domain (\SS^s_3) \supset \{ \rho < \alpha r_{\can, \eps} \} \cap \{ \eta_3 < 1 \} \cap \M^s$.
\item  \label{ass_round_Cyl_d} $\eta_3$ satisfies Assertions~\ref{ass_cutoff_a}--\ref{ass_cutoff_c}, \ref{ass_cutoff_e} of Lemma~\ref{lem_cutoff_almost_extinct} with respect to the new constants $D, C_m$ and for $N = [\delta^{-1}]$.
\item  \label{ass_round_Cyl_e} $g'_3 = g$ on $\{ \rho > 10^{-1} r_{\can,\eps}\}$.
\item \label{ass_round_Cyl_f} $|\nabla^{m_1} \partial_\t^{m_2} ( g^{\prime,s}_3-g^s )| \leq \delta \rho^{-m_1-2m_2}$ for $m_1, m_2 = 0, \ldots, [\delta^{-1}]$.
\item \label{ass_round_Cyl_g} If $(\M^s_t, g^s_t)$ is homothetic to the round sphere or a quotient of the round cylinder for some $t \geq 0$, then $g^{\prime,s}_{3,t} = g^s_t$.
\item  \label{ass_round_Cyl_h} For every spherical fiber $\mathcal{O}$ of $\SS^s_3$ there is a family of local spatial vector fields $(Y_\mathcal{O}^{s'})_{s' \in X}$ defined in a neighborhood of $\mathcal{O}$ in $\cup_{s \in X} \M^s$ such that  for all $s' \in X$ the vector field $\partial^{s'}_\t + Y^{s'}_\mathcal{O}$ preserves $\SS^{s'}_3$ and $|\nabla^{m_1} \partial_\t^{m_2} Y^{s'}_\mathcal{O}| < \delta \rho^{1-m_1 - 2m_2}$ for $m_1, m_2 = 0, \ldots, [\delta^{-1}]$.
\end{enumerate}
\end{lemma}

Note that $\alpha$, $D$ and $C_m$ are independent of $\delta$.

\begin{proof}
Let $\delta^*, \delta^\#, \delta' > 0$ be constants that we will determine in the course of the proof.
Apply Lemma~\ref{lem_Bryant_rounding} for $\delta^*, \delta = \delta^\#$ and consider the family of metrics $(g^{\prime, s}_2)_{s \in X}$, the continuous function $\eta_2$ and the family of spherical structures $(\SS^s_2)_{s \in X}$.

If $\delta^* \leq \ov\delta^*$, then for every $x \in \M^s_t$ that satisfies Assertion~\ref{ass_round_Bry_e2} or \ref{ass_round_Bry_e3} of Lemma~\ref{lem_Bryant_rounding} we can find a unique constant mean curvature (CMC) sphere or CMC-projective space $x \in \Sigma_x \subset \M^s_t$ with respect to $g^{\prime,s}_{2,t}$ whose diameter is $< 100 \rho (x)$.
Moreover, the induced metric $g^{\prime,s}_{2,t} |_{\Sigma_x}$ can, after rescaling, be assumed to be arbitrarily close to the round sphere or projective space if $\delta^*$ is chosen appropriately small.
So if $\delta^* \leq \ov\delta^*$, then $g'_x := \RD^2 (g^{\prime,s}_{2,t} |_{\Sigma_x})$ defines a round metric on $\Sigma_x$; here $\RD^2$ denotes the rounding operator from Subsection~\ref{subsec_RD}.
If $x \in \domain (\SS^s_2)$, then $\Sigma_x$ is a regular fiber of $\SS^s_2$ and the induced metric is round, so $g'_x = g^{\prime,s}_{2,t} |_{\Sigma_x}$.
It follows that for sufficiently small $\delta^*$ the spheres $\Sigma_x \subset \M^s_t$ and volume-normalizations of the metrics $g'_x$, for all $x$ satisfying Assertion~\ref{ass_round_Bry_e2} or \ref{ass_round_Bry_e3} of Lemma~\ref{lem_Bryant_rounding}, can be used to define a spherical structure $\SS^{\prime,s}$ on an open subset of $\M^s$ that extends the spherical structure $\SS^s_2$.
By construction, the family $(\SS^{\prime,s})_{s \in X}$ is transversely continuous and
\begin{equation} \label{eq_domain_SS_prime_contains}
 \{ \rho_{g} < \alpha r_{\can, \eps} \} \cap \{ \eta_2 < 1 \} \cap \M^s \subset \domain(\SS^{\prime,s}). 
\end{equation}
The family $(\SS^s_2)_{s \in X}$ will arise by restricting $(\SS^{\prime,s})_{s \in X}$ to a smaller domain.

From now on we fix $\delta^* > 0$ such that the construction in the previous paragraph can be carried out.
We can then also fix $D \geq \underline{D} (\delta^*)$, $\alpha \leq \ov\alpha (\delta^*)$ and $C_m \geq \underline{C}_m (\delta^*)$ according to Lemma~\ref{lem_Bryant_rounding}.

We will now construct a family of metrics on $\domain(\SS^{\prime,s})$ that are compatible with $\SS^{\prime,s}$.
For this purpose fix $s \in X$ and consider a regular spherical fiber $\Sigma \subset \M^s_t$ of $\SS^{\prime,s}$.
Let $g_\Sigma$ be a multiple of the standard round metric induced by $\SS^{\prime,s}$ with the property that:
\begin{enumerate}
\item \label{prop_CMC_rounding_1} The areas of $\Sigma$ with respect to $g_\Sigma$ and the induced metric $g^{\prime,s}_{2,t} |_{\Sigma}$ agree.
\end{enumerate}
Note that $g_\Sigma = g'_x$ from before if $\Sigma = \Sigma_x$.
If $\Sigma$ is a fiber of $\SS^s_2$, then we have $g_\Sigma = g^{\prime,s}_{2,t} |_{\Sigma}$.
By passing to a local two-fold cover (see Lemma~\ref{lem_local_spherical_struct}), we can also define $g_\Sigma$ for any singular fibers $\Sigma \approx \IR P^2$.

Next, fix a unit normal vector field $N$ along $\Sigma$ (with respect to $g^{\prime,s}_2$) and consider all spatial vector fields $Z$ defined in a neighborhood of $\Sigma$ such that:
\begin{enumerate}[start=2]
\item \label{prop_CMC_rounding_2} $Z$ preserves $\SS^{\prime,s}$
\item \label{prop_CMC_rounding_3} $\int_\Sigma \langle Z, N \rangle_{g^{\prime,s}_2} d\mu_{g_\Sigma}=\int_\Sigma  d\mu_{g_\Sigma}$
\end{enumerate}
Any two such vector fields $Z_1, Z_2$ differ along $\Sigma$ by a Killing field on $(\Sigma,g_\Sigma)$.
So there is a unique vector field $Z_{\Sigma, N}$ along $\Sigma$ with minimal $L^2$-norm (with respect to $g^{\prime,s}_2$) that arises as a restriction of a vector field $Z$ satisfying Properties~\ref{prop_CMC_rounding_2} and \ref{prop_CMC_rounding_3} to $\Sigma$.
Note that $Z_{\Sigma, - N} = - Z_{\Sigma, N}$.
If $\Sigma$ is a fiber of $\SS^s_2$, then $Z_{\Sigma, N} = N$.

We can now define a metric $g^{\prime\prime, s}$ for $\ker d\t$ on $\domain(\SS^{\prime,s})$ such that for every regular fiber $\Sigma$ of $\SS^{\prime,s}$:
\begin{enumerate}[start=4]
\item $g^{\prime\prime,s} |_{\Sigma} = g_\Sigma$
\item $Z_{\Sigma, N}$ are unit normal vector fields with respect to $g^{\prime\prime,s}$.
\end{enumerate}
Then $g^{\prime\prime,s}$ is smooth and transversely continuous near regular fibers.
The regularity near singular fibers $\approx \IR P^2$ can be seen by passing to a local two-fold cover.
On $\domain (\SS^s_2)$ we have $g^{\prime\prime, s} = g^{\prime, s}_2$, so $g^{\prime\prime, s}$ is also regular near singular fibers that are points.
This shows that $g^{\prime\prime,s}$ is smooth and transversely continuous everywhere.

Let us now discuss the closeness of $g^{\prime\prime, s}$ to $g^{\prime,s}_2$ and the existence of a local vector field as in Assertion~\ref{ass_round_Cyl_h}.
Fix a fiber $\Sigma \subset \M^s_t$ of $\SS^{\prime, s}$ that is not a fiber of $\SS^s_2$ and pick a point $x \in \Sigma$ and two orthonormal vectors $v_1, v_2 \in T_x \M^s_t$ (with respect to $g^{\prime\prime, s}_{t}$) that are tangent to $\Sigma$.
We can uniquely extend $v_i$ to a family of vectors $v_{i,t'}$ along the curve $t' \mapsto x(t') \in \M^s_{t'}$, for $t'$ close to $t$, such that $v_{1,t'}, v_{2,t'}$ remain orthonormal and tangent to the spherical fiber $\Sigma_{x(t')}$ through $x(t')$ and such that $\frac{d}{dt'} v_{1,t'}$ is normal to $v_{2,t'}$.
Using the exponential map, we can find a unique family of homothetic embeddings $\beta_{t'} : S^2 \to \M^s_{t'}$ such that $\beta_{t'} (S^2) = \Sigma_{x(t')}$ and such that for some fixed orthonormal tangent vectors $u_1, u_2$ of $S^2$ the images $d\beta_{t'} (u_i)$ are positive multiples of $v_{i,t'}$.
Using the normal exponential map to $\Sigma_{x(t')}$, we can extend these maps to charts of the form
\[ \chi^s_{t'} : S^2 \times (-a^s, a^s) \longrightarrow \M^s_{t'} \]
By repeating the same procedure for $s'$ near $s$, starting with points $x^{s'}$ and vectors $v_i^{s'}$ that depend continuously on $s'$, we can extend $(\chi^s_{t'})$ to a transversely continuous family of charts $\chi = (\chi^{s'}_{t'})$.
These charts induce a transversely continuous family of local $O(3)$-actions that are isometric with respect to $g^{\prime\prime,s}$ and compatible with $\SS^{\prime,s}$.
As in the proof of Lemma~\ref{lem_Bryant_rounding} we can define $\partial^{s'}_\t + Y^{s'}_{\chi}$ to be the average of $\partial^{s'}_\t$ under this action near $\Sigma$.
Then $\partial^{s'}_\t + Y^{s'}_{\chi}$ preserves $\SS^{\prime,s'}$ for $s'$ near $s$.
A limit argument yields that if $\eps \leq \ov\eps (\delta'), \delta^\# \leq \ov\delta^\# (\delta')$, then near $\Sigma$
\begin{equation}\label{eq_gppYchi_again}
|\nabla^{m_1} \partial_\t^{m_2} ( g^{\prime\prime, s} - g_2^{\prime,s} )| \leq \delta' \rho^{-m_1 - 2m_2}, \qquad
|\nabla^{m_1} \partial_\t^{m_2} Y^{s'}_\chi | \leq \delta' \rho^{1-m_1 - 2m_2}
\end{equation}
for $m_1, m_2 = 0, \ldots, [(\delta')^{-1}]$.

Lastly, we construct the metrics $g^{\prime, s}_3$ by interpolating between $ g_2^{\prime,s}$ and $g^{\prime\prime, s}$ using the cutoff function $\nu$ from Subsection~\ref{subsec_rounding_conventions}.
If $x \in \domain (\SS^{\prime, s})$, then set
\[ (g^{\prime,s}_3)_x := (g^{\prime, s}_2)_x + \nu \bigg( \frac{\alpha \cdot r_{\can, \eps} (x)}{10^2 \rho_{g} (x)} \bigg) \cdot \nu \big( 2-2 \eta_2 (x) \big) \cdot \big( (g^{\prime\prime,s})_x - (g^{\prime, s}_2)_x \big), \]
otherwise let $(g^{\prime,s}_2)_x := (g^{\prime, s}_1)_x$.
This defines a smooth and transversely continuous family of metrics due to (\ref{eq_domain_SS_prime_contains}).
On $\{ \rho > 10^{-1} r_{\can, \eps} \}$ we have $g^{\prime}_3 = g^{\prime}_2 = g$, assuming $\alpha \leq 1$, which implies Assertion~\ref{ass_round_Cyl_e}.
As in the proof of Lemma~\ref{lem_Bryant_rounding} the bound (\ref{eq_gppYchi_again}) implies Assertion~\ref{ass_round_Cyl_f} if $\delta' \leq \ov\delta' ( \delta, (C_m (\delta^*)))$ and $\delta^\# \leq \ov\delta^\# (\delta)$ and we can again assume that $\frac12 \rho_{g'_2} < \rho_{g'_3} < 2 \rho_{g'_2}$ and $\frac12 \rho_{g} < \rho_{g'_2} < 2 \rho_{g}$.

Next note that $g'_3 = g''$ on 
\[ \{ \rho_{g} <   10^{-2} \alpha r_{\can, \eps} \} \cap \{ \eta_2 < \tfrac12 \} \supset \{ \rho_{g'_3} < \tfrac14 10^{-2} \alpha r_{\can, \eps} \} \cap \{ \eta_2 < \tfrac12 \}. \]
So if $\SS^s_3$ denotes the restriction of $\SS^{\prime, s}$ to the subset $\{ \rho_{g'_3} < \frac14 10^{-2} \alpha r_{\can, \eps} \} \cap \{ \eta_2 > \frac12 \}$,
then Assertions~\ref{ass_round_Cyl_a}, \ref{ass_round_Cyl_b} hold by construction and Assertion~\ref{ass_round_Cyl_c} holds if we replace $\alpha$ by $\frac1{16} 10^{-2} \alpha$ and set $\eta_3 (x) := \nu (2 \eta_2 (x) )$.

Finally, Assertion~\ref{ass_round_Cyl_h} of this lemma holds due to (\ref{eq_gppYchi_again}) and Lemma~\ref{lem_Bryant_rounding}\ref{ass_round_Bry_j}, assuming $\delta' \leq \delta$, and Assertions~\ref{ass_round_Cyl_d} and \ref{ass_round_Cyl_g} hold after adjusting the constants $C_m$ appropriately. 
\end{proof}

\subsection{Modification of the time vector field} \label{subsec_modify_dt}
Next, we will modify the time-vector fields $\partial^s_\t$ on $\{ \eta_3 > 0 \} \cap \{ \rho < \alpha r_{\can, \eps} \}$ so that they preserve the spherical structures $\SS^s_3$.

\begin{lemma} \label{lem_dtprime}
For every $\delta > 0$ and assuming that $\alpha \leq \ov\alpha$, $D \geq \underline{D}$, $C_m \geq \underline{C}_m$ and $\eps \leq \ov\eps(\delta)$  there are:
\begin{itemize}
\item a transversely continuous family of smooth metrics $(g^{\prime, s}_{4})_{s \in X}$ on $\ker d\t$, 
\item a continuous function $\eta_4 : \cup_{s \in X} \M^s \to [0,1]$ that is smooth on each fiber $\M^s$ and transversely continuous in the smooth topology,
\item a transversely continuous family of spherical structures $(\SS^s_4)_{s \in X}$ on open subsets of $\M^s$,
\item  a transversely continuous family of smooth vector fields $(\partial^{\prime,s}_{\t,4})_{s \in X}$ on $\M^s$ that satisfy $\partial^{\prime,s}_{\t,4}  \t = 1$,
\end{itemize}
such that the assertions of Lemma~\ref{lem_round_cyl} still hold for $(g^{\prime,s}_4)_{s \in X}$, $\eta_4$ and $(\SS^s_4)_{s \in X}$ and such that in addition for all $s \in X$:
\begin{enumerate}[label=(\alph*), start=9]
\item \label{ass_round_VF_i} $\partial^{\prime,s}_{\t,4}$ preserves $\SS^s_4$.
\item \label{ass_round_VF_j} $\partial^{\prime,s}_{\t,4} = \partial^s_\t$ on $\{ \rho > 10^{-1}  r_{\can, \eps}  \}$.
\item \label{ass_round_VF_k} $|\nabla^{m_1} \partial_\t^{m_2} (\partial^{\prime,s}_{\t,4} - \partial^s_\t)| \leq \delta \rho^{1-m_1 - 2m_2}$ for $m_1, m_2 = 0, \ldots, [\delta^{-1}]$.
\item \label{ass_round_VF_l} If $(\M^s_0,g^s_0)$ is homothetic to the round sphere or a quotient of the round cylinder, then $\partial^{\prime,s}_{\t,4} = \partial^s_\t$.
\end{enumerate}
\end{lemma}

\begin{proof}
Apply Lemma~\ref{lem_round_cyl} for $\delta$ replaced by some $\delta^\# > 0$, which we will determine in the course of the proof depending on $\delta$, and fix the constants $\alpha$, $D$, $C_m$.
Assume in the following that $\frac12 g < g'_3 < 2 g$.
Set $(g^{\prime,4}_s )_{s \in X} := (g^{\prime,3}_s )_{s \in X}$.
Fix some $s \in X$ for now and set $U^s := \{ \eta_3 < 1 \} \cap \{ \rho_{g'_3} < \frac12 \alpha r_{\can, \eps} \} \cap \M^s \subset \domain (\SS^s_3)$.
We will define 
\begin{equation} \label{eq_def_partialtprime}
 \partial^{\prime,s}_{\t,4} := \partial^s_\t + \eta^* Z^s, 
\end{equation}
where $Z^s$ is a spatial vector field on $U^s$ and $\eta^* : \cup_{s \in X} \M^s \to [0,1]$ is a smooth and transversely continuous cutoff function with support on $\cup_{s \in X} U^s$.

Let us describe the construction of $Z^s$.
Consider a spherical fiber $\mathcal{O} \subset U^s \cap \M^s_t$.
Call a vector field $Z'$ in $\M^s_t$ along $\mathcal{O}$ \emph{admissible} if $\partial^s_\t + Z'$ can be extended to a vector field on neighborhood of $\mathcal{O}$ in $\M^s$ that preserves $\SS^s_3$.
Using a local chart near $\mathcal{O}$ in $\M^s$, one can see that the space of admissible vector fields along $\mathcal{O}$ is affine and finite dimensional.
More specifically, if $\mathcal{O}$ is a point, then there is only one admissible vector field along $\mathcal{O}$ and otherwise the difference of any two admissible vector fields is equal to the sum of a Killing field on $\mathcal{O}$ and a parallel normal vector field to $\mathcal{O}$.
Let now $Z'_{\mathcal{O}}$ be the admissible vector field along $\mathcal{O}$ whose $L^2$-norm is minimal and define $Z^s$ on $U^s$ such that $Z^s|_{\mathcal{O}} := Z'_{\mathcal{O}}$ for every spherical fiber $\mathcal{O} \subset U^s$.
Then $Z^s$ is well defined.

We will now show that $Z^s$ is smooth, transversely continuous in the smooth topology and small in the sense of Assertion~\ref{ass_round_VF_k}.
For this purpose, consider a family of defined vector field $(Y^{s'}_\mathcal{O})_{s' \in X}$ from Lemma~\ref{lem_round_cyl}\ref{ass_round_Cyl_h} near a spherical fiber $\mathcal{O}$.
As in the proofs of Lemmas~\ref{lem_Bryant_rounding} and \ref{lem_round_cyl}, we can construct a transversely continuous family of local $O(3)$-actions $(\zeta^{s'})$ that are compatible with $\SS^{s'}_3$ and isometric with respect to $g^{\prime, s'}_3$.
By definition $Y^{s'}_\mathcal{O} - Z^{s'}$ restricted to every spherical fiber $\mathcal{O}' \subset \M^{s'}$ near $\mathcal{O}$ that is not a point equals the $L^2$-projection of $Y_\mathcal{O}|_{\mathcal{O}'}$ onto the subspace spanned by Killing fields and parallel normal vector fields along $\mathcal{O}'$.
If $\mathcal{O}'$ is a point, then $Y^{s'}_\mathcal{O} - Z^{s'}$ restricted to every spherical fiber $\mathcal{O}'$ vanishes.
So a representation theory argument implies that
\[ Y^{s'}_\mathcal{O} - Z^{s'} |_U = \frac1{|O(3)|} \int_{O(3)} (1 + 3 \tr A) (\zeta^{s'}_A)_* Y^{s'}_{\mathcal{O}} \, dA, \]
where $\zeta^{s'}_A := \zeta^{s'} (A, \cdot)$.
This implies the desired regularity properties of $Z^s$ and the bound from Assertion~\ref{ass_round_VF_k} for $\delta^\# \leq \ov\delta^\# (\delta)$ if we had $\partial^{\prime,s}_{\t,4} = \partial^s_\t +  Z^s$.

It remains to construct the cutoff function $\eta^*$.
Let $\nu$ be the cutoff function from Subsection~\ref{subsec_rounding_conventions} and set
\[ \eta^* (x) := \nu \bigg( \frac{\alpha \cdot r_{\can, \eps} (x)}{10^2 \rho_{g'_3} (x)} \bigg) \cdot \nu \big( 2 \eta_2 (x) \big) . \]
Then $\supp \eta^* \subset \cup_{s \in X} U^s$, so if we define $\partial^{\prime,s}_{\t,4}$ as in (\ref{eq_def_partialtprime}), then Assertion~\ref{ass_round_VF_k} holds for $\delta^\# \leq \ov\delta^\# (\delta, (C_m))$ and $\eps \leq \ov\eps (\delta, (C_m))$.
Next, note that $\eta^* \equiv 1$ on $U' := \{ \rho_{g'_3} < 10^{-2} \alpha r_{\can, \eps} \} \cap \{ \eta_3 > \frac12 \}$.
So if we define $\SS^{s}_4$ to be the restriction of $\SS^{s}_3$ to $U' \cap \M^s$, then Assertion~\ref{ass_round_VF_i} holds.
Assertions~\ref{ass_round_Cyl_a}--\ref{ass_round_Cyl_h} of Lemma~\ref{lem_round_cyl} continue to hold if we replace $\alpha$ by $\frac12 10^{-2} \alpha$, set $\eta_4 (x) := \nu (2 \eta_3 (x) )$ and adjust $C_m$ appropriately.
Assertion~\ref{ass_round_VF_j} of this lemma holds assuming $\alpha \leq \frac12$ and Assertion~\ref{ass_round_VF_l} is true due to construction.
\end{proof}

\subsection{Extension of the structure until the metric is almost round} \label{subsec_extend_to_almost_round}
Next we modify each metric $g^{\prime,s}_4$ on the support of $\eta_4$ such that it remains compatible with a spherical structure $\SS^s_5$ until it has almost constant curvature.
We will also choose a new cutoff function $\eta_5$ whose support is contained in the support of $\eta_4$ and which measures the closeness of $g^{\prime,s}_4$ to a constant curvature metric.

Recall in the following that $\eta_4$ is only non-zero in components that have positive sectional curvature, bounded normalized diameter and will become extinct in finite time.
Therefore our construction will only take place in product domains of $\M^s$, which can be described by conventional Ricci flows.
Our strategy will be to construct new metrics $g^{\prime,s}_5$ by evolving the metrics $g^{\prime,s}_4$ on the support of $\eta_4$ forward by a certain amount of time under the Ricci flow.
By the continuous dependence  of the Ricci flow on its the initial data, this flow remains close to $g^s$ for some time.
Moreover, any symmetry of $g^{\prime,s}_4$ will be preserved by the flow and therefore the new metrics $g^{\prime,s}_5$ will be compatible with a spherical structure for a longer time.
By choosing our constants appropriately, we can ensure that $g^{\prime,s}_5$ is compatible with a spherical structure until a time close enough to the corresponding extinction time.
After this time the remaining flow is sufficiently close to a quotient of the round sphere.

\begin{lemma} \label{lem_exten_almost_round}
For every $\delta > 0$ and assuming that $\alpha \leq \ov\alpha (\delta)$,  $C^*_m \geq \underline{C}^*_m$ and $\eps \leq \ov\eps(\delta)$  there are:
\begin{itemize}
\item a transversely continuous family of smooth metrics $(g^{\prime, s}_{5})_{s \in X}$ on $\ker d\t$, 
\item a continuous function $\eta_5 : \cup_{s \in X} \M^s \to [0,1]$ that is smooth on each fiber $\M^s$ and transversely continuous in the smooth topology,
\item a transversely continuous family of spherical structures $(\SS^s_5)_{s \in X}$ on open subsets of $\M^s$,
\item  a transversely continuous family of smooth vector fields $(\partial^{\prime,s}_{\t,5})_{s \in X}$ on $\M^s$ that  satisfy $\partial^{\prime,s}_{\t,5}  \t = 1$,
\end{itemize}
such that for any $s \in X$ and $t \geq 0$:
\begin{enumerate}[label=(\alph*)]
\item \label{ass_extend_alm_round_aa} $\partial^s_\t \eta_5 \geq 0$.
\item \label{ass_extend_alm_round_a} $\eta_5$ is constant on every connected component of $\M^{s}_t$.
\item \label{ass_extend_alm_round_b} Any connected component $\C \in \M^{s}_t$ on which $\eta_5  > 0$ is compact and there is a time $t_\C > t$ such that:
\begin{enumerate}[label=(c\arabic*)]
\item  $\C$ survives until time $t'$ for all $t' \in [t, t_\C)$ and no point in $\C$ survives until or past time $t_\C$.
\item The universal cover $(\td\C (t'), g^{\prime,s}_{5,t'})$ is $\delta$-close to the round sphere modulo rescaling for all $t' \in [t, t_\C)$.
\item  $\rho < 10^{-1} r_{\can, \eps}$ on $\C (t')$ for all $t' \geq [t, t_\C)$.
\end{enumerate}
\item \label{ass_extend_alm_round_c} $| \partial_\t^m \eta_5 | \leq C^*_m \rho^{-2m}$ for  $m = 0, \ldots, [\delta^{-1}]$.
\item \label{ass_extend_alm_round_d} $\eta_5$, $g'_5$, $\SS_5$, $\partial'_{\t, 5}$ satisfy Assertions~\ref{ass_round_Cyl_a}--\ref{ass_round_Cyl_c}, \ref{ass_round_Cyl_e}--\ref{ass_round_Cyl_g} of Lemma~\ref{lem_round_cyl} and all assertions of Lemma~\ref{lem_dtprime}.
\end{enumerate}
\end{lemma}

It will be important later that the constants $C^*_m$ are independent of $\delta$.

\begin{proof}
Let $ \delta^\#, \delta',\delta'' > 0$, $A < \infty$ be constants whose values will be determined depending on $\delta$ in the course of the proof.
Apply Lemma~\ref{lem_dtprime} with $\delta$ replaced by $\delta^\#$ and consider the families $(g^{\prime,s}_4)_{s \in X}, (\partial^{\prime,s}_{\t,4})_{s \in X}$, $(\SS^s_4)_{s \in X}$, the cutoff function $\eta_4$ and the constants $\alpha, D, C_m$.
Assume that $\delta^\#$ is chosen small enough such that $\frac12 \rho_{g} < \rho_{g'_4} < 2 \rho_{g}$.

Set $\partial'_{\t, 5} := \partial'_{\t, 4}$ everywhere and $g'_5 := g'_4$, $\SS_5 := \SS_4$ and $\eta_5 := \eta_4 = 0$ on $\{ \eta_4 = 0 \}$.
Therefore, it suffices to consider only components $\C \subset \M^s_t$ where $\eta_4 > 0$.
Recall that by Lemma~\ref{lem_cutoff_almost_extinct}\ref{ass_cutoff_c}, the union of these components consists of pairwise disjoint product domains.
For any such component $\C \subset \M^s_t$ choose $t^{\min}_\C < t < t^{\max}_\C$ minimal/maximal such that $\C$ survives until all $t' \in (t^{\min}_\C, t^{\max}_\C)$ and set
\[ U_\C := \cup_{t' \in (t^{\min}_\C, t^{\max}_\C)} \C (t'). \]
Then $U_\C$ is a product domain (with respect to $\partial^s_\t$ and $\partial^{\prime,s}_{\t, 5}$) and for any two components $\C, \C' \subset \M^s_t$ on which $\eta_4 > 0$ the product domains $U_{\C}, U_{\C'}$ are either equal or disjoint.
For the remainder of the proof let us fix some $s \in X$ and a product domain of the form $U_\C$.
Recall that $\eta_4 = 0$ on $\C(t') \subset U_\C$ for $t'$ close to $t^{\min}_\C$.
We will describe how to define $g^{\prime,s}_5$, $\eta_5$ and $\SS^s_5$ on $U_\C$ such that all assertions of this lemma hold and $g^{\prime,s}_5 = g^{\prime,s}_4$, $\eta_5 = \eta_4 = 0$ and $\SS^s_5 = \SS^s_4$ on $\{ \eta_4 = 0 \} \cap U_\C$.
It will be clear that the same construction can be performed for every $U_\C$ and that the resulting objects have the desired regularity properties.

We first represent the flows $g^s$ and $g^{\prime,s}_4$ restricted to $U_\C$ by a smooth family of metrics on $\C$, which satisfy an equation that is similar to the volume normalized Ricci flow equation.
For this purpose consider the flow $\Phi : \C \times (T_1, T_2) \to U_\C$ of the vector field $\td\partial^s_\t := \wh\rho_g^2 \cdot \partial^{\prime,s}_{\t, 4}$.
By Lemma~\ref{lem_cutoff_almost_extinct}\ref{ass_cutoff_c5} we have $|\td\partial^s_\t  \t| = \wh\rho_g^2 \leq C (t^{\max}_\C - \t)$.
So $T_2 = \infty$.
Express $g^s$ and $g^{\prime,s}_4$ as families of metrics  $(\td{g}_t)_{t \in (T_1, \infty)}$, $(\td{g}_{4,t})_{t \in (T_1, \infty)}$ and define $\td\eta_4 : (T_1, \infty) \to [0,1]$ as follows:
\[ \td{g}_t :=\Phi^*_t  \big( \wh\rho^{-2}_g  g^{s} \big), \qquad \td{g}_{4,t} := \Phi^*_t  \big( \wh\rho_g^{-2}  g^{\prime,s}_4 \big), \qquad \td\eta_4 (t) := \eta_4 (\Phi_t (\C)) . \]
Note that if $\partial^{\prime,\t}_{\t, 4} = \partial^s_{\t}$, then $(\td{g}_t)$ is the volume normalization of the flow $g^s$ restricted to $U_\C$.
In the general case, consider the family of vector fields $(Z_t)_{t \in (T_1, \infty)}$
\[ Z_t := \Phi^*_t \big( \wh\rho_g^2 (\partial^s_{\t}-\partial^{\prime,s}_{\t, 4}) \big). \]
Then $(\td{g}_t)$ satisfies the following normalized flow equation with a correctional Lie derivative:
\[ \partial_t \td{g}_t + \mathcal{L}_{Z_t} \td{g}_t = - 2 \Ric_{\td{g}_t} + \frac{2}{3  V(\C, \td{g}_t)} \int_\C R_{\td{g}_t} d\mu_{\td{g}_t} \cdot \td{g}_t . \]
We have $\partial_t \td\eta_4 \geq 0$ by Lemma~\ref{lem_cutoff_almost_extinct}\ref{ass_cutoff_a}.
By our previous discussion we have $\td\eta_4 (t) \equiv 0$ for $t$ near $T_1$ and $\td\eta_4 (t) > 0$ for large $t$.
Let $T_0 \in (T_1, \infty)$ be maximal such that $\td\eta_4 |_{(T_1, T_0]} \equiv 0$.

\begin{Claim} \label{cl_bounds_vol_normalized_RF}
There are constants $C'_m < \infty$ such that if $\delta^\# \leq \ov\delta^\# (\delta')$, $\eps \leq \ov\eps (\delta')$, then $T_0 - (\delta')^{-2} > T_1$ and for all $t \in (T_0 - (\delta')^{-2}, \infty)$ and $m_1, m_2 = 0, \ldots, [(\delta')^{-1}]$ we have
\begin{multline*}
 | \partial_t^{m_2} \td\eta_4 | \leq C'_{m_2} \rho_{\td{g}}^{-2m_2}, \qquad |\nabla^{m_1} \partial_t^{m_2} ( \td{g}_4 -\td{g} )| \leq \delta' \rho_{\td{g}}^{-m_1-2m_2}, \\
  \qquad |\nabla^{m_1} \partial_t^{m_2} Z|  \leq \delta' \rho_{\td{g}}^{1-m_1-2m_2}  . 
\end{multline*}
\end{Claim}

\begin{proof}
The first statement is a consequence of Lemma~\ref{lem_rcan_control}, assuming that $\eps \leq \ov\eps (\delta')$.
The other bounds follow via a standard limit argument using Lemma~\ref{lem_cutoff_almost_extinct}\ref{ass_cutoff_e}, Lemma~\ref{lem_round_cyl}\ref{ass_round_Cyl_f} and Lemma~\ref{lem_dtprime}\ref{ass_round_VF_k}
\end{proof}

For any metric $g'$ on $\C$ and any $\Delta T \geq 0$ denote by $\RF_{\Delta T} (g')$ the result of evolving $g'$ by the volume normalized Ricci flow equation
\begin{equation} \label{eq_vol_norm_RF}
 \partial_t g'_t = - 2 \Ric_{g'_t} + \frac{2}{3  V(\C, g'_t)} \int_\C R_{g'_t} d\mu_{g'_t} \cdot g'_t, \qquad g'_0 = g' 
\end{equation}
for time $\Delta T$, if possible.
Note that if $g'$ is compatible with some spherical structure $\SS'$, then so is $RF_{\Delta T} (g')$.
We now define $(\td{g}_{5,t})_{t \in (T_1, \infty)}$ and $\td\eta_5 : (T_1, \infty) \to [0,1]$ as follows:
\[ \td{g}_{5,t} := \RF_{\td\eta_4(t) A} \td{g}_{4, t-\td\eta_4(t) A}, \qquad \td\eta_5 (t) := \begin{cases} 0 & \text{if $t \leq T_0$} \\
 \td\eta_4(t -A) &\text{if $t > T_0$} \end{cases}. \]

\begin{Claim} \label{cl_translation_by_A}
If $\delta' \leq \ov\delta' (A, \delta'')$ and $\eps \leq \ov\eps (A, \delta'')$, then $(\td{g}_{5,t})_{t \in (T_1, \infty)}$ and $\td\eta_5 : (T_1, \infty) \to [0,1]$ are well defined and smooth and for all $m_1, m_2 = 0, \ldots, [(\delta'')^{-1}]$ we have
\[ | \partial_t^{m_2} \td\eta_5 | \leq C'_{m_2} \rho_{\td{g}}^{-2m_2}, \qquad |\nabla^{m_1} \partial_t^{m_2} ( \td{g}_5 -\td{g} )|  \leq \delta'' \rho_{\td{g}}^{-m_1-2m_2}  . \]
Moreover, $\td\eta_5 = \td\eta_4$ and $\td{g}_5 = \td{g}_4$ on $(T_1, T_0)$ and if $\Phi_{t - \td\eta_4(t) A} (\C) \subset \domain (\SS^s_4)$, then $\td{g}_{5,t}$ is compatible with the pullback of $\SS_4$ via $\Phi_{t - \td\eta_4(t) A}$.
\end{Claim}

\begin{proof}
The first bound is a direct consequence of Claim~\ref{cl_bounds_vol_normalized_RF}.
The second bound follows via a standard limit argument (the limit being a volume normalized Ricci flow).
The last two statements follow by construction.
\end{proof}

We can now choose $g^{\prime,s}_5$ and $\eta_5$ on $U_{\C}$ such that for all $t \in (T_1, \infty)$
\[  \td{g}_{5,t} := \Phi^*_t  \big( \wh\rho_g^{-2}  g^{\prime,s}_5 \big), \qquad \td\eta_5 (t) := \eta_5 (\Phi_t (\C)). \]
Note that $g^{\prime,s}_5 = g^{\prime,s}_4$ and $\eta_5 = \eta_4$ on $ \{ \eta_4 = 0 \} \cap U_{\C}$.
For any $t \in (T_1, \infty)$ for which $\Phi_{t - \td\eta_4(t) A} (\C) \subset \domain (\SS^s_4)$ consider the push-forward of $\SS^s_4 |_{\Phi_{t - \td\eta_4(t) A} (\C)}$ onto $\Phi_t (\C)$ via the diffeomorphism $\Phi_t \circ \Phi^{-1}_{t - \td\eta_4(t) A}$.
The union of these spherical structures, for all $t \in (T_1,\infty)$ defines a new spherical structure $\SS^s_5$ on $U_\C$, which is compatible with $g^{\prime,s}_5$ on $U_\C$ and equal to $\SS^s_4$ on $U_\C \cap \{ \eta_4 = 0 \}$.

By Lemma~\ref{lem_cutoff_almost_extinct}\ref{ass_cutoff_c} we have $\{ \rho > 10^{-1} r_{\can, \eps} \}  \cap U_\C  \subset \{ \eta_4 = 0 \} \cap  U_\C$.
It follows that Assertions~\ref{ass_round_Cyl_a}, \ref{ass_round_Cyl_b},  \ref{ass_round_Cyl_e},  \ref{ass_round_Cyl_g} of Lemma~\ref{lem_round_cyl} and all Assertions of Lemma~\ref{lem_dtprime} hold on $U_\C$ for $g^{\prime,s}_5$, $\eta_5$ and $\SS^s_5$.
Assertions~\ref{ass_extend_alm_round_aa} and \ref{ass_extend_alm_round_a} of this lemma hold by construction.
By Claim~\ref{cl_translation_by_A} and another limit argument we can choose constants $C^*_m = C^*_m ((C'_m)) < \infty$ such that Assertion~\ref{ass_extend_alm_round_c} of this lemma and Assertion~\ref{ass_round_Cyl_f} of Lemma~\ref{lem_round_cyl} hold on $U_\C$ if $\delta'' \leq \ov\delta'' (\delta)$, $\eps \leq \ov\eps (\delta)$.
Note here that the constants $C^*_m$ can be chosen independently of $A$.
The following claim implies that the remaining assertions of this lemma hold on $U_\C$.

\begin{Claim}
If $A \geq \underline{A} (D,\delta)$, $c \leq \ov{c}(A)$ and $\eps \leq \ov\eps(D,\delta)$, then Assertion~\ref{ass_extend_alm_round_b} of this lemma holds on $U_\C$ and Assertion~\ref{ass_round_Cyl_c} of Lemma~\ref{lem_round_cyl} holds on $U_\C$ for $g^{\prime,s}_5$, $\eta_5$ and $\SS^s_5$ if we replace $\alpha$  by $c \alpha$.
\end{Claim}

\begin{proof}
Consider first Assertion~\ref{ass_extend_alm_round_b} of this lemma.
Due to the fact that $\{ \eta_5 > 0 \} \cap U_\C \subset \{ \eta_4 > 0 \} \cap U_\C$ and Assertion~\ref{ass_cutoff_c} of Lemma~\ref{lem_cutoff_almost_extinct}, it suffices to show that if $\td\eta_5(t) > 0$, then the universal cover of $(\C,\td{g}_t)$  is $\delta$-close to the round sphere modulo rescaling if $A \geq \underline{A} (D,\delta)$ and $\eps \leq \ov\eps (D,\delta)$.
Fix such a $t \in (T_1, \infty)$ with $\td\eta_5(t) > 0$ and assume that the universal cover of $(\C,\td{g}_t)$  was not $\delta$-close to the round sphere modulo rescaling.
Then $t-A > T_0$ and $\td\eta_4(t') > 0$ for all $t' \in [t - A, t]$.
So Assertions~\ref{ass_cutoff_c1}, \ref{ass_cutoff_c3}, \ref{ass_cutoff_c4} of Lemma~\ref{lem_cutoff_almost_extinct} hold on $\Phi_{t'} (\C)$ for all $t' \in [t - A, t]$.
We obtain that $\sec_{\td{g}_{t'}} > 0$ on $\C \times [t-A,t]$ and  $\diam_{\td{g}_{t'}} \C < D \rho$ on $\C$ for all $t' \in [t-A,t]$.
In addition, the $\eps$-canonical neighborhood assumption holds on $\C \times [t-A,t]$.
So if for no choice of $A$ and $\eps$ the universal cover of $(\C,\td{g}_t)$ was  $\delta$-close to the round sphere modulo rescaling, then we could apply a limit argument using Theorem~\ref{Thm_kappa_compactness_theory} and obtain an ancient solution $(M^\infty, (\td{g}^\infty_{t'})_{t' \leq 0})$ to the volume normalized Ricci flow equation (\ref{eq_vol_norm_RF}) that satisfies 
\begin{equation} \label{eq_sec_diam_tdg_infty}
\sec_{\td{g}^\infty} \geq 0, \qquad  \diam_{\td{g}^\infty} \C \leq D \rho
\end{equation}
everywhere.
Moreover every time-slice of $(\td{g}^\infty_{t'})_{t' \leq 0}$ is isometric to a time-slice of a $\kappa$-solution and $(\td{g}^\infty_{t'})_{t' \leq 0}$ cannot homothetic to a shrinking round sphere.
Due to the positivity of the scalar curvature, it follows that reparameterizing $(\td{g}^\infty_{t'})_{t' \leq 0}$ yields an ancient Ricci flow, and therefore a $\kappa$-solution.
The second bound in (\ref{eq_sec_diam_tdg_infty}) implies via Theorem~\ref{Thm_kappa_sol_classification}\ref{ass_kappa_sol_classification_d} that this $\kappa$-solution is homothetic to the round sphere, in contradiction to our assumptions.

For Assertion~\ref{ass_round_Cyl_c} of Lemma~\ref{lem_round_cyl} note that by the canonical neighborhood assumption we have
\[ | \td\partial^s_\t \rho_{g} | = \wh\rho_g^2 \cdot |\partial^{\prime,s}_{\t, 4} \rho_g | \leq C(D) \rho_g. \]
It follows that for any $(x,t) \in \C \times (T_1, \infty)$
\[ \rho (\Phi_{t - \td\eta_4(t)A} (x)) ) \leq e^{C(D) A} \rho(\Phi_t( x) ). \]
Set $c := e^{-C(D) A}$ and assume that $\rho_{g^s} (\Phi_t(x)) < c  \alpha r_{\can, \eps} (\Phi_t (x))$ and $\td\eta_5 (t) = \eta_5 (\Phi_t (x)) < 1$.
Then 
\[ \rho (\Phi_{t - \td\eta_4(t)A} (x)) ) < \alpha r_{\can, \eps} (\Phi_t (x)) \leq \alpha r_{\can, \eps} (\Phi_{t - \td\eta_4(t)A} (x)). \]
If $\td\eta_4 (t) < 1$, then $\td\eta_4 ( t - \td\eta_4(t)A ) \leq \td\eta_4(t) < 1$ and if $\td\eta_4(t) = 1$, then $\td\eta_4 ( t - \td\eta_4(t)A ) = \td\eta_4 ( t - A ) = \td\eta_5(t) < 1$.
Therefore
\[ \eta_4 (\Phi_{t - \td\eta_4(t)A} (x)) < 1. \]
By Lemma~\ref{lem_round_cyl}\ref{ass_round_Cyl_c} it follows that $\Phi_{t - \td\eta_4(t)A} (x) \in \domain (\SS^s_4)$ and therefore by construction $\Phi_t (x) \in \domain (\SS^s_5)$.
\end{proof}
By repeating the construction above for all product domains $U_\C$, we obtain that all assertions of this lemma hold on the union of all $U_\C$.
On the complement of these product domain we have $g'_5 = g'_4$, $\SS_5 := \SS_4$ and $\eta_5 = \eta_4 = 0$.
So Assertions~\ref{ass_extend_alm_round_aa}--\ref{ass_extend_alm_round_c} hold trivially on this complement and Assertion~\ref{ass_extend_alm_round_d} holds due to Lemmas~\ref{lem_round_cyl} and \ref{lem_dtprime} assuming $\delta^\# \leq \delta$.
\end{proof}

\subsection{Modification in almost round components and proof of the main theorem} \label{subsec_proof_rounding}
Lastly, we will construct $(g^{\prime,s})_{s \in X}$ and $(\partial^{\prime,s}_\t)_{s \in X}$ by modifying $(g^{\prime,s}_5)_{s \in X}$ and $(\partial^{\prime,s}_{\t,5})_{s \in X}$ on $\{ \eta_5 > 0 \}$.
We will also construct a family of spherical structures $(\SS^s)_{s \in X}$ by restricting and extending the family $(\SS^s_5)_{s \in X}$.
These objects will form the family of $\RR$-structures whose existence is asserted in Theorem~\ref{Thm_rounding}.

\begin{proof}[Proof of Theorem~\ref{Thm_rounding}.]
Let $ \delta^\#, \delta'> 0$, $A < \infty$ be constants, whose values will be determined depending on $\delta$ in the course of the proof.
Apply Lemma~\ref{lem_exten_almost_round} with $\delta$ replaced by $\delta^\#$ and consider the families $(g^{\prime,s}_5)_{s \in X}, (\partial^{\prime,s}_{\t,4})_{s \in X}$, $(\SS^s_5)_{s \in X}$, the cutoff function $\eta_5$ and constants $\alpha (\delta^\#)$ and $C^*_m$.
Assume that $\delta^\#$ is chosen small enough such that $\frac12 \rho_{g} < \rho_{g'_5} < 2 \rho_{g}$.

Let us first define the family of metrics $g^{\prime,s}$ and vector fields $\partial^{\prime,s}_\t$.
Fix $s \in X$ and consider a component $\C \subset \M^s_t$ on which $\eta_5 > 0$.
By Lemma~\ref{lem_exten_almost_round}\ref{ass_extend_alm_round_b} the universal cover of $(\C, g^{\prime,s}_{5,t})$ is $\delta^\#$-close to the round sphere modulo rescaling.
So if $\delta^\# \leq \ov\delta^\#$, then we can define a smooth family of metrics $g^{\prime\prime,s}$ on $\{ \eta_5 > 0 \} \cap \M^s$ such that
\[ g^{\prime\prime,s}_t |_\C = \RD^3 (g^{\prime,s}_{5,t} |_\C) \]
for any such component $\C$; here $\RD^3$ is the rounding operator from Subsection~\ref{subsec_RD}.
Similarly as in the proof of Lemma~\ref{lem_dtprime}, we can consider the space of vector fields $Z'$ on $\C$ that can be extended to spatial vector fields $Z''$ near $\C$ in $\M^s$ such that the flow of $\partial^{\prime,s}_{\t, 5} + Z''$ consists of homotheties with respect to $g^{\prime\prime,s}$.
The difference of any two such vector fields is a Killing field on $(\C, g^{\prime\prime,s}_t |_{\C})$, therefore this space is affine linear and finite dimensional.
Let $Z_\C$ be the vector field of this space whose $L^2$-norm with respect to $g^{\prime\prime,s}_t$ is minimal.
Define the spatial vector field $Z^s$ on $\{ \eta_5 > 0 \} \cap \M^s$ such that $Z^s |_\C = Z_\C$ for any component $\C \subset \M^s_t \cap \{ \eta_5 > 0\}$.
Note that since our construction is invariant under isometries, $g^{\prime\prime,s}$ is still compatible with $\SS^s_5$ and $\partial^{\prime\prime,s}_{\t}$ still preserves $\SS^s_5$.
Lemma~\ref{lem_round_cyl}\ref{ass_round_Cyl_f}, Lemma~\ref{lem_dtprime}\ref{ass_round_VF_k} and a standard limit argument imply that if $\delta^\# \leq \ov\delta^\# (\delta')$, $\eps \leq \ov\eps (\delta')$, then
\begin{equation}\label{eq_gpp_Z_proof_thm}
|\nabla^{m_1} \partial_\t^{m_2} ( g^{\prime\prime, s} - g_5^{\prime,s} )| \leq \delta' \rho^{-m_1 - 2m_2}, \qquad
|\nabla^{m_1} \partial_\t^{m_2} Z^{s} | \leq \delta' \rho^{1-m_1 - 2m_2}
\end{equation}
for $m_1, m_2 = 0, \ldots, [(\delta')^{-1}]$.

Define $g^{\prime,s}:= g^{\prime,s}_5$ and $\partial^{\prime,s}_\t := \partial^{\prime,s}_{\t, 5}$ on $\{ \eta_5 = 0 \}$ and for all $x \in \{ \eta_5 > 0 \} \cap \M^s_t$ set
\begin{align*}
  \big( g^{\prime,s}_t  \big)_x &:= \big( g^{\prime,s}_{5,t} \big)_x + \nu (2\eta_5(x)) \cdot \big( \big( g^{\prime\prime,s}_t \big)_x - \big( g^{\prime,s}_{5,t} \big)_x \big), \\
  \big( \partial^{\prime,s}_{\t}  \big)_x &:= \big( \partial^{\prime,s}_{\t,5} \big)_x + \nu (2\eta_5(x)) \cdot \big( \big( \partial^{\prime\prime,s}_\t \big)_x - \big( \partial^{\prime,s}_{\t,5} \big)_x \big).
\end{align*}
Then $g^{\prime,s} = g^{\prime\prime,s}$ and $\partial^{\prime,s}_\t = \partial^{\prime\prime,s}_{t}$ on $\{ \eta_5 > \frac12 \} \cap \M^s$.
Assertion~\ref{ass_thm_rounding_c} of this theorem follows using (\ref{eq_gpp_Z_proof_thm}), Lemma~\ref{lem_exten_almost_round}\ref{ass_extend_alm_round_c}, Lemma~\ref{lem_round_cyl}\ref{ass_round_Cyl_f}, Lemma~\ref{lem_dtprime}\ref{ass_round_VF_k}, assuming $\delta' \leq \ov\delta' (\delta, (C^*_m))$, $\delta^\# \leq \delta$ and $\eps \leq \ov\eps (\delta, (C^*_m))$.
By our discussion of the previous paragraph, $g^{\prime,s}$ is still compatible with $\SS^s_5$ and $\partial^{\prime,s}_{\t}$ still preserves $\SS^s_5$.

Next let us construct $\SS^s$, $U^s_{S2}$ and $U^s_{S3}$.
By Lemma~\ref{lem_round_cyl}\ref{ass_round_Cyl_c} and Lemma~\ref{lem_exten_almost_round}\ref{ass_extend_alm_round_b}, we can find a universal constant $c > 0$ such that if  $\delta^\# \leq \ov\delta^\#$, then 
\[ \domain (\SS^s_5) \supset \{ \wh\rho_g < c \alpha r_{\can, \eps} \} \cap \{0 < \eta_5 < 1 \} \cap \M^s. \]
Set
\[ U^s_{S3} := \{ \wh\rho_g < c \alpha r_{\can, \eps} \} \cap \{ \tfrac12 < \eta_5  \} \cap \M^s. \]
Then $g^{\prime,s}$ restricted to every time-slice of $U^s_{S3}$ has constant curvature and the flow of $\partial^{\prime,s}_\t$ restricted to $U^s_{S3}$ consists of homotheties with respect to $g^{\prime,s}$.

Before constructing $\SS^s$ and $U^s_{S2}$, we need to improve the family of spherical structures $\SS^s_5$.
By restricting $\SS^s_5$, we obtain a transversely continuous family of spherical structures $(\SS^s_6)_{s \in X}$ such that
\begin{multline*}
 \domain (\SS^s_5) \supset \domain (\SS^s_6)  \\
 = \big( \domain (\SS^s_5) \cap \{ \eta_5 < \tfrac14 \} \big) \cup \big( \{ \wh\rho_g < c \alpha r_{\can, \eps} \} \cap \{ 0 < \eta_5 < 1 \} \cap \M^s \big). 
\end{multline*}
So there is a universal constant $c' > 0$ such that if $\delta^\# \leq \ov\delta^\#$, then
\begin{align*}
 \domain (\SS^s_6) &\supset \{ \rho_g < c' \alpha r_{\can, \eps} \} \cap \{  \eta_5 < 1 \} \cap \M^s, \\
 \domain (\SS^s_6) \cap \{ \tfrac12 < \eta_5  \} &= U^s_{S3} \cap \{ \eta_5 < 1 \}  .
\end{align*}

Next we construct a family of spherical structures $\SS^s$ by extending the domains of the spherical structures $\SS^s_6$.
Fix $s \in X$.
Assume that $\delta^\# \leq \ov\delta^\#$, $\eps \leq \ov\eps$ such that by Lemma~\ref{lem_exten_almost_round}\ref{ass_extend_alm_round_b}
\[ \partial^{\prime}_{\t} \bigg( \frac{\wh\rho_{g}}{r_{\can, \eps}} \bigg) < 0 \qquad \text{on} \quad \{ \eta_5 > 0 \}.  \]
Therefore, for any component $\C \subset \M^s_t \cap U^s_{S3}$ we have $\C (t') \subset U^s_{S3}$ for all $t' \in [t, t_\C)$.
So every component $W \subset U^s_{S3}$ is either contained in $\{ \eta_5 = 1 \}$, and therefore disjoint from $\domain (\SS^s_6)$, or it intersects $\domain (\SS^s_6)$ in a connected product domain.
Let $U^s_{S2}$ be the union of $\domain (\SS^s_6)$ with all components of the second type.
Then $U^s_{S3} \setminus U^s_{S2}$ is open.
Using the flow of $\partial^{\prime,s}_\t$ we can extend $\SS^s_5$ to a spherical structure $\SS^s$ on $U^s_{S2}$ that is compatible with $g^{\prime,s}$ and preserved by $\partial^{\prime,s}_\t$.
Recall here that the flow of $\partial^{\prime,s}_\t$ restricted to every  component $W \subset U^s_{S3}$ consists of homotheties with respect to $g^{\prime,s}$.
By construction, $\cup_{s \in X} U^s_{S2}$ is open and the family of spherical structures $(\SS^s)_{s \in X}$ is transversely continuous.

We have shown so far that $(\mathcal{R}^s := ( g^{\prime, s}, \lb \partial^{\prime, s}_\t,  \lb U^s_{S2}, \lb U^s_{S3}, \lb \mathcal{S}^s))_{s \in X}$ is a transversely continuous family of $\RR$-structures.
Assertions~\ref{ass_thm_rounding_a}, \ref{ass_thm_rounding_b} and \ref{ass_thm_rounding_e} of this theorem hold by setting
\[ r_{\rot, \delta} (r, t) := \tfrac12 c' \alpha r_{\can, \eps (\delta)} (r,t), \qquad C := 4 (c' \alpha)^{-1}. \]
Assertion~\ref{ass_thm_rounding_d} is a consequence of Lemma~\ref{lem_round_cyl}\ref{ass_round_Cyl_g} and Lemma~\ref{lem_dtprime}\ref{ass_round_VF_l}.
\end{proof}

\section{Preparatory results}
In this section we collect several results that will be useful in Sections~\ref{sec_partial_homotopy} and \ref{sec_deforming_families_metrics}.

\subsection{Spherical structures}
We will need the following two lemmas on spherical structures.

\begin{lemma} \label{lem_spherical_struct_classification}
Consider a spherical structure $\SS$ on a connected 3-manifold with boundary $M$ such that $\domain (\SS) = M$.
Then one of the following cases holds:
\begin{enumerate}[label=(\alph*)]
\item $\SS$ only consists of regular fibers and $M$ is diffeomorphic to one of the following models:
\[  S^2 \times (0,1), \; S^2 \times [0,1),\;  S^2 \times [0,1], \; S^2 \times S^1 \]
\item $\SS$ has exactly one singular fiber and this fiber is a point and $M$ is diffeomorphic to one of the following models:
\[ B^3, \; D^3 \]
\item $\SS$ has exactly one singular fiber and this fiber is $\approx \IR P^3$ and $M$ is diffeomorphic to one of the following models:
\[  \big( S^2 \times (-1,1) \big) / \IZ_2, \; \big( S^2 \times [-1,1] \big) / \IZ_2 \]
\item $\SS$ has exactly two singular fibers, both of which are points and $M \approx S^3$.
\item $\SS$ has exactly two singular fibers, both of which are $\approx \IR P^3$ and $M \approx (S^2 \times S^1) / \IZ_2$.
Here $\IZ_2$ acts as the antipodal map on $S^2$ and as a reflection with two fixed points on $S^1$.
\item $\SS$ has exactly two singular fibers, one of which is a point and the other which is $\approx \IR P^3$ and $M \approx \IR P^3$.
\end{enumerate}
\end{lemma}

\begin{proof}
Let $X$ be the quotient of $M$ by the spherical fibers.
By Lemma~\ref{lem_local_spherical_struct}, $X$ is homeomorphic to a 1-manifold with boundary and every boundary point corresponds to a model of the form $S^2 \times [0,1)$, $D^3$ or $(S^2 \times [-1,1]) / \IZ_2$.
\end{proof}

\begin{lemma} \label{Lem_cont_fam_sph_struct}
Consider a 3-manifold with boundary $M$ and a compact, connected, 3-di\-men\-sion\-al submanifold with boundary $Y \subset M$.
Let $(\SS^s)_{s \in D^n}$, $n \geq 0$, be a transversely continuous family of spherical structures defined on open subsets $Y \subset U^s \subset M$ such that $\partial Y$ is a union of regular fibers of $\SS^s$ for all $s \in D^n$.
Assume that there is a transversely continuous family of Riemannian metrics $(g^s)_{s \in D^n}$ on $M$ that is compatible with $(\SS^s)_{s \in X}$.

Then there is an open subset $Y \subset V \subset M$ and a transversely continuous family of embeddings $(\omega_s : V \to M)_{s \in D^n}$ such that $\omega_s (Y) = Y$ for all $s \in D^n$ and such that the pullbacks of $\SS^s$ via $\omega_s$ are constant in $s$.
\end{lemma}

We remark that the existence of the family of metrics $(g^s)_{s \in D^n}$ could be dropped from the assumptions of the lemma, as it follows easily from the other assumptions.

\begin{proof}
By Lemma~\ref{lem_spherical_struct_classification} the number and types of singular fibers of $\SS^s$ restricted to $Y$ is constant in $s$.
\medskip

\textit{Case 1: $\partial Y \neq \emptyset$. \quad}
Pick a component $\Sigma \subset \partial Y$.
Using the exponential map on $(\Sigma, g^s |_{\Sigma})$ we can find a continuous family of diffeomorphisms $(\varphi^s : S^2 \to \Sigma)_{s \in X}$ that are homotheties as maps from $(S^2, g_{S^2})$ to $(\Sigma, g^s |_{\Sigma})$.
For every $s \in D^n$ let $\nu^s$ be the unit normal vector field to $\Sigma$ pointing towards the interior of $Y$ and consider the normal exponential map: 
\[ \psi^s : (z,r) \longmapsto \exp_{\varphi^s(z)}^{g^s} ( r \, \nu^s (\varphi^s(z))). \]
Choose $r_{\max}^s > 0$ maximal such that $\psi^s$ is injective on $S^2 \times [0, r^s_{\max})$ and such that $\psi^s ( S^2 \times [0, r^s_{\max}) ) \subset Y$.
After replacing $(g^s)_{s \in X}$ with $( (r^s_{\max})^{-2} g^s )_{s \in X}$, we may assume that $r^s_{\max} \equiv 1$.

If $\Sigma \not\subset \partial M$, then we can find a uniform $\eps_1 > 0$ such that $\psi^s$ is defined and injective on $S^2 \times [-\eps_1, 1)$.
If $\Sigma \subset \partial M$, then set $\eps_1 := 0$.
If $Y$ has another boundary component $\Sigma_2$, then we can similarly find a constant $\eps_2 \geq 0$ such that $\psi^s$ is defined and injective on $S^2 \times [-\eps_1, 1+\eps_2]$ and $\eps_2 > 0$ if $\Sigma_2 \not\subset \partial M$.

Let $Y_0^s \subset Y$ be the union of regular spherical fibers of $\SS^s$ and let $s_0 \in D^n$ be an arbitrary point.
Using the maps $(\psi^s)_{s \in D^n}$ we can construct an open subset $Y^{s_0}_0 \subset \td{V} \subset M$ and a continuous family of embeddings $(\td\omega_s : \td{V} \to M )_{s \in D^n}$ such that $\td\omega_s (Y^s_0) = Y^{s_0}_0$ and such that the pullbacks of $\SS^s$ restricted to $Y^s_0$ are constant in $s$.
Due to the construction of the maps $(\td\omega_s)_{s \in D^n}$ via the exponential map and since $Y_0^s \subset Y$ is dense, we can extend these maps to maps $(\omega_s)_{s \in D^n}$ with the desired properties.

\medskip
\textit{Case 2:  $\partial Y = \emptyset$. \quad}

\medskip
\textit{Case 2a: $\SS^s$ contains a singular fiber $\approx \IR P^2$. \quad}
Similarly as in Case~1, we can use the exponential map to construct a continuous family of homothetic embeddings $( \varphi^s : \IR P^2 \to Y )_{s \in D^n}$ such that for every $s \in D^n$ the image $\varphi^s ( \IR P^2 )$ is a singular fiber of $\SS^s$.
We can now construct $(\omega_s)$ similarly as in Case~1 using the normal exponential map.

\medskip
\textit{Case 2b: $\SS^s$ contains a singular fiber that is a point. \quad}
We can find a continuous family of points $(p^s)_{s \in D^n}$ such that $\{ p^s \}$ is a singular fiber for $\SS^s$ for all $s \in D^n$.
Choose a continuous family of isometric maps $(\varphi^s : \IR^3 \to T_{p^s} Y)_{s \in X}$.
After rescaling the metric $g^s$, we may assume that $\injrad (M, g^s, p^s) = 1$ for all $s \in D^n$.
The remainder of the proof is similar as in Case 1.

\medskip
\textit{Case 2c: $\SS^s$ only consists of regular fibers. \quad}
By Lemma~\ref{lem_spherical_struct_classification} $Y \approx S^2 \times S^1$.
Pick a point $p \in Y$.
As in the proof of Case~1 we can find a continuous family of diffeomorphisms $(\varphi^s : S^2 \to \Sigma)_{s \in X}$ that are homotheties between $(S^2, g_{S^2})$ and fibers $p \in \Sigma^s \subset Y$ of $\SS^s$.
Using the normal exponential map and after rescaling the metric $g^s$, as in Case~1, we may further construct a continuous family of  covering maps $(\td\psi^s : S^2 \times \IR \to Y)_{s \in D^n}$ such that the pullback of $\SS^s$ via each $\td\psi^s$ agrees with the standard spherical structure on $S^2 \times \IR$ and such that $S^2 \times [0,1)$ is a fundamental domain.
Then $\td\psi^s (z,r+1) = \td\psi^s ( A^s(r) z, r)$ for some continuous family of smooth maps $(A^s : \IR \to O(3))_{s \in D^n}$.
Since $D^n$ is contractible, we can find a continuous family of smooth maps $(\td{B}^s : [0,1] \to O(3))_{s \in D^n}$ such that $A^s (0) \td{B}^s(1) = B^s(0)$ for all $s \in D^n$.
These maps can be extended to a continuous family of smooth maps $(\td{B}^s : [0,1] \to O(3))_{s \in D^n}$ such that
\[ A^s(r) B^s (r+1) = B^s (r) \]
for all $s \in D^n$ and $r \in \IR$.
Set $\ov\psi^s (z,r) := \td\psi^s (B^s (r) z , r)$.
Since
\begin{multline*}
 \ov\psi^s (z,r+1) = \td\psi^s (B^s (r+1) z , r+1) = \td\psi^s (A^s(r) B^s (r+1) z , r) \\ = \td\psi^s (B^s (r) z , r) = \ov\psi^s (z,r),
\end{multline*}
the family $(\ov\psi^s)_{s \in D^n}$ descends to a continuous family of diffeomorphisms $(\psi^s : S^2 \times S^1 \to Y)_{s \in D^n}$ such that $(\omega^s := \psi^s \circ (\psi^{s_0})^{-1})_{s \in D^n}$ has the desired properties for any fixed $s_0 \in D^n$.
\end{proof}

\subsection{PSC-conformal metrics}
In this subsection we introduce a conformally invariant condition on the class of compact Riemannian $3$-manifolds with round boundary components.  This property behaves well with respect to the geometric operations that arise in our main construction;  it may be a viewed as a relative version of the property of being conformally equivalent to a PSC metric.  

\begin{definition}[PSC-conformality] \label{Def_PSC_conformal}
A compact Riemannian 3-manifold with boundary $(M,g)$ is called {\bf PSC-conformal} if there is a smooth positive function $w \in C^\infty (M)$ such that:
\begin{enumerate}
\item $w^4 g$ has positive scalar curvature.
\item $w$ restricted to each boundary component of $M$ is constant.
\item Every boundary component of $(M, w^4 g)$ is totally geodesic and isometric to the standard round 2-sphere.
\end{enumerate}
\end{definition}

Note that in the case $\partial M = \emptyset$, the manifold $(M,g)$ is PSC-conformal if and only if its Yamabe constant is positive.

By expressing the conditions above in terms of $w$, we obtain:

\begin{lemma} \label{Lem_PSC_conformal_analytic}
A compact Riemannian 3-manifold with boundary $(M,g)$ is PSC-conformal if and only if all its boundary components are homothetic to the round sphere and there is a function $w \in C^\infty (M)$ such that:
\begin{enumerate}[label=(\arabic*)]
\item \label{prop_PSC_conformal_analytic_1} $w^4 g$ has positive scalar curvature, or equivalently, $8 \triangle w - R w < 0$ on $M$.
\item \label{prop_PSC_conformal_analytic_2} $w^4 |_{\partial M}$ is equal to the sectional curvature of the induced metric on $\partial M$.
\item \label{prop_PSC_conformal_analytic_3} $A_{\partial M} =( \nu_{\partial M} w^4 )g$, where $A_{\partial M}$ denotes the second fundamental form and $\nu_{\partial M}$ the inward pointing unit vector field to $\partial M$ (note that this implies that $\partial M$ is umbilic).
\end{enumerate}
\end{lemma}

The next lemma shows that the PSC-conformal property is open ---  as is the standard PSC-condition --- if we restrict to  variations with a specific behavior on the boundary.

\begin{lemma} \label{Lem_PSC_conformal_open}
Let $M$ be a compact 3-manifold with boundary and $(g_s)_{s \in X}$ a continuous family of Riemannian metrics.
Assume that for all $s \in X$ all boundary components of $(M,g_s)$ are umbilic and homothetic to the round sphere.
If $(M,g_{s_0})$ is PSC-conformal for some $s_0 \in X$, then so is $(M,g_s)$ for $s$ near $s_0$.
Moreover we may choose conformal factors $w_s$ satisfying Lemma~\ref{Lem_PSC_conformal_analytic} which varying continuously with $s$.
\end{lemma}

\begin{proof}
This is a consequence of Lemma~\ref{Lem_PSC_conformal_analytic}.
Choose $w_{s_0}$ such that Properties~\ref{prop_PSC_conformal_analytic_1}--\ref{prop_PSC_conformal_analytic_3} of Lemma~\ref{Lem_PSC_conformal_analytic} hold for $(M, g_{s_0})$.
We can extend $w_{s_0}$ to a continuous family of functions $w_s \in C^\infty(M)$ such that Properties~\ref{prop_PSC_conformal_analytic_1} and \ref{prop_PSC_conformal_analytic_3} of Lemma~\ref{Lem_PSC_conformal_analytic} hold for all $s \in X$.
Then Property~\ref{prop_PSC_conformal_analytic_2} holds for $s$ near $s_0$.
\end{proof}

Next we show that the PSC-conformal property remains preserved if we enlarge a given Riemannian manifold by domains that allow a spherical structure.
This fact will be important in the proof of Proposition~\ref{prop_extending}.

\begin{lemma} \label{Lem_PSC_conformal_enlarge}
Consider a compact Riemannian 3-manifold with boundary $(M,g)$, let $Z \subset M$ be a compact 3-dimensional submanifold with boundary and let $\SS$ be a spherical structure on $M$.
Suppose that:
\begin{enumerate}[label=(\roman*)]
\item $g$ is compatible with $\SS$.
\item $(Z,g)$ is PSC-conformal.
\item $\ov{M \setminus Z}$ is a union of spherical fibers of $\SS$.
\end{enumerate}
Then $(M,g)$ is also PSC-conformal.
\end{lemma}

Note that in the case $Z = \emptyset$ Lemma~\ref{Lem_PSC_conformal_enlarge} implies that $(M,g)$ is PSC-conformal if $M$ admits a spherical structure that is compatible with $g$.

\begin{proof}
We first argue that we may assume without loss of generality that $\ov{M \setminus Z}$ is connected, is a union of regular fibers and is disjoint from $\partial M$.
To see this let $\{ \Sigma_1, \ldots, \Sigma_m \}$ be the set of all singular fibers in $M \setminus Z$, boundary components of $\partial M$ that are contained in $M \setminus Z$ and at least one regular fiber in $M \setminus Z$.
Choose pairwise disjoint closed neighborhoods $V_j \subset M \setminus Z$ of each $\Sigma_j$ that are each unions of spherical fibers and diffeomorphic to $D^3$ or $(S^2 \times [-1,1])/\IZ_2$ or $S^2 \times [0,1]$.
Note that $(V_j, g)$ is PSC-conformal for all $j = 1, \ldots, m$, because in the first case $g |_{V_j}$ is conformally equivalent to a metric with positive scalar curvature that is cylindrical near the boundary and in the second and third case $g |_{V_j}$ is conformally equivalent to (a quotient) of the round cylinder.
So all assumptions of the lemma are still satisfied if we replace $Z$ by $Z' := Z \cup_{j=1}^m V_j$.
Therefore, we may assume that $\ov{M \setminus Z}$ is disjoint from all singular fibers and $\partial M$ and that $Z \neq \emptyset$.
Furthermore, by induction it suffices to consider the case in which $M \setminus Z$ is connected.
So by Lemma~\ref{lem_spherical_struct_classification} we have $\ov{M \setminus Z} \approx S^2 \times [0,1]$.

Choose an embedding $\phi : S^2 \times (-L-\eps, L+\eps) \to M$, $L, \eps > 0$ such that $\ov{M \setminus Z} = \phi (S^2 \times [0,L])$ and such that $\phi^* g = f^4 (g_{S^2} + dr^2)$ for some smooth function $f : (-L-\eps,L+\eps) \to \IR_+$.
Let $w \in C^\infty (Z)$ such that all properties of Lemma~\ref{Lem_PSC_conformal_analytic} hold.
Then $\phi^* (w^4 g) = (\td{w} f)^4 (g_{S^2} + dr^2)$ for some smooth function $\td{w} : (-L-\eps , -L] \cup [L,L+\eps) \to \IR_+$ and we have
\[ (\td{w} f)(\pm L) = 1, \qquad (\td{w} f)' (\pm L) = 0. \]
By smoothing the function
\[ r \mapsto \frac1{f(r)} \begin{cases} ( \td{w}f)( r)  &\text{if $r \in (-L-\eps , -L] \cup [L,L+\eps)$} \\ 1 & \text{if $r \in (-L, L)$} \end{cases}, \]
we can find a smooth function $\td{w}^*: (-L-\eps,L+\eps) \to \IR_+$ such that $\td{w}^* = \td{w}$ near the ends of the domain and such that $(\ov{w} f)^4 (g_{S^2} + dr^2)$ has positive scalar curvature.
We can then choose $w^* \in C^\infty (M)$ such that $w^* = w$ outside the image of $\phi$ and $w^* \circ \phi = \td{w}^*$.
This function satisfies all properties of Lemma~\ref{Lem_PSC_conformal_analytic}, showing that $(M,g)$ is PSC-conformal.
\end{proof}

Next, we discuss a criterion that will help us identify PSC-conformal manifolds in Section~\ref{sec_deforming_families_metrics}.

\begin{lemma} \label{lem_CNA_SS_implies_PSC_conformal}
There are constants $\eps, c > 0$ such that the following holds.
Suppose that $(M,g)$ is a (not necessarily) complete Riemannian 3-manifold, $\SS$ is a spherical structure on $M$ that is compatible with $g$ and $Z \subset M$ is a compact 3-dimensional submanifold with the property that for some $r > 0$:
\begin{enumerate}[label=(\roman*)]
\item $\partial Z$ is a union of regular spherical fibers of $\SS$.
\item $(Z,g)$ has positive scalar curvature.
\item Every point of $Z \cap \{ \rho <  r \}$ satisfies the $\eps$-canonical neighborhood assumption.
\item $\rho \leq c r$ on $\partial Z$.
\item $\{ \rho <  r \} \subset \domain \SS$.
\end{enumerate}
Then $(Z,g)$ is PSC-conformal.
\end{lemma}

\begin{proof}
The constants $\eps$ and $c$ will be determined in the course of the proof.
We may assume without loss of generality that $Z$ is connected and $\partial Z \neq \emptyset$.
Moreover, we may assume that $Z$ is not a union of spherical fibers, in which case the PSC-conformality is trivial due to Lemma~\ref{Lem_PSC_conformal_enlarge}.
Therefore $Z$ must contain a point with $\rho \geq r$.
Choose $\lambda \in (\sqrt{c}/2,\sqrt{c})$ such that $\rho \neq \lambda r$ on any singular spherical fiber in $Z$ and consider the subset $Z_0 := \{ \rho \leq \lambda \} \cap Z \subset Z$.
Then $\ov{Z \setminus Z_0}$ is a union of spherical fibers.
Let $Z_1 \subset Z$ be the union of $Z_0$ with all components of $Z \setminus Z_0$ that are disjoint from $\partial Z$.
Then $\ov{Z \setminus Z_1}$ is a union of spherical fibers and $\rho \equiv \lambda r$ on $\partial Z_1$.
Due to Lemma~\ref{Lem_PSC_conformal_enlarge} it suffices to to show that $(Z_1, g)$ is PSC-conformal.

To see this it suffices to show the following claim.

\begin{Claim}
If $\eps \leq \ov\eps$, $c \leq \ov{c}$, then for every  component $\Sigma \subset \partial Z_1$ there is a collar neighborhood $U_\Sigma \subset Z_1$ consisting of regular spherical fibers and a function $u_\Sigma \subset C^\infty (U_\Sigma)$ such that:
\begin{enumerate}[label=(\alph*)]
\item \label{ass_cl_U_Sigma_tot_geod_a} $u_\Sigma > 0$.
\item \label{ass_cl_U_Sigma_tot_geod_b} $u_\Sigma \equiv 1$ outside of a compact subset of $U_\Sigma$.
\item \label{ass_cl_U_Sigma_tot_geod_c} $8 \triangle u_\Sigma - R u_\Sigma < 0$; therefore $u_\Sigma^4 g$ has positive scalar curvature.
\item \label{ass_cl_U_Sigma_tot_geod_d} $\Sigma$ is totally geodesic in $(Z_1,  u_\Sigma^4 g |_\Sigma )$ and is isometric to the round 2-sphere of scale $\lambda r$.
\end{enumerate}
\end{Claim}

Note that by our assumption $Z_1 \not\subset \domain (\SS)$, which implies that the neighborhoods $U_\Sigma$ are pairwise disjoint.

\begin{proof}
By rescaling we may assume without loss of generality that $\lambda r = 1$.
Fix sequences $\eps^i, c^i \to 0$ and consider a sequence of counterexamples $M^i, Z^i, Z^i_1, \Sigma^i, g^i$, $r^i$ to the claim.
Choose points $x^i \in \Sigma^i$.
After passing to a subsequence, we may assume that $(M^i, g^i, x^i)$ converge to the final time-slice of a pointed $\kappa$-solution $(\ov{M}, \ov{g}, \ov{x})$ with $\rho (\ov{x}) = 1$; note that $(M^i, g^i)$ cannot be isometric to the round sphere for large $i$ by our assumption that $M^i$ contains a point with $\rho \geq r^i$.

We claim that $(\ov{M}, \ov{g})$ is homothetic to the round cylinder. 
Otherwise $\ov{M}$ would be either compact or one-ended (see Theorem~\ref{Thm_kappa_sol_classification}).
So $\Sigma^i$ bounds a compact domain $D^i \subset M^i$ for large $i$ on which $C_*^{-1} < \rho < C_*$ for some constant $C_* < \infty$ that is independent of $i$.
Since $Z^i_1$ contains a point of scale $\rho \geq r^i \geq \lambda_i^{-1} \to \infty$ this implies that the interiors of $D^i$ and $Z^i_1$ are disjoint.
Since $\rho \leq c^i r^i = c^i / \lambda_i \to 0$ on $\partial Z^i$ we have $D^i \subset Z^i_1$, which contradicts our construction of $Z^i_1$ for large $i$.

Since $(\ov{M}, \ov{g})$ is homothetic to the round cylinder, we can choose $U_{\Sigma^i} \subset Z_1^i$ to be larger and larger tubular neighborhoods of $\Sigma^i$.
Assertions~\ref{ass_cl_U_Sigma_tot_geod_a}--\ref{ass_cl_U_Sigma_tot_geod_d} can therefore be achieved easily for large $i$.
\end{proof}
This finishes the proof of the lemma.
\end{proof}

Lastly, we discuss further properties of the PSC-conformal condition, which will be central to the proof of Proposition~\ref{prop_move_remove_disk}.
We begin by showing that the conformal factor $w$ from Definition~\ref{Def_PSC_conformal} or  Lemma~\ref{Lem_PSC_conformal_analytic} can be chosen to be of a standard form near a given point.

\begin{lemma} \label{Lem_w_std_form}
Let $(M,g)$ be PSC-conformal and $p \in \Int M$.
Then there is a constant $a \in \IR$ such and a smooth function $w \in C^\infty (M)$ satisfying Properties~\ref{prop_PSC_conformal_analytic_1}--\ref{prop_PSC_conformal_analytic_3} of Lemma~\ref{Lem_PSC_conformal_analytic} and such that near $p$ we have
\begin{equation} \label{eq_w_std_form}
 w = w (p) - a \cdot d^2 (p,\cdot) . 
\end{equation}
\end{lemma}

\begin{proof}
We first show that we can arrange $w$ such that $\nabla w( p) = 0$.
For this purpose, fix some $w \in C^\infty (M)$ satisfying Properties~\ref{prop_PSC_conformal_analytic_1}--\ref{prop_PSC_conformal_analytic_3} of Lemma~\ref{Lem_PSC_conformal_analytic} and assume that $\nabla w (p) \neq 0$.
Choose a small geodesic $\gamma : (-\eps, \eps) \to M$ with $\gamma(0) = p$ and $\gamma' (p) =\nabla w (p) /  |\nabla w (p)|$.
Choose a sequence $\alpha_i \to 0$ with $\alpha_i > 0$.

\begin{Claim}
There exists a sequence of even functions $\varphi_i \in C^\infty (\IR)$ and numbers $t_i > 0$ such that for large $i$:
\begin{enumerate}[label=(\arabic*)]
\item \label{ass_cl_varphi_even_1} $0 \leq \varphi_i \leq 1$.
\item \label{ass_cl_varphi_even_2} $\varphi_i \equiv 0$ on $[1, \infty)$.
\item \label{ass_cl_varphi_even_3} $\varphi'_i \leq 0$ on $[0, \infty)$.
\item \label{ass_cl_varphi_even_4} $\alpha_i^2 \varphi'_i (t_i) = - |\nabla w(p)|$.
\item \label{ass_cl_varphi_even_5} $\varphi''_i (r) + 1.5 r^{-1} \varphi'_i (r) \leq 1$  on $[0, \infty)$.
\end{enumerate}
\end{Claim}

\begin{proof}
Let $\nu : \IR \to [0,1]$ be an even cutoff function with $\nu \equiv 1$ on $[-\frac12, \frac12]$ and $\nu \equiv 0$ on $(-\infty, -1] \cup [1, \infty)$ and $\nu' \leq 0$ on $[0, \infty)$.
Let $\delta > 0$ and consider the function
\[ \psi_\delta (r) := \delta \cdot  \nu (r) \cdot |r|^{-1/2}. \]
For small $\delta$, this function satisfies Properties \ref{ass_cl_varphi_even_2}, \ref{ass_cl_varphi_even_3} and \ref{ass_cl_varphi_even_5} wherever defined.
Our goal will be to choose positive constants $\delta_i, t_i \to 0$ and let $\varphi_i$ be a smoothing of $\max \{  \psi_{\delta_i}, .9 \}$.
In order to ensure that Properties~\ref{ass_cl_varphi_even_1}--\ref{ass_cl_varphi_even_5} hold, we require that $0 \leq \psi_{\delta_i} \leq \frac12$ on $[t_i / 2, \infty)$ and $\psi'_{\delta_i} (t_i) = - \alpha_i^{-2}  |\nabla w(p)|$.
These conditions are equivalent to
\[ 0 < t_i < \tfrac12, \qquad \delta_i (t_i/2)^{-1/2} \leq \tfrac12, \qquad
\delta_i t_i^{-1.5} =  2 \alpha_i^{-2}  |\nabla w(p)|, \]
which can be met for large $i$ such that $\delta_i, t_i \to 0$.
\end{proof}

Set now $q_i := \gamma(-t_i)$ and
\[ \td{w}_i := w + \alpha_i^3  \varphi_i \bigg( \frac{d (q_i, \cdot )}{\alpha_i} \bigg). \]
Then $\td{w}_i \to w$ in $C^0$ and for large $i$
\[ \triangle \td{w}_i - \triangle w \leq  \alpha_i \bigg( \varphi'' + \triangle d(q_i, \cdot) \cdot \varphi' \bigg)
\leq  \alpha_i \bigg( \varphi'' + \frac{1.5}{d(q_i,\cdot)} \varphi' \bigg) \leq \alpha_i \to 0. \]
Therefore, for large $i$ the function $\td{w}_i$ satisfies Properties~\ref{prop_PSC_conformal_analytic_1}--\ref{prop_PSC_conformal_analytic_3} of Lemma~\ref{Lem_PSC_conformal_analytic}.
Moreover $\nabla \td{w}_i =( |\nabla w (p)| + \alpha_i^2 \varphi'_i (t_i) ) \gamma'(p) = 0$ for large $i$.

This shows that we can find a function $w \in C^\infty (M)$ satisfying Properties~\ref{prop_PSC_conformal_analytic_1}--\ref{prop_PSC_conformal_analytic_3} of Lemma~\ref{Lem_PSC_conformal_analytic} and $\nabla w (p) = 0$.
Fix $w$ for the remainder of the proof.
Choose some $a \in \IR$ such that for $i := w(p) - a \cdot d^2 (p, \cdot)$ we have $8 \triangle u - R u < 0$ near $p$.
Our goal will be to interpolate between $w$ and $u$.
Let $f : \IR \to [0,1]$ be a smooth cutoff function with $f \equiv 1$ on $(-\infty, -2]$ and $f \equiv 0$ on $[-1, \infty)$.
Choose a sequence of positive numbers $\eps_i \to 0$ and set
\[ \nu_i (r) := f (\eps_i \log r ), \qquad \td{w}_i := w + \nu_i (d(p, \cdot))  (u-w). \]
Then for large $i$ the function $\td{w}_i$ satisfies (\ref{eq_w_std_form}) and we have $\td{w}_i = w$ near $\partial M$.
It remains to show that for we have $8 \triangle \td{w}_i - R \td{w}_i < 0$ for large $i$.
To see this, observe first that $\td{w}_i \to w$ in $C^0$.
Next we compute with $r := d(p, \cdot)$
\begin{multline*}
 \triangle \td{w}_i - \triangle w = \triangle \nu_i (r) (u-w) + 2 \nabla \nu_i \nabla (u-w) + \nu_i \triangle (u-w) \\
 \leq \nu_i'' (u-w) + \nu_i' \cdot \triangle r \cdot (u-w) + 2 |\nu'_i| \cdot |\nabla (u-w)| + \nu_i \triangle (u-w) \\
 \leq C \eps_i r^{-2} \cdot C r^2 + C \eps_i r^{-1} \cdot C r^{-1} \cdot C r^2 + C \eps_i r^{-1} \cdot C r + \nu_i \triangle (u-w).
\end{multline*}
It follows that
\[ \triangle \td{w}_i \leq C \eps_i + \nu_i \triangle u + (1-\nu_i ) \triangle w. \]
This implies that
\[ 8 \triangle \td{w}_i - R \td{w}_i \leq C \eps_i + \nu_i (8 \triangle u - R u) + (1-\nu_i) (8 \triangle w - R w) . \]
The sum of the second and third term is strictly negative and independent of $i$; so since $C \eps_i \to 0$, the right-hand side is strictly negative for large $i$.
\end{proof}

\begin{lemma}
\label{lem_thick_annulus_psc_conformal}
Let $M$ be a compact 3-manifold with boundary and consider a continuous family of Riemannian metrics $(g_s)_{s \in X}$ on $M$, where $X$ is a compact topological space.
Suppose that $(M, g_s)$ is PSC-conformal for all $s \in X$.

Consider a continuous family of embeddings $(\mu_s : B^3 (1) \to M)_{s \in X}$ and suppose that for every $s \in X$ the pullback metric $\mu^*_s g_s$ is  compatible with the standard spherical structure on $B^3(1)$.
Then there is a constant $0 < r_0 < 1$ such that for all $0 < r \leq r_0$ and $s \in X$ the Riemannian manifold $(M \setminus \mu_s (B^3 (r)), g_s)$ is PSC-conformal.
\end{lemma}

\begin{proof}
Due to the openness of the PSC-conformal condition from Lemma~\ref{Lem_PSC_conformal_open} and the fact that $X$ is compact, it suffices to prove the lemma for a single $s \in X$.
So let us write in the following $g = g_s$ and $\mu = \mu_s$.
In addition, by Lemma~\ref{Lem_PSC_conformal_enlarge}, it suffices to show that there is an $r_0 > 0$ such that $(M \setminus \mu (B^3 (0,r_0)), g)$ is PSC-conformal.
Using the exponential map, we may moreover assume that $\mu^* g$ is even invariant under the standard $O(3)$-action.

By Lemma~\ref{Lem_w_std_form} we may choose $w \in C^\infty (M)$ satisfying Properties~\ref{prop_PSC_conformal_analytic_1}--\ref{prop_PSC_conformal_analytic_3} of Lemma~\ref{Lem_PSC_conformal_analytic} and (\ref{eq_w_std_form}) for $p = \mu (0)$.
Then $( B^3 (1) \setminus \{ 0 \}, \mu_1^* (w^4 g))$ is isometric to 
\begin{equation} \label{eq_f4_cyl}
 \big( S^2 \times (0,\infty), f^4 ( g_{S^2} + dt^2  )  \big),
\end{equation}
where $f : (0, \infty) \to \IR_+$ is smooth with
\begin{equation} \label{eq_f_conditions_to_0}
 \lim_{t \to \infty} f (t) = \lim_{t \to \infty} f' (t) = 0, \qquad 8 f'' - 2 f < 0. 
\end{equation}
The last condition is equivalent to the statement that the metric (\ref{eq_f4_cyl}) has positive scalar curvature.

It remains to show that there is a number $t_0 > 0$ and a smooth function $\td{f} : (0, t_0] \to \IR_+$ such that 
\begin{enumerate}
\item $\td{f} = f$ near $0$, 
\item $8\td{f}'' - 2 \td{f} < 0$, 
\item $\td{f} (t_0) = 1$ and
\item $\td{f}'(t_0) = 0$.
\end{enumerate}
Due to a standard smoothing argument it suffices to construct $\td{f}$ to be piecewise smooth and with the property that $\frac{d}{dt^-} f \geq \frac{d}{dt^+} f$ at every non-smooth point.
Moreover, since every function with $\td{f}'(t_0') \geq 0$ can be continued by a constant function for $t \geq t_0'$, we may replace Property~(4) by $\td{f}' (t_0) \geq 0$.

Let us now construct $\td{f}$.
The conditions (\ref{eq_f_conditions_to_0}) imply that $2f' (t_1) > - f(t_1)$ and $f(t_1) < 1$ for some $t_1 > 0$, because otherwise we have for large $t$
\[ 4(f')^2 - f^2 > 0, \qquad (4(f')^2 - f^2)' = (8f'' - 2 f) f' > 0,\]
which contradicts the first two limits of (\ref{eq_f_conditions_to_0}).
Fix $t_1$ and $\delta > 0$ and set
\[ \td{f} (t) := \begin{cases} f(t) & \text{if $t \leq t_1$} \\ f(t_1) \cosh (\frac{t-t_1}{2+\delta}) + (2+\delta) f'(t_1) \sinh( \frac{t-t_1}{2+\delta}) & \text{if $t > t_1$} \end{cases}. \]
For sufficiently small $\delta$ we have $(2+\delta) f'(t_1) + f(t_1) > 0$ and therefore $\lim_{t \to \infty} \td{f}(t) = \infty$ and $\td{f}(t) > 0$ for all $t \geq t_1$.
Thus we can choose $t_0 > t_1$ such that $\td{f} (t_0) = 1$ and $\td{f}' (t_0) \geq 0$.
\end{proof}

\subsection{Extending symmetric metrics}
The following proposition will be used in the proof of Proposition~\ref{prop_extending}.
Its purpose will be to extend the domain of a family of metrics by a subset that is equipped with a family of spherical structures, while preserving the PSC-conformal condition.

\begin{proposition} \label{prop_extending_symmetric}
Consider a 3-manifold with boundary $M$ and a compact connected 3-di\-men\-sion\-al submanifold with boundary $Y \subset M$.
Let $(\SS^s)_{s \in D^n}$, $n \geq 0$, be a transversely continuous family of spherical structures defined on open subsets $Y \subset U^s \subset M$ such that $\partial Y$ is a union of regular fibers of $\SS^s$ for all $s \in D^n$.
Consider the following continuous families of metrics:
\begin{itemize}
\item $(g^1_{s})_{s \in D^n}$ on $M$
\item $(g^2_{s,t})_{s \in D^n, t \in [0,1]}$ on $\ov{M \setminus Y}$
\item $(g^3_{s,t})_{s \in \partial D^n, t \in [0,1]}$ on $M$
\end{itemize}
Assume that $g^1_s$, $g^2_{s,t}$ and $g^3_{s,t}$ are compatible with $\SS^s$ for all $s \in D^n$ or $\partial D^n$ and $t \in [0,1]$ and assume that the following compatiblity conditions hold:
\begin{enumerate}[label=(\roman*)]
\item \label{prop_lem_extending_symmetric_i} $g^1_s = g^2_{s, 0}$ on $\ov{M \setminus Y}$ for all $s \in D^n$.
\item \label{prop_lem_extending_symmetric_ii} $g^1_s = g^3_{s,0}$ on $M$ for all $s \in \partial D^n$.
\item \label{prop_lem_extending_symmetric_iii} $g^2_{s,t} = g^3_{s,t}$ on $\ov{M \setminus Y}$ for all $s \in \partial D^n$, $t \in [0,1]$.
\end{enumerate}
Then there is a continuous family of metrics $(h_{s,t})_{s \in D^n, t \in [0,1]}$ on $M$ such that $h_{s,t}$ is compatible with $\SS^s$ for all $s \in D^n$, $t \in [0,1]$ and such that:
\begin{enumerate}[label=(\alph*)]
\item \label{ass_lem_extending_symmetric_a} $h_{s,0} = g^1_s$ on $M$ for all $s \in D^n$.
\item \label{ass_lem_extending_symmetric_b} $h_{s,t} = g^2_{s,t}$ on $\ov{M \setminus Y}$ for all $s \in D^n$, $t \in [0,1]$.
\item \label{ass_lem_extending_symmetric_c} $h_{s,t} = g^3_{s,t}$ on $M$ for all $s \in \partial D^n$, $t \in [0,1]$.
\end{enumerate}
\end{proposition}

The proof of this proposition will occupy the remainder of this subsection.
Our strategy will be to successively simplify the statement in a sequence of lemmas.

\begin{proof}
By Lemma~\ref{Lem_cont_fam_sph_struct} there is a connected open neighborhood $Y \subset V \subset M$ and a continuous family of embeddings $(\omega_s : V \to M)_{s \in D^n}$ such that $\omega_s (Y) = Y$ and such that the following holds: if $\SS^{\prime,s}$ denotes the pullback of $\SS^s$ via $\omega_s$, then $V$ is a union of fibers of $\SS^{\prime,s}$ and $\SS^{\prime,s}$ restricted to $V$ is constant in $s$.
After replacing $M$ by $V$, $(\SS^s)_{s \in D^n}$ by $(\SS^{\prime,s})_{s \in D^n}$, $(g^1_s)_{s \in D^n}$ by $(\omega_s^* g^1_s)_{s \in D^n}$ etc., we may assume without loss of generality that $\domain (\SS^s) = M$ and that $\SS^s =: \SS$ is constant in $s$.

Since $(D^n \times [0,1], \partial D^n \times [0,1] \cup D^n \times \{ 0 \} )$ is homeomorphic $(D^n \times [0,1], D^n \times \{ 0 \})$, we can simplify the proposition further by removing the family $(g^3_{s,t})_{s \in \partial D^n, t \in [0,1]}$ as well as Assumptions~\ref{prop_lem_extending_symmetric_ii}, \ref{prop_lem_extending_symmetric_iii} and Assertion~\ref{ass_lem_extending_symmetric_c}.

If $\partial Y = \emptyset$, then $M \setminus Y = \emptyset$, so Assumption~\ref{prop_lem_extending_symmetric_i} and Assertion~\ref{ass_lem_extending_symmetric_b} are vacuous and we can simply set $h_{s,t} := g^1_s$.
If $\partial Y \neq \emptyset$, then $Y$ is diffeomorphic to one of the following manifolds (see Lemma~\ref{lem_spherical_struct_classification}):
\[ S^2 \times [0,1], \; (S^2 \times [-1,1])/ \IZ_2, \; D^3. \]
By removing collar neighborhoods of $\partial Y$ from $Y$, we can construct a compact 3-manifold with boundary $Z \subset \Int Y$ that is a union of fibers of $\SS$ such that $Y' := \ov{Y \setminus Z}$ is diffeomorphic to a disjoint union of copies of $S^2 \times I$.
Define $g^{\prime,2}_{s,t} := g^2_{s,t}$ on $\ov{M \setminus Y}$ and $g^{\prime,2}_{s,t} := g^1_{s,0}$ on $Z$.
Then $(g^{\prime,2}_{s,t})_{s \in D^n, t \in [0,1]}$ is defined on $\ov{M \setminus Y'}$.
By replacing $Y$ by $Y'$ and $(g^{2}_{s,t})_{s \in D^n, t \in [0,1]}$ by $(g^{\prime,2}_{s,t})_{s \in D^n, t \in [0,1]}$, we can reduce the proposition to the case in which $Y$ is diffeomorphic to a disjoint union of copies of $S^2 \times [0,1]$.
Moreover, due to Assertion~\ref{ass_lem_extending_symmetric_b}, we can handle each component of $Y$ separately.
Therefore the proposition can be reduced to Lemma~\ref{lem_extending_symmetric_simplified} below.
\end{proof}

\begin{lemma} \label{lem_extending_symmetric_simplified}
Let $n \geq 0$ and let  $(M, Y) := ( S^2 \times (-2, 2), S^2 \times [-1,1])$.
Let $\SS$ be the standard spherical structure on $M$ and consider continuous families of metrics:
\begin{itemize}
\item $(g^1_s)_{s \in D^n}$ on $M$,
\item $(g^2_{s,t})_{s \in  D^n, t \in [0,1]}$ on $\ov{M \setminus Y}$
\end{itemize}
that are all compatible with $\SS$ and that satisfy the compatibility condition $g^1_s = g^2_{s,0}$ on $\ov{M \setminus Y}$ for all $s \in D^n$.
Then there is a continuous family of metrics $(h_{s,t})_{s \in D^n, t \in [0,1]}$ on $M$ such that:
\begin{enumerate}[label=(\alph*)]
\item $h_{s,t}$ is compatible with $\SS$  for all $s \in D^n$ and $t \in [0,1]$.
\item $h_{s,0} = g^1_s$ on $M$ for all $s \in D^n$,
\item $h_{s,t} = g^2_{s,t}$ on $\ov{M \setminus Y}$ for all $s \in D^n$ and $t \in [0,1]$.
\end{enumerate}
\end{lemma}

\begin{proof}
By Lemma~\ref{lem_compatible_metric_general_form} the metrics $g^1_s$ and $g^2_{s,t}$ are of the form
\[ a^2(r) g_{S^2} + b^2(r) dr^2 + \sum_{i=1}^3 c_i (r)  (dr \, \xi_i + \xi_i \, dr), \]
where $a,b, c_i$ are smooth functions and $a,b > 0$.
So by considering the functions $\log a, \log b, c_i$,  the lemma can be reduced to Lemma~\ref{lem_extending_symmetric_simplified_functions} below.
\end{proof}

\begin{lemma} \label{lem_extending_symmetric_simplified_functions}
Let $n \geq 0$ and consider continuous families of functions:
\begin{itemize}
\item $(f^1_s \in C^\infty ((-2,2)))_{s \in D^n}$,
\item $(f^2_{s,t} \in C^\infty ( (-2,-1] \cup [1,2))_{s \in  D^n, t \in [0,1]}$
\end{itemize}
that satisfy the compatibility condition $f^1_s = f^2_{s,0}$ on $(-2,-1] \cup [1,2)$ for all $s \in D^n$.
Then there is a continuous family of functions $(f^3_{s,t}\in C^\infty ((-2,2)))_{s \in D^n, t \in [0,1]}$ such that:
\begin{enumerate}[label=(\alph*)]
\item $f^3_{s,0} = f^1_s$ on $(-2,2)$ for all $s \in D^n$,
\item $f^3_{s,t} = f^2_{s,t}$ on $(-2,-1] \cup [1,2)$ for all $s \in D^n$ and $t \in [0,1]$.
\end{enumerate}
\end{lemma}

\begin{proof}
It suffices to prove the lemma in the case in which $f^1_{s,t} \equiv 0$, because after applying the lemma to the functions $\td{f}^1_s :\equiv 0$ and $\td{f}^2_{s,t} := f^2_{s,t} - f^1_s$, resulting in a family of functions $\td{f}^3_{s,t}$, we may set $f^3_{s,t} := \td{f}^3_{s,t} + f^1_s$.
So assume in the following that $f^1_{s,t} \equiv 0$ and note that this implies $f^2_{s,0} \equiv 0$.
By Seeley's Theorem \cite{Seeley1964}, we may extend the family $(f^2_{s,t})$ to a continuous family $(\ov{f}^2_{s,t} \in C^\infty ( (-2,-.8) \cup (.8,2))_{s \in  D^n, t \in [0,1]}$ such that $\ov{f}^2_{s,0} \equiv 0$ for all $s \in D^n$.
Let $\eta \in C^\infty ((-2,2))$ be a cutoff function such that $\eta \equiv 1$ on $(-2,-1] \cup [1,2)$ and $\eta \equiv 0$ on $[-.9,.9]$.
Then $f^3_{s,t} := \eta \ov{f}^2_{s,t}$ has the desired properties.
\end{proof}

\subsection{Extending symmetric metrics on the round sphere}
The following proposition is similar in spirit to Proposition~\ref{prop_extending_symmetric} and will also be used in the proof of Proposition~\ref{prop_extending}.
It concerns deformations of metrics compatible with a family of spherical structures $(\SS^s)$ if the starting metric is the round sphere or an isometric quotient of it.
In this case the spherical structures $\SS^s$ are not uniquely determined by the metric.
Due to this ambiguity, the spherical structure $\SS^s$ may not be defined for certain parameters $s$; in this case we will require the associated metrics to be multiples of a fixed round metric.

\begin{proposition} \label{prop_exten_round}
Let $(M, g^*)$ be a compact 3-manifold of constant sectional curvature $1$ and let $\Delta^n$ be the standard $n$-simplex for $n \geq 0$.
Consider the following data:
\begin{itemize}
\item a continuous function $\lambda : \Delta^n \to \IR_+$,
\item a continuous family of metrics $(k_{s,t})_{s \in \partial \Delta^n, t \in [0,1]}$ on $M$
\item an open subset $A \subset \Delta^n$,
\item a closed subset $E \subset \partial \Delta^n$ such that $E \subset A$
\item a transversely continuous family of spherical structures $(\SS^s)_{s \in A}$ on $M$.
\end{itemize}
Assume that
\begin{enumerate}[label=(\roman*)]
\item $k_{s,0} = \lambda^2 (s) g^*$ for all $s \in \partial \Delta^n$
\item For all $s \in A$ the metric $g^*$ is compatible with $\SS^s$.
\item \label{prop_exten_round_iii} For all $s \in E$ and $t \in [0,1]$ the metric $k_{s,t}$ is compatible with $\SS^s$.
\item For all $s \in \partial \Delta^n \setminus E$ and $t \in [0,1]$ the metric $k_{s,t}$ is a multiple of $g^*$.
\end{enumerate}
Then there is a continuous family of metrics $(h_{s,t})_{s \in \Delta^n, t \in [0,1]}$ on $M$ and a closed subset $E' \subset \Delta^n$ such that:
\begin{enumerate}[label=(\alph*)]
\item  \label{ass_exten_round_aa} $h_{s,0} = \lambda^2 (s) g^*$ for all $s \in \Delta^n$.
\item \label{ass_exten_round_a} $h_{s,t} = k_{s,t}$ for all $s \in \partial \Delta^n$, $t \in [0,1]$.
\item \label{ass_exten_round_b} $E' \subset A$.
\item \label{ass_exten_round_c} For any $s \in A$ and $t \in [0,1]$ the metric $h_{s,t}$ is compatible with $\SS^s$.
\item \label{ass_exten_round_d} For all $s \in  \Delta^n \setminus E'$ and $t \in [0,1]$ the metric $h_{s,t}$ is a multiple of $g^*$.
\end{enumerate}
\end{proposition}

Note that if $M \not\approx S^3, \IR P^3$, then $E= A = \emptyset$, in which case the proposition is trivial.

\begin{proof}
Let us first reduce the proposition to the case in which either $E = \emptyset$ or $A = \Delta^n$.
To see this, assume that the proposition is true in these two cases, in any dimension.
Choose a simplicial refinement of $\Delta^n$ that is fine enough such that every subsimplex $\sigma \subset \Delta^n$ is either fully contained in $A$ or disjoint from $E$.
Then we may successively construct $(h_{s,t})$ over the skeleta of this decomposition.
More specifically, let $0 \leq k \leq n$ and assume by induction that either $k = 0$ or that some family of metrics $(h^{k-1}_{s,t})$ has been constructed over the $(k-1)$-dimensional skeleton $X^{k-1}$ of $\Delta^n$ such that Assertions~\ref{ass_exten_round_aa}--\ref{ass_exten_round_d} hold for some closed subset $E'_{k-1} \subset X^{k-1}$, where in Assertion~\ref{ass_exten_round_d} we have to replace the difference $\Delta^n \setminus E'$ by $X^{k-1} \setminus E'_{k-1}$.
For every $k$-simplex $\sigma \subset \Delta^n$ consider the closed subset $E_\sigma := E'_{k-1} \cap \partial \sigma$ and apply the proposition to find an extension $(h^\sigma_{s,t})_{s \in \sigma, t \in [0,1]}$ of $(h^{k-1}_{s,t})$ to $\sigma$ and a closed subset $E'_\sigma \subset \sigma$ such that Assertions~\ref{ass_exten_round_aa}, \ref{ass_exten_round_b}--\ref{ass_exten_round_d} of the proposition hold.
Set $E'_k := \cup_{\sigma \subset \Delta^n, \dim \sigma = k} E'_\sigma$ and combine all families $(h^\sigma_{s,t})$ to a family $(h^k_{s,t})$ over $X^k$.
This finishes the induction.
Then $(h_{s,t}) := (h^n_{s,t})$ and $E' := E'_n$ are the desired data for this proposition.

So it remains to prove the proposition in the two cases $E = \emptyset$ and $A = \Delta^n$.

Consider first the case in which $E = \emptyset$.
Then $k_{s,t} = \mu^2 (s,t) g^*$ for some continuous function $\mu : \partial \Delta^n \times [0,1] \to \IR_+$ that satisfies $\mu(s,0) = \lambda (s)$ for all $s \in \partial \Delta^n$.
Let $\lambda_0 : \Delta^n \times [0,1] \to \IR_+$ be a continuous function such that $\lambda_0 (s,0) = \lambda (s)$ for all $s \in \Delta^n$ and $\lambda_0 (s,t)= \mu(s,t)$ for all $s \in \partial \Delta^n$, $t \in [0,1]$.
Then $(h_{s,t} := \lambda_0^2 (s,t) g^*)$ and $E' := \emptyset$ are the desired data for this proposition.

Lastly, consider the case in which $A = \Delta^n$.
So $\SS^s$ is defined for all $s \in \Delta^n$.
For all $s \in \partial\Delta^n - E$ and $t \in [0,1]$ the metric $k_{s,t}$ is a multiple of $g^*$ and therefore compatible with $\SS^s$.
For all $s \in  E$ and $t \in [0,1]$ the metric $k_{s,t}$ is compatible with $\SS^s$ by Assumption~\ref{prop_exten_round_iii}.
Therefore $k_{s,t}$ is compatible with $\SS^s$ for all $s \in \partial \Delta^n$ and $t \in [0,1]$.
Set $E' := \Delta^n$.
Then Assertion~\ref{ass_exten_round_d} becomes vacuous and the existence of $(h_{s,t})$ is a direct consequence of Proposition~\ref{prop_extending_symmetric}.
\end{proof}

\subsection{The conformal exponential map} \label{subsec_conf_exp}
In this subsection, we will introduce the \emph{conformal exponential map}, which will be used in the rounding construction of Subsection~\ref{subsec_Rounding}.
The conformal exponential map will produce a set of canonical local coordinates near a point $p \in M$ in a Riemannian manifold $(M,g)$ with the following properties:
\begin{itemize}
\item If $g$ is locally rotationally symmetric about $p$, then so is the metric in these local coordinates.
\item If the metric is  conformally flat near $p$, then the metric expressed in these local coordinates is conformally equivalent to the standard Euclidean metric.
\item The metric expressed in these local coordinates agrees with the Euclidean metric up to first order at the origin.
\end{itemize}

In the following let $(M, g)$ be a 3-dimensional Riemannian manifold.
Recall that the Schouten tensor is defined as follows:
\[ S = \Ric - \frac14 R g \]
If $\td{g} = e^{2\phi} g$, then the Schouten tensor $\td{S}$ of $\td{g}$ can be expressed as
\begin{equation} \label{eq_Schouten_conf_trafo}
 \td{S}_{ij} = S_{ij} - ( \nabla^2_{ij} \phi - \nabla_i \phi \nabla_j \phi ) - \frac12 |\nabla \phi |^2 g_{ij}. 
\end{equation}
If $\td{S} \equiv 0$, which implies that $g$ is  conformally flat, then we can recover the function $\phi$ by integrating an expression involving $S$ twice along curves.
We will now carry out this integration even in the case in which $g$ is not conformally flat.
In doing so, we will construct a locally defined function $\phi$, with the property that if $g$ is  conformally flat, then $e^{2\phi} g$ is flat.

Let $p \in M$ be a point with injectivity radius $\injrad(p)$.
By standard ODE-theory there is a maximal radius $r=r_{g,p} \in (0, \injrad(p)]$, depending continuously on the metric $g$ and the point $p$, such that we can solve the following ODE for 1-forms radially along arclength geodesics $\gamma : [0,r) \to M$ emanating at $\gamma(0)=p$:
\[ \alpha (p) = 0, \qquad \nabla_{\gamma'} \alpha = S(\gamma', \cdot) + \alpha( \gamma' ) \alpha - \frac12 |\alpha |^2   g (\gamma', \cdot ). \]
We claim that this defines a unique smooth 1-form $\alpha$ on $B(p,r)$.
To see that $\alpha$ is smooth consider the exponential map $v \mapsto \gamma_v (1) = \exp_p (v)$ and notice that the parallel transports $\alpha_v (s) := P^{\gamma_v}_{0,s} \alpha$ of $\alpha (\gamma_v (s))$ to $p$ satisfies the ODE
\[ \alpha_v (0) = 0, \qquad \alpha'_v (s) =   (P^{\gamma_v}_{0,s} S)(v, \cdot) + \alpha_v (s) (v) \alpha_v (s) - \frac12 |\alpha_v^2 (s)| g_p( v, \cdot ). \]
The value $\alpha_v (1) = P^{\gamma_v}_{0,1} \alpha$ depends smoothly on the vector $v \in T_p M$.
Next, we construct the following function $\phi \in C^\infty (B(p,r))$ by radial integration along geodesics $\gamma$ emanating from $p$:
\[ \phi (p) = 0, \qquad \nabla_{\gamma'} \phi = \alpha (\gamma'). \]
The smoothness of $\phi$ follows similarly as before.

Note that by uniqueness of solutions to ODEs we have:

\begin{lemma} \label{Lem_recover_phi}
Let $\gamma : [0, r) \to M$ be a geodesic emanating at $\gamma (0)=p$ and let $t \in [0,r)$.
Suppose that there is a neighborhood $\gamma(t) \in U' \subset B(p,r)$ and a function $\phi' \in C^\infty (U')$ such that $e^{2\phi'} g$ is flat.
If $\phi' = \phi$, $d\phi' = d\phi$ at $\gamma (t)$, then the same is true along the component of $\gamma ([0,r)) \cap U'$ containing $\gamma(t)$.
\end{lemma}

We can now define the conformal exponential map.

\begin{definition}[Conformal exponential map]
After constructing $\phi \in C^\infty (B(p,r))$ as above, set
\[ \confexp_p := \exp_{e^{2 \phi} g, p} : U^{conf}_p \longrightarrow B(p,r), \]
where $U^{conf}_p := \exp_{e^{2 \phi} g, p}^{-1} (B(p,r)) \subset T_p M$.\end{definition}

The following proposition summarizes the properties of the conformal exponential map:

\begin{proposition} \label{Prop_properties_confexp}
Let $B(p,r') \subset B(p,r)$ be a possibly smaller ball and consider the pullback $g' := \confexp_p^* g$ onto $\confexp_p^{-1} (B(p,r')) \subset T_p M$.
Denote by $g_{\eucl,p}$ the standard Euclidean metric on $T_p M$. 
Then:
\begin{enumerate}[label=(\alph*)]
\item \label{ass_properties_confexp_a} If $(g_s)_{s \in X}$ is a continuous family of Riemannian metrics on $M$ (in the smooth topology), then the corresponding family of conformal exponential maps $(\confexp_{g_s,p})_{s \in X}$ also depends continuously on $s$ (in the smooth topology).
\item \label{ass_properties_confexp_b} If $g$ restricted to $B(p,r')$ is rotationally symmetric about $p$, i.e. if $g$ is compatible with a spherical structure on $B(p,r')$ with singular fiber $\{ p \}$, then $\confexp_p^{-1} (B(p,r'))$ and $g'$ invariant under the standard $O(3)$-action.
\item \label{ass_properties_confexp_c} If $g$ restricted to $B(p,r')$ is  conformally flat, then $g' = e^{2\phi'} g_{\eucl,p}$ for some $\phi' \in C^\infty (\confexp_p^{-1} (B(p,r')))$. 
\item \label{ass_properties_confexp_d} At $p$ we have $g' - g_{\eucl,p} = \partial ( g' - g_{\eucl,p} ) = 0$.
\end{enumerate}
\end{proposition}

\begin{proof}
Assertion~\ref{ass_properties_confexp_a} follows by construction.
For Assertion~\ref{ass_properties_confexp_b} observe that if $g$ is rotationally symmetric on $B(p,r')$, then so is $\phi$.

For Assertion~\ref{ass_properties_confexp_c} it suffices to show that $e^{2\phi} g$ restricted to $B(p,r')$ is flat.
For this purpose choose an arclength geodesic $\gamma : [0,r') \to M$ emanating at $\gamma (0) = p$.
Let $t_0 \in [0, r')$ be maximal with the property that $e^{2\phi} g$ is flat in a neighborhood $U \subset B(p,r')$ of $\gamma ([0,t_0))$.
By Lemma~\ref{Lem_recover_phi} above and Lemma~\ref{Lem_global_conf_flat_U_sc} below we have $t_0 > 0$.
Assume now that $t_0 < r'$ and set $q := \gamma (t_0)$.
By Lemma~\ref{Lem_global_conf_flat_U_sc} we can find a function $\phi' \in C^\infty (V)$ defined in a neighborhood $q \in V \subset B(p,r')$ such that $e^{2\phi'} g$ is flat and $\phi' (q) = \phi (q)$, $d\phi'_q = d\phi_q$.
By Lemma~\ref{Lem_recover_phi}, we obtain that $\phi = \phi'$, $d\phi = d\phi'$ along $\gamma |_{(t_1,t_0]}$ for some $t_1 \in [0, t_0)$.
So by the uniqueness statement in Lemma~\ref{Lem_global_conf_flat_U_sc} we have $\phi' = \phi$ on the connected component of $U \cap V$ containing $\gamma([t_1,t_0))$.
So $e^{2\phi} g$ is flat in a neighborhood of $\gamma ([0,t_0])$, in contradiction to the maximal choice of $t_0$.

For Assertion~\ref{ass_properties_confexp_d} observe that $\phi (p) = d\phi_p = 0$ and $e^{2\phi \circ \confexp_p} g'  - g_{\eucl,p} = \exp^*_{e^{2\phi} g,p} (e^{2\phi} g) - g_{\eucl, p} = O(r^2)$.
\end{proof}

\begin{lemma} \label{Lem_global_conf_flat_U_sc}
If $(M, g)$ is  conformally flat near some point $q \in M$, then for any $a \in \IR$, $\alpha \in T_q^* M$ there is a neighborhood $q \in V \subset M$ and a $\phi \in C^\infty (V)$ such that $e^{2\phi} g$ is flat and $\phi (q) = a$, $d\phi_q = \alpha$.
Moreover $\phi$ is unique modulo restriction to a smaller neighborhood of $q$.\end{lemma}

\begin{proof}
By the local conformal flatness, we can find an open neighborhood $q \in V' \subset M$ and a $\phi' \in C^\infty (V')$ such that $g' = e^{2\phi'} g$ is flat on $V'$.
Since $e^{2\phi} g = e^{2 (\phi - \phi')} g'$ for any smooth function $\phi$ that is defined near $p$, we may replace $M$ by $V'$ and $g$ by $e^{2\phi} g$.
So by isometrically identifying $(M,g)$ with a subset of $\IR^n$ we may assume without loss of generality that $M \subset \IR^3$ and $g = g_{\eucl}$.

Recall that by (\ref{eq_Schouten_conf_trafo}) local conformal flatness is equivalent to the PDE
\begin{equation} \label{eq_conf_flat_psi}
 -\nabla^2_{ij} \phi + \nabla_i \phi \nabla_j \phi - \frac12 |\nabla \phi|^2 \delta_{ij} = 0 
\end{equation}
If $\alpha = 0$, then we can set $\phi \equiv a$, otherwise we can set $\phi (x) := -2\log |x-y| + 2 \log |y| + a$ for $y \in \IR^3$ with $2 \frac{y}{|y|^2} = \alpha$.
This proves existence.
Uniqueness follows by viewing (\ref{eq_conf_flat_psi}) restricted to lines as an ODE in $\nabla \phi$, as was done in the beginning of this subsection.
\end{proof}

\begin{remark}\label{rmk_Cotton_York}
We recall that a metric is  conformally flat if and only the Schouten tensor satisfies the Cotton-York condition (see \cite{Cotton-1899}):
\[ \nabla_i S_{jk} - \nabla_j S_{ik} = 0 \]
With some more effort Assertion~\ref{ass_properties_confexp_c} of Proposition~\ref{Prop_properties_confexp} can be deduced directly from the Cotton-York condition.
In the context of this paper the Cotton-York condition is, however, not essential, which is why details have been omitted.
\end{remark}

\subsection{Rounding metrics} \label{subsec_Rounding}
The goal of this subsection is to prove the following result, which states that a  metric on the standard $3$-ball can be deformed  to make it compatible with the standard spherical structure on a smaller ball, while preserving the properties of conformal flatness and compatibility with the spherical structure on (most) disks $D(r)$ for $r\in (0,1]$ and PSC-conformality.

The following is the main result of this subsection.

\begin{proposition} \label{Prop_rounding}
Let $X$ be a compact topological space, $X_{PSC} \subset X$ a closed subset and $0 < \ov{r}_1 < 1$.
Consider a continuous family of Riemannian metrics $(h_s)_{s \in X}$ on the unit ball $B^3 \subset \IR^3$ such there is a continuous family of positive functions $(w_s \in C^\infty(B^3))_{s \in X_{PSC}}$ with the property that $w_s^4 h_s$ has positive scalar curvature.

Then there is a continuous family of Riemannian metrics $(h'_{s,u})_{s \in X, u \in [0,1]}$ and a constant $0 < r_1 < \ov{r}_1$ such that for all $s \in X$ and $u \in [0,1]$:
\begin{enumerate}[label=(\alph*)]
\item \label{ass_Prop_rounding_a} $h'_{s,u} = h_s$ on $B^3 \setminus B^3 (\ov{r}_1)$.
\item \label{ass_Prop_rounding_b} $h'_{s,0} = h_s$.
\item \label{ass_Prop_rounding_c} $h'_{s,1} = h_s$ is compatible with the standard spherical structure on $D^3 (r_1)$.
\item \label{ass_Prop_rounding_d} If $h_s$ is compatible with the standard spherical structure on $D^3(r)$ for some $r \in [\ov{r}_1,1]$, then so is $h'_{s,u}$.
\item \label{ass_Prop_rounding_e} If $h_s$ is conformally flat, then so is $h'_{s,u}$.
\item \label{ass_Prop_rounding_f} If $s \in X_{PSC}$, then $w_{s,u}^{\prime,4} h'_{s,u}$ has positive scalar curvature, for some $w'_{s,u} \in C^\infty(B^3)$ that agrees with $w_s$ on $B^3 \setminus B^3 (\ov{r}_1)$.
\end{enumerate}
\end{proposition}

We will reduce Proposition~\ref{Prop_rounding} to Lemma~\ref{Lem_rounding_intermediate} below.

\begin{lemma} \label{Lem_rounding_intermediate}
Assume that we are in the same situation as Proposition~\ref{Prop_rounding} and assume additionally that for all $s \in X$
\begin{enumerate}[label=(\roman*)]
\item \label{hyp_Lem_rounding_intermediate_i} $(h_s)_0 = (\partial h_s)_0 = 0$.
\item \label{hyp_Lem_rounding_intermediate_ii} If $h_s$ is  conformally flat, then on $B^3(\ov{r}_1)$ it is even conformally equivalent to the standard Euclidean metric $g_{\eucl}$.
\end{enumerate}
Then there is a continuous family of Riemannian metrics $(h'_{s,u})_{s \in X, u \in [0,1]}$ and a constant $0 < r_1 < \ov{r}_1$ such that Assertions~\ref{ass_Prop_rounding_a}--\ref{ass_Prop_rounding_f} of Proposition~\ref{Prop_rounding}  hold for all $s \in X$ and $u \in [0,1]$.
\end{lemma}

In order to show that Lemma~\ref{Lem_rounding_intermediate} implies Proposition~\ref{Prop_rounding} we need the following lemma:

\begin{lemma} \label{Lem_existence_phi_s_u}
Let $(h_s)_{s \in X}$ be a continuous family of Riemannian metrics on $B^3$ and let $0 < \ov{r}_1 < 1$.
Then there is a constant $0 < \ov{r}'_1 < \ov{r}_1$ and a continuous family of diffeomorphisms $(\phi_{s,u} : B^3 \to B^3)_{s \in X, u \in [0,1]}$ such that for all $s \in X$ and $u \in [0,1]$:
\begin{enumerate}[label=(\alph*)]
\item \label{ass_Lem_existence_phi_s_u_a} $\phi_{s,u} = \id$ on $B^3 \setminus B^3 (\ov{r}_1)$.
\item \label{ass_Lem_existence_phi_s_u_b} $\phi_{s,0} = \id$
\item \label{ass_Lem_existence_phi_s_u_c} The pullback $\phi^*_{s,1} h_s$ agrees with the Euclidean metric $g_{\eucl}$ at the origin up to first order.
\item \label{ass_Lem_existence_phi_s_u_d} If $h_s$ is  conformally flat, then $\phi^*_{s,1} h_s$ restricted to $B^3(\ov{r}'_1)$ is conformally equivalent to $g_{\eucl}$.
\item \label{ass_Lem_existence_phi_s_u_e} If $h_s$ is compatible with the standard spherical structure on $D^3(r)$ for some $r \in [\ov{r}_1,1]$, then so is $\phi^*_{s,u} h_s$.
\end{enumerate}
\end{lemma}

\begin{proof}
Choose $r_2 > 0$ small enough such that $\confexp_{0, h_s} |_{B^3(r_2)} : B^3 (r_2) \to B^3$ is defined for all $s \in X$.
Let $S^+_3$ the set of symmetric positive definite $3 \times 3$ matrices and denote by $I \in S^+_3$ the identity matrix.
For any $s \in X$ choose the matrix $L_s \in S^+_3$ with the property $(h_s)_0 (L_s v, L_s w) = v \cdot w$ for any $v, w \in \IR^3$, where the latter denotes the standard Euclidean inner product.
Then $L_s$ is continuous in $s$ and we can find an $r_3 > 0$ such that the maps
\[ \psi_{s,u} := u \confexp_{0, h_s} \circ L_s |_{B^3(r_3)} + (1-u) \id_{B^3(r_3)} : B^3 (r_3) \to B^3 \]
are defined and continuous for all $s \in X$, $u \in [0,1]$.
Note that for all $s \in X$, $u \in [0,1]$ we have
\[ \psi_{s,u} (0) = 0, \qquad (d\psi_{s,u})_0 = u L_s + (1-u) I \]
and $\psi_{s,1}^* h_s$ agrees with $g_{\eucl}$ at the origin up to first order.

\begin{Claim}
There is a continuous family of diffeomorphisms $(\zeta_A : B^3 \to B^3)_{A \in S^+_3}$, parameterized by the set of symmetric positive definite $3 \times 3$ matrices, such that for all $A \in S^+_3$:
\begin{enumerate}[label=(\alph*)]
\item \label{ass_cl_zeta_A_a} $\zeta_A= \id$ on $B^3 \setminus B^3 (\ov{r}_1)$
\item $\zeta_A (0) = 0$.
\item $(d\zeta_A)_0 = A$.
\item $\zeta_{I}  = \id_{B^3}$.
\end{enumerate}
\end{Claim}

\begin{proof}
Let $S_3$ be the set of symmetric $3 \times 3$ matrices and recall that $\exp : S_3 \to S_3^+$ is a diffeomorphism.
For any compactly supported vector field $V$ on $B^3$ let $\zeta'_V : B^3 \to B^3$ be the flow of $V$ at time $1$.
If $V_0 = 0$, then $(d \zeta_V)_0 = \exp (dV_0)$.
By taking linear combinations, we can find a continuous (or linear) family of vector fields $(V_A)_{A \in S_3}$ on $B^3$ whose support lies inside $B^3(\ov r_1)$ such that $(V_A)_0 = 0$ and $(dV_A)_0 = A$.
We can then set $\zeta_{\exp (A)} := \zeta'_{V_A}$.
\end{proof}

Let $\eta : \IR^3 \to [0,1]$ be a smooth, rotationally symmetric cutoff function with $\eta \equiv 1$ on $B^3(1)$ and $\eta \equiv 0$ on $\IR^3 \setminus B^3(2)$.
Let $r_4 > 0$ and set $\eta_{r_4} := \eta (v/r_4)$ and 
\[ \phi_{s, u} := \eta_{r_4} \psi_{s, u} + (1-\eta_{r_4} ) \zeta_{u L_s + (1-u) I}. \]
A standard limit argument shows that the family $(\phi_{s})_{s \in X}$ consists of diffeomorphisms for sufficiently small $r_4$.

Assertion~\ref{ass_Lem_existence_phi_s_u_a} holds due to Assertion~\ref{ass_cl_zeta_A_a} of the Claim if $r_4$ is chosen sufficiently small.
Assertion~\ref{ass_Lem_existence_phi_s_u_b} holds since $\psi_{s,0} = \eta_{r_4} \id + (1-\eta_{r_4}) \zeta_I = \id$.
Assertions~\ref{ass_Lem_existence_phi_s_u_c} and \ref{ass_Lem_existence_phi_s_u_d} hold for small $\ov{r}'_1$ since $\phi_{s,1} = \psi_{s,1}$ near the origin and due to the discussion preceding the Claim.
For Assertion~\ref{ass_Lem_existence_phi_s_u_e} assume that $h_s$ is compatible with the standard spherical structure on $D^3(r)$ for some $r \in [\ov{r}_1, 1]$.
Then $h_s$ agrees with $g_{\eucl}$ at the origin up to first order, so $L_s = I$.
Moreover $\confexp_{0,h_s}$ preserves the standard spherical structure on $\confexp^{-1}_{0,h_s} (D^3 (r))$.
Therefore $\phi_{s,u}$ preserves the standard spherical structure on $\phi^{-1}_{s,u} (D^3 (r)) = D^3(r)$, since $r \geq \ov{r}_1$.
This implies Assertion~\ref{ass_Lem_existence_phi_s_u_e}. 
\end{proof}

\begin{proof}[Proof that Lemma~\ref{Lem_rounding_intermediate} implies Proposition~\ref{Prop_rounding}]
Consider the constant $0 < \ov{r}'_1 < \ov{r}_1$ and the family of diffeomorphisms $(\phi_{s,u} : B^3 \to B^3)_{s \in X, u \in [0,1]}$ from Lemma~\ref{Lem_existence_phi_s_u}.
Apply Lemma~\ref{Lem_rounding_intermediate} with $\ov{r}_1$ replaced by $\ov{r}'_1$ to $(\phi^*_{s,1} h_s)_{s \in X}$, consider the constant $0 < r_1 < \ov{r}_1$ and call the resulting family of metrics $(\td{h}'_{s,u})_{s \in X, u \in [0,1]}$.
Set
\[ h'_{s,u} := \begin{cases} \phi^*_{s, 2u} h_s & \text{if $u \in [0, \frac12]$} \\
\td{h}'_{s,2u-1} & \text{if $u \in (\frac12, 1]$} \end{cases} \]
Note that this family is continuous due to  Lemma~\ref{Lem_rounding_intermediate}\ref{ass_Prop_rounding_b}.

We claim that $(h'_{s,u})_{s \in X, u \in [0,1]}$ satisfies Assertions~\ref{ass_Prop_rounding_a}--\ref{ass_Prop_rounding_f} of this proposition.
For Assertion~\ref{ass_Prop_rounding_a} observe that on $B^3 \setminus B^3 (\ov{r}_1)$ we have $h'_{s,u} = \phi^*_{s, 2u} h_s = h_s$ if $u \in [0,\frac12]$, due to Lemma~\ref{Lem_existence_phi_s_u}\ref{ass_Lem_existence_phi_s_u_a} and $h'_{s,u} = h_s$ on $B^3 \setminus B^3 (\ov{r}_1) \subset B^3 \setminus B^3 (\ov{r}'_1)$ if $u \in [\frac12,1]$, due to Lemma~\ref{Lem_rounding_intermediate}\ref{ass_Prop_rounding_a}.
Assertion~\ref{ass_Prop_rounding_b} holds since $\phi_{s,0} = \id$; compare with Lemma~\ref{Lem_existence_phi_s_u}\ref{ass_Lem_existence_phi_s_u_b}.
Assertions~\ref{ass_Prop_rounding_c}, \ref{ass_Prop_rounding_e} and \ref{ass_Prop_rounding_f} hold due to the same assertions of Lemma~~\ref{Lem_rounding_intermediate}, because $h'_{s,1} = \td{h}'_{s,1}$ and due to Lemma~\ref{Lem_existence_phi_s_u}\ref{ass_Lem_existence_phi_s_u_a}.
Consider now Assertion~\ref{ass_Prop_rounding_d} and assume that $h_s$ is compatible with the standard spherical structure on $D^3 (r)$ for some $r \in [\ov{r}_1,1]$.
By Lemma~\ref{Lem_existence_phi_s_u}\ref{ass_Lem_existence_phi_s_u_e} the same is true for the pullbacks $\phi^*_{s,u} h_s$ and therefore by Lemma~\ref{Lem_rounding_intermediate}\ref{ass_Prop_rounding_d}, the same is true for $\td{h}'_{s,u}$.
\end{proof}

\begin{proof}[Proof of Lemma~\ref{Lem_rounding_intermediate}.]
Let $f : \IR \to [0,1]$ be a cutoff function with $f \equiv 1$ on $(-\infty, -2]$ and $f\equiv 0$ on $[-1, \infty)$.
Let $\eps > 0$ be a constant that we will determine later and set
\[ \nu (s) := f(\eps \log s ). \]
Fix a metric $\ov g$ on $B^3$ with the following properties:
\begin{enumerate}
\item $\ov g$ is $O(3)$-invariant.
\item $(\ov g_{ij})_0 = (\partial \ov g_{ij})_0 = 0$.
\item $\ov g$ is conformally equivalent to the Euclidean metric $g_{\eucl}$.
\item $w_s^4 \ov{g}$ has positive scalar curvature at the origin for all $s \in X_{PSC}$.
\end{enumerate}
Set
\[ h'_{s,u} := \big(1- u \cdot \nu (d_{h_s} (0, \cdot)) \big) h_s + u \cdot \nu (d_{h_s} (0, \cdot)) \ov{g}, \]
where $d_{h_s} (0, \cdot)$ denotes the radial distance function with respect to the metric $h_s$.

It remains to check that Assertions~\ref{ass_Prop_rounding_a}--\ref{ass_Prop_rounding_f} of Proposition~\ref{Prop_rounding} hold for sufficiently small $\eps$ and $r_1$.
Assertions~\ref{ass_Prop_rounding_a}--\ref{ass_Prop_rounding_c} trivially hold for sufficiently small $\eps$ and $r_1$.

For Assertion~\ref{ass_Prop_rounding_d} observe that if $h_s$ is compatible with the standard spherical structure on $D^3(r)$, then $d_{h_s} (0, \cdot)$ restricted to $D^3(r)$ is $O(3)$-invariant.
Moreover, $\ov{g}$ is compatible with the standard spherical structure on $D^3(r)$ as well.
It follows that $h'_{s,u}$ is compatible with the standard spherical structure on $D^3(r)$.

For Assertion~\ref{ass_Prop_rounding_e} assume that $h_s$ is  conformally flat.
Then by Assumption~\ref{hyp_Lem_rounding_intermediate_ii} $h_s$ is conformally equivalent to $g_{\eucl}$ on $B^3(\ov{r}_1)$.
Since $\ov{g}$ is conformally equivalent to $g_{\eucl}$ as well, this implies that $h'_{s,u}$ is conformally equivalent to $g_{\eucl}$ on $B^3(\ov{r}_1)$.
The conformal flatness on $B^3 \setminus B^3(\ov{r}_1)$ follows from Assertion~\ref{ass_Prop_rounding_a}

Lastly, we claim that Assertion~\ref{ass_Prop_rounding_f} holds for $w'_{s,u} = w_s$ if $\eps$ is chosen small enough.
To see this note first that there is a constant $C < \infty$ that is independent of $\eps, s, u$ such that
\[ \big| (h'_{s,u})_{ij} - \ov{g}_{ij} \big| \leq C r^2, \]
\[ \big| \partial_i (h'_{s,u})_{jk} \big| \leq C |\nu' (r)| \cdot |h_s-\ov{g}| + C r \leq C \eps r + C r, \]
\begin{multline*}
 \big| \partial^2 ( h'_{s,u})_{ij} - (1-u \nu) \cdot \partial^2 (h_s)_{ij} - u \nu \cdot \partial^2 \ov{g}_{ij} \big|  \\
 \leq C (|\nu''| + r^{-1} |\nu'| )  |h_{s} - \ov{g}|  + C  |\nu'| \cdot | \partial (h'_{s,u} - \ov{g})| 
 \leq C \eps^2 + C \eps .
\end{multline*}
Since $h'_{s,u} = h_s$ if $d_{h_s} (0, \cdot) > e^{-1/\eps}$ and since the $w_s$ is uniformly bounded in the $C^2$-norm for all $s \in X_{PSC}$, we obtain that
\[ | R_{w_s^4 h'_{s,u}} - (1-u\nu) R_{w_s^4 h'_{s,u}}  - u\nu  R_{w_s^4 \ov{g}}| \leq c \]
if $\eps \leq \ov\eps (c)$.
By compactness of $X_{PSC}$ we obtain that $R_{w_s^4 h'_{s,u}} > 0$ for sufficiently small $\eps$.
\end{proof}

\section{Partial homotopies} \label{sec_partial_homotopy}
In this section we introduce partial homotopies and certain modification moves, which will be used in Section~\ref{sec_deforming_families_metrics}.

\subsection{General setup} \label{subsec_gen_setup}
For the following discussion we fix a pair of finite simplicial complexes $(\mathcal{K}, \mathcal{L})$, where $\mathcal{L} \subset \mathcal{K}$ is a subcomplex.
We will denote by $L \subset K$ the geometric realizations of $\mathcal{L} \subset \mathcal{K}$.
When there is no chance of confusion, we will refer to the pair $(K,L)$ instead of $(\mathcal{K}, \mathcal{L})$.

Consider a fiber bundle $E \to K$ over $K$ whose fibers are smooth compact Riemannian 3-manifolds.
We will view this bundle as a continuous family of Riemannian manifolds $(M^s, g^s)_{s \in K}$ (see Remark~\ref{rmk_fiber_bundle_construction}).
Note that a particularly interesting case is the case in which $(M^s)_{s \in K}$ is given by a trivial family of the form $(M, g^s)_{s \in K}$, where $(g^s)_{s \in K}$ is a continuous family of Riemannian metrics on a fixed compact 3-manifold $M$.

\begin{definition}
A metric $g$ on a compact 3-manifold $M$ is called a {\bf CC-metric} if $(M,g)$ is homothetic to a quotient of the round sphere or round cylinder.
\end{definition}

If $L \neq \emptyset$, then we assume that the metrics $g^s$, $s \in L$, are CC-metrics.
If none of the $M^s$ are diffeomorphic to a spherical space form or a quotient of a cylinder, then no such metrics exist on any $M^s$ and therefore we must have $L = \emptyset$.

We will also fix a closed subset $K_{PSC} \subset K$ with the property that $(M^s, g^s)$ has positive scalar curvature for all $s \in K_{PSC}$.

Our ultimate goal in this and the following section (see Theorem~\ref{Thm_main_deform_to_CF}) will be to construct a transversely continuous family of metrics $(h^s_t)_{s \in K, t \in [0,1]}$ on $(M^s)_{s \in K}$ with the property that:
\begin{enumerate}[label=(\arabic*)]
\item \label{prop_family_1} $h^s_0 = g^s$ for all $s \in K$,
\item \label{prop_family_2} $h^s_1$ is conformally flat and PSC-conformal for all $s \in K$,
\item \label{prop_family_3} $h^s_t$ are CC-metrics for all $s \in L$, $t \in [0,1]$,
\item \label{prop_family_4} $(M^s, h^s_t)$ is PSC-conformal for any $s \in K_{PSC}$, $t \in [0,1]$; see Definition~\ref{Def_PSC_conformal}.
\end{enumerate}
Here $(h^s_t)_{s \in K, t \in [0,1]}$ is said to be transversely continuous if it is transversely continuous in every family chart of $(M^s)_{s \in K}$, or equivalently, if $(h^s_t)_{(s,t) \in  K \times [0,1]}$ is transversely continuous in the sense of Definition~\ref{def_continuity_smooth_objects} on the continuous family $(M^s \times \{ t \})_{(s,t) \in  K \times [0,1]}$.

In order to achieve this, we will apply Theorem~\ref{thm_existence_family_k} to find a continuous family of singular Ricci flows $(\M^s )_{s \in K}$ whose family of time-0-slices $(\M^s_0, g^s_0)_{s \in K}$ is isomorphic to $(M^s, g^s)_{s \in K}$; in the following we will identify both objects.
By Theorem~\ref{Thm_sing_RF_uniqueness} for all $s \in L$ all time-slices of $\M^s$ are CC-metrics.
By Theorem~\ref{Thm_PSC_preservation} the flow $\M^s$ has positive scalar curvature for all $s \in K_{PSC}$.

In the following we will write $\M^s = (\M^s, \t^s, g^s, \partial^s_\t)$ and we fix a transversely continuous family of $\RR$-structures
\[ \mathcal{R}^s = ( g^{\prime, s}, \lb \partial^{\prime, s}_\t,  \lb U^s_{S2}, \lb U^s_{S3}, \lb \mathcal{S}^s)  \]
for each $\M^s$.
Such a family exists due to Theorem~\ref{Thm_rounding} and we can ensure that
\begin{equation} \label{eq_RR_trivial_over_L}
g^{\prime,s} = g^s, \quad \partial^{\prime,s}_\t = \partial^s_\t \qquad \text{for all} \quad s \in L.
\end{equation}
We will discuss further geometric and analytic properties of this structure in Section~\ref{sec_deforming_families_metrics}; for the purpose of this section it suffices to assume that $(\RR^s)_{s \in K}$ satisfies the properties of Definitions~\ref{Def_R_structure} and \ref{Def_RR_structure_transverse_cont}.

Let $T \geq 0$.
The goal of this section will be to introduce a new type of partially defined homotopy starting from the the family of metrics $( g^{\prime,s}_{T})_{s \in K}$ on the family of time-$T$-slices $(\M^s_T)_{s \in K}$.
We will see that if $T = 0$, then under a certain conditions this partial homotopy implies the existence of the desired family $(h^s_t)_{s \in K, t \in [0,1]}$ satisfying Properties~\ref{prop_family_1}--\ref{prop_family_4}.
We will moreover discuss ``moves'' that will allow us to ``improve'' a given partial homotopy and enable us to flow it backwards in time, i.e. decrease the time parameter $T$.

\subsection{Definition of a partial homotopy}
Let $X$ be a topological space.

\begin{definition}[Metric deformation] \label{Def_metric_deformation}
A {\bf metric deformation over $X$} is a pair $(Z,   (g_{s,t})_{s \in X, t \in [0,1]})$ with the following properties:
\begin{enumerate}[label=(\arabic*)]
\item $Z$ is a compact 3-manifold with boundary whose boundary components are spheres.
\item $(g_{s,t})_{s \in X, t \in [0,1]}$ is a continuous family of Riemannian metrics.
\item For all $s \in X$, the Riemannian manifold $(Z, g_{s,1})$ is conformally flat and PSC-conformal.
\end{enumerate}
\end{definition}

We can now define a partial homotopy.
For this purpose, fix some $T \geq 0$ and consider the simplicial pair $(K,L)$, the continuous family of singular Ricci flows $(\M^s)_{s \in K}$ over $K$, as well as the family of $\RR$-structures $(\RR^s)_{s \in K}$ from Subsection~\ref{subsec_gen_setup}.

\begin{definition}[Partial homotopy] \label{Def_partial_homotopy}
For every simplex $\sigma \subset K$ consider a metric deformation $( Z^\sigma, \lb (g^\sigma_{s,t})_{s \in \sigma, t \in [0,1]})$ and a transversely continuous family of embeddings $(\psi^\sigma_s : Z^\sigma \to \M^s_T )_{s \in \sigma}$.
We call $\{ ( Z^\sigma, \lb (g^\sigma_{s,t})_{s \in \sigma, t \in [0,1]}, \lb (\psi^\sigma_s  )_{s \in \sigma}) \}_{\sigma \subset K}$ a {\bf partial homotopy (at time $T$ relative to $L$ for the the family of $\RR$-structures $(\RR^{s})_{s \in K}$)} if the following holds:
\begin{enumerate}[label=(\arabic*)]
\item \label{prop_def_partial_homotopy_1} For all $s \in \sigma \subset K$ we have $(\psi^{ \sigma}_s)^* g^{\prime,s}_{T} = g^{\sigma}_{s, 0}$.
\item \label{prop_def_partial_homotopy_2} For all $s \in \tau \subsetneq \sigma \subset K$ we have $ \psi^{ \sigma}_s ( Z^{ \sigma} ) \subset \psi^{ \tau}_s ( Z^{ \tau} )$.
\item \label{prop_def_partial_homotopy_3} For all $s \in \tau \subsetneq \sigma \subset K$ and $t \in [0,1]$ we have $((\psi^{ \tau}_s )^{-1} \circ \psi^{ \sigma}_s )^* g^\tau_{s,t} = g^\sigma_{s,t}$.
\item \label{prop_def_partial_homotopy_4} For each $s \in \tau \subsetneq \sigma \subset K$ and for the closure $\C$ of each component of $Z^{ \tau}  \setminus((\psi^{ \tau}_s )^{-1} \circ \psi^{ \sigma}_s ) ( Z^{ \sigma} )$ one (or both) of the following is true:
\begin{enumerate}[label=(\roman*)]
\item \label{prop_def_partial_homotopy_4i} $\psi^\tau_s (\C) \subset U_{S2}^s$ and $\psi^\tau_s (\C)$ is a union of  spherical fibers.
Moreover, for every $t \in [0,1]$ the metric $(\psi^\tau_s)_* g^\tau_{s,t}$ restricted to $\psi^\tau_s (\C)$ is compatible with the restricted spherical structure.
\item \label{prop_def_partial_homotopy_4ii} $\partial \C = \emptyset$, $\psi^\tau_s (\C) \subset U^s_{S3}$ and for every $s' \in \tau$ near $s$ the metric $g^\tau_{s',t}$ restricted to $\C$ is a multiple of $g^\tau_{s',0}$ for all $t \in [0,1]$.
\end{enumerate}
\item \label{prop_def_partial_homotopy_5} For every $\sigma \subset K$ and every component $\Sigma \subset \partial Z^\sigma$ the image $\psi^\sigma_s (\Sigma)$ is a regular fiber of $\SS^s$ for all $s \in \sigma$.
Moreover, there is an $\eps > 0$ that is independent of $s$ such that for all $t \in [0,1]$ the metric $(\psi^\sigma_s)_*  g^{\sigma}_{s,t}$ is compatible with $\SS^s$ on an $\eps$-collar neighborhood of $\psi^\sigma_s (\Sigma)$ inside $\psi^\sigma_s (Z^\sigma)$.
\item \label{prop_def_partial_homotopy_6} For every $s \in \sigma \subset L$ we have $\psi^\sigma_s (Z^\sigma) = \emptyset$ or $\M^s_T$ and the metrics $g^\sigma_{s,t}$, $t \in [0,1]$ are either multiples of the same constant curvature metric or they are isometric to quotients of the round cylinder and admit the same local isometric $O(3)$-actions.
\end{enumerate}
We say that the partial homotopy is {\bf PSC-conformal over $s \in K$} if for any simplex $\sigma \subset K$ with $s \in \sigma$ and any $t \in [0,1]$ the Riemannian manifold $(Z^\sigma, g^\sigma_{s,t})$ is PSC-conformal.

If $Z^\sigma = \emptyset$ for all $\sigma \subset K$, then the partial homotopy is called {\bf trivial}.
\end{definition}

In other words, a partial homotopy is given by metric deformations $(\psi^\sigma_s)_* g^\sigma_{s,t}$ on the $s$-dependent domains $\psi^\sigma_s (Z^\sigma) \subset \M^s_T$ starting from $(\psi^\sigma_s)_* g^\sigma_{s,0} = g^{\prime,s}_T$ (see Property~\ref{prop_def_partial_homotopy_1}).
For a fixed $s \in K$ there may be several such domains and deformations, for different simplices $\sigma \subset K$ containing $s$.
Property~\ref{prop_def_partial_homotopy_2} states that these domains are nested and decrease in size as the dimension of $\sigma$ increases.
Property~\ref{prop_def_partial_homotopy_3} states that any two deformations $(\psi^\sigma_s)_* g^\sigma_{s,t}$, $(\psi^\tau_s)_* g^\tau_{s,t}$ ($\tau \subset \sigma$) for the same parameter $s$ agree on the smaller domain $\psi^\sigma_s (Z^\sigma)$ and Property~\ref{prop_def_partial_homotopy_4} imposes a symmetry condition of the larger deformation $(\psi^\tau_s)_* g^\tau_{s,t}$ on the difference $\psi^\tau_s (Z^\tau) \setminus \psi^\tau_s (Z^\sigma)$.
The use of the parameter $s'$ in Property~\ref{prop_def_partial_homotopy_4}\ref{prop_def_partial_homotopy_4ii} ensures that $\sigma$ admits an \emph{open} over $\sigma = U^\sigma_{(i)} \cup U^\sigma_{(ii)}$ such that Property~\ref{prop_def_partial_homotopy_4}\ref{prop_def_partial_homotopy_4i} holds for all $s \in U^\sigma_{(i)}$ and Property~\ref{prop_def_partial_homotopy_4}\ref{prop_def_partial_homotopy_4ii} holds for all $s \in U^\sigma_{(ii)}$; this fact will be important in the proof of Proposition~\ref{prop_extending} (see also Proposition~\ref{prop_exten_round}).
Property~\ref{prop_def_partial_homotopy_5} is a technical property, which will allow us to extend a deformation $(\psi^\sigma_s)_* g^\sigma_{s,t}$ to a larger domain using metrics that are $\SS^s$ compatible.
It could potentially be replaced by a condition requiring the boundary components of $\psi^\sigma_s (Z^\sigma)$ to be umbilic round spheres with respect to all metrics $(\psi^\sigma_s)_* g^\sigma_{s,t}$, plus some conditions on the higher derivatives.
In the cylindrical case Property~\ref{prop_def_partial_homotopy_6} implies that for fixed $s \in \sigma$ the metrics $g^\sigma_{s,t}$, $t \in [0,1]$ can locally be expressed as $a_t^2 g_{S^2} + b_t^2 dr^2$ for some continuously varying $a_t, b_t > 0$.

On a more philosophical level, Definition~\ref{Def_partial_homotopy} formalizes the idea of a continuous family of metric deformations on subsets of $\M^s_T$, which are defined up to some ambiguity.
This ambiguity is ``supported'' on the differences $\psi^\tau_s (Z^\tau) \setminus \psi^\tau_s (Z^\sigma)$, which we will later choose to be subsets of small scale $\rho$.
The deformations  $(\psi^\sigma_s)_* g^\sigma_{s,t}$ restricted to these differences are required to be of a very controlled symmetric form.
This will allow us to argue that the ambiguity expressed in Definition~\ref{Def_partial_homotopy} is ``contractible'' in a certain sense.

Lastly, let us comment on the use of the PSC-conformal condition in Definition~\ref{Def_partial_homotopy}.
The reason that we are using this condition instead of the standard positive scalar curvature condition is rather subtle, but it will be central in the proof of Proposition~\ref{prop_extending} below.
In short, it has to do with poor contractibility properties of certain spaces of warped product metrics with positive scalar curvature.
The PSC-conformal condition, on the other hand, is much more forgiving; for instance, a metric is automatically PSC conformal if it is compatible with a spherical structure whose domain is the entire manifold.
 
\subsection{Constructing the desired family of metrics from a partial homotopy}
Our strategy in Section~\ref{sec_deforming_families_metrics} will be to inductively construct partial homotopies at a sequence of decreasing times $T_i$ with $T_i = 0$ for some large $i$.
The following proposition will allow us to convert a partial homotopy at time $0$ to a conventional homotopy.
It essentially states that we can construct a family $(h^s_t)_{t \in [0,1],s \in K}$ satisfying Properties~\ref{prop_family_1}--\ref{prop_family_4} from Subsection~\ref{subsec_gen_setup} from a partial homotopy at time $0$ if the maps $\psi^\sigma_s$ are all surjective.

\begin{proposition} \label{prop_partial_homotopy_standard_homotopy}
Suppose there is a partial homotopy $\{ ( Z^\sigma, \lb (g^\sigma_{s,t})_{s \in \sigma, t \in [0,1]}, \lb (\psi^\sigma_s  )_{s \in \sigma}) \}_{\sigma \subset K}$ at time $0$ relative to $L$ with the property that $\psi^\sigma_s (Z^\sigma) = \M^s_0$ for all $s \in \sigma \subset K$.
Then there is a family of family of metrics $(h^s_t)_{t \in [0,1],s \in K}$ satisfying Properties~\ref{prop_family_1}--\ref{prop_family_3} from  Subsection~\ref{subsec_gen_setup}.

Moreover for any $s \in K$, $t \in [0,1]$ the following holds.
If the partial homotopy is PSC-conformal at $s$, then $(M^s, g^s_t)$ is PSC-conformal (compare with Property~\ref{prop_family_4} above).
\end{proposition}

\begin{proof}
By Definition~\ref{Def_partial_homotopy}\ref{prop_def_partial_homotopy_3} we have $(\psi^\tau_s)_* g^\tau_{s,t} = (\psi^\sigma_s)_* g^\sigma_{s,t}$ for all $s \in \tau \subsetneq \sigma \subset K$, $t \in [0,1]$.
So we can define $h^s_t := (\psi^\sigma_s)_* g^\sigma_{s,t}$ for any $s \in \sigma \subset K$.
The asserted properties of the family $(M^s = \M^s_0, (h^s_t)_{t \in [0,1]})_{s \in K}$ are direct consequences of Definition~\ref{Def_partial_homotopy}. 
\end{proof}

\subsection{Moving a partial homotopy backwards in time}
The following proposition allows us to construct a partial homotopy $\{ ( Z^\sigma, \lb (\ov g^\sigma_{s,t})_{s \in \sigma, t \in [0,1]}, \lb (\ov\psi^\sigma_s  )_{s \in \sigma}) \}_{\sigma \subset K}$ at an earlier time $T' \leq T$ from a partial homotopy $\{ ( Z^\sigma, \lb (g^\sigma_{s,t})_{s \in \sigma, t \in [0,1]}, \lb (\psi^\sigma_s  )_{s \in \sigma}) \}_{\sigma \subset K}$ at time $t$ such that the domains $\ov\psi^\sigma_s (Z^\sigma)$ arise by flowing $\psi^\sigma_s (Z^\sigma)$ backwards by the flow of $\partial^{\prime,s}$ by $T - T'$.
In order to achieve this we require that the differences $\psi^\tau_s (Z^\tau) \setminus \psi^\tau_s (Z^\sigma)$ remain in the support of $\RR^s$.

In the following we denote by $x^{\partial^{\prime,s}_\t} (t')$ and $X^{\partial^{\prime,s}_\t}(t')$ the image of $x \in \M^s_t$ and $X \subset \M^s_t$) under the time $(t'-t)$-flow of the vector field $\partial^{\prime,s}_\t$; this is the same notion as in Definition~\ref{def_points_in_RF_spacetimes} with $\partial_\t$ replaced by $\partial^{\prime,s}_\t$.

\begin{proposition} \label{prop_move_part_hom_bckwrds}
Consider a partial homotopy $\{ ( Z^\sigma, \lb ( g^\sigma_{s,t})_{s \in \sigma, t \in [0,1]}, \lb (\psi^\sigma_s  )_{s \in \sigma}) \}_{\sigma \subset K}$ at time $T$ relative to $L$ and let $T' \in [0, T]$.
Assume that:
\begin{enumerate}[label=(\roman*)]
\item \label{hyp_lem_move_part_hom_bckwrds_i} For all $s \in \sigma \subset K$ all points of $\psi^\sigma_s(Z^\sigma)$ survive until time $T'$ with respect to the flow of $\partial^{\prime,s}_\t$.
\item \label{hyp_lem_move_part_hom_bckwrds_ii} For all $s \in \tau \subsetneq \sigma \subset K$ and $t' \in [T', T]$ we have
 \[
\big( \psi^\tau(Z^\tau) \setminus \psi^\sigma_s(Z^\sigma) \big)^{\partial^{\prime,s}_\t} (t') \subset U^s_{S2} \cup U^s_{S3}. \]
\end{enumerate}
Then there is a partial homotopy of the form $\{ ( Z^\sigma, \lb (\ov g^\sigma_{s,t})_{s \in \sigma, t \in [0,1]}, \lb (\ov\psi^\sigma_s  )_{s \in \sigma}) \}_{\sigma \subset K}$ at time $T'$ relative to $L$ such that for all $s \in \sigma \subset K$:
\begin{enumerate}[label=(\alph*)]
\item \label{ass_lem_move_part_hom_bckwrds_a} $\ov\psi^\sigma_s (Z^\sigma) = (\psi^\sigma_s (Z^\sigma))^{\partial^{\prime,s}_\t}(T')$.
\item \label{ass_lem_move_part_hom_bckwrds_b} If $\{ ( Z^\sigma, \lb (\ov g^\sigma_{s,t})_{s \in \sigma, t \in [0,1]}, \lb (\ov\psi^\sigma_s  )_{s \in \sigma}) \}_{\sigma \subset K}$ is PSC-conformal over $s$ and if $g^{\prime,s}_{t'}$ restricted to $(\psi^\sigma_s(Z^\sigma))^{\partial^{\prime,s}_\t} (t')$ is PSC-conformal for all $t' \in [T',T]$, then $\{ ( Z^\sigma, \lb (g^\sigma_{s,t})_{s \in \sigma, t \in [0,1]}, \lb (\psi^\sigma_s  )_{s \in \sigma}) \}_{\sigma \subset K}$ is also PSC-conformal over $s$.
\end{enumerate}
\end{proposition}

\begin{proof}
We can define maps $( \psi^{\sigma}_s  )_{t'} : Z^\sigma \to \M^s_{t'}$, $t' \in [T',T]$, as follows:
\[ \big( \psi^{\sigma}_s  \big)_{t'} (z) := \big( \psi^{\sigma}_s (z) \big)^{\partial^{\prime,s}_\t} (t'). \]
Set
\begin{equation*} \label{eq_def_tdpsi}
 \ov\psi^{\sigma}_s := \big( \psi^{\sigma}_s  \big)_{T'} 
\end{equation*}
and
\[ \ov{g}^\sigma_{s,t} := \begin{cases} (\psi^\sigma_s)^*_{T' + 2 t  (T - T')} g_{T' + 2 t  (T - T')} & \text{if $t \in [0, \tfrac12]$} \\ g^\sigma_{s, 2 t - 1} & \text{if $t \in [\tfrac12, 1]$}  \end{cases} \]
Then Assertions~\ref{ass_lem_move_part_hom_bckwrds_a} and \ref{ass_lem_move_part_hom_bckwrds_b} hold automatically.
Let us now verify all properties of Definition~\ref{Def_partial_homotopy} for $\{ ( Z^\sigma, \lb (\ov g^\sigma_{s,t})_{s \in \sigma, t \in [0,1]}, \lb (\ov\psi^\sigma_s  )_{s \in \sigma}) \}_{\sigma \subset K}$.

By construction $(Z^\sigma, (\ov{g}^\sigma_{s,t})_{s \in \sigma, t \in [0,1]})$ is a metric deformation for all $\sigma \subset K$.
Properties~\ref{prop_def_partial_homotopy_1}--\ref{prop_def_partial_homotopy_3} of Definition~\ref{Def_partial_homotopy} hold by construction as well.

For Property~\ref{prop_def_partial_homotopy_4} consider the closure $\C$ of a component of
\[ Z^{ \tau}  \setminus((\ov\psi^{ \tau}_s )^{-1} \circ \ov\psi^{ \sigma}_s ) ( Z^{ \sigma} ) = Z^{ \tau}  \setminus((\psi^{ \tau}_s )^{-1} \circ \psi^{ \sigma}_s ) ( Z^{ \sigma} ). \]
By Assumption~\ref{hyp_lem_move_part_hom_bckwrds_ii} we have 
\begin{equation} \label{eq_CCinUS2US3}
 \C^{\partial^{\prime,s}_\t} (t') \subset U^s_{S2} \cup U^s_{S3} \qquad \text{for all}  \quad t' \in [T', T].
\end{equation}
We can apply Definition~\ref{prop_def_partial_homotopy_4} to $\C$ and the original partial homotopy, which implies that there are two cases.

\medskip
\textit{Case 1: Property~\ref{prop_def_partial_homotopy_4}\ref{prop_def_partial_homotopy_4i} holds for the original partial homotopy. \quad}
So $\psi^\tau_s (\C) \subset U^s_{S2}$, $\psi^\tau_s (\C)$ is a union of fibers of $\SS^s$ and $(\psi_s^\tau)_* g^\tau_{s,t}$ restricted to $\psi^\tau_s (\C)$ is compatible with $\SS^s$.
By (\ref{eq_CCinUS2US3}) and Definition~\ref{Def_R_structure}\ref{prop_def_RR_1} we know that $\C(t') \subset U^s_{S2}$ for all $t' \in [T', T]$ and thus by Definition~\ref{Def_R_structure}\ref{prop_def_RR_2}--\ref{prop_def_RR_4} the set $\ov\psi^\tau_s (\C) = \C^{\partial^{\prime,s}_\t} (T')$ is a union of fibers of $\SS^s$ and $(\ov\psi^\tau_s)_* \ov{g}^\tau_{s,t}$ restricted to $\ov\psi^\tau_s (\C)$ is compatible with $\SS^s$.

\medskip
\textit{Case 2: Property~\ref{prop_def_partial_homotopy_4}\ref{prop_def_partial_homotopy_4ii} holds for the original partial homotopy. \quad}
So $\partial\C = \emptyset$, $\psi^\tau_s (\C) \subset U^s_{S3}$ and for every $s' \in \tau$ near $s$ the metric $g^\tau_{s', t}$, $t \in [0,1]$, is a multiple of $g^\tau_{s', 0}$.
By (\ref{eq_CCinUS2US3}) and Definition~\ref{Def_R_structure}\ref{prop_def_RR_5} we have $\C^{\partial^{\prime,s}_\t}  (T') \subset U^s_{S2}$ or $\C^{\partial^{\prime,s}_\t}  (T') \subset U^s_{S3}$.

\medskip
\textit{Case 2a: $\C (T') \subset U^s_{S2}$. \quad}
By Definition~\ref{Def_R_structure}\ref{prop_def_RR_6} there is a $T^* \in [T', T)$ such that $\C^{\partial^{\prime,s}_\t}  (t') \subset U^s_{S3}$ for all $t' \in (T^*, T]$ and $\C^{\partial^{\prime,s}_\t}  (t') \subset U^s_{S2} \setminus U^s_{S3}$ for all $t' \in [T', T^*]$.
Using Definition~\ref{Def_R_structure}\ref{prop_def_RR_2}--\ref{prop_def_RR_4}, \ref{prop_def_RR_7} we can argue as in Case~1 that $(\ov\psi^\tau_s)_* \ov{g}^\tau_{s,t}$ restricted to $\ov\psi^\tau_s (\C)$ is compatible with $\SS^s$.

\medskip
\textit{Case 2b: $\C (T') \subset U^s_{S3}$. \quad}
By Definition~\ref{Def_R_structure}\ref{prop_def_RR_6} we have $\C^{\partial^{\prime,s}_\t}  (t') \subset U^s_{S3}$ for all $t' \in [T', T]$.
By openness of $\cup_{s' \in K} U^{s'}_{S3} \subset \cup_{s' \in K} \M^{s'}$ the same is true for $s' \in \tau$ near $s$.
So by Definition~\ref{Def_R_structure}\ref{prop_def_RR_7} for $s' \in \tau$ near $s$ we obtain that $\ov{g}^\tau_{s,t}$ is a multiple of $\ov{g}^\tau_{s,0}$.

\medskip
Next, consider Property~\ref{prop_def_partial_homotopy_5} of Definition~\ref{Def_partial_homotopy}.
Fix some $\Sigma \subset \partial Z^\sigma$.
Then $\psi^\sigma_s (\Sigma) \subset U^s_{S2}$.
We can again argue as before that $(\psi^\sigma_s (\Sigma))^{\partial^{\prime,s}_\t} (t') \subset U^s_{S2}$ for all $t' \in [T', T]$.
Therefore, in a neighborhood of $(\psi^\sigma_s (\Sigma))^{\partial^{\prime,s}_\t} (t')\subset U^s_{S2}$ the metric $g^{\prime,s}_{t'}$ is compatible with $\SS^s$.
Property~\ref{prop_def_partial_homotopy_5} now follows from Definition~\ref{Def_R_structure}\ref{prop_def_RR_2}--\ref{prop_def_RR_4}.

Property~\ref{prop_def_partial_homotopy_6} is a consequence of (\ref{eq_RR_trivial_over_L}).
\end{proof}

\subsection{Passing to a simplicial refinement}
Consider the simplicial pair $(\mathcal{K}, \mathcal{L})$ with geometric realization $(K,L)$, as discussed in Subsection~\ref{subsec_gen_setup}.
Let $\mathcal{K}'$ be a simplicial refinement of $\mathcal{K}$, let $\mathcal{L}'$ be the corresponding simplicial refinement of $\mathcal{L}$, and identify the geometric realizations of $(\mathcal{K}',  \mathcal{L}')$ with $(K,L)$.
The next proposition states that given a partial homotopy respecting the simplicial structure $\mathcal{K}$, we may construct a canonical partial homotopy respecting the simplicial structure $\mathcal{K}'$.

\begin{proposition} \label{prop_simp_refinement}
Let $\mathcal{K}'$ be a simplicial refinement of $\mathcal{K}$.
If $\{ ( Z^\sigma, \lb ( g^\sigma_{s,t})_{s \in \sigma, t \in [0,1]}, \lb (\psi^\sigma_s  )_{s \in \sigma}) \}_{\sigma \subset K}$ is a partial homotopy at time $T$ relative to $L$ that respects the simplicial structure $\mathcal{K}$, then there is a partial homotopy $\{ ( \ov{Z}^\sigma, \lb (\ov g^\sigma_{s,t})_{s \in \sigma, t \in [0,1]}, \lb (\ov\psi^\sigma_s  )_{s \in \sigma}) \}_{\sigma \subset K}$ at time $T$ relative to $L$ that respects the simplicial structure $\mathcal{K}'$ such that:
\begin{enumerate}[label=(\alph*)]
\item For any $\sigma' \in \mathcal{K}'$ the following holds:
If $\sigma \in \mathcal{K}$ is the simplex with the smallest dimension such that $\sigma \supset \sigma'$, then $\ov\psi^{\sigma'}_s (\ov{Z}^{\sigma'}) = \psi^{\sigma}_s (Z^{\sigma})$ and $(\ov\psi^{\sigma'}_s)_* \ov{g}^{\sigma'}_{s,t} = (\psi^{\sigma}_s)_* g^\sigma_{s,t}$ for all $s \in \sigma'$, $t \in [0,1]$.
\item If $\{ ( Z^\sigma, \lb ( g^\sigma_{s,t})_{s \in \sigma, t \in [0,1]}, \lb (\psi^\sigma_s  )_{s \in \sigma}) \}_{\sigma \subset K}$ is PSC-conformal over some $s \in K$, then so is $\{ ( \ov{Z}^\sigma, \lb (\ov g^\sigma_{s,t})_{s \in \sigma, t \in [0,1]}, \lb (\ov\psi^\sigma_s  )_{s \in \sigma}) \}_{\sigma \subset K}$.
\end{enumerate}
\end{proposition}

\begin{proof} 
For any $\sigma' \in \mathcal{K}'$, then there is a unique simplex $\sigma_{\sigma'} \in \mathcal{K}$ that has minimal dimension and satisfies $\sigma_{\sigma'}  \supset \sigma'$.
Set
\[ ( \ov{Z}^{\sigma'
}, (\ov g^{\sigma'}_{s,t})_{s \in \sigma', t \in [0, 1]},  ( \psi^{\sigma'}_s )_{s \in \sigma'}) := ( Z^{\sigma_{\sigma'}}, (g^{\sigma_{\sigma'}}_{s,t})_{s \in \sigma_{\sigma'}, t \in [0, 1]}, ( \psi^{\sigma_{\sigma'}}_s )_{s \in \sigma_{\sigma'}}). \]
It is easy to see that this new data still defines a partial homotopy that satisfies the assertions of this proposition.
\end{proof}

\subsection{Enlarging a partial homotopy}
In the following we prove that we can enlarge the images $\psi^\sigma_s (Z^\sigma)$ by certain subsets contained in $U^s_{S2} \cup U^s_{S3}$ that are either unions of spherical fibers or round components.

\begin{proposition} \label{prop_extending}
Consider a partial homotopy $\{ ( Z^\sigma, \lb (g^\sigma_{s,t})_{s \in \sigma, t \in [0,1]}, \lb (\psi^\sigma_s  )_{s \in \sigma}) \}_{\sigma \subset K}$ at time $T$ relative to $L$.
Fix some simplex $\sigma \subset K$.
Let $\wh{Z}^\sigma$ be a compact 3-manifold with boundary, $\iota^\sigma : Z^\sigma \to \wh{Z}^\sigma$ an embedding and $(\wh\psi^\sigma_s : \wh{Z}^\sigma \to \M^s_T)_{s \in \sigma}$ a continuous family of embeddings.
Assume that for all $s \in \sigma$:
\begin{enumerate}[label=(\roman*)]
\item \label{ass_ph_extend_i} $\psi^\sigma_s = \wh\psi^\sigma_s \circ \iota^\sigma$.
\item \label{ass_ph_extend_ii} If $s \in \tau \subset \partial \sigma$, then $\wh\psi^\sigma_s (\wh{Z}^\sigma) \subset \psi^\tau_s (Z^\tau)$.
\item \label{ass_ph_extend_iii} For the closure $Y$ of every component of $\wh{Z}^\sigma \setminus \iota^\sigma (Z^\sigma)$ one of the following is true uniformly for all $s \in \sigma$:
\begin{enumerate}[label=(iii-\arabic*)]
\item \label{ass_ph_extend_iii-1} $\wh\psi^\sigma_s (Y)$ is a union of fibers of $\SS^s$. 
\item \label{ass_ph_extend_iii-2} $\partial Y = \emptyset$, $\wh\psi^\sigma_s (Y) \subset U^s_{S3}$ and $(\wh\psi^\sigma_s)^* g^{\prime,s}_{T}$ a multiple of the same constant curvature metric for all $s \in \sigma$.
\end{enumerate}
\item \label{ass_ph_extend_iv} If $\sigma \subset L$, then $\wh\psi^\sigma_s (\wh{Z}^\sigma) \setminus \psi^\sigma_s ( Z^\sigma) = \emptyset$ or $\M^s_T$ and $g^{\prime,s}_T$ is a CC-metric.
\end{enumerate}
Then there is a family of metrics $(\wh{g}^\sigma_{s \in \sigma, t \in [0,1]})$ on $\wh{Z}^\sigma$ such that 
\begin{equation} \label{eq_enlarged_partial_h}
 \{ ( Z^{\sigma'}, \lb (g^{\sigma'}_{s,t})_{s \in \sigma', t \in [0,1]}, \lb (\psi^{\sigma'}_s  )_{s \in \sigma'}) \}_{\sigma' \subset K, \sigma' \neq \sigma} \cup \{ ( \wh{Z}^\sigma, \lb (\wh{g}^\sigma_{s,t})_{s \in \sigma, t \in [0,1]}, \lb (\wh\psi^\sigma_s  )_{s \in \sigma}) ) \} 
\end{equation}
is a partial homotopy at time $T$ relative to $L$.
Moreover, if the original partial homotopy was PSC-conformal over some $s \in K$, then the new partial homotopy is PSC-conformal over $s$ as well.
\end{proposition}

\begin{proof}
After replacing $Z^\sigma$ with $\iota^\sigma (Z^\sigma)$, $\psi^\sigma_s$ with $\wh\psi^\sigma_s |_{\iota^\sigma (Z^\sigma)}$ and $g^\sigma_{s,t}$ with $\iota^\sigma_* g^\sigma_{s,t}$, we may assume without loss of generality  that $Z^\sigma \subset \wh{Z}^\sigma$ is an embedded submanifold, $\iota^\sigma = \id_{Z^\sigma}$ and $\psi^\sigma_s = \wh\psi^\sigma_s |_{Z^\sigma}$.
Next, note that $\wh{Z}^\sigma \setminus Z^\sigma$ consists of finitely many connected components.
In the following we will assume without loss of generality that this difference only contains one connected component.
The proposition in its full generality will then follow by successively adding connected components to $Z^\sigma$.
Let $Y$ be the closure of $\wh{Z}^\sigma \setminus Z^\sigma$.

By Definition~\ref{Def_partial_homotopy}\ref{prop_def_partial_homotopy_3}, for every two simplicies $\tau_1 \subset \tau_2 \subset \partial \sigma$ and $s \in \tau_1$, $t \in [0,1]$ we have on $\wh{Z}^\sigma$:
\[ \big( (\psi^{\tau_2}_s)^{-1} \circ \wh\psi^\sigma_s \big)^* g^{\tau_2}_{s,t}
= \big( (\psi^{\tau_2}_s)^{-1} \circ \wh\psi^\sigma_s \big)^* \big( (\psi^{\tau_1}_s)^{-1} \circ \psi^{\tau_2}_s \big)^* g^{\tau_1}_{s,t}
= \big( (\psi^{\tau_1}_s)^{-1} \circ \wh\psi^\sigma_s \big)^* g^{\tau_1}_{s,t}. \]
Therefore, there is a continuous family of metrics $(k_{s,t})_{s \in \partial\sigma, t \in [0,1]}$ on $\wh{Z}^\sigma$ that satisfies
\begin{equation} \label{eq_k_gtau}
 k_{s,t} = \big( (\psi^{\tau}_s)^{-1} \circ \wh\psi^\sigma_s \big)^* g^{\tau}_{s,t} \qquad \text{for all} \quad s \in \tau \subset \partial\sigma. 
\end{equation}

\begin{Claim}
We have
\begin{alignat}{3}
k_{s,0} &= \big( \wh\psi^\sigma_s \big)^*  g^{\prime,s}_{T} & \qquad &\text{on} \quad \wh{Z}^\sigma &\qquad &\text{for}\quad s \in \partial\sigma, \label{eq_k_psi_g} \\
k_{s,t} &= g^{\sigma}_{s,t}  & \qquad &\text{on} \quad Z^\sigma &\qquad &\text{for}\quad s \in \partial\sigma, \; t \in [0,1] \label{eq_k_iotag} \\
g^{\sigma}_{s,0} &= \big( \wh\psi^\sigma_s \big)^*  g^{\prime,s}_{T}  & \qquad &\text{on} \quad Z^\sigma &\qquad &\text{for}\quad s \in \sigma, \label{eq_psi_g_iotag}
\end{alignat}
\end{Claim}

\begin{proof}
For (\ref{eq_k_psi_g}), observe that for all $s \in \tau \subset \partial\sigma$
\[ k_{s,0} = \big( (\psi^{\tau}_s)^{-1} \circ \wh\psi^\sigma_s \big)^* (\psi^\tau_s)^* g^{\prime,s}_{T} =  \big( \wh\psi^\sigma_s \big)^*  g^{\prime,s}_{T}. \]
(\ref{eq_k_iotag}) and (\ref{eq_psi_g_iotag}) follow from $\psi^\sigma_s = \wh\psi^\sigma_s |_{Z^\sigma}$ and
\[  k_{s,t} \big|_{Z^\sigma} = \big( (\psi^{\tau}_s)^{-1} \circ \psi^\sigma_s \big)^* g^{\tau}_{s,t} 
=  g^{\sigma}_{s,t}. \qedhere\]
\end{proof} \medskip 

Our goal will be to construct a metric deformation $(\wh{Z}^\sigma, (\wh{g}^\sigma_{s,t})_{s \in \sigma, t \in [0,1]})$ such that, among other things,
\begin{alignat}{3}
\wh{g}^\sigma_{s,t} &= k_{s,t} &\qquad &\text{on} \quad \wh{Z}^\sigma &\qquad &\text{for} \quad s \in \partial\sigma, \; t \in [0,1] \label{eq_h_k} \\
\wh{g}^\sigma_{s,0} &= \big(\wh\psi^\sigma_{s} \big)^* g^{\prime,s}_{T}  &\qquad &\text{on} \quad \wh{Z}^\sigma &\qquad & \text{for} \quad s \in \sigma \label{eq_h_psig} \\
\wh{g}^\sigma_{s,t} &=  g^\sigma_{s,t} &\qquad &\text{on} \quad Z^\sigma &\qquad & \text{for} \quad s \in \sigma, t \in [0,1] \label{eq_h_iotag}
\end{alignat}

\begin{Claim}
If (\ref{eq_h_k})--(\ref{eq_h_iotag}) hold, then (\ref{eq_enlarged_partial_h}) satisfies Properties \ref{prop_def_partial_homotopy_1}--\ref{prop_def_partial_homotopy_3} of Definition~\ref{Def_partial_homotopy}.
\end{Claim}

\begin{proof}
Property~\ref{prop_def_partial_homotopy_1} of Definition~\ref{Def_partial_homotopy} holds due to (\ref{eq_h_psig}).
Property~\ref{prop_def_partial_homotopy_2} holds by Assumptions~\ref{ass_ph_extend_i} and \ref{ass_ph_extend_ii}.
For Property~\ref{prop_def_partial_homotopy_3} observe that by (\ref{eq_k_gtau}) and (\ref{eq_h_k}) for any $s \in \tau \subset \partial\sigma$ we have
\[ \big((\psi^{\tau}_s)^{-1} \circ \wh\psi^\sigma_s \big)^* g^{\tau}_{s,t} = k_{s,t}  = \wh{g}^\sigma_{s,t} \qquad \text{on} \quad \wh{Z}^\sigma. \]
On the other hand, for any $s \in \sigma \subset \partial \tau$ we have by Assumption~\ref{ass_ph_extend_i}
\[ \psi^\tau_s (Z^\tau) \subset \psi^\sigma_s (Z^\sigma) = \wh\psi^\sigma_s (Z^\sigma). \]
Thus $(\wh\psi^{\sigma}_s )^{-1} \circ \psi^\tau_s  =  (\psi^{\sigma}_s )^{-1} \circ \psi^\tau_s : Z^\tau \to Z^\sigma$ and by (\ref{eq_h_iotag}) we have on $Z^\tau$
\[  \big((\wh\psi^{\sigma}_s)^{-1} \circ \psi^\tau_s \big)^* \wh{g}^\sigma_{s,t}=  \big((\psi^{\sigma}_s)^{-1} \circ \psi^\tau_s \big)^*g^\sigma_{s,t} = g^\tau_{s,t}. \qedhere \]
\end{proof}

We will now construct $(\wh{g}^\sigma_{s,t})_{s \in \sigma, t \in [0,1]}$ such that (\ref{eq_h_k})--(\ref{eq_h_iotag}) hold.

\bigskip
\textit{Case 1: $\sigma\not\subset L$.}\quad
We distinguish the two cases from Assumption~\ref{ass_ph_extend_iii}.
\medskip

\textit{Case 1a: Assumption~\ref{ass_ph_extend_iii-1} holds for all $s \in \sigma$. \quad}
For every $s \in \sigma$ let $\SS^{\prime,s}$ be the pull back of $\SS^s$ along $\wh\psi^\sigma_s$.
Then $(\SS^{\prime,s})_{s \in \sigma}$ is a transversely continuous family of spherical structures defined on neighborhoods of $Y$ in $\wh{Z}^\sigma$, $Y$ is a union of fibers of  $\SS^{\prime,s}$ and $(\wh\psi^\sigma_s)^* g^{\prime,s}_T$ is compatible with $\SS^{\prime,s}$ for all $s \in \sigma$.
Definition~\ref{Def_partial_homotopy}\ref{prop_def_partial_homotopy_5} implies that there is an $\eps > 0$ such that for every $t \in [0,1]$ the metric $g^\sigma_{s,t}$ restricted to an $\eps$-collar neighborhood of $\partial Z^\sigma$ inside $Z^\sigma$ is compatible with $\SS^{\prime,s}$.
So after restricting $\SS^{\prime, s}$ to a smaller domain, we may assume in the following that for all $s \in \sigma$ we have $Y \subset \domain \SS^{\prime, s}$ and that $g^\sigma_{s,t}$ is compatible with $\SS^{\prime, s}$ on $Z^\sigma \cap \domain ( \SS^{\prime,s} )$ for all $t \in [0,1]$.

\begin{Claim}
$k_{s,t}$ is compatible with $\SS^{\prime,s}$ for all $s \in \partial\sigma$ and $t \in [0,1]$
\end{Claim}

\begin{proof}
Note first that due to (\ref{eq_k_iotag}) the metric $k_{s,t}$ restricted to $Z^\sigma$ is compatible with $\SS^{\prime,s}$.
So it remains to check the compatibility only on $Y$.

Choose $\tau \subset \partial \sigma$ such that $s \in \tau$.
The set $((\psi^\tau_s)^{-1} \circ \wh\psi^\sigma_s)(Y)$ is contained in the closure $\C$ of a component of $Z^\tau \setminus ((\psi^\tau_s)^{-1} \circ \psi^\sigma_s)(Z^\sigma)$.
By Property~\ref{prop_def_partial_homotopy_4} of Definition~\ref{Def_partial_homotopy} there are two cases.

In the first case (Case~\ref{prop_def_partial_homotopy_4}\ref{prop_def_partial_homotopy_4i}) the image $\psi^\tau_s (\C)$ is a union of fibers of $\SS^s$ and $((\wh\psi^\sigma_s)_* k_{s,t} )|_{\wh\psi^\sigma_s(Y)}  = ((\psi^\tau_s)_* g^\tau_{s,t})|_{\wh\psi^\sigma_s(Y)}$ is compatible with $\SS^s$ for all $t \in [0,1]$.
It follows that $k_{s,t} |_Y$ is compatible with $\SS^{\prime,s}$ for all $t \in [0,1]$.

In the second case (Case~\ref{prop_def_partial_homotopy_4}\ref{prop_def_partial_homotopy_4i}) $\partial\C = \emptyset$, $\psi^\tau_s (\C) \subset U^s_{S3}$ and $g^\tau_{s,t}$ restricted to $\C$ is a multiple of $g^\tau_{s,0}$ for all $t \in [0,1]$.
It follows that $((\wh\psi^\sigma_s)_* k_{s,t} )|_{\wh\psi^\sigma_s(Y)}= ((\psi^\tau_s)_* g^\tau_{s,t})|_{\wh\psi^\sigma_s(Y)}$ is a multiple of $((\wh\psi^\sigma_s)_* k_{s,0} )|_{\wh\psi^\sigma_s(Y)}= ((\psi^\tau_s)_* g^\tau_{s,0})|_{\wh\psi^\sigma_s(Y)}$ for all $t \in [0,1]$.
So $k_{s,t}|_Y$ is a multiple of $k_{s,0}|_Y = ((\wh\psi^\sigma_s)^* g^{\prime,s}_T)|_Y$ for all $t \in [0,1]$, which is compatible with $\SS^{\prime,s}$.
\end{proof}

We can now construct $(\wh{g}^\sigma_{s,t})_{s \in \sigma, t \in [0,1]}$ using Proposition~\ref{prop_extending_symmetric}.
\medskip

\textit{Case 1b: Assumption~\ref{ass_ph_extend_iii-2} holds for all $s \in \sigma$. \quad}
Recall that in this case $\wh{Z}^\sigma$ is a disjoint union of $Z^\sigma$ and $Y$, $\wh\psi^\sigma_s (Y) \subset U^s_{S3}$ and there is a constant curvature metric $g^*$ on $Y$ and a continuous function $\lambda : \sigma \to \IR_+$ that $((\wh\psi^\sigma_s)^* g^{\prime,s}_{T} )|_Y = \lambda^2(s) g^*$ for all $s \in \sigma$.
Set $(\wh{g}^\sigma_{s,t}) := (g^\sigma_{s,t})$ on $Z^\sigma$.
Then (\ref{eq_h_iotag}) is satisfied and it remains to specify $(\wh{g}^\sigma_{s,t})$ on $Y$.

Let 
\[ A := \{ s \in \sigma \;\; : \;\; \wh\psi^\sigma (Y) \subset U^s_{S2} \} \] and for every $s \in A$ let $\SS^{\prime,s}$ be the pull back of $\SS^s$ to $Y$ along $\wh\psi^\sigma|_Y$.
Note that $A$ is open in $\sigma$, $(\SS^{\prime,s})_{s \in A}$ is transversely continuous and $(\wh\psi^\sigma_s)^* g^{\prime,s}_T$ is compatible with $\SS^{\prime,s}$ for all $s \in A$.
Moreover, if $A \neq \emptyset$, then $(Y, g^*)$ is isometric to the standard $S^3$ or $\IR P^3$.

Next let us analyze the family $(k_{s,t})_{s \in \partial\sigma, t \in [0,1]}$.
For any $s \in \tau \subset \partial\sigma$ the set $\C_{\tau,s} := ((\psi^\tau_s)^{-1} \circ \wh\psi^\sigma_s) (Y)$ is a component of $Z^\tau$.
Since $\tau$ is connected we have $\C_{\tau,s} = \C_\tau$ for all $s \in \tau$.
By Definition~\ref{Def_partial_homotopy}\ref{prop_def_partial_homotopy_4}, (\ref{eq_k_gtau}) and (\ref{eq_k_psi_g}) there is a closed subset $E_\tau \subset A \cap \tau$ such that:
\begin{enumerate}
\item For all $s \in E_\tau$ and $t \in [0,1]$ the metric $k_{s,t}|_Y$ is compatible with $\SS^{\prime,s}$.
\item For all $s \in \tau \setminus E_\tau$ and $t \in [0,1]$ the metric $k_{s,t} |_Y$ is a multiple of $k_{s,0} |_Y = (\wh\psi^\sigma_s g^{\prime,s}_T) |_Y = \lambda^2 (s) g^*$.
\end{enumerate}
Set $E := \cup_{\tau \subset \partial\sigma} E_\tau \subset A$ and notice that $E$ is closed.
By Proposition~\ref{prop_exten_round} we can construct $(\wh{g}^\sigma_{s,t} |_Y)_{s \in \sigma, t \in [0,1]}$ such that (\ref{eq_h_k}) and (\ref{eq_h_psig}) hold and such that $\wh{g}^\sigma_{s,t}|_Y$ is compatible with $\SS^{\prime,s}$ for all $s \in A$ and $t \in [0,1]$.
Moreover, $\wh{g}^\sigma_{s,t}|_Y$ is a multiple of $g^*$ for all $s \in A \setminus E'$ for some closed subset $E' \subset A \subset \sigma$.

\bigskip
\textit{Case 2: $\sigma\subset L$ \quad} 
We may assume that $Y \neq \emptyset$. 
By Assumption~\ref{ass_ph_extend_iv} we have $\partial Y = \emptyset$, $\wh{Z}^\sigma = Y$ and $Z^\sigma = \emptyset$.

\medskip
\textit{Case 2a: $Y$ is a spherical space form. \quad}
Due to Definition~\ref{Def_partial_homotopy}\ref{prop_def_partial_homotopy_6} and (\ref{eq_k_psi_g}) we have $k_{s,t} = \lambda^2_{s,t} ( \wh\psi^\sigma_s )^* g^{\prime, s}_T$ for all $s \in \partial \sigma$, $t \in [0,1]$, where $(s,t) \mapsto \lambda_{s,t} > 0$ is continuous and $\lambda_{s,0} = 1$.
Extend this function to a continuous function $(s,t) \mapsto \td\lambda_{s,t} > 0$ on $\sigma \times [0,1]$ and set $h_{s,t} := \td\lambda^2_{s,t} ( \wh\psi^\sigma_s )^* g^{\prime, s}_T$.

\medskip
\textit{Case 2b: $Y$ is a quotient of a cylinder. \quad}
As in Case 1a let $\SS^{\prime,s}$ be the pull back of $\SS^s$ along $\wh\psi^\sigma_s$.
For any $s \in \sigma$ we can split $( \wh\psi^\sigma_s )^* g^{\prime, s}_T = u_s + v_s$ into its components tangential and orthogonal to the fibers of $\SS^{\prime,s}$.
For the same reasons as in Case~2a we have $k_{s,t} = \lambda^2_{s,t} u_s + \mu^2_{s,t} v_s$, where $(s,t) \mapsto \lambda_{s,t} > 0$, $(s,t) \mapsto \mu_{s,t} > 0$ are continuous functions with $\lambda_{s,0} = \mu_{s,0} = 1$.
Let $\td\lambda_{s,t}, \td\mu_{s,t} : \sigma \times [0,1] \to \IR_+$ be continuous extensions of these functions and set $h_{s,t} := \td\lambda^2_{s,t}  u_s + \td\mu^2_{s,t}  v_s$.

\bigskip
We now verify that $(\wh{g}^\sigma_{s,t})_{s \in \sigma, t \in [0,1]}$ satisfies the required properties in all cases.
For this purpose, recall that by (\ref{eq_h_psig}) for all $s \in \sigma$ the metric $\wh g^\sigma_{s,1}$ restricted to $Z^\sigma$ is conformally flat and PSC-conformal and $\wh g^\sigma_{s,1}$ restricted to a neighborhood of $\wh Z^\sigma \setminus \Int Z^\sigma$ is compatible with a spherical structure and therefore also conformally flat.
So by Lemma~\ref{Lem_PSC_conformal_enlarge} we obtain that $(\wh{Z}, \wh g^\sigma_{s,1})$ is conformally flat and PSC-conformal for all $s \in \sigma$, which implies that $(\wh{Z}^\sigma, (\wh{g}^\sigma_{s,t})_{s \in \sigma, t \in [0,1]})$ is a metric deformation.

Next, consider Definition~\ref{Def_partial_homotopy}.
Properties~\ref{prop_def_partial_homotopy_5}, \ref{prop_def_partial_homotopy_6} of Definition~\ref{Def_partial_homotopy} hold by construction.

Let us now verify Property~\ref{prop_def_partial_homotopy_4} of Definition~\ref{Def_partial_homotopy}.
First assume that $s \in \tau \subset \partial\sigma$ and consider the embedding
\[ F := (\psi^\tau_s)^{-1} \circ \wh\psi^\sigma_s : \wh{Z}^\sigma \longrightarrow Z^\tau. \]
Consider the closure $\C$ of a component of $Z^\tau \setminus F(\wh{Z}^\sigma) = Z^\tau \setminus ( F(Z^\sigma) \cup F(Y))$.
Let $\C'$ be the closure of a component of $Z^\tau \setminus F(Z^\sigma)$ such that $\C' \supset \C$.
If $\C' = \C$, then there is nothing to show.
So assume that $\C' \supsetneq \C$.
Then $\C' \supset F(Y)$ and $\C$ is the closure of a component of $\C' \setminus F(Y)$.
This implies that $\partial Y \neq \emptyset$ and therefore the construction of $(\wh{g}^\sigma_{s,t})_{s \in \sigma, t \in [0,1]}$ was covered in Case~1a.
Let us now consider the two possibilities of Property~\ref{prop_def_partial_homotopy_4} of Definition~\ref{Def_partial_homotopy} for $\C'$:
\begin{enumerate}[label=(\roman*)]
\item $\psi^\tau_s(\C')$ is a union of spherical fibers and $(\psi^\tau_s)_* g^\tau_{s,t}$ restricted to $\psi^\tau_s (\C')$ is compatible with the restricted spherical structure for all $t \in [0,1]$.
Since $\psi^\tau_s (F(Y)) = \wh\psi^\sigma_s (Y)$ is a union of spherical fibers, the same is true if we replace $\C'$ by $\C$.
\item $\C'$ is closed, $\psi^\tau_s (\C') \subset U^s_{S3}$ and for every $s' \in \tau$ near $s$ the metric $g^\tau_{s',t}$ restricted to $\C$ is a multiple of $g^\tau_{s', 0}$ for all $t \in [0,1]$.
Again, since $\psi^\tau_s (F(Y)) = \wh\psi^\sigma_s (Y)$ is a union of spherical fibers, we obtain using Definition~\ref{Def_R_structure}\ref{prop_def_RR_1} that $\psi^\tau_s (\C') \subset U^s_{S2}$.
This implies that $(\psi^\tau_s)_* g^\tau_{s,t}$ restricted to $\psi^\tau_s (\C)$, being a multiple of $(\psi^\tau_s)_* g^\tau_{s,0} = g^{\prime,s}_T$ restricted to $\psi^\tau_s (\C)$, is compatible with the restricted spherical structure.
Hence this case implies Case (i).
\end{enumerate}

Next assume that $s \in \sigma \subset \partial\tau$ and consider the embedding
\[ F := \big(\wh\psi^\sigma_s \big)^{-1} \circ \psi^\tau_s : Z^\tau \longrightarrow \wh{Z}^\sigma. \]
Consider the closure $\C$ of a component of $\wh{Z}^\sigma \setminus F(Z^\tau)$.
If $\C$ is also the closure of a component of $Z^\sigma \setminus F(Z^\tau)$, then we are done, so assume that this is not the case.
It follows that $\C \supset Y$. 
Moreover, if $\C'_1, \ldots, \C'_m$ denote the closures of all components of $Z^\sigma \setminus F(Z^\tau)$ and $I \subset \{ 1, \ldots, m \}$ denotes the set of indices with the property that $\C'_i \cap Y \neq \emptyset$, then $\C= Y \cup_{i \in I} \C'_i$.
If $(\wh{g}^\sigma_{s,t})_{s \in \sigma, t \in [0,1]}$ was constructed according to Cases 1b or Case 2, then $I = \emptyset$, $\C = Y$ and the conditions in Property~\ref{prop_def_partial_homotopy_4} of Definition~\ref{Def_partial_homotopy} hold by the discussions in these cases.
So assume that $(\wh{g}^\sigma_{s,t})_{s \in \sigma, t \in [0,1]}$ was constructed according to Case 1a.
Since $\partial\C'_i \neq \emptyset$ for all $i \in I$, Property~\ref{prop_def_partial_homotopy_4} of Definition~\ref{Def_partial_homotopy} implies that for all $i \in I$ the set $\wh\psi^\sigma_s (\C'_i)$ is a union of spherical fibers and for every $t \in [0,1]$ the metric $(\wh\psi^\sigma_s)_* \wh{g}^\sigma_{s,t}$ restricted to $\wh\psi^\sigma_s (\C'_i)$ is compatible with the restricted spherical structure.
By construction, the same is true if we replace $\C'_i$ by $Y$.
Therefore $\wh\psi^\sigma_s (\C)$ is a union of spherical fibers and $(\wh\psi^\sigma_s)_* \wh{g}^\sigma_{s,t}$ restricted to $\wh\psi^\sigma_s (\C)$ is compatible with the restricted spherical structure
It follows that $\C$ satisfies Property~\ref{prop_def_partial_homotopy_4}\ref{prop_def_partial_homotopy_4i} of Definition~\ref{Def_partial_homotopy} (with $\sigma$ and $\tau$ switched).

Lastly, assume that the original partial homotopy was PSC-conformal over some $s \in K$.
We will argue that (\ref{eq_enlarged_partial_h}) is PSC-conformal over $s$ as well.
For this purpose it remains to consider the case $s \in \sigma$, and we only need to show that $(\wh{Z}^\sigma, \wh{g}^\sigma_{s,t})$ is PSC-conformal for all $t \in [0,1]$.
Recall that $(Z^\sigma, g^\sigma_{s,t})$ is PSC-conformal.
In Cases~1b, 2a, 2b, $Y$ is a component of $\wh{Z}^\sigma$ and we can use Lemma~\ref{Lem_PSC_conformal_enlarge} to show that $(Y,  \wh{g}^\sigma_{s,t})$ is PSC-conformal.
In Case~1a the metric $\wh{g}^\sigma_{s,t}$ is compatible with a spherical structure on $\wh{Z}^\sigma$ with the property that $Y$ is a union of spherical fibers.
Therefore, we obtain again using Lemma~\ref{Lem_PSC_conformal_enlarge} that $(\wh{Z}^\sigma, \wh{g}^\sigma_{s,t})$ is PSC-conformal.
\end{proof}

\subsection{Removing a disk from a partial homotopy}
In this subsection we will show that one can modify a partial homotopy by the removal of  a disk of controlled size from some $Z^\sigma$, without disturbing the partial homotopy property.  
To get an idea of (one of the situations in which) this operation will eventually be applied, consider a single singular Ricci flow $\M$ that undergoes a degenerate neck pinch at some time $T_0>0$.  Right after the singular time, one observes a cap region modelled on Bryant soliton, whose scale decreases to zero as $t\searrow T_0$.  For certain times $T>T_0$ close to $T_0$, we will find that there are two equally admissible truncations of $\M_T$, which agree with one another modulo the removal of an approximate Bryant $3$-disk region.  The following proposition will enable one to adjust a partial homotopy to accommodate the removal of such a disk region.

\begin{proposition}[Disk removal] \label{prop_move_remove_disk}
Consider a partial homotopy $\{ ( Z^\sigma, \lb (g^\sigma_{s,t})_{s \in \sigma, t \in [0,1]}, \lb (\psi^\sigma_s  )_{s \in \sigma}) \}_{\sigma \subset K}$ at time $T$ relative to $L$ and a closed subset $K_{PSC} \subset K$.  
Fix some simplex $\sigma_0 \subset K$ with $\si_0\cap L=\emptyset$, and assume that there is a collection of continuous families of embeddings $\{(\nu_{s,j} : D^3(1)=:D^3 \to \M^s_T )_{s \in \sigma_0}\}_{1\leq j\leq m}$ such that the following holds:
\begin{enumerate}[label=(\roman*)]
\item \label{lem_move_remove_disk_i} For all $s \in \sigma_0$ the embeddings $\{\nu_{s,j}\}_{1\leq j\leq m}$ have pairwise disjoint images and for every  $j \in \{ 1,\ldots,m \}$ we have $\nu_{s,j} (D^3) \subset \psi^{\sigma_0}_s (\Int Z^{\sigma_0}) \cap  U^s_{S2}$.
\item \label{lem_move_remove_disk_ii} For every $s \in \sigma_0$, $j \in \{1,\ldots,m \}$, the embedding $\nu_{s,j}$ carries the standard spherical structure on $D^3$ to $\SS^s$ restricted to $\nu_{s,j}(D^3)$.
\item \label{lem_move_remove_disk_iii} $\nu_{s,j} (D^3) \cap \psi^\tau_s (Z^\tau) = \emptyset$ whenever $s\in{\sigma_0} \subsetneq \tau \subset K$, $j\in \{1,\ldots,m\}$.
\item \label{lem_move_remove_disk_iv} If $\si_0$ is a maximal simplex, i.e. $\sigma_0$ is not properly contained in any other simplex $\sigma_1 \subset K$, then we assume additionally that for any $\tau \subset K$ with the property that $\tau \cap {\sigma_0} \neq \emptyset$ the following holds for all $s \in {\sigma_0} \cap \tau$, $j \in \{ 1,\ldots,m \}$:
The image $\nu_{s,j} (D^3)$ does not contain an entire component of $\psi^{\tau}_s (Z^{\tau})$.
\item \label{lem_move_remove_disk_v} $\{ ( Z^\sigma, \lb (g^\sigma_{s,t})_{s \in \sigma, t \in [0,1]}, \lb (\psi^\sigma_s  )_{s \in \sigma}) \}_{\sigma \subset K}$ is PSC-conformal over every $s\in K_{PSC}$.
\item \label{lem_move_remove_disk_vi} For all $s \in \sigma_0 \cap K_{PSC}$ the Riemannian manifold 
 \[ \big( \psi^{\sigma_0}_s (Z^{\sigma_0})  \setminus \cup_{j=1}^m \nu_{s,j} (\Int D^3), g^{\prime,s}_T \big) \]
 is PSC-conformal.
\end{enumerate}

Then we can find a partial homotopy $\{ (\td Z^\sigma, \lb (\td g^\sigma_{s,t})_{s \in \sigma, t \in [0,1]}, \lb (\td\psi^\sigma_s  )_{s \in \sigma}) \}_{\sigma \subset K}$ at time $T$ relative to $L$ such that the following holds:
\begin{enumerate}[label=(\alph*)]
\item \label{lem_move_remove_disk_a}  $\td\psi^\tau_s (\td{Z}^\tau) = \psi^\tau_s (Z^\tau)$ for all $s\in \tau$, if $\tau \neq {\sigma_0}$.
\item \label{lem_move_remove_disk_b}  $\td\psi^{\sigma_0}_s (\td{Z}^{\sigma_0}) = \psi^{\sigma_0}_s (Z^{\sigma_0}) \setminus \cup_{j=1}^m \nu_{s,j} (\Int D^3)$ for all $s\in {\sigma_0}$.
\item \label{lem_move_remove_disk_c}  $\{ ( \td Z^\sigma, \lb (\td g^\sigma_{s,t})_{s \in \sigma, t \in [0,1]}, \lb (\td\psi^\sigma_s  )_{s \in \sigma}) \}_{\sigma \subset K}$ is PSC-conformal over every $s\in K_{PSC}$.
\end{enumerate}
\end{proposition}

Before proceeding with the proof, we give an indication of some of the main points.
For simplicity we will restrict to the case when only one disk is removed, and  drop the index $j$.  
The main objective in the removal procedure is to ensure that the removed disk satisfies  Definition~\ref{Def_partial_homotopy}\ref{prop_def_partial_homotopy_4} --- the $\SS$-compatibility of the metric deformation $(g^{\sigma_0}_{s,t})$ --- on the disk.  

If $\si_0$ is not maximal, we choose some simplex $\tau\supsetneq \si_0$.
By applying Assumption~\ref{lem_move_remove_disk_iii} and Definition~\ref{Def_partial_homotopy}\ref{prop_def_partial_homotopy_4} to the pair $\si_0\subsetneq\tau$, we find that $(\psi^{\si_0}_s)_*g^{\si_0}_{s,t}$ is already $\SS^s$-compatible on the entire disk $\nu_s(D^3)$.  Hence we may simply remove $\nu_s(D^3)$ from $\psi^{\si_0}_s(Z^{\si_0})$ and the verification of Definition~\ref{Def_partial_homotopy}  amounts to unwinding of definitions.  

When $\si_0$ is maximal, the metric deformation $(\psi^{\si_0}_s)_*g^{\si_0}_{s,t}$ is typically not $\SS^s$-compatible on the $3$-disk $\nu_s(D^3)$, so we must modify it to respect the compatibility conditions on the system of metric deformations in a partial homotopy.
This necessitates adjustments to metric deformations $(g^\tau_{s,t})$ for $\tau$ near $\si_0$ as well.  To first approximation, the modification process involves two steps: we first apply a rounding procedure to  $(\psi^{\si_0}_s)_*g^{\si_0}_{s,t}$ on a small ball centered at the singular fiber $\nu_s(0)$ of $\SS^s$ to produce a metric deformation that is $\SS^s$-compatible near $\nu_s(0)$; then we push forward the resulting metric deformation by an $\SS^s$-compatible diffeomorphism to inflate the region of $\SS^s$-compatibility so that it covers $\nu_s(D^3)$.  The actual process is more involved than this for several reasons:  One issue is that we need to cut off the modification procedure in a neighborhood of $\si_0$.  To address this, we upgrade the two steps mentioned above --- the rounding procedure and the push forward by an inflationary diffeomorphism --- into continuous families depending on a  parameter lying in the interval $[0,1]$, and implement the cutoff by arranging for this control parameter to be supported close to $\si_0\subset K$.  Another major constraint on the modification process is the $\SS$-compatibility required by Definition~\ref{Def_partial_homotopy}\ref{prop_def_partial_homotopy_4} for various pairs $\tau\subsetneq \si$.   The assumptions of Proposition~\ref{prop_move_remove_disk} 
imply that for pairs $\tau\subsetneq\si$ and $s\in \tau$ near $\si_0$, the portion of the $3$-disk $\nu_s(D^3)$ where Definition~\ref{Def_partial_homotopy}\ref{prop_def_partial_homotopy_4} imposes $\SS$-compatibility is either a disk  $\nu_s(D^3(r))$ or an annulus $\nu_s(A^3(r,r'))$, where $r$ is bounded away from $0$.  This enables us to guarantee that the two step modification process --- rounding and inflation --- leaves the $\SS$-compatibility undisturbed.

\begin{proof}
By induction we may reduce the proposition to the case $m=1$ (Alternatively, the construction given below may be localized near each disk, so that the case when $m>1$ follows by applying the same argument simultaneously).
Note here that due to Lemma~\ref{Lem_PSC_conformal_enlarge} Assumptions~\ref{lem_move_remove_disk_ii} and \ref{lem_move_remove_disk_vi} imply that for any subset $I \subset \{ 1, \ldots, m \}$ the Riemannian manifold
 \[ \big( \psi^{\sigma_0}_s (Z^{\sigma_0}) \setminus \cup_{j \in I} \nu_{s,j} (\Int D^3), g^{\prime,s}_T \big) \]
 is PSC-conformal.
So we may assume in the following that $m=1$ and drop the index $j$ from now on.

Using the exponential map based at $\nu_s (0)$ and the fact that $\cup_{s \in K} U^s_{S2}$ is open, we can construct a neighborhood $\sigma_0 \subset U \subset K$ and a continuous family of embeddings $(\ov\nu_s : D^3 (2) \to \M^s)_{s \in U}$ such that $\nu_s = \ov\nu_s |_{D^3 (1)}$ if $s \in \sigma_0$ and such that:
\begin{enumerate}[label=(\arabic*)]
\item \label{prop_proof_disk_removal_1} For all $s \in \sigma_0$ we have $\ov\nu_s (D^3(2)) \subset \psi^{\sigma_0}_s (\Int Z^{\sigma_0})$.
\item \label{prop_proof_disk_removal_2} For all $s \in U$ we have $\ov\nu_s (D^3(2)) \subset U^s_{S2}$ and $\ov\nu_s$ carries the standard spherical structure on $D^3(2)$ to the restriction of $\SS^s$ to $\ov\nu_s(D^3(2))$.  
\item \label{prop_proof_disk_removal_3} $\ov\nu_s (D^3(2)) \cap \psi^\tau_s (Z^\tau) = \emptyset$ whenever ${\sigma_0} \subsetneq \tau \subset K$.
\item \label{prop_proof_disk_removal_4} If $\sigma_0$ is maximal, then for any $\tau \subset K$ and $s \in \tau \cap U$ the image $\ov\nu_s (D^3(2))$ does not contain an entire component of $\psi^{\tau}_s (Z^{\tau})$.
\end{enumerate}
In the following we will write $\nu_s$ instead of $\ov\nu_s$ for simplicity.

Next, we exploit the invariance of Definition~\ref{Def_partial_homotopy} under precomposition by diffeomorphisms to argue that we may assume in addition, without loss of generality: 
\begin{enumerate}[label=(\arabic*), start=5]
\item \label{prop_proof_disk_removal_5} There is an embedding $\mu : D^3 (2) \to Z^{\sigma_0}$ with the property that $\nu_s = \psi^{\sigma_0}_s \circ \mu$ for all $s \in \sigma_0$.
\end{enumerate}
More specifically, let $(\chi_s : Z^{\sigma_0} \to Z^{\sigma_0})_{s \in {\sigma_0}}$ be a continuous family of diffeomorphisms that are equal to the identity near $\partial Z^{\sigma_0}$ and such that $\chi_s^{-1} \circ (\psi^{\sigma_0}_s)^{-1} \circ \nu_s$ is constant in $s$.
Set
\[ \big( \ov\psi^{\sigma_0}_s := \psi^{\sigma_0}_s \circ \chi_s : Z^{\sigma_0} \to Z^{\sigma_0} \big)_{s \in {\sigma_0}}, \qquad \ov{g}^{\sigma_0}_{s,t} :=  \chi_s^* g_{s,t}^{\sigma_0}.  \]
If we replace $(g^{\sigma_0}_{s,t})_{s \in \sigma_0, t \in [0,1]}, \lb (\psi^{\sigma_0}_s  )_{s \in \sigma_0}$ by $(\ov{g}^{\sigma_0}_{s,t})_{s \in \sigma_0, t \in [0,1]}, \lb (\ov\psi^{\sigma_0}_s  )_{s \in \sigma_0}$, then the conditions for a partial homotopy are preserved, and $\mu := (\psi^{\sigma_0}_s)^{-1} \circ \nu_s$ is constant in $s$.

Next, we claim that the following is true:
\begin{enumerate}[label=(\arabic*), start=6]
\item \label{prop_proof_disk_removal_6} If $\tau \subsetneq \sigma \subset K$, $s \in \tau \cap U$ and $\C$ is the closure of a component of $Z^{ \tau}  \setminus((\psi^{ \tau}_s )^{-1} \circ \psi^{ \sigma}_s ) ( Z^{ \sigma} )$ with $\psi^\tau_s(\C) \cap \nu_s (D^3(2)) \neq \emptyset$, then Case~\ref{prop_def_partial_homotopy_4}\ref{prop_def_partial_homotopy_4i} of Definition~\ref{Def_partial_homotopy}\ref{prop_def_partial_homotopy_4} holds for $\C$.
\end{enumerate}
In fact, if $\C$ satisfies Definition~\ref{Def_partial_homotopy}\ref{prop_def_partial_homotopy_4}\ref{prop_def_partial_homotopy_4i}, then we are done; so assume that it satisfies Definition~\ref{Def_partial_homotopy}\ref{prop_def_partial_homotopy_4}\ref{prop_def_partial_homotopy_4ii}.
Then $\partial\C = \emptyset$, $\psi^{\tau}_s (\C) \subset U^s_{S3}$ and $g^{\tau}_{s,t}$ restricted to $\C$ is a multiple of $g^{\tau}_{s,0}$.
Since $\nu_s(D^3(2))\subset \psi^\tau_s(\C) \subset U_{S2}$, we conclude using Definition~\ref{Def_R_structure}\ref{prop_def_RR_1} that $\psi^{\tau}_s(\C) \subset U^s_{S2}$.  
So since $\partial \C = \emptyset$, we obtain that $\psi^{\tau}_s(\C)$ is a union of spherical fibers.
By Definition~\ref{Def_R_structure}\ref{prop_def_RR_4} we obtain that $(\psi^{\tau}_s)_* g^{\tau}_{s,0} = g^{\prime,s}_T$, and thus also $(\psi^{\tau}_s)_* g^{\tau}_{s,t}$, restricted to $\psi^{\tau}_s(\C)$ is compatible with $\SS^s$.

\medskip
\textit{Case 1: $\sigma_0 \subsetneq  \sigma_1$ for some simplex $\sigma_1 \subset K$. \quad}
Let us first apply Property~\ref{prop_proof_disk_removal_6} above for $\tau = \sigma_0$ and $\sigma = \sigma_1$.
By Property~\ref{prop_proof_disk_removal_3} above, for any $s \in \sigma_0$, there is a component $\C \subset Z^{ \sigma_0}  \setminus((\psi^{ \sigma_0}_s )^{-1} \circ \psi^{ \sigma_1}_s ) ( Z^{ \sigma_1} )$ with $\psi^{\sigma_0}_s (\C) \supset \nu_s (D^3(2))$.
So by Properties~\ref{prop_proof_disk_removal_2}, \ref{prop_proof_disk_removal_5}, \ref{prop_proof_disk_removal_6} above we obtain:
\begin{enumerate}[label=(\arabic*), start=7]
\item \label{prop_proof_disk_removal_7} For all $s \in \sigma_0$, $t \in [0,1]$ the pullback $\mu^* g^{\sigma_0}_{s,t}$ is compatible with the standard spherical structure on $D^2(2)$.
\end{enumerate}

Now set
\[ \td{Z}^{\sigma_0} := Z^{\sigma_0} \setminus \mu^{-1} (\Int D^3(1)), \qquad
\td\psi^{\sigma_0}_s := \psi^{\sigma_0}_s \big|_{\td{Z}^{\sigma_0}}, \qquad
\td{g}^{\sigma_0}_{s,t} := g^{\sigma_0}_{s,t} \big|_{\td{Z}^{\sigma_0}} \]
and 
\[ (\td Z^\sigma, \lb (\td g^\sigma_{s,t})_{s \in \sigma, t \in [0,1]}, \lb (\td\psi^\sigma_s  )_{s \in \sigma}) := ( Z^\sigma, \lb ( g^\sigma_{s,t})_{s \in \sigma, t \in [0,1]}, \lb (\psi^\sigma_s  )_{s \in \sigma}) \]
for all $\sigma \neq \sigma_0$.
Then Assertions~\ref{lem_move_remove_disk_a}, \ref{lem_move_remove_disk_b} of this proposition hold automatically.

Let us now verify that $\{ (\td Z^\sigma, \lb (\td g^\sigma_{s,t})_{s \in \sigma, t \in [0,1]}, \lb (\td\psi^\sigma_s  )_{s \in \sigma}) \}_{\sigma \subset K}$ is a partial homotopy.

First, we argue that $(\td Z^{\sigma_0}, \lb (\td g^{\sigma_0}_{s,t})_{s \in \sigma_0, t \in [0,1]})$ is a metric deformation (see Definition~\ref{Def_metric_deformation}).
Since $(Z^{\sigma_0}, \lb ( g^{\sigma_0}_{s,t})_{s \in \sigma_0, t \in [0,1]})$ is a metric deformation, the only non-trivial property is the PSC-conformality of $(\td Z^{\sigma_0}, \lb \td g^{\sigma_0}_{s,1})$, which follows from Lemma~\ref{Lem_PSC_conformal_enlarge} with $M=\td Z^{\si_0}$, $Z=(\psi^{\si_0}_s)^{-1}(\psi^{\si_1}_s(Z^{\si_1}))\subset \td Z^{\si_0}$ and $g=\td g^{\si_0}_{s,1}$.

Properties~\ref{prop_def_partial_homotopy_1}--\ref{prop_def_partial_homotopy_3}, \ref{prop_def_partial_homotopy_6} of Definition~\ref{Def_partial_homotopy} clearly hold for $\{ (\td Z^\sigma, \lb (\td g^\sigma_{s,t})_{s \in \sigma, t \in [0,1]}, \lb (\td\psi^\sigma_s  )_{s \in \sigma}) \}_{\sigma \subset K}$.  
Property~\ref{prop_def_partial_homotopy_5} is unaffected by the modification, except for the boundary component $\mu(\D D^3(1)) \subset \td Z^{\sigma_0}$, for which Property~\ref{prop_def_partial_homotopy_5} follows from  Property~\ref{prop_proof_disk_removal_7} above.

We now verify Property~\ref{prop_def_partial_homotopy_4}.   
Note that it holds for pairs of simplices $ \tau\subsetneq \sigma$ when $\sigma_0\not\in\{\tau,\sigma\}$, because $\{ ( Z^\sigma, \lb (g^\sigma_{s,t})_{s \in \sigma, t \in [0,1]}, \lb (\psi^\sigma_s  )_{s \in \sigma}) \}_{\sigma \subset K}$ is a partial homotopy.  

Suppose that $s\in \sigma_0\subsetneq\si $ for some simplex $\si \subset K$.  
Then the collection of components of $\td Z^{\sigma_0}\setminus ((\td\psi^{\sigma_0}_s)^{-1}\circ\td\psi^{\si}_s)(\td Z^{\si})$ is the same as the collection of components of  $Z^{\si_0}\setminus ((\psi^{\sigma_0}_s)^{-1}\circ\psi^{\si}_s)(Z^{\si})$, except for the one containing $\mu (D^3(1))$; let $\td\C$  denote its  closure and $\C$ denote the closure of the corresponding component of $Z^{\sigma_0}\setminus ((\psi^{\sigma_0}_s)^{-1}\circ\psi^{\si}_s)(Z^{\si})$.  
So $\td{\C} = \C \setminus \mu (\Int D^3(1))$.
By Properties~\ref{prop_proof_disk_removal_2}, \ref{prop_proof_disk_removal_6} above, we obtain that $\td\C$ satisfies Property~\ref{prop_def_partial_homotopy_4}\ref{prop_def_partial_homotopy_4i} of Definition~\ref{Def_partial_homotopy}.

Next suppose that $s\in \tau \subsetneq  \sigma_0$.  
Then the closures of the component $\td Z^{\tau}\setminus ((\td\psi^{\tau}_s)^{-1}\circ\td\psi^{\si_0}_s)(\td Z^{\si_0})$ are the same as those of $ Z^{\tau}\setminus ((\psi^{\tau}_s)^{-1}\circ\psi^{\si_0}_s)(Z^{\si_0})$ plus the component $\C:=(\psi^{\tau}_s)^{-1}(\nu_s(D^3(1)))$.  
It follows from Property~\ref{prop_proof_disk_removal_7} that $\C$ satisfies Definition~\ref{Def_partial_homotopy}\ref{prop_def_partial_homotopy_4}\ref{prop_def_partial_homotopy_4i}.

We now verify Assertion~\ref{lem_move_remove_disk_c} of this proposition.  Suppose that $\si\subset K$ is a simplex and $s\in K_{PSC}\cap\si$.  
 If $\si\neq\si_0$, then $\td Z^\si=Z^\si$ and $\td g^\si_{s,t}=g^\si_{s,t}$ for all $t\in [0,1]$, so $(\td Z^\sigma, \td g^{\si}_{s,t})$ is PSC-conformal by assumption.   If $\si=\si_0$, then for every $t\in[0,1]$ we may apply Lemma~\ref{Lem_PSC_conformal_enlarge} with $M=\td Z^{\si_0}$, $Z=(\psi^{\si_0}_s)^{-1}(\psi^{\si_1}_s(Z^{\si_1}))\subset \td Z^{\si_0}$ and $g=\td g^{\si_0}_{s,t}$ to conclude that $\td g^{\si_0}_{s,t}$ is PSC-conformal. 

\medskip
\textit{Case 2: $\sigma_0$ is a maximal simplex of $K$. \quad}
In this case, by Properties \ref{prop_proof_disk_removal_1}--\ref{prop_proof_disk_removal_6} above imply that the assumptions of Lemma~\ref{lem_make_compatible_on_disk} below hold.  So the proposition follows from Lemma~\ref{lem_make_compatible_on_disk} below.
\end{proof}

\begin{lemma} \label{lem_make_compatible_on_disk}
Consider a partial homotopy $\{ ( Z^\sigma, \lb (g^\sigma_{s,t})_{s \in \sigma, t \in [0,1]}, \lb (\psi^\sigma_s  )_{s \in \sigma}) \}_{\sigma \subset K}$ at time $T$ relative to $L$ and a closed subset $K_{PSC} \subset K$.
Fix some simplex $\sigma_0 \subset K$, $\sigma_0\cap L=\emptyset$ and a neighborhood $\sigma_0 \subset U \subset K$ and assume that there is a continuous family of embeddings $(\nu_s : D^3(2) \to \M^s )_{s \in U}$ such that the following holds:
\begin{enumerate}[label=(\roman*)]
\item \label{lem_make_compatible_on_disk_i} $\si_0$ is a maximal simplex. 
\item \label{lem_make_compatible_on_disk_ii} There is an embedding $\mu : D^3 (2) \to \Int Z^{\sigma_0}$ such that $\nu_s = \psi^{\sigma_0}_s \circ \mu$ for all $s \in \sigma_0$.
\item \label{lem_make_compatible_on_disk_iii} For all $s \in U$ we have $\nu_s (D^3(2)) \subset  U^s_{S2}$ and the embedding $\nu_s$ carries the standard spherical structure on $D^3$ to $\SS^s$ restricted to $\nu_{s}(D^3(2))$.
\item \label{lem_make_compatible_on_disk_iv} For any $\tau \subset K$ and $s \in \tau \cap U$ the image $\nu_s (D^3(2))$ does not contain an entire component of $\psi^{\tau}_s (Z^{\tau})$.
\item \label{lem_make_compatible_on_disk_vii} If $\tau \subsetneq \sigma \subset K$, $s \in \tau \cap U$ and $\C$ is the closure of a component of $Z^{ \tau}  \setminus((\psi^{ \tau}_s )^{-1} \circ \psi^{ \sigma}_s ) ( Z^{ \sigma} )$ with $\psi^\tau_s (\C) \cap \nu_s (D^3(2)) \neq \emptyset$, then Case~\ref{prop_def_partial_homotopy_4}\ref{prop_def_partial_homotopy_4i} of Definition~\ref{Def_partial_homotopy}\ref{prop_def_partial_homotopy_4} holds for $\C$.
\item \label{lem_make_compatible_on_disk_v} $\{ ( Z^\sigma, \lb (g^\sigma_{s,t})_{s \in \sigma, t \in [0,1]}, \lb (\psi^\sigma_s  )_{s \in \sigma}) \}_{\sigma \subset K}$ is PSC-conformal over every $s \in K_{PSC}$.
\item \label{lem_make_compatible_on_disk_vi} For all $s \in \sigma_0 \cap K_{PSC}$ the Riemannian manifold $(\psi^{\sigma_0}_s (Z^{\sigma_0} )\setminus \nu_s (\Int D^3(1)), \lb g^{\prime,s}_T)$ is PSC-conformal.
\end{enumerate}

Then letting 
\begin{equation}
\label{eqn_tilde_z_sigma_definition}
\td Z^\si:=\begin{cases}
Z^\si & \text{if $\si\neq\si_0$,}\\
 Z^{\si_0}\setminus \mu (\Int D^3(1)) &  \text{if $\si=\si_0$}\,
\end{cases}, \qquad \td\psi^\sigma_s := \psi^\sigma_s \big|_{\td Z^\sigma},
\end{equation} 
we can find continuous families $(\td{g}^\sigma_{s,t})_{s \in \sigma, t \in [0,1]}$ of metrics on $\td{Z}^\sigma$ such that:
\begin{enumerate}[label=(\alph*)]
\item $\{ ( \td Z^\sigma, \lb (\td g^\sigma_{s,t})_{s \in \sigma, t \in [0,1]}, \lb (\td\psi^\sigma_s  )_{s \in \sigma}) \}_{\sigma \subset K}$ is a partial homotopy at time $T$ relative to $L$.
\item $\{ ( \td Z^\sigma, \lb (\td g^\sigma_{s,t})_{s \in \sigma, t \in [0,1]}, \lb (\td\psi^\sigma_s  )_{s \in \sigma}) \}_{\sigma \subset K}$ is PSC-conformal  over every $s \in K_{PSC}$.
\end{enumerate}
\end{lemma}

\begin{proof}
In the following, we will construct the families of metrics $(\td{g}^\sigma_{s,t})$.
The construction will be performed locally on $\nu_s (D^3(2))$ and on $U$, meaning that the metrics $(\td\psi^\sigma_s)_* \td{g}^\sigma_{s,t}$ and $(\psi^\sigma_s)_* g^\sigma_{s,t}$ will (if at all) only differ on $\nu_s (D^3(2))$ if $s \in U$ and we will choose $\td{g}^\sigma_{s,t} = g^\sigma_{s,t}$ if $s \not\in U$.

Before proceeding, we observe that without loss of generality we may assume that each $g^\sigma_{s,t}$ is constant in $t$ for all $t \in [0,\frac12]$.
In fact, in our partial homotopy we may replace each $g^\sigma_{s,t}$ by
\[ \td{g}^{\sigma}_{s,t} := \begin{cases} g^\sigma_{s,0} & \text{if $t \in [0, \tfrac12]$} \\
 g^\sigma_{s, 2t - 1} & \text{if $t \in [\tfrac12,1]$} \end{cases} \]
and Definition~\ref{Def_partial_homotopy} will still be satisfied.
So assume that this is the case from now on.

In the first claim we will analyze the intersections of $\nu_s (D^3(2))$ with the images $\psi^\sigma_s (Z^\sigma)$.
We will find that $\nu_s (D^3(2))$ is either fully contained in this image or the intersection equals an annulus whose inner circle has radius bounded away from $0$.

\begin{Claim} \label{cl_2_mu}
After possibly shrinking the neighborhood $U$ of $\sigma_0$ we can find a small $r_0 \in (0,1)$ such that:
\begin{enumerate}[label=(\alph*)]
\item \label{cl_2_mu_a} $U$ is contained in the union of all simplices $\tau \subset K$ that intersect ${\sigma_0}$.
\item \label{cl_2_mu_b} For any $\tau \subset K$ there are two cases: 
\begin{enumerate}[label=(b\arabic*)]
\item \label{cl_2_mu_b1} $\nu_s (D^3(2)) \subset \psi^\tau_s (Z^\tau)$  for all $s \in \tau \cap U$ or 
\item \label{cl_2_mu_b2} $\nu_s (D^3(r_0)) \cap \psi^\tau_s (Z^\tau) = \emptyset$ for all $s \in \tau \cap U$.
In view of Assumption~\ref{lem_make_compatible_on_disk_iv} this implies that the complement 
$D^3(2) \setminus  \nu_s^{-1} ( \psi^\tau_s  ( \Int(Z^\tau)))$
is either empty or of the form $D^3(r)$ for some $r \in (r_0,2]$.
\end{enumerate}
\end{enumerate}
\end{Claim}

\begin{proof}
This follows by an openness argument and Assumption~\ref{lem_make_compatible_on_disk_iv} of the lemma.
Let $U = U_1 \supset U_2 \supset \ldots \supset \sigma_0$ a sequence of open subsets with $\cap_{i=1}^\infty U_i = \sigma_0$.
Then Assertion~\ref{cl_2_mu_a} holds if we replace $U$ by $U_i$ for large $i$.
We will now argue that Assertion~\ref{cl_2_mu_b} holds as well for large $i$.

Consider some simplex $\tau \subset K$.
If $\tau = \sigma_0$, then Property~\ref{cl_2_mu_b1} holds by Assumption~\ref{lem_make_compatible_on_disk_ii} of the lemma.
If $\tau \subsetneq  \sigma_0$, then $\psi^{\sigma_0}_s (Z^{\sigma_0}) \subset \psi^\tau_s (Z^\tau)$ and thus Property~\ref{cl_2_mu_b1} holds as well.
Assume now that $\tau \not\subset \sigma_0$, but $\tau \cap \sigma_0 \neq \emptyset$.
By Assumption~\ref{lem_make_compatible_on_disk_iv} of the lemma, for any $s \in \tau \cap U$ we have either $\nu_s (0) \not\in \psi^\tau_s (Z^\tau)$ or $\nu_s (D^3(2)) \subset \psi^\tau_s (Z^\tau)$.
Consider the set $S \subset \tau \cap U$ of parameters for which $\nu_s (D^3(2)) \subset \psi^\tau_s (Z^\tau)$.
Then $S$ is closed in $\tau \cap U$ by definition, but by Assumption~\ref{lem_make_compatible_on_disk_iv} of the lemma it is also open in $\tau \cap U$.
So there is a subset $\tau \cap \sigma_0 \subset U_\tau \subset \tau \cap U$ that is open in $\tau$ such that either $\nu_s (0) \not\in \psi^\tau_s (Z^\tau)$ or $\nu_s (D^3(2)) \subset \psi^\tau_s (Z^\tau)$ uniformly for all $s \in U_\tau$.
Thus Assertion~\ref{cl_2_mu_b} holds for $\tau$ if $i$ is large enough such that $U_i \cap \tau \subset U_\tau \cap \tau$ and $r_0$ is chosen small enough.
Since $K$ is finite, this implies Assertion~\ref{cl_2_mu_b} for large $i$.
\end{proof}

As mentioned before, our goal will be to modify the metrics $(\psi^\sigma_s)_* g^\sigma_{s,t}$ only on $\nu_s (D^3(2))$.
It will therefore be useful to consider the pullbacks $\nu_s^* (\psi^\sigma_s)_* g^\sigma_{s,t}$.
Since these pullbacks may only be defined on annular regions, it will be suitable later to construct a family of extensions to an entire disk $D^3(1.99) \subset D^3(2)$, which for technical reasons will have to be slightly smaller than $D^3(2)$.

\begin{Claim} \label{cl_3_h}
After possibly shrinking the neighborhood $U$ of $\sigma_0$, we can find a continuous family of metrics $(h_{s,t})_{s \in U, t \in [0,1]}$ on $D^3(1.99)$ such that the following holds for any $ \tau \subset K$, $s \in \tau \cap U$, $t \in [0,1]$:
\begin{enumerate}[label=(\alph*)]
\item \label{cl_3_h_a} $h_{s,t} = \nu^*_s ( \psi^\tau_s)_* g^\tau_{s,t}$ on $D^3(1.99) \cap \nu^{-1}_s ( \psi^\tau_s (Z^\tau))$.
\item \label{cl_3_h_b} The metric $h_{s,t}$ restricted to $$D^3(1.99) \setminus  \nu_s^{-1} \big( \psi^\tau_s  ( \Int Z^\tau) \big)$$ is compatible with the standard spherical structure on $D^3(1.99)$. 
\end{enumerate}
\end{Claim}

\begin{proof}
Let $\tau_1, \ldots, \tau_m \subset K$ be a list of all simplices of $K$, indexed in such a way that $\dim \tau_j$ is non-decreasing in $j$.
We will proceed by induction and construct a sequence $U = U_0 \supset U_1 \supset \ldots$ of open neighborhoods of $\sigma_0$ and a sequence of families $(h^j_{s,t})_{s \in U_j \cap \cup_{i=1}^j \tau_i, t \in [0,1]}$ such that Assertions~\ref{cl_3_h_a} and \ref{cl_3_h_b} hold if $s \in U_j \cap \cup_{i=1}^j \tau_i$.
Let $j \in \{ 1,\ldots, m \}$ and assume by induction that we have already constructed $U_{j-1}$ and $(h^{j-1}_{s,t})$.
Note that $(h^{j-1}_{s,t})$ is defined on $\partial \tau_j \cap U_{j-1}$.

Our goal will now be to possibly shrink $U_{j-1}$ and extend $(h^{j-1}_{s,t})$ over the interior of $\tau_j \cap U_{j}$ such that Assertions~\ref{cl_3_h_a} and \ref{cl_3_h_b} continue to hold.
We first claim that this is the case if for all $s \in \tau_j \cap U_j$ and $t \in [0,1]$ we have:
\begin{enumerate}[label=(\arabic*)]
\item \label{prop_h_construction_1} $h^j_{s,t} = \nu^*_s ( \psi^{\tau_j}_s)_* g^{\tau_j}_{s,t}$ on $D^3(1.99) \cap \nu^{-1}_s ( \psi^{\tau_j}_s (Z^{\tau_j}))$.
\item \label{prop_h_construction_2} $h^j_{s,t}$ restricted to $D^3(1.99) \setminus  \nu_s^{-1} ( \psi^{\tau_j}_s  ( \Int Z^{\tau_j}) )$ is compatible with the standard spherical structure on $D^3(1.99)$.
\item \label{prop_h_construction_3} $h^j_{s,t} = h^{j-1}_{s,t}$ if $s \in \partial \tau_j \cap U_j$.
\end{enumerate}
In fact, if $\tau \subset K$ with $\tau_j \not\subset \tau$, then $\tau$ is disjoint from the interior of $\tau_j$ and Assertions~\ref{cl_3_h_a} and \ref{cl_3_h_b} hold by induction.
If $\tau = \tau_j$, then Assertions~\ref{cl_3_h_a} and \ref{cl_3_h_b} trivially follow from Properties~\ref{prop_h_construction_1} and \ref{prop_h_construction_2} above.
If $\tau_j \subsetneq \tau$, then by the definition of a partial homotopy we have $\psi^\tau_s (Z^\tau) \subset \psi^{\tau_j}_s (Z^{\tau_j})$ and $(\psi^\tau_s)_* g^\tau_{s,t} = (\psi^{\tau_j}_s)_* g^{\tau_j}_{s,t}$ on $\psi^\tau_s (Z^\tau)$.
So Property~\ref{prop_h_construction_1} implies Assertion~\ref{cl_3_h_a}.
Consider now Assertion~\ref{cl_3_h_b}.
Due to Assumption~\ref{lem_make_compatible_on_disk_iii} of the lemma and Property~\ref{prop_h_construction_2} above it suffices to show that on $\nu_s (D^3(1.99)) \cap (\psi^{\tau_j}_s (Z^{\tau_j}) \setminus \psi^\tau_s (Z^\tau))$ the metrics $(\nu_s)_* h^j_{s,t} = (\psi^{\tau_j}_s)_* g^{\tau_j}_{s,t}$, $t \in [0,1]$, are compatible with $\SS^s$.
For this purpose let $\C$ be the closure of a component of $Z^{\tau_j} \setminus ((\psi^{\tau_j}_s)^{-1} \circ \psi^\tau_s )( Z^\tau)$ such that $\nu_s (D^3(1.99))\cap \psi^{\tau_j}_s (\C) \neq \emptyset$.
By Assumption~\ref{lem_make_compatible_on_disk_vii} the image  $\psi^{\tau_j}_s (\C)$ is a union of spherical fibers and $ (\psi^{\tau_j}_s)_* g^{\tau_j}_{s,t}$ restricted to $\psi^{\tau_j}_s (\C)$ is compatible with the spherical structure.
This finishes the proof of Assertion~\ref{cl_3_h_b}.

It remains to choose $U_j$ and $(h^j_{s,t})$ satisfying Properties~\ref{prop_h_construction_1}--\ref{prop_h_construction_3} above.
If $\tau_j$ satisfies Case~\ref{cl_2_mu_b1} of Claim~\ref{cl_2_mu}, then we can simply set 
\[  h^j_{s,t} := \nu_s^* (\psi^{\tau_j}_s)_* g^{\tau_j}_{s,t}. \]
So assume that $\tau_j$ satisfies Case~\ref{cl_2_mu_b2}.
By the remark in Claim~\ref{cl_2_mu}\ref{cl_2_mu_b2}, there is a continuous function $r_j : U_{j-1} \cap \tau_j \to (r_0, 2]$ such that for all $s \in U_{j-1} \cap \tau_j$
\[ B^3 (2) \cap  \nu_s^{-1} (\psi^{\tau_j}_s (Z^{\tau_j}))  = B^3 (2) \cap \ov{A^3 (r_j (s), 2)} . \]
Note that in the case $r_j(s) = 2$ this set is empty.
Next observe that it suffices to construct $h^j_{s,t}$ for $s$ in a neighborhood $V_{s_0} \subset \tau_j$ of any parameter $s_0 \in \tau_j \cap \sigma_0$.
The desired family can then be constructed using a partition of unity.
So fix some $s_0 \in \tau_j \cap \sigma_0$.
If $r_j (s_0) > 1.99$, then Property~\ref{prop_h_construction_1} is vacuous in a neighborhood of $s_0$ and Properties~\ref{prop_h_construction_2}, \ref{prop_h_construction_3} can be satisfied by extending $(h^{j-1}_{s,t})$ by an arbitrary family of metrics on $D^3(1.99)$ that are compatible with the standard spherical structure.
If $r_j (s_0) \leq 1.99$, then $r_j < 2$ in a neighborhood of $s_0$ in $\tau_j$ and Properties~\ref{prop_h_construction_1}--\ref{prop_h_construction_3} can be satisfied by extending $\nu_s^* (\psi^{\tau_j}_s)_* g^{\tau_j}_{s,t}$ onto $D^3(1.99)$ as in the proof of Proposition~\ref{prop_extending_symmetric}.
This finishes the proof of the claim.
\end{proof}

We will now mainly work with the family $(h_{s,t})$.
In the next step we apply a rounding procedure at the origin.
The resulting family $(h'_{s,t})$ will be compatible with the standard spherical structure on a small disk $D^3(r_1) \subset D^3 (1.99)$.

\begin{Claim} \label{cl4_hprime}
Let $r_0$ be the constant from Claim~\ref{cl_2_mu}.
We can find a smaller open neighborhood $U'  \Subset U$ of ${\sigma_0}$ and a continuous family of metrics $(h'_{s,t})_{s \in U, t \in [0,1]}$ on $D^3(1.99)$, such that the following holds for some $r_1 \in (0,r_0/2)$:
\begin{enumerate}[label=(\alph*)]
\item \label{cl4_hprime_a} For all $s \in U'$ and $t \in [\frac12,1]$ the metric $h'_{s,t}$ is compatible with the standard spherical structure on $D^3 (r_1)$.
\item \label{cl4_hprime_b} $h'_{s,0} = h_{s,0}$ for all $s \in U$.
\item \label{cl4_hprime_c} $h'_{s,t} = h_{s,t}$ for all $s \in U \setminus U'$ and $t \in [0,1]$.
\item \label{cl4_hprime_d} $h'_{s,t} = h_{s,t}$ on $A^3 (r_0, 1.99)$ for all $s \in U$ and $t \in [0,1]$.
\item \label{cl4_hprime_e} $h'_{s,1}$ is conformally flat for all $s \in U$.
\item \label{cl4_hprime_f} If for some $s \in U$, $t \in [0,1]$ and $r\in [r_0,1.99]$ the metric $h_{s,t}$ is compatible with the standard spherical structure on $D^3(r)$, then so is $h'_{s,t}$.
\item \label{cl4_hprime_g} For any $s \in U$, $s \in \sigma \subset K$ and $t \in [0,1]$ the following holds: 
Consider the metric
\begin{equation} \label{eq_k_definition_cases}
 k^\sigma_{s,t} := \begin{cases} (\nu_s)_* h'_{s,t} & \text{on $\nu_s (D^3(1.99))$} \\ (\psi_{s}^\sigma)_* g^\sigma_{s,t} & \text{on $\psi^\sigma_s (Z^\sigma) \setminus \nu_s (D^3(1.99))$} \end{cases} 
\end{equation}
If $s \in K_{PSC}$ or $t = 1$, then $(\psi^\sigma_s (Z^\sigma), k^\sigma_{s,t})$ is PSC-conformal.
\end{enumerate}
\end{Claim}

\begin{proof} 
Choose $\delta_1 \in C^0_c (U)$, $0 \leq \delta_1 \leq 1$ such that $\delta_1 \equiv 1$ on a neighborhood $U'$ of $\si_0$ and choose $\delta_2 \in C^0 ([0,1])$ such that $\delta_2 (0) = 0$ and $\delta_2 \equiv 1$ on $[\frac12, 1]$.
Let $(h'_{s,t})_{(s,t) \in \supp (\delta_1) \times [0,1]}$ be the result of applying Proposition~\ref{Prop_rounding} with $u = \delta_1 (s) \delta_2 (t)$, $\ov r_1 = r_0$ (from Claim~\ref{cl_2_mu}), $X = \supp (\delta_1) \times [0,1]$ and 
\begin{equation} \label{eq_XPSC_removing}
 X_{PSC} = \big( (\supp(\delta_1) \cap K_{PSC}) \times [0,1] \big) \cup \big( \supp(\delta_1) \times \{ 1 \} \big) 
\end{equation}
on $D^3 (1.99)$.
Let $r_1 \in (0, r_0)$ be the constant produced by Proposition~\ref{Prop_rounding}.
By Proposition~\ref{Prop_rounding}\ref{ass_Prop_rounding_b}, we can extend $(h'_{s,t})_{(s,t) \in \supp (\delta_1) \times [0,1]}$ to a continuous family over $U \times [0,1]$ by setting $h'_{s,t} := h_{s,t}$ whenever $s \not\in \supp (\delta_1)$.

Proposition~\ref{Prop_rounding}\ref{ass_Prop_rounding_c} implies Assertion~\ref{cl4_hprime_a}, because $\delta_1 (s) \delta_2 (t) = 1$ for $s \in U'$ and $t \in [\frac12, 1]$.
Assertions~\ref{cl4_hprime_b}--\ref{cl4_hprime_d} follow immediately from Proposition~\ref{Prop_rounding}\ref{ass_Prop_rounding_a}, \ref{ass_Prop_rounding_c}.
For Assertion~\ref{cl4_hprime_e} notice that by Claim~\ref{cl_3_h}, the metric $h_{s,1}$ is locally either conformally flat or compatible with the standard spherical structure, and therefore conformally flat.
By Proposition~\ref{Prop_rounding}\ref{ass_Prop_rounding_e} the rounding procedure retains this property.
Assertion~\ref{cl4_hprime_f} is a restatement of Proposition~\ref{Prop_rounding}\ref{ass_Prop_rounding_d}.

Assertion ~\ref{cl4_hprime_g} holds by assumption if $\psi^\sigma_s(Z^\sigma)\cap \nu_s(D^3(r_0))=\emptyset$ since in this case $k^\sigma_{s,t}=(\psi^\sigma_{s})_*g^\sigma_{s,t}$ (see Assertion~\ref{cl4_hprime_d}).
Hence by Claim~\ref{cl_2_mu} we may assume $\psi^\sigma_{s}(Z^\sigma)\supset \nu_s(D^3(2))$ for all $s \in \sigma$, and in this case PSC-conformality follows from Proposition~\ref{Prop_rounding}\ref{ass_Prop_rounding_f}.  
To see this, note that by Assumption~\ref{lem_make_compatible_on_disk_v}, for every $(s,t) \in X_{PSC}$ from (\ref{eq_XPSC_removing}) we know that $(Z^\sigma, g^\sigma_{s,t})$ is PSC-conformal, so there exists a function $\wh w_s\in C^\infty(Z^\sigma)$ satisfying the conditions in Lemma~\ref{Lem_PSC_conformal_analytic}.  
Hence by Lemma~\ref{Lem_PSC_conformal_open} and a partition of unity argument we may assume that $\{\wh w_s\}_{s\in X_{PSC}}$ is a continuous family of smooth functions.  
Setting $w_s:=\wh w_s\circ\mu$, we apply Proposition~\ref{Prop_rounding}\ref{ass_Prop_rounding_f}, and let $w'_{s,t}$ be the resulting functions.  
Now letting 
\[ \wh w'_{s,t} := \begin{cases} w'_{s,t}\circ \nu_s^{-1} & \text{on $\nu_s (B^3(1.99))$} \\ \wh w_{s,t} \circ (\psi^\sigma_s)^{-1} & \text{on $\psi^\sigma_s (Z^\sigma) \setminus \nu_s (B^3(1.99))$} \end{cases} \]
we see that $\wh w_{s,t}'$ satisfies the conditions in Lemma~\ref{Lem_PSC_conformal_analytic} for $(\psi^\sigma_s(Z^\sigma),k^\sigma_{s,t})$.
\end{proof}

Next, we stretch the metrics $h'_{s,t}$ radially.

\begin{Claim} \label{cl_5_hprimeprime}
We can find a continuous family of metrics $(h''_{s,t})_{s \in U, t \in [0,1]}$ on $D^3(1.99)$, such that the following holds:
\begin{enumerate}[label=(\alph*)]
\item \label{cl_5_hprimeprime_a} For all $s \in {\sigma_0}$ the metric $h''_{s,t}$ is compatible with the standard spherical structure on $D^3 (1.1)$.
\item \label{cl_5_hprimeprime_b} $h''_{s,0} = h_{s,0}$ for all $s \in U$.
\item \label{cl_5_hprimeprime_c} $h''_{s,t} = h_{s,t}$ for all $s \in U \setminus U'$ and $t \in [0,1]$.
\item \label{cl_5_hprimeprime_d} $h''_{s,t} = h_{s,t}$ on $A^3 (1.98, 1.99)$ for all $s \in U$ and $t \in [0,1]$.
\item \label{cl_5_hprimeprime_e} $h''_{s,1}$ is  conformally flat for all $s \in U$.
\item \label{cl_5_hprimeprime_f}   If some $s \in U$, $t \in [0,1]$ and $r \in [r_0, 1.99]$ the metric $h_{s,t}$ is compatible with the standard spherical structure on $D^3(r)$, then so is $h''_{s,t}$.
\item \label{cl_5_hprimeprime_g} For any $s \in  U$, $s \in \sigma \subset K$ and $t \in [0,1]$ the following holds: 
Consider the metric
\[ \td{k}^\sigma_{s,t} := \begin{cases} (\nu_s)_* h''_{s,t} & \text{on $\nu_s (D^3(1.99))$} \\ (\psi_{s,t}^\sigma)_* g^\sigma_{s,t} & \text{on $\psi^\sigma_s (Z^\sigma) \setminus \nu_s (D^3(1.99))$} \end{cases} \]
restricted to $\psi^\sigma_s(\td Z^\sigma)$, where $\td Z^\sigma$ is as in (\ref{eqn_tilde_z_sigma_definition}).  
If $s \in K_{PSC}$ or $t = 1$, then $(\td\psi^\sigma_s (\td Z^\sigma), \td k^\sigma_{s,t})$ is PSC-conformal.
\end{enumerate}
\end{Claim}

\begin{proof} 
Fix a continuous family of diffeomorphisms 
\[ (\Phi_u : D^3(1.99) \to D^3(1.99))_{u \in (0,1]} \] 
with the following properties: 
\begin{enumerate}
[label=(\Alph*)]
\item \label{item_phi_properties_1} $\Phi_u (x) = f(x,u) x$ for some scalar function $f : D^3(1.99) \times (0,1] \to (0,1]$.
\item \label{item_phi_properties_2} $\Phi_u = \id$ on $A^3 (1.98, 1.99)$.
\item $\Phi_1 = \id$.
\item For any $x \in D^3 (1.97)$ we have $|\Phi_u (x)| < u$.
\end{enumerate}
Fix a continuous function $\delta_1 \in C_c^0 (U')$ with $0 \leq \delta_1 \leq 1$ and support in $U'$ such that $\delta_1 \equiv 1$ on ${\sigma_0}$ and a continuous function $\delta_2 \in C^0 ([0,1])$ with $0 \leq \delta_2 \leq 1$ such that $\delta_2 (0) = 0$ and $\delta_2 \equiv 1$ on $[\frac12, 1]$.
Let $r_2\in(0,r_1/2]$ be a constant whose value we will determine later and set
\[ h''_{s,t} := \Phi^*_{1- (1-r_2) \delta_1(s) \delta_2(t)} h'_{s,t}. \]

Let us first prove Assertion~\ref{cl_5_hprimeprime_a}.
Let $s \in \sigma_0$ and $t \in [0,1]$.
If $t \in [\frac12, 1]$, then we have $h''_{s,t} = \Phi^*_{r_2} h'_{s,t}$ and $\Phi_{r_2} ( D^3(1.1) ) \subset D^3(r_1)$.
So by Claim~\ref{cl4_hprime}\ref{cl4_hprime_a}, the metric $h''_{s,t}$ is compatible with the standard spherical structure on $D^3(1.1)$.
On the other hand, if $t \in [0, \frac12]$, then by Claim~\ref{cl_3_h}\ref{cl_3_h_a} and the fact that $g^{\sigma_0}_{s,t}=g^{\sigma_0}_{s,0}$ for all $t\in [0,\frac12]$ we have
\begin{equation*} \label{eq_h_p_h_nu_psi_g}
  h_{s,t} = \nu^*_s (\psi^{\sigma_0}_s)_* g^{\sigma_0}_{s,t} = \nu^*_s (\psi^{\sigma_0}_s)_* g^{\sigma_0}_{s,0} = \nu^*_s g^{\prime,s}_T. 
\end{equation*}
So by Assumption~\ref{lem_make_compatible_on_disk_iii} the metric $h_{s,t}$ is compatible with the standard spherical structure on $D^3 (1.99)$.
By Claim~\ref{cl4_hprime}\ref{cl4_hprime_f} the same is true for $h'_{s,t}$.

For Assertion~\ref{cl_5_hprimeprime_b}, observe that by Claim~\ref{cl4_hprime}\ref{cl4_hprime_b} for all $s \in U$
\[ h''_{s,0} = \Phi^*_1 h'_{s,0} = h'_{s,0} = h_{s,0}. \]
Similarly, for Assertion~\ref{cl_5_hprimeprime_c}, we have by Claim~\ref{cl4_hprime}\ref{cl4_hprime_c} for all $s \in U \setminus U'$, $t \in [0,1]$
$$ h''_{s,t} = \Phi^*_1 h'_{s,t} = h'_{s,t} = h_{s,t}. $$
Assertion~\ref{cl_5_hprimeprime_d} follows from  Property \ref{item_phi_properties_2} along with Claim~\ref{cl4_hprime}\ref{cl4_hprime_d}.
Assertion~\ref{cl_5_hprimeprime_e} follows from Claim~\ref{cl4_hprime}\ref{cl4_hprime_e}.

Assertion~\ref{cl_5_hprimeprime_f} follows from Claim~\ref{cl4_hprime}\ref{cl4_hprime_f} and Property \ref{item_phi_properties_1} above.
More specifically, if $h_{s,t}$ is compatible with the standard spherical structure on $D^3(r)$, then $h'_{s,t}$ restricted to $D^3(r)$ is as well and therefore, $h''_{s,t}$ is compatible with the standard spherical structure on $\Phi^{-1}_{1- (1-r_2) \delta_1(s) \delta_2(t)} ( D^3(r)) \supset D^3(r)$.

Lastly, consider Assertion~\ref{cl_5_hprimeprime_g}.

First suppose that $\sigma\neq\si_0$ or $t \in [0,\frac12]$.
We claim that in this case
\begin{equation} \label{eq_td_Z_sigma_PSC_conf}
(\td{Z}^\sigma,  g^\sigma_{s,t} |_{\td{Z}^\sigma}) \quad \text{is PSC-conformal if $s \in \sigma \cap K_{PSC}$ or $t =1$.}
\end{equation}
In fact, if $\sigma \neq \sigma_0$, then (\ref{eq_td_Z_sigma_PSC_conf}) follows from Assumption~\ref{lem_make_compatible_on_disk_v}, Definition~\ref{Def_metric_deformation} and the fact that $\td{Z}^\sigma = Z^\sigma$.
On the other hand, if $\sigma = \sigma_0$, $s \in \sigma_0 \cap K_{PSC}$ and $t \in [0,\frac12]$, then $g^{\sigma_0}_{s,t} = g^{\sigma_0}_{s,0} = (\psi^{\sigma_0}_s)^* g^{\prime,s}_T$.
So $(\td{Z}^{\sigma_0},  g^{\sigma_0}_{s,t} |_{\td{Z}^\sigma})$ is isometric to $(\td\psi^{\sigma_0} (\td{Z}^{\sigma_0}), g^{\prime,s}_T)$, which is PSC-conformal by Assumption~\ref{lem_make_compatible_on_disk_vi}.

Let us now continue with the proof of Assertion~\ref{cl_5_hprimeprime_g} if $\sigma\neq\si_0$ or $t \in [0,\frac12]$.
If $\nu_s (D^3 (1.99)) \subset \psi^\sigma_s (\td Z^\sigma)$ (which precludes $\sigma = \sigma_0$), then we are done by Claim~\ref{cl4_hprime}\ref{cl4_hprime_g}, because the metrics $h'_{s,t}$ and $h''_{s,t}$ are isometric to one another, which implies that the extensions of $(\nu_s)_* h'_{s,t}$ and $(\nu_s)_* h''_{s,t}$ by $(\psi_{s}^\sigma)_* g^\sigma_{s,t}$ onto $\psi^\sigma_s (\td Z^\sigma)$ are isometric.
The same is true if $\nu_s (D^3 (1.99))$, $\psi^\sigma_s (\td Z^\sigma)$ are disjoint.
If $\nu_s (D^3 (1.99)) \not\subset \psi^\sigma_s (\td Z^\sigma)$ and both subsets are not disjoint, then by Claim~\ref{cl_2_mu}\ref{cl_2_mu_b2} or the definition of $\td Z^{\sigma_0}$
\[ D^3 (1.99) \cap  \nu_s^{-1} (\psi^\sigma_s (\td Z^\sigma))  = \ov{A^3 (r,1.99 )} \]
for some $r \in (r_0, 1.99]$.
In this case, the metric $h''_{s,t}$ restricted to 
\[ \Phi^{-1}_{1- (1-r_2) \delta_1(s) \delta_2(t)} (  \ov{A^3 (r, 1.99)} ) \subset \ov{A^3 (r, 1.99)} \]
is isometric to $h'_{s,t}$ restricted to $ \ov{A^3 (r, 1.99)}$.
By Claim~\ref{cl4_hprime}\ref{cl4_hprime_d} and Claim~\ref{cl_3_h}\ref{cl_3_h_a} we have $h'_{s,t} = h_{s,t} = \nu^*_s (\psi^\sigma_s)_* g^{\sigma}_{s,t}$ on $ \ov{A^3 (r, 1.99)}$.
It follows that $\td{k}^\sigma_{s,t}$ restricted to
\[  \big( \psi^\sigma_s (\td Z^\sigma) \setminus \nu_s (D^3(1.99)) \big) \cup \nu_s \big( \Phi^{-1}_{1- (1-r_2) \delta_1(s) \delta_2(t)} (  \ov{A^3 (r, 1.99)} ) \big)   \]
is isometric to $(\td Z^\sigma, g^\sigma_{s,t})$, which is PSC-conformal if $s \in \sigma \cap K_{PSC}$ or $t=1$, due to (\ref{eq_td_Z_sigma_PSC_conf}).
On the other hand, $\td{k}^\sigma_{s,t}$ restricted to the closure of 
\begin{equation} \label{eq_A_Phi_A}
 \nu_s \big( \ov{A^3 (r, 1.99)} \setminus \Phi^{-1}_{1- (1-r_2) \delta_1(s) \delta_2(t)} (  \ov{A^3 (r, 1.99)} ) \big) 
\end{equation}
is isometric to $h'_{s,t}$ restricted to 
\[ \Phi_{1- (1-r_2) \delta_1(s) \delta_2(t)} (\ov{A^3 (r, 1.99)}) \setminus   \ov{A^3 (r, 1.99)}  \subset D^3 (r). \]
By Claim~\ref{cl_3_h}\ref{cl_3_h_b}, we know that $h_{s,t}$ is compatible with the standard spherical structure on $D^3(r)$ and thus by Claim~\ref{cl4_hprime}\ref{cl4_hprime_f} the same is true for $h'_{s,t}$ (recall that $r \geq r_0$).
It follows that $\td{k}^\sigma_{s,t}$ restricted to the closure of (\ref{eq_A_Phi_A}) is compatible with $\SS^s$.
Thus by Lemma~\ref{Lem_PSC_conformal_enlarge} we conclude that $(\psi^\sigma_s (\td Z^\sigma), \td{k}^\sigma_{s,t})$ is PSC-conformal if $s \in \sigma \cap K_{PSC}$ or $t=1$.

Now suppose $\sigma=\si_0$ and $t \in [\frac12,1]$.
Consider the family of metrics $( k^{\si_0}_{s,t})$ on $\psi^{\si_0}_s(Z^{\si_0})$ from (\ref{eq_k_definition_cases}).
By Claim~\ref{cl4_hprime}\ref{cl4_hprime_g} we know that $(Z^{\si_0},(\psi^{\si_0}_s)^* k^{\si_0}_{s,t})$ is PSC-conformal for all $(s,t) \in (\sigma_0 \cap K_{PSC}) \times [\frac12,1] \cup \sigma_0 \times \{ 1 \}$.
By Lemma~\ref{lem_thick_annulus_psc_conformal}, we may choose $r_2\in (0,r_1)$ such that $(\psi^{\si_0}_s)^* k^{\si_0}_{s,t}$ is also PSC-conformal on $Z^{\si_0} \setminus\mu (\Int D^3(r_2))$ for the same $(s,t)$. 
Assertion~\ref{cl_5_hprimeprime_g} now follows from the fact that $( \psi^{\si_0}_s(\td Z^{\si_0}) = \psi^{\si_0}_s(Z^{\si_0})\setminus\nu_s(\Int D^3(1)), \td{k}^{\sigma_0}_{s,t})$ is isometric to $k^{\sigma_0}_{s,t}$ restricted to $\psi^{\sigma_0}_s (Z^{\si_0} \setminus\mu (\Int D^3(r_2)))$.
\end{proof}

\bigskip

For every $s \in \sigma \subset K$ and $t \in [0,1]$, we can now define
\[ \td{g}^\sigma_{s,t} := \begin{cases} g^\sigma_{s,t} & \text{on $Z^\sigma \setminus (\psi^\sigma_s)^{-1} (\nu_s (D^3 (1.99)))$ if $s \in U$ or on $Z^\sigma$ if $s \not\in U$} \\ 
(\psi^\sigma_s)^* (\nu_s)_* h''_{s,t} & \text{on $(\psi^\sigma_s)^{-1} (\nu_s (D^3 (1.99)))$ if $s \in U$} \end{cases} \]

To complete the proof of Lemma~\ref{lem_make_compatible_on_disk}, we now verify that $\{ (\td Z^\si, \lb (\td{g}^\sigma_{s,t})_{s \in \sigma, t \in [0,1]}, \lb (\td\psi^\sigma_s)_{s \in \sigma}) \}_{\sigma \subset K}$ is a partial homotopy at time $T$ relative to $L$ that is PSC-conformal over all $s \in K_{PSC}$.

By Claim~\ref{cl_3_h}\ref{cl_3_h_a} and Claim~\ref{cl_5_hprimeprime}\ref{cl_5_hprimeprime_c}, \ref{cl_5_hprimeprime_d}, the metrics $\td{g}^\sigma_{s,t}$ are smooth and depend continuously on $s, t$.
By Claim~\ref{cl_5_hprimeprime}\ref{cl_5_hprimeprime_e}, \ref{cl_5_hprimeprime_g}, $(\td{Z}^\sigma, \td{g}^\sigma_{s,1})$ is conformally flat and PSC-conformal.
So $(\td Z^\sigma, (\td{g}^\sigma_{s,t}))$ are metric deformations.

By Claim~\ref{cl_5_hprimeprime}\ref{cl_5_hprimeprime_g} the Riemannian manifold $(\td Z^\sigma, \td g^\sigma_{s,t})$ is PSC-conformal for all $t \in [0,1]$ and $s \in \sigma \subset K$ if $s \in K_{PSC} \cap U$.
If $s \in K_{PSC} \setminus U$, then $(\td Z^\sigma, \td g^\sigma_{s,t})=( Z^\sigma,  g^\sigma_{s,t})$ is PSC-conformal by Assumption~\ref{lem_make_compatible_on_disk_v}.

Let us now verify the properties of Definition~\ref{Def_partial_homotopy}.

By Claim~\ref{cl_5_hprimeprime}\ref{cl_5_hprimeprime_b} and Claim~\ref{cl_3_h}\ref{cl_3_h_a} we have $(\psi^\sigma_s)^* g^{\prime,s}_T = g^\sigma_{s,0} = \td{g}^\sigma_{s,0}$.
This verifies Property~\ref{prop_def_partial_homotopy_1}.
Property~\ref{prop_def_partial_homotopy_2} does not concern $\td{g}^\sigma_{s,t}$, so it remains true.
Property~\ref{prop_def_partial_homotopy_3} is immediate from the definition of $\td{g}^\sigma_{s,t}$.
Property~\ref{prop_def_partial_homotopy_6} remains unchanged.
It remains to verify Properties~\ref{prop_def_partial_homotopy_4}, \ref{prop_def_partial_homotopy_5} in Definition~\ref{Def_partial_homotopy}.

For Property~\ref{prop_def_partial_homotopy_4} consider the closure $\C$ of a component of $Z^{ \tau}  \setminus((\psi^{ \tau}_s )^{-1} \circ \psi^{ \sigma}_s ) ( Z^{ \sigma} )$, for some $\tau\subsetneq \si$, $s\in \tau\cap U$.
We may assume that $\psi^\tau_s (\C) \cap \nu_s (D^3 (1.99)) \neq \emptyset$, because otherwise the property holds trivially.  
By Claim~\ref{cl_2_mu}\ref{cl_2_mu_b} we have $\psi^\tau_s (\C) \cap \nu_s (D^3 (1.99))  = \nu_s ( \ov {A^3(r_\tau,r_\sigma)})$ for some   $r_\sigma \in (r_0, 2]$, $r_\tau\in \{0\}\cup (r_0,r_\sigma)$.
Thus by Assumption~\ref{lem_make_compatible_on_disk_vii}, Claim \ref{cl_3_h} and Claim~\ref{cl_5_hprimeprime}\ref{cl_5_hprimeprime_f} we have that $(\psi^\tau_s)_* \td{g}^\tau_{s,t}$ is compatible with $\SS^s$ on $\psi^\tau_s (\C)$.

Lastly, consider Property~\ref{prop_def_partial_homotopy_5}.
Let $\sigma \subset K$, $s \in \sigma \cap U$ and consider a boundary component $\Sigma \subset \partial Z^\sigma$ with $\psi^\sigma_s (\Sigma) \subset \nu_s (D^3 (1.99))$.
By Claim~\ref{cl_3_h} we have that $h_{s,t}$ restricted to a neighborhood of the disk bounded by $\nu_s^{-1} (\Sigma)$ is compatible with the standard spherical structure.
Thus by Claim~\ref{cl_5_hprimeprime}\ref{cl_5_hprimeprime_f}, so is $h''_{s,t}$, which proves Property~\ref{prop_def_partial_homotopy_5}.
\end{proof}

\section{Deforming families of metrics towards families of conformally flat metrics}\label{sec_deforming_families_metrics}
\subsection{Statement of the main result and setup} \label{subsec_deforming_main_results}
Similarly as in Subsection~\ref{subsec_gen_setup} we fix a pair $(K,L)$ of topological spaces that is homeomorphic to the the geometric realization of a pair of finite simplicial complexes $(\mathcal{K}, \mathcal{L})$ where $\mathcal{L} \subset \mathcal{K}$ is a subcomplex.
We will mostly refer to the pair $(K,L)$ instead of $(\mathcal{K}, \mathcal{L})$ if there is no chance of confusion.

In this section we will show the following theorem.

\begin{theorem} \label{Thm_main_deform_to_CF}
Consider a continuous family $(M^s, g^s)_{s \in K}$ of Riemannian manifolds.
Suppose that $M^s$ is diffeomorphic to a connected sum of spherical space forms and copies of $S^2 \times S^1$ for all $s \in K$ and that $(M^s, g^s)$ is a CC-metric for all $s \in L$.
Let $K_{PSC} \subset K$ be a closed subset with the property that $(M^s,g^s)$ has positive scalar curvature for all $s \in K_{PSC}$.
Then there is a continuous family of Riemannian metrics $(h^s_t)_{s \in K, t \in [0,1]}$ on $(M^s)_{s \in K}$ such that for all $s \in K$:
\begin{enumerate}[label=(\alph*)]
\item $h^s_0 = g^s$.
\item $h^s_1$ is conformally flat and PSC-conformal.
\item If $s \in L$, then $h^s_t$ is a CC-metric for all $t \in [0,1]$.
\item If $s \in K_{PSC}$, then $(M^s, h^s_t)$ is PSC-conformal for all $t \in [0,1]$.
\end{enumerate}
\end{theorem}

We will now reduce Theorem~\ref{Thm_main_deform_to_CF} to Lemma~\ref{Lem_main_existence_partial_homotopy} below, which concerns the existence of certain partial homotopies.
By Theorem~\ref{thm_existence_family_k} there is a continuous family of singular Ricci flows $(\M^s)_{s \in K}$ with initial condition $(M^s, g^s)_{s \in K}$.
We may identify $(\M^s_0, g^s_0) = (M^s, g^s)$ for all $s \in K$.
By uniqueness, for all $s \in L$  all time-slices of $\M^s$ are CC-metrics. Moreover, by Theorem~\ref{Thm_PSC_preservation} the flow $\M^s$ has positive scalar curvature for all $s \in K_{PSC}$.
By Theorem~\ref{Thm_extinction_time} there is a uniform time $T_0 < \infty$ at which these flows become extinct, i.e. $\M^s_t = \emptyset$ for all $t \geq T_0$ and $s \in K$.

Next, we invoke Theorem~\ref{Thm_rounding} for some $\delta > 0$, which we will choose later.
We obtain a transversely continuous family of $\RR$-structures 
\[ \RR^s = ( g^{\prime, s}, \partial^{\prime, s}_\t, U^s_{S2}, U^s_{S3}, \SS^s ) \]
on $(\M^s)_{s \in K}$.
Recall from Theorem~\ref{Thm_rounding} that
\[ U_{S2}^s \cup U_{S3}^s = \big\{ x \in \M^s \;\; : \;\; \rho_{g^{\prime, s}} (x) < r_{\rot, \delta} (r_{\initial} (M^s, g^s) , \t(x)) \big\}. \]
Due to the uniform extinction time $T_0$ and Assertion~\ref{ass_thm_rounding_e} of Theorem~\ref{Thm_rounding}, we can multiply the metrics $g^s$ with a large constant and assume without loss of generality that
\begin{equation} \label{eq_rs_bigger_C0}
 r_{\initial} (M^s, g^s), \; r_{\can, \delta} (r_{\initial} (M^s, g^s) , t), \; r_{\rot, \delta} (r_{\initial} (M^s, g^s) , t) > C_0 
\end{equation}
 for all $s \in K$ and $t \geq 0$ for which $\M^s_t \neq \emptyset$, where $10 < C_0 < \infty$ is a constant that we will choose later.
Therefore, we have 
\begin{equation} \label{eq_US2US310}
 \{ \rho_{g^{\prime, s}} < C_0 \}  \subset U^s_{S2} \cup U^s_{S3} \qquad \text{for all} \quad s \in K. 
\end{equation}
In this section we will exclusively work with the objects $g^{\prime, s}, \partial^{\prime,s}_\t$ instead of $g^s, \partial^s_\t$ and we will often omit the index in expressions of the form ``$\rho_{g^{\prime, s}}$''.
The following lemma summarizes all further properties of $g^{\prime,s}$ and $\partial^{\prime,s}_t$ that we will use in this section.

Fix an arbitrary constant $\Lambda > 100$ for the remainder of this section.

\begin{lemma} \label{Lem_further_properties_gpdtp}
If $C_0 \geq \underline{C}_0 (\Lambda)$ and $\delta \leq \ov\delta (\Lambda)$, then:
\begin{enumerate}[label=(\alph*)]
\item \label{ass_further_properties_gpdtp_a} $g^{\prime,s}_{0}  = g^s$ for all $s \in K$.
\item \label{ass_further_properties_gpdtp_aa} $\rho > 1$ on $\M^s_0$ for all $s \in K$.
\item \label{ass_further_properties_gpdtp_b} There is some $T_{\ext} < \infty$ such that $\M^s_{t} = \emptyset$ for all $t \geq T_{\ext}$ and $s \in K$.
\item \label{ass_further_properties_gpdtp_c} For any $r > 0$ the restriction of $\pi : \cup_{s \in K} \M^s \to K$ to $\{ \rho \geq r \}$ is proper.
\item \label{ass_further_properties_gpdtp_d} There is a constant $\theta = \theta (r) \in (0, r^2]$ such that for any $s \in K$, $t_1, t_2 \geq 0$ with $|t_1 - t_2| \leq \theta$ the following is true.
If $x \in \M^s_{t_2}$ with $\rho (x) > r / 10$, then the point $x( t_1 )$ is defined and we have:
\[ |\rho(x) - \rho(x(t_1))| < 10^{-3} \rho(x) \]
\item \label{ass_further_properties_gpdtp_e} If in Assertion~\ref{ass_further_properties_gpdtp_c} we have $t_1 \leq t_2$ and $\rho (x (t_1) ) \leq \rho (x) \leq 10$, then there are embedded disks $D' \subset D \subset U^s_{S2} \cap \M^s_{t_2}$ with $x \in D'$ that are the union of spherical fibers of $\SS^s$ and such that $\rho > .9 \rho (x)$ on $D$, $\rho > 2 \Lambda^3 \rho (x)$ on $\partial D$, $\rho < 2 \rho (x)$ on $D'$ and such that $D'$ contains a singular spherical fiber of the form $\{ x' \} \subset D'$.
\item \label{ass_further_properties_gpdtp_f} For any $s \in L$ the following is true:
\begin{enumerate}[label=(g\arabic*)]
\item If $(M, g^s)$ is homothetic to a quotient of the round sphere, then the flow of $\partial'_\t$ induces a homothety between $(\M^s_0, g^{\prime,s}_{0})$ and $(\M^s_t, g^{\prime,s}_{t})$ for all $t$ for which it is defined.
\item If $(M,g^s)$ is homothetic to a quotient of the round cylinder, then for all $t \geq 0$ for which $\M^s_t \neq \emptyset$ the Riemannian manifold $(\M^s_t, g^{\prime,s}_{t})$ is homothetic to a (possibly different) quotient of the round cylinder and the flow of $\partial'_\t$ preserves the local isometric $O(3)$-actions on each time-slice $\M^s_t$.
\end{enumerate}
\item \label{ass_further_properties_gpdtp_g} If $(M^s, g^s )$ has positive scalar curvature, then $g^{\prime s}$ has positive scalar curvature on every time-slice.
Moreover if $t \geq 0$ and if $Y \subset \M^s_t$ is a compact 3-dimensional submanifold whose boundary components are regular spherical fibers of $\SS^s$ and $\rho \leq 1$ on $\partial Y$, then $(Y, g^{\prime,s}_t)$ is PSC-conformal.
\end{enumerate}
\end{lemma}

\begin{proof}
Assertion~\ref{ass_further_properties_gpdtp_a} is a consequence of Theorem~\ref{Thm_rounding}\ref{ass_thm_rounding_b}.
Assertion~\ref{ass_further_properties_gpdtp_aa} holds for $C_0 \geq \underline{C}_0$.
Assertion~\ref{ass_further_properties_gpdtp_b} follows from Theorem~\ref{Thm_extinction_time}.
 Assertion~\ref{ass_further_properties_gpdtp_c} follows from Assertion~\ref{ass_further_properties_gpdtp_b} and Theorem~\ref{Thm_properness_fam_sing_RF}.
Assertion~\ref{ass_further_properties_gpdtp_d} is a consequence of Assertion~\ref{ass_further_properties_gpdtp_c}.

For Assertion~\ref{ass_further_properties_gpdtp_e} let $\delta' > 0$ be a constant whose value we will determine later.
By Lemma~\ref{lem_bryant_increasing_scale}, assuming $C_0 \geq \underline{C}_0 (\delta')$, $\delta \leq \ov\delta (\delta')$ (see (\ref{eq_rs_bigger_C0})), the pointed Riemannian manifold $(\M_{t_2}^s, g^s_{t_2}, x)$ is $\delta'$-close to the pointed Bryant soliton $(M_{\Bry}, \lb g_{\Bry}, \lb x_{\Bry})$ at scale $\rho(x)$.
If $\delta' \leq \ov\delta'$, then $.99 \rho_{g^s} \leq \rho_{g^{\prime,s}} \leq 1.01 \rho_{g^s}$ on $\{ \rho < C_0 \}$.
Moreover if $\delta' \leq \ov\delta'$, then the scalar curvature of $g^{\prime,s}$ attains a unique maximum at some $x' \in B(x, .01\rho(x))$.
Therefore $x,x' \in U^s_{S2}$ and $\{ x' \}$ is a singular fiber.
Let $D'$ be the union of spherical fibers intersecting the minimizing geodesic between $x, x'$.
Then the asserted properties of $D'$ hold for $\delta' \leq \ov\delta'$ and the existence of $D$ holds if $C_0 \geq \underline{C}_0 (\Lambda)$ and $\delta' \leq \ov\delta' (\Lambda)$.

Assertion~\ref{ass_further_properties_gpdtp_f} follows by uniqueness of singular Ricci flows, Theorem~\ref{Thm_sing_RF_uniqueness}, and Theorem~\ref{Thm_rounding}\ref{ass_thm_rounding_d}.

The first part of Assertion~\ref{ass_further_properties_gpdtp_g} holds if $\delta \leq \ov{\delta}$.
The second part follows using Lemma~\ref{lem_CNA_SS_implies_PSC_conformal} if $C_0 \geq \underline{C}_0$ and $\delta \leq \ov{\delta}$.
Observe that we may assume without loss of generality that $Y$ is disjoint from $U_{S3}^s$, because all components of $U_{S3}^s \cap \M^s_t$ are PSC-conformal.
\end{proof}

From now on let us fix the constants $C_0, \delta$ from Lemma~\ref{Lem_further_properties_gpdtp}, as well as the family of $\RR$-structures $(\RR^s)_{s \in K}$.

Let $n := \dim K$ and set $r_k := \Lambda^{-4n + 4k-4}$.
So 
\[ 0 < r_0 < \ldots < r_{n-1} < r_n = \Lambda^{-4}, \qquad \Lambda^4 r_k = r_{k+1}. \]

By Proposition~\ref{prop_partial_homotopy_standard_homotopy} and Lemma~\ref{Lem_further_properties_gpdtp}\ref{ass_further_properties_gpdtp_aa}, Theorem~\ref{Thm_main_deform_to_CF} can be reduced to the following lemma.

\begin{lemma} \label{Lem_main_existence_partial_homotopy}
For any $T \geq 0$ there is a simplicial refinement of $\mathcal{K}$ and a partial homotopy $\{ ( Z^\sigma, \lb (g^\sigma_{s,t})_{s \in \sigma, t \in [0,1]}, \lb (\psi^\sigma_s  )_{s \in \sigma}) \}_{\sigma \subset K}$ at time $T$ relative to $L$ for $(\RR^s)_{s \in K}$ that satisfies the following a priori assumptions for every $k$-simplex $\si\subset K$ and all $s \in \sigma$:
\begin{enumerate}[label=(APA \arabic*), leftmargin=*]
\item \label{APA1} $\{ \rho > 1 \} \cap \M^s_T \subset \psi^\sigma_s (Z^\sigma) \subset \{ \rho > r_k \}$
\item \label{APA2} Every component of $\psi^\sigma_s (Z^\sigma)$ contains a point with $\rho > \Lambda^2 r_k$.
\item \label{APA3} $\rho > \Lambda r_k$ on $\psi^\sigma_s (\partial Z^\sigma)$.
\item \label{APA4} $\{ ( Z^\sigma, \lb (g^\sigma_{s,t})_{s \in \sigma, t \in [0,1]}, \lb (\psi^\sigma_s  )_{s \in \sigma}) \}_{\sigma \subset K}$ is PSC-conformal  over every $s \in K_{PSC}$.
\end{enumerate}
\end{lemma}

The remainder of this section is devoted to the proof of Lemma~\ref{Lem_main_existence_partial_homotopy}, which proceeds by induction.
If $T \geq T_{\ext}$, then the assertion of the lemma is true, as we can choose the trivial partial homotopy.
So it remains to show that if the lemma is true for some time $T$, then it also holds at time $T - \Delta T$, where $0 < \Delta T \leq \min \{ T, \theta (r_0) \}$, where $\theta$ is the constant from Lemma~\ref{Lem_further_properties_gpdtp}\ref{ass_further_properties_gpdtp_d}.
For this purpose, fix $T$, $\Delta T$ for the remainder of the section and consider some simplicial refinement $\mathcal{K}'$ of $\mathcal{K}$ and a partial homotopy $\{ ( Z^\sigma, \lb (g^\sigma_{s,t})_{s \in \sigma, t \in [0,1]}, \lb (\psi^\sigma_s  )_{s \in \sigma}) \}_{\sigma \subset K}$ at time $T$ satisfying the a priori assumptions \ref{APA1}--\ref{APA4}.
Our goal in the next subsections will be to construct a partial homotopy at time $T - \Delta T$ that satisfies a priori assumptions \ref{APA1}--\ref{APA4}, after passing to a simplicial refinement of $\mathcal{K}'$.

\subsection{Passing to a simplicial refinement}
Our strategy will be to modify (or improve) the partial homotopy $\{ ( Z^\sigma, \lb (g^\sigma_{s,t})_{s \in \sigma, t \in [0,1]}, \lb (\psi^\sigma_s  )_{s \in \sigma}) \}_{\sigma \subset K}$ at time $T$ so that by evolving it backwards in time by $\Delta T$ (using Proposition~\ref{prop_move_part_hom_bckwrds}) we obtain another partial homotopy that satisfies a priori assumptions \ref{APA1}--\ref{APA4}.
The modification will be carried out by successive application of the modification moves described in Proposition~\ref{prop_extending} (Enlarging a partial homotopy) and \ref{prop_move_remove_disk} (Removing a disk from a partial homotopy).
As a preparation, we will first find a simplicial refinement of $K$ (using Proposition~\ref{prop_simp_refinement}) such that over each simplex we can choose certain continuous data, which will later serve as a blueprint for these modification moves.
The main result of this subsection is the following lemma:

\begin{lemma}[Passing to a simplicial refinement] \label{lem_simplicial_refinement}
After passing to a simplicial refinement of $\mathcal{K}'$ and modifying the partial homotopy $\{ ( Z^\sigma, \lb (g^\sigma_{s,t})_{s \in \sigma, t \in [0,1]}, \lb (\psi^\sigma_s  )_{s \in \sigma}) \}_{\sigma \subset K}$ to respect this refined structure, we may assume in addition to a priori assumptions \ref{APA1}--\ref{APA4} that for any simplex $\sigma \subset K$ there are the following data:
\begin{itemize}
\item a compact manifold with boundary $\widehat{Z}^\sigma$, 
\item an embedding $\iota^\sigma: Z^\sigma \to \widehat{Z}^\sigma$, 
\item a continuous family of embeddings $(\widehat{\psi}^\sigma_s : \widehat{Z}^\sigma \to \M^s_{T})_{s \in \sigma}$ and
\item continuous families of embeddings $(\nu^\sigma_{s,j} : D^3 \to \M^s)_{s \in \sigma, j = 1, \ldots, m_\sigma}$,
\end{itemize}
such that for all $s \in \sigma$, $j = 1, \ldots, m_\sigma$ and $k = \dim \sigma$:
\begin{enumerate}[label=(\alph*)]
\item  \label{ass_refinement_a} 
$\psi_{s}^{\sigma}  = \wh\psi^{\sigma}_{s}  \circ \iota^{\sigma}$.
\item  \label{ass_refinement_b} 
For the closure $\C$ of every component of $\wh{Z}^{\sigma} \setminus \iota^{\sigma}( Z^\sigma)$ one (or both) of the following is true uniformly for all $s \in \sigma$: $\wh\psi^\sigma_s (\C)$ is a union of fibers of $\SS^{s}$ or $\partial\C = \emptyset$ and  $\wh\psi^\sigma_{s} (\C) \subset U_{S3}^{s}$.
In the second case the metrics $(\wh\psi^\sigma_s)^* g'_{s,T}$, $s \in \sigma$, are multiples of each another.
\item  \label{ass_refinement_c} 
 If $\sigma \subset L$, then $\wh\psi^\sigma_s (\wh{Z}^\sigma )$ is empty or equal to $\M^s_T$.
\item \label{ass_refinement_d} 
$\widehat\psi^\sigma_s (\widehat{Z}^\sigma) \setminus \psi^\sigma_s (Z^\sigma) \subset \{ \rho >  2r_{k} \}$.
\item \label{ass_refinement_e} 
$\{ \rho > \frac12 \Lambda^3 r_{k} \} \cap \M^s_T \subset \widehat\psi^\sigma_s (\widehat{Z}^\sigma)$.
\item \label{ass_refinement_f} 
Every component of $\wh\psi^\sigma_s (\wh{Z}^\sigma)$ that does not contain a component of $\psi^\sigma_s (Z^\sigma)$ contains a point of scale $\rho > 2 \Lambda^2 r_k$.
\item \label{ass_refinement_g} 
In every component of $\M^{s}_{T} \setminus \wh\psi^\sigma_{s}(\Int \wh{Z}^\sigma)$ that intersects $\wh\psi^\sigma_{s}(\wh{Z}^\sigma)$ there is a point with $\rho < 4 r_k$.
\item  \label{ass_refinement_h} 
$\rho > 2\Lambda r_k$ on $ \wh\psi^{\sigma}_{s}(\partial \wh{Z}^\sigma )\setminus \psi^\sigma_{s} ( \partial Z^\sigma )$.
\item  \label{ass_refinement_h_old} 
 If $\sigma \cap L\neq\emptyset$, then $m_\sigma = 0$.
\item \label{ass_refinement_disjoint} The images $\nu^\sigma_{s,1} (D^3), \ldots, \nu^\sigma_{s,m_\sigma} (D^3)$ are pairwise disjoint.
 \item  \label{ass_refinement_ii} 
 $\nu^\sigma_{s,j} (D^3) \subset \psi^\sigma_s (\Int Z^\sigma)$.
\item  \label{ass_refinement_i} 
$\nu^\sigma_{s,j} (D^3) \subset U^s_{S2}$ and the embedding $\nu^\sigma_{s,j}$ carries the standard spherical structure on $D^3$ to $\SS^s$ restricted to $\nu^\sigma_{s,j}(D^3)$.
\item  \label{ass_refinement_j} 
$\nu^\sigma_{s,1} (D^3) \cup \ldots \cup \nu^\sigma_{s,m_\sigma} (D^3)$ contains all singular spherical fibers of $\psi^\sigma_s (Z^\sigma) \cap U^s_{S2}$ that are points and on which $\rho < 4 r_k$.
\item  \label{ass_refinement_k} 
$\rho < 4 \Lambda r_k$ on $\nu^\sigma_{s,j} (D^3)$.
\item  \label{ass_refinement_l} 
 $\rho > 2\Lambda r_k$ on $\nu^\sigma_{s,j} (\partial D^3)$
\end{enumerate}
\end{lemma}

The idea of the proof is the following.
For every $s_0 \in \sigma \subset K$ we first construct continuous data $Z^{s_0, \sigma}, \iota^{s_0, \sigma}, \wh\psi^{s_0, \sigma}_s, \nu^{s_0, \sigma}_{s,j}$ that satisfy the assertions of Lemma~\ref{lem_simplicial_refinement} for parameters $s$ that are close enough to $s_0$.
We therefore obtain an open covering of $K$ consisting of subsets over which this data is defined.
Our simplicial refinement of $\mathcal{K}'$ will later be taken to be subordinate to this open cover.

Let us first construct $Z^{s_0, \sigma}, \iota^{s_0, \sigma}, \wh\psi^{s_0, \sigma}_s$ for $s$ near some $s_0 \in \sigma \subset K$.

\begin{lemma}
For every $s_0 \in \sigma \subset K$ there is a neighborhood $U_{s_0, \sigma} \subset K$ of $s_0$, a compact manifold with boundary $\widehat{Z}^{s_0,\sigma}$, an embedding $\iota^{s_0, \sigma}: Z^\sigma \to \widehat{Z}^{s_0,\sigma}$ and a continuous family of embeddings $(\widehat{\psi}^{s_0,\sigma}_{s} : \widehat{Z}^{s_0,\sigma} \to \M^{s}_{T})_{s \in U_{s_0, \sigma} \cap \sigma}$ such that  Assertions \ref{ass_refinement_a}--\ref{ass_refinement_h} of Lemma~\ref{lem_simplicial_refinement} hold for all $s \in U_{s_0, \sigma} \cap \sigma$ and $k = \dim \sigma$ if we replace $(Z^{ \sigma}, \iota^{ \sigma}, \wh\psi^{ \sigma}_s)$ with $(Z^{s_0, \sigma}, \iota^{s_0, \sigma}, \wh\psi^{s_0, \sigma}_s)$.
\end{lemma}

\begin{proof}
By a priori assumption \ref{APA1} and (\ref{eq_US2US310}) we have
\[ \M^{s_0}_{T} \setminus \psi^{\sigma}_{s_0} ( \Int Z^\sigma ) \subset U^{s_0}_{S2} \cup U^{s_0}_{S3}. \]
It follows from Definition~\ref{Def_R_structure}\ref{prop_def_RR_1} and \ref{Def_partial_homotopy} that each component of this difference is contained in $U^{s_0}_{S2}$ and is a union of fibers of $\SS^{s_0}$ or in $U^{s_0}_{S3}$ and is homothetic to a quotient of a standard sphere.
Let $Y \subset \M^{s_0}_{T} \setminus \psi^\sigma_{s_0} ( \Int Z^\sigma )$ be the set of points on which $\rho > 2 r_k$.
Define $Y' \subset Y$ to be the union of all connected components of $Y$ that contain a point of scale $\rho > 2 \Lambda^2 r_k$ or that intersect  $\psi^\sigma_{s_0} ( \partial Z^\sigma )$.
Note that any boundary component of $Y'$ that is contained in $Y'$ is also contained in $\psi^\sigma_{s_0} (\partial Z^\sigma)$.
Furthermore, any connected component of $Y'$ is contained in $U^{s_0}_{S3}$ and is compact or in $U^{s_0}_{S2}$ and is a union of fibers of $\SS^{s_0}$.
Therefore, every non-compact component of $Y'$ is a union of spherical fibers and must be diffeomorphic to one of the following manifolds (see Lemma~\ref{lem_spherical_struct_classification}):
\[  S^2 \times [0, 1), \quad S^2 \times (0,1), \quad (S^2 \times (-1,1))/\IZ_2, \quad B^3 \]
Consider such a component $\C \subset Y'$.
Call $\C$ \emph{good} if it contains a point with $\rho > 2 \Lambda^2 r_k$.
By construction, bad components must intersect $\psi^\sigma_{s_0} ( \partial Z^\sigma )$ and must therefore be either compact or diffeomorphic to $S^2 \times [0, 1)$.
Suppose for a moment that $\C$ is good.
Since $\rho \to 2 r_k$ near the ends of $\C$, we can find a minimal compact domain $\C' \subset \C$ such that 
\begin{enumerate}
\item $\C'$ is a union of spherical fibers.
\item $\C \setminus \C'$ is a union of neighborhoods of the ends of $\C$; so each component is diffeomorphic to $S^2 \times (0,1)$.
\item $\rho < 4 \Lambda r_k$ on $\C \setminus \C'$.
\end{enumerate}
Call $\C'$ the \emph{core} of $\C$.
We now define $Y''$ to be the union of compact components of $Y'$ and the cores of non-compact good components of $Y'$.
Set $\widehat{Z}^\sigma := \psi^\sigma_{s_0} ( Z^\sigma ) \cup Y''$.

By Lemma~\ref{lem_chart_near_compact_subset} and isotopy extension, we can define $\wh\psi^{s_0,\sigma}_{s} : \wh Z^\sigma \to \M^s_T$ for $s \in \sigma$ close to $s_0$ such that $\wh\psi^{s_0,\sigma}_{s} ( \partial \widehat{Z}^\sigma  )$ consists of spherical fibers and such that Assertion \ref{ass_refinement_a} of Lemma~\ref{lem_simplicial_refinement} holds. 
We can moreover construct $\wh\psi^{s_0,\sigma}_{s}$ in such a way that for every component $\C \subset Y$ with the property that $\partial \C = \emptyset$ and $Y \subset U^{s_0}_{S3}$ the metric $(\wh\psi^{s_0,\sigma}_{s})^* g_{s,T}$ restricted to $\C$ is a multiple of the same constant curvature metric for all $s$ close to $s_0$.
Then Assertion \ref{ass_refinement_b} of Lemma~\ref{lem_simplicial_refinement} holds for all $s$ close to $s_0$.
By construction, Assertions \ref{ass_refinement_c}--\ref{ass_refinement_f}, \ref{ass_refinement_h} of Lemma~\ref{lem_simplicial_refinement} hold for $s = s_0$.
Next, we argue that the same is true for Assertion~\ref{ass_refinement_g}.
Assume by contradiction that $\rho \geq 4r_k$ on some component $\C^* \subset \M^s_T \setminus \wh\psi^\sigma_s (\Int \wh Z^\sigma)$ that intersects $\wh\psi^\sigma_s ( \wh Z^\sigma)$.
Then the component $\C^{**} \subset \M^s_T \setminus \psi^\sigma_s (\Int  Z^\sigma)$ containing $\C^*$ is a union of $\C^*$ with components of $Y$.
Since $\rho \to 2 r_k$ near the open ends of $Y$, we find that $\C^{**} \subset Y$.
However this implies that $\C^{**}$ is a compact component of $Y$ and therefore $\C^{**} \subset Y''$.
Since Assertions \ref{ass_refinement_c}--\ref{ass_refinement_h} are open, they also hold for $s$ sufficiently close to $s_0$.
\end{proof}

\begin{lemma}
\label{lem_nu_s_js}
For every $s_0 \in \sigma \subset K$ there is a neighborhood $V_{s_0, \sigma} \subset K$ of $s_0$ and continuous families of embeddings $(\nu^{s_0,\sigma}_{s,j} : D^3 \to \M^s)_{s \in V_{s_0, \sigma} \cap \sigma, j = 1, \ldots, m_{s_0, \sigma}}$ such that Assertions~\ref{ass_refinement_disjoint}--\ref{ass_refinement_l} of Lemma~\ref{lem_simplicial_refinement} hold for all $s \in V_{s_0, \sigma} \cap \sigma$, $j = 1, \ldots, m_{s_0, \sigma}$ and $k = \dim \sigma$ if we replace $(\nu^\sigma_{s,j} : D^3 \to \M^s)_{s \in \sigma, j = 1, \ldots, m_\sigma}$ with $(\nu^{s_0,\sigma}_{s,j} : D^3 \to \M^s)_{s \in V_{s_0, \sigma} \cap \sigma, j = 1, \ldots, m_{s_0, \sigma}}$.
Moreover, instead of Assertion~\ref{ass_refinement_h_old} we have:
\begin{enumerate}[label=(\alph*$\,'$), start=9]
\item \label{ass_find_disk_i_prime} If $V_{s_0, \sigma} \cap L \neq \emptyset$, then $m_{s_0,\sigma} = 0$.
\end{enumerate}
\end{lemma}

\begin{proof}
Let $E \subset \cup_{s \in K} U^s_{S2}$ be the union of all spherical fibers that are points and on which $\rho \leq 4 r_k$.
Then $E$ is closed in $\cup_{s \in K} U^s_{S2}$ and $E \cap \psi^\sigma_{s_0} (Z^\sigma) =: \{ x_1, \ldots, x_{m_{s_0, \sigma}} \}$ consists of finitely many points.

For every $j = 1, \ldots, m_{s_0, \sigma}$ let $Y_j \subset \M^{s_0}_T$ be the union of all open disks $X \subset \M^{s_0}_T$ with the property that $x_j \in X$, $X \setminus \{ x_j \}$ consists of regular spherical fibers and $\rho \leq 3 \Lambda r_k$ on $X$.
Then $Y_j$ is also an open disk.
By a priori assumption \ref{APA2}, no component of $\psi^\sigma_{s_0} (Z^\sigma)$ is fully contained in the closure $\ov{Y}_j$ of $Y_j$.
So, in particular, $\partial \ov{Y}_j \neq \emptyset$, which implies that $\partial \ov{Y}_j$ is a regular fiber and therefore $\ov{Y}_j$ is a closed disk.
As the boundaries of both subsets $Y_j$ and $\psi^\sigma_{s_0} (Z^\sigma)$ consist of spherical fibers and $x \in Y_j \cap \psi^\sigma_{s_0} (Z^\sigma) \neq\emptyset$, we also obtain that $\ov{Y}_j \subset \psi^\sigma_{s_0} (\Int Z^\sigma)$.

For any two $j,j' = 1, \ldots, m_{s_0, \sigma}$, $j \neq j'$ the disks $\ov{Y}_j$, $\ov{Y}_{j'}$ are pairwise disjoint, because otherwise their union would be a connected component of $\M^{s_0}_{T}$, which would again contradict a priori assumption \ref{APA2}.

For every $j = 1, \ldots, m_{s_0, \sigma}$ choose a continuous family of points $(x'_{s,j} \in \M^s_T)$ for $s$ close to $s_0$ such that $x'_{s_0,j} = x_j$.
Using the exponential map based at $x'_{s,j}$, we can construct continuous families of embeddings $(\nu^{s_0,\sigma}_{s,j} : D^3 \to \M^s)$ for $s$ close to $s_0$ such that $x'_{s,j} \in \nu^{s_0,\sigma}_{s,j} ( 0)$ and $\nu^{s_0,\sigma}_{s_0,j} (D^3) = \ov{Y}_j$.
For $s$ close to $s_0$, these disks are pairwise disjoint and satisfy Assertion~\ref{ass_refinement_i} of Lemma~\ref{lem_simplicial_refinement}.
Assertions~\ref{ass_refinement_ii}, \ref{ass_refinement_j}--\ref{ass_refinement_l} of Lemma~\ref{lem_simplicial_refinement} hold for $s = s_0$ by construction and therefore by openness, they also hold for $s \in V_{s_0, \sigma}$, where $V_{s_0, \sigma}$ is a small enough neighborhood of $s_0$. 
For Assertion~\ref{ass_find_disk_i_prime} we can distinguish two cases: If $s_0 \not\in L$, then we can choose $V_{s_0, \sigma} \cap L \neq\emptyset$.
If $s_0 \in L$, then $\rho$ is constant on $\M^{s_0}_T$ and therefore by construction $m_{s_0, \sigma} = 0$.
\end{proof}

\begin{proof}[Proof of Lemma~\ref{lem_simplicial_refinement}.]
For every $s_0 \in K$ let
\[ W_{s_0} := \bigcap_{s_0 \in \sigma \subset K} (U_{s_0, \sigma} \cap V_{s_0, \sigma} ). \]
Then $W_{s_0}$ is still an open neighborhood of $s_0$ and $K = \cup_{s_0 \in K} W_{s_0}$.
Let now $\mathcal{K}''$ be a refinement of $\mathcal{K}'$ that is subordinate to this cover and for every simplex $\sigma \in \mathcal{K}''$ of this refinement let 
\begin{multline*}
 \big( \wh{Z}^\sigma, \iota^\sigma,( \wh\psi^\sigma_s)_{s \in \sigma}, (\nu^{\sigma}_{s,j} )_{s \in \sigma, j = 1, \ldots, m_\sigma} \big) \\ 
 := \big(\wh{Z}^{s_\sigma, \sigma}, \iota^{s_\sigma,\sigma}, ( \wh\psi^{s_\sigma, \sigma}_s)_{s \in \sigma},  (\nu^{s_\sigma,\sigma}_{s,j} )_{s \in \sigma, j = 1, \ldots, m_{s_\sigma, \sigma}} \big), 
\end{multline*}
 where $s_\sigma \in \sigma$ is chosen such that $\sigma \subset W_{s_\sigma}$.
 
 We can now modify the partial homotopy $\{ ( Z^\sigma, \lb (g^\sigma_{s,t})_{s \in \sigma, t \in [0,1]}, \lb (\psi^\sigma_s  )_{s \in \sigma}) \}_{\sigma \subset K}$ according to Proposition~\ref{prop_simp_refinement} to respect the refinement $\mathcal{K}''$.
 To see that a priori assumptions \ref{APA1}--\ref{APA4} continue to hold, let $s \in \sigma''$ where $\sigma'' \in \mathcal{K}''$, $k'' = \dim \sigma''$, and choose $\sigma' \in \mathcal{K}'$, $k' = \dim \sigma'$ of minimal dimension such that $\sigma' \supset \sigma''$.
 Then $k' \geq k''$ and $\psi^{\sigma''}_s (Z^{\sigma''}) = \psi^{\sigma'}_s (Z^{\sigma'})$.
 Therefore we have
 \[ \{ \rho > 1 \} \cap \M^s_T \subset \psi^{\sigma''}_s (Z^{\sigma''}) \subset \{ \rho > r_{k'} \} \subset \{ \rho > r_{k''} \}, \]
 every component of $\psi^{\sigma''}_s (Z^{\sigma''})$ contains a point with $\rho > \Lambda^2 r_{k'} \geq \Lambda^2 r_{k''}$ and $\rho > \Lambda r_{k'} \geq \Lambda r_{k''}$ on $\psi^{\sigma''}_s (\partial Z^{\sigma''})$.
\end{proof}

\subsection{Improving the partial homotopy}
Our next goal will be to construct a new partial homotopy by extending the domain of the partial homotopy $\{ ( Z^\sigma, \lb (g^\sigma_{s,t})_{s \in \sigma, t \in [0,1]}, \lb (\psi^\sigma_s  )_{s \in \sigma}) \}_{\sigma \subset K}$ according to the maps $(\wh\psi^\sigma_s)_{s \in \sigma}$ (using Proposition~\ref{prop_extending}) and by removing the disks $(\nu^\sigma_{s,j} (D^3))_{s \in \sigma, j = 1, \ldots, m_\sigma}$ (using Proposition~\ref{prop_move_remove_disk}).
In the next subsection we will move the new partial homotopy backwards in time by the time step $\Delta T$.
The following lemma will be used to verify that the resulting partial homotopy satisfies a priori assumptions \ref{APA1}--\ref{APA4}.

\begin{lemma} \label{lem_newAPA}
With the choices of $(\widehat{\psi}^\sigma_s )_{s \in \sigma}$ and $(\nu^\sigma_{s,j})_{s \in \sigma, j = 1, \ldots, m_\sigma}$ from Lemma~\ref{lem_simplicial_refinement} the following holds for all $s \in \sigma \subset K$.
All points of
\begin{equation} \label{def_X_psi}
 X^\sigma_s :=  \wh\psi^{\sigma}_s (\wh Z^{\sigma})  \setminus (\nu^\sigma_{s,1}(B^3) \cup \ldots \cup \nu^\sigma_{s,m_\sigma} (B^3))
\end{equation}
survive until time $T - \Delta T$ and we have:
\begin{enumerate}[label=(\alph*)]
\item \label{ass_newAPA_a} $\{ \rho >  \Lambda^3 r_{k} \} \cap \M^s_T \subset X^\sigma_s (T - \Delta T) \subset \{ \rho >  r_{k} \}$
\item \label{ass_newAPA_b} Every component of of $X^\sigma_s$ contains a point $x$ with $\rho (x(T-\Delta T)) > \Lambda^2 r_k$.
\item \label{ass_newAPA_c} $\rho > \Lambda r_k$ on $\partial X^\sigma_s (T - \Delta T)$.
\end{enumerate}
\end{lemma}

\begin{proof}
Fix $s \in \sigma \subset K$.
By a priori assumption \ref{APA1} and Lemma~\ref{lem_simplicial_refinement}\ref{ass_refinement_d} we have $\rho > r_k$ on $X^\sigma_s$.
So by our choice of $\Delta T$ all points of $X^\sigma_s$ survive until time $T - \Delta T$.
We first show:

\begin{Claim}
If there is some $x \in \wh\psi^\sigma_s (\wh{Z}^\sigma)$ with $\rho (x(T - \Delta T)) \leq \rho (x) \leq 10$, then there is an embedded disk $D \subset \M^s_T$ such that:
\begin{enumerate}[label=(\alph*)]
\item \label{ass_disk_a} $x \in \Int D$.
\item \label{ass_disk_b} $\rho > 2 \Lambda^3 r_k$ on $\partial D$.
\item \label{ass_disk_c} If $\rho (x) > 5 r_k$, then $D \subset \wh\psi^\sigma_s ( \Int \wh{Z}^\sigma)$.
\item \label{ass_disk_d} If $\rho (x (T - \Delta T)) <  r_k$, then $x \in \nu^\sigma_{s,1} (B^3) \cup \ldots \cup \nu^\sigma_{s,m_\sigma} (B^3)$.
\end{enumerate}
\end{Claim}

\begin{proof}
Let $x, x' \in D' \subset D$ be the data from  Lemma~\ref{Lem_further_properties_gpdtp}\ref{ass_further_properties_gpdtp_e}.
Assertions~\ref{ass_disk_a} and \ref{ass_disk_b} follow immediately using a priori assumption \ref{APA1}.

By a priori assumption \ref{APA2} and Lemma~\ref{lem_simplicial_refinement}\ref{ass_refinement_e} we have $\partial D \subset  \wh\psi^\sigma_s (\wh{Z}^\sigma)$.
So if Assertion~\ref{ass_disk_c} was false, then $D$ must contain a component of $\M^s_T \setminus \wh\psi^\sigma_s ( \Int \wh{Z}^\sigma)$ that intersects $\wh\psi^\sigma_s (\wh{Z}^\sigma)$ and on which $\rho > .9 \rho (x) >4 r_k$, in contradiction to Lemma~\ref{lem_simplicial_refinement}\ref{ass_refinement_g}.

Lastly, we verify Assertion~\ref{ass_disk_d}.
Assume that $\rho (x(T- \Delta T))< r_k$.
Then by our choice of $\Delta T$ we have $\rho (x) < 2 r_k$, see Lemma~\ref{Lem_further_properties_gpdtp}\ref{ass_further_properties_gpdtp_c}.
So by Lemma~\ref{lem_simplicial_refinement}\ref{ass_refinement_d} we have $x \in \psi^\sigma_s (Z^\sigma)$.
By a priori assumption \ref{APA3} and the fact that $\rho < 2 \rho (x) < 4 r_k$ on $D'$, we obtain that $x' \in D' \subset \psi^\sigma_s (Z^\sigma)$.
So by Lemma~\ref{lem_simplicial_refinement}\ref{ass_refinement_j} we have $x' \in \nu^\sigma_{s,j} (D^3)$ for some $j \in \{ 1, \ldots, m_\sigma \}$.
By Lemma~\ref{lem_simplicial_refinement}\ref{ass_refinement_l} we have $x \in D' \subset \nu^\sigma_{s,j} (B^3)$.
This finishes the proof of Assertion~\ref{ass_disk_d}.
\end{proof}

We can now verify the assertions of this lemma.
The first inclusion of Assertion~\ref{ass_newAPA_a} follows from  Lemma~\ref{lem_simplicial_refinement}\ref{ass_refinement_e}, \ref{ass_refinement_k} and our choice of $\Delta T$, see Lemma~\ref{Lem_further_properties_gpdtp}\ref{ass_further_properties_gpdtp_c}.
The second inclusion is a consequence of a priori assumption \ref{APA1},  Lemma \ref{lem_simplicial_refinement}\ref{ass_refinement_d}, Assertion~\ref{ass_disk_d} of the Claim and our choice of $\Delta T$.

For Assertion~\ref{ass_newAPA_b} consider a component $\C$ of $X^\sigma_s$ and let $\C' \subset \wh{Z}^\sigma$ be the component with $\C \subset \wh\psi^\sigma_s (\C')$.
Assume by contradiction that $\rho \leq \Lambda^2 r_k$ on $\C (T-\Delta T)$.
So by our choice of $\Delta T$ and Lemma~\ref{lem_simplicial_refinement}\ref{ass_refinement_k} we have 
\begin{equation} \label{eq_rho_on_CCprime}
 \rho \leq \Lambda^2 r_k \quad \text{on} \quad \big(\wh\psi^\sigma_s (\C') \big)(T - \Delta T) \quad \Longrightarrow \quad \rho < 2 \Lambda^2 r_k \quad \text{on} \quad \wh\psi^\sigma_s (\C'). 
\end{equation}
By Lemma~\ref{lem_simplicial_refinement}\ref{ass_refinement_f} there is a component $\C'' \subset Z^\sigma$ such that $\psi^\sigma_s (\C'') \subset \wh\psi^\sigma_s (\C')$.
Due to a priori assumption \ref{APA2} we can find a point $x \in \psi^\sigma_s (\C'')$ such that $\rho (x) > \Lambda^2 r_k$.
Since $\rho (x(T- \Delta T) \leq \Lambda^2 r_k$, Assertion~\ref{ass_disk_c} of the Claim implies that there is an embedded disk $D \subset \wh\psi^\sigma_s (\C')$ with $\rho > 2 \Lambda^3 r_k$ on $\partial D$, in contradiction to (\ref{eq_rho_on_CCprime}).

Lastly, we verify Assertion~\ref{ass_newAPA_c}.
Consider a component $\Sigma \subset \partial X^\sigma_s$.
If $\Sigma = \nu^\sigma_{s,j} (\partial D^3)$ for some $j = 1, \ldots, m_\sigma$, then we are done by Lemma~\ref{lem_simplicial_refinement}\ref{ass_refinement_l} and our choice of $\Delta T$.
Assume now that $\Sigma \subset \wh\psi^\sigma_s (\partial \wh{Z}^\sigma)$.
If $\Sigma \not\subset \psi^\sigma_s (\partial Z^\sigma)$, then we are done by  Lemma~\ref{lem_simplicial_refinement}\ref{ass_refinement_h} and our choice of $\Delta T$.
Assume now that $\Sigma \subset \psi^\sigma_s (\partial Z^\sigma)$.
By a priori assumption \ref{APA3} we have $\rho > \Lambda r_k$ on $\Sigma$.
If $\rho (\Sigma (T - \Delta T)) \leq \Lambda r_k$, then we can apply Assertion~\ref{ass_disk_c} of the Claim to find an embedded disk $D \subset \wh\psi^\sigma_s (\Int \wh{Z}^\sigma )$ whose interior intersects $\Sigma$, contradicting the fact that $\Sigma \subset \wh\psi^\sigma_s (\partial \wh{Z}^\sigma)$.
\end{proof}

The next lemma ensures that all necessary containment relationships hold if we successively modify the partial homotopy $\{ ( Z^\sigma, \lb (g^\sigma_{s,t})_{s \in \sigma, t \in [0,1]}, \lb (\psi^\sigma_s  )_{s \in \sigma}) \}_{\sigma \subset K}$ according to the data $(\wh\psi^\sigma_s)_{s \in \sigma}$ and $(\nu^\sigma_{s,j} (D^3))_{s \in \sigma, j = 1, \ldots, m_\sigma}$ over all simplices of $K$.

\begin{lemma} \label{lem_containment_psi}
For any two simplices $\sigma, \tau \subset K$ and $s \in \tau \cap \sigma$ we have for all $j =1, \ldots, m_\sigma$:
\begin{enumerate}[label=(\alph*)]
\item  \label{ass_containment_psi_a} If $\dim \tau < \dim \sigma$, then $\wh\psi^\sigma_s (\wh{Z}^\sigma) \subset \wh\psi^\tau_s (\wh{Z}^\tau) \setminus  (\nu^\tau_{s,1} (B^3) \cup \ldots \cup \nu^\tau_{s,m_\tau}(B^3)) $.
\item \label{ass_containment_psi_b} If $\dim \tau < \dim \sigma$, then $\nu^\tau_{s,j}(D^3) \cap \psi^\sigma_s (Z^\sigma) = \nu^\tau_{s,j}(D^3) \cap \wh\psi^\sigma_s (\wh Z^\sigma) = \emptyset$.
\item \label{ass_containment_psi_c} If $\dim \tau \leq \dim \sigma$, then the image $\nu^\tau_{s,j} (D^3)$ does not contain an entire component of $\psi^{\sigma}_s (Z^{\sigma})$.
\item \label{ass_containment_psi_d} If $\dim \tau \geq \dim \sigma$, then the image $\nu^\tau_{s,j} (D^3)$ does not contain an entire component of $\wh\psi^\sigma_s (\wh{Z}^\sigma) \setminus  (\nu^\sigma_{s,1} (B^3) \cup \ldots \cup \nu^\sigma_{s,m_\sigma}(B^3))$.
\end{enumerate}
\end{lemma}

\begin{proof}
By a priori assumption \ref{APA1} and Lemma~\ref{lem_simplicial_refinement}\ref{ass_refinement_d}  we have $\rho > r_{\dim \sigma}$ on $\wh\psi^\sigma_s (\wh{Z}^\sigma)$.
On the other hand, by Lemma~\ref{lem_simplicial_refinement}\ref{ass_refinement_e}, \ref{ass_refinement_k} the set $\wh\psi^\tau_s (\wh{Z}^\tau) \setminus  (\nu^\tau_{s,1} (B^3) \cup \ldots \cup \nu^\tau_{s,m_\tau}(B^3))$ contains all points in $\M^s_T$ of scale $\rho > \frac12 \Lambda^3 r_{\dim \tau}$ and $r_{\dim \sigma} > \frac12 \Lambda^3 r_{\dim \tau}$.
This implies Assertion~\ref{ass_containment_psi_a}.
Assertion~\ref{ass_containment_psi_b} is a direct consequence of Assertion~\ref{ass_containment_psi_a}.
If $\dim \tau \neq \dim \sigma$, then Assertion~\ref{ass_containment_psi_c} follows from Assertion~\ref{ass_containment_psi_b} and Assertion~\ref{ass_containment_psi_d} follows from Assertion~\ref{ass_containment_psi_a}, because $\nu^\tau_{s,j}(D^3) \subset \psi^\tau_s (Z^\tau) \subset \wh \psi^\tau_s (\wh Z^\tau)$.
So assume that $\dim \tau = \dim \sigma =: k$.
By a priori assumption \ref{APA2} and Lemma~\ref{lem_simplicial_refinement}\ref{ass_refinement_f} every component of $\psi^{\sigma}_s (Z^{\sigma})$ or $\wh\psi^\sigma_s (\wh{Z}^\sigma) \setminus  (\nu^\sigma_{s,1} (B^3) \cup \ldots \cup \nu^\sigma_{s,m_\sigma}(B^3))$ contains a point with $\rho > \Lambda^2 r_{k}$.
On the other hand, by Lemma~\ref{lem_simplicial_refinement}\ref{ass_refinement_k}, we have $\rho < 4 \Lambda r_k < \Lambda^2 r_k$ on $\nu^\tau_{s,j} (D^3)$.
So $\nu^\tau_{s,j} (D^3)$ cannot fully contain any such component.
\end{proof}

We now modify the partial homotopy $\{ ( Z^\sigma, \lb (g^\sigma_{s,t})_{s \in \sigma, t \in [0,1]}, \lb (\psi^\sigma_s  )_{s \in \sigma}) \}_{\sigma \subset K}$ using Proposition~\ref{prop_extending} and \ref{prop_move_remove_disk}.
We proceed inductively over the dimension of the skeleton.
More specifically, we claim that for every $k \geq 0$ we can construct a partial homotopy at time $T$ relative to $L$ that has the form
\begin{equation} \label{eq_whtd_partial_homotopy}
 \big\{ (\td{Z}^{\sigma}, (\td{g}^{\sigma}_{s,t} )_{s \in \sigma, t \in [0,1]}, (\td\psi^{\sigma}_s)_{s \in \sigma}  ) \big\}_{\substack{\sigma \subset K, \\ \dim \sigma < k}} \cup
\{ (Z^{\sigma}, (g^{\sigma}_{s,t})_{s \in \sigma, t \in [0,1]}, (\psi^{\sigma}_s)_{s \in \sigma} ) \}_{\substack{\sigma \subset K, \\ \dim \sigma \geq k}} 
\end{equation}
and for which
\[ \td\psi^\sigma_s (\td{Z}^\sigma) = X^\sigma_s \]
for all $s \in \sigma \subset K$ with $\dim \sigma < k$, where $X^\sigma_s$ is defined in (\ref{def_X_psi}).
Note that if $k = 0$, then (\ref{eq_whtd_partial_homotopy}) can be taken to be the original partial homotopy $\{ (Z^\sigma, \lb (g^\sigma_{s,t})_{s \in \sigma, t \in [0,1]}, \lb (\psi^\sigma_s)_{s \in \sigma}) \}_{\sigma \subset K}$.
If (\ref{eq_whtd_partial_homotopy}) has been constructed for some $k$, then we can construct another partial homotopy of form (\ref{eq_whtd_partial_homotopy}) for $k+1$ by first applying Proposition~\ref{prop_extending}, using the data $\wh{Z}^\sigma, \iota^\sigma, (\wh\psi^\sigma_s)_{s \in \sigma}$,
and then Proposition~\ref{prop_move_remove_disk}, using the data $( \nu^\sigma_{s, j})_{s \in \sigma, j = 1, \ldots, m_\sigma}$ over all simplices $\sigma \subset K$ of dimension $k+1$.
Lemma~\ref{Lem_further_properties_gpdtp}\ref{ass_further_properties_gpdtp_g}, 
Lemma~\ref{lem_simplicial_refinement}\ref{ass_refinement_a}--\ref{ass_refinement_c}, \ref{ass_refinement_disjoint}--\ref{ass_refinement_i} and Lemma~\ref{lem_containment_psi} ensure that the assumptions of Proposition~\ref{prop_extending} and \ref{prop_move_remove_disk} hold.
So by induction we obtain:

\begin{lemma} \label{lem_apply_all_moves}
There is a simplicial refinement of $\mathcal{K}$ and a partial homotopy $\{ ( \td{Z}^\sigma, \lb (\td{g}^\sigma_{s,t})_{s \in \sigma, t \in [0,1]}, \lb (\td\psi^\sigma_s  )_{s \in \sigma}) \}_{\sigma \subset K}$ at time $T$ relative to $L$ for $(\RR^s)_{s \in K}$ such that for all $s \in \sigma \subset K$ the set $X^\sigma_s := \td\psi^\sigma_s (\td{Z}^\sigma)$ satisfies Assertions~\ref{ass_newAPA_a}--\ref{ass_newAPA_c} of Lemma~\ref{lem_newAPA}.
Moreover, $\{ ( \td{Z}^\sigma, \lb (\td{g}^\sigma_{s,t})_{s \in \sigma, t \in [0,1]}, \lb (\td\psi^\sigma_s  )_{s \in \sigma}) \}_{\sigma \subset K}$ is still PSC-conformal over every $s \in K_{PSC}$.
\end{lemma}

\subsection{Construction of a partial homotopy at time $T - \Delta T$}
Apply Proposition~\ref{prop_move_part_hom_bckwrds} with $T' = T - \Delta T$ to the partial homotopy $\{ ( \td{Z}^\sigma, \lb (\td{g}^\sigma_{s,t})_{s \in \sigma, t \in [0,1]}, \lb (\td\psi^\sigma_s  )_{s \in \sigma}) \}_{\sigma \subset K}$.
The assumptions of this lemma are met due to Lemma~\ref{lem_newAPA} and (\ref{eq_US2US310}).
We obtain a partial homotopy $\{ ( \td Z^\sigma, \lb (\ov g^\sigma_{s,t})_{s \in \sigma, t \in [0,1]}, \lb (\ov\psi^\sigma_s  )_{s \in \sigma}) \}_{\sigma \subset K}$ at time $T- \Delta T$ with $\ov\psi^\sigma_s ( \td Z^\sigma) = X^\sigma_s (T - \Delta T)$.
So by Lemma~\ref{lem_newAPA}, a priori assumptions \ref{APA1}--\ref{APA3} hold for this new partial homotopy.
A priori assumption \ref{APA4} holds due to Proposition~\ref{prop_move_part_hom_bckwrds}\ref{ass_lem_move_part_hom_bckwrds_b} and Lemmas~\ref{Lem_further_properties_gpdtp}\ref{ass_further_properties_gpdtp_g}, \ref{lem_apply_all_moves}.

This concludes our induction argument and proves Lemma~\ref{lem_newAPA}, which, as discussed in Subsection~\ref{subsec_deforming_main_results} implies Theorem~\ref{Thm_main_deform_to_CF}.

\section{Proofs of the main theorems}
\label{sec_proofs_main_theorems}
Theorem~\ref{Thm_main_general_case} from Section~\ref{sec_introduction} is a direct consequence of the following theorem.

\begin{theorem} \label{Thm_main_conf_flat}
Let $(K,L)$, $L \subset K$, be a pair of topological spaces that is homeomorphic to the geometric realization of a pair of finite simplicial complexes.
Let $M$ be a connected sum of spherical space forms and copies of $S^2 \times S^1$.
Consider a fiber bundle $\pi : E \to K$ whose fibers are homeomorphic to $M$ and whose structure group is $\Diff (M)$.
Let $(g^s)_{s \in K}$ be a continuous family of fiberwise Riemannian metrics such that $( \pi^{-1} (s) , g^s)$ is isometric to a compact quotient of the standard round sphere or standard round cylinder for all $s \in L$.

Then there is a continuous family of Riemannian metrics $(h^s_t)_{s \in K, t \in [0,1]}$ such that for all $s \in K$ and $t \in [0,1]$
\begin{enumerate}[label=(\alph*)]
\item \label{Thm_main_conf_flat_a} $h^s_0 = g^s$.
\item \label{Thm_main_conf_flat_b} $h^s_1$ is conformally flat and has positive scalar curvature.
\item \label{Thm_main_conf_flat_c} If $s \in L$, then $h^s_t = g^s$.
\item \label{Thm_main_conf_flat_d} If $(M^s, g^s)$ has positive scalar curvature, then so does $(M^s, h^s_t)$.
\end{enumerate}
\end{theorem}

\begin{proof}
Due to Remark~\ref{rmk_fiber_bundle_construction} we can view the fiber bundle $\pi : E \to K$ as a continuous family of Riemannian manifolds $(M^s := \pi^{-1} (s), g^s)_{s \in K}$ with $M^s \approx M$.
Let $K_{PSC} \subset K$ be a closed subset with the property that $(M^s, g^s)$ has positive scalar curvature for all $s \in K_{PSC}$.
We first show the theorem with the following weaker assertions; we will later explain how to strengthen these assertions:
\begin{enumerate}[label=(\alph*$'$), start=3]
\item \label{Thm_main_conf_flat_cp} If $s \in L$, then one of the following is true for all $t \in [0,1]$:
\begin{enumerate}[label=(c$'$\arabic*)]
\item $M \approx S^3 / \Gamma$ and $h^s_t$ is a constant multiple of $g^s$.
\item $M \approx (S^2 \times \IR) / \Gamma$ and $(M^s, h^s_t)$ is a quotient of the round cylinder and the local isometric $O(3)$-actions of $(M^s, g^s)$ and $(M^s, h^s_t)$ are the same.
\end{enumerate}
\item \label{Thm_main_conf_flat_dp} $(M^s, h^s_t)$ has positive scalar curvature for all $s \in K_{PSC}$ and $t \in [0,1]$.
\end{enumerate}
Consider the family of metrics $(h^s_t)_{s \in K, t \in [0,1]}$ produced by Theorem~\ref{Thm_main_deform_to_CF}, based on the family $(M^s, g^s)_{s \in K}$.
We claim that there is a continuous family of positive functions $(w^s_t \in C^\infty (M^s))_{s \in K, t \in [0,1]}$ such that for all $s \in K$ and $t \in [0,1]$:
\begin{enumerate}[label=(\arabic*)]
\item \label{prop_w_s_t_1} $w_0^s = 1$.
\item \label{prop_w_s_t_2} $w_t^s = 1$ if $s \in L$.
\item \label{prop_w_s_t_3} $(w^s_t)^4 h^s_t$ has positive scalar curvature, or equivalently $8 \triangle w^s_t - R_{h^s_t} w^s_t < 0$, if $(s,t) \in K \times \{ 1 \} \cup L \times [0,1]$.
\end{enumerate}
In fact, for any $(s_0, t_0) \in K \times [0,1]$ there is a neighborhood $U_{s_0, t_0} \subset K \times [0,1]$ and a continuous family of positive functions $(w^{s_0,s}_{t_0,t})_{(s,t) \in U_{s_0, t_0}}$ satisfying Properties~\ref{prop_w_s_t_1}--\ref{prop_w_s_t_3}.
Moreover, Properties~\ref{prop_w_s_t_1}--\ref{prop_w_s_t_3} remain valid under convex combination.
Therefore $(w^s_t)$ can be constructed from the families $(w^{s_0,s}_{t_0,t})$ using a partition of unity.
The family $(\td{h}^s_t := (w^s_t)^4 h^s_t)_{s \in K, t \in [0,1]}$ satisfies Assertions~\ref{Thm_main_conf_flat_a}, \ref{Thm_main_conf_flat_b}, \ref{Thm_main_conf_flat_cp}, \ref{Thm_main_conf_flat_dp}.

We will now argue that the theorem also holds with Assertion~\ref{Thm_main_conf_flat_c} replaced by \ref{Thm_main_conf_flat_cp} only.
For this purpose, let $K_{PSC} \subset K$ be the set of parameters $s \in K$ for which $(M^s, g^s)$ has non-negative scalar curvature.
Then $K_{PSC}$ is closed.
After evolving the metrics $g^s$ by the Ricci flow for some small uniform time $\tau > 0$, we can produce families of metrics $(g^s_t)_{s \in X, t \in [0,\tau]}$ such that $g^s_0 = g^s$ for all $s \in K$ and such that $g^s_t$ even has positive scalar curvature for all $s \in K_{PSC}$ and $t \in (0, \tau]$.
Apply our previous discussion to the family to the family $(M^s, g^s_\tau)_{s \in K}$, resulting in a family $(h^{\prime, s}_t)$ satisfying Assertions~\ref{Thm_main_conf_flat_a}, \ref{Thm_main_conf_flat_b}, \ref{Thm_main_conf_flat_cp}, \ref{Thm_main_conf_flat_dp}.
Then we can obtain the desired family $(h^s_t)$ by concatenating the resulting family $(h^{\prime, s}_t)$ with $(g^s_t)$.

So assume in the following that $(h^s_t)_{s \in K, t \in [0,1]}$ satisfies Assertions~\ref{Thm_main_conf_flat_a}, \ref{Thm_main_conf_flat_b}, \ref{Thm_main_conf_flat_c}, \ref{Thm_main_conf_flat_dp}.
It remains to construct a family $(h^{\prime\prime,s}_t)_{s \in K, t \in [0,1]}$, that in addition satisfies Assertion~\ref{Thm_main_conf_flat_c}.
By gluing together local trivializations of the vector bundle $\pi : E \to K$, we can find neighborhoods $L \subset U \Subset V \subset K$ and a transversely continuous bundle map $(F, f) : (E,K) \to (E, K)$; meaning that $\pi \circ F = f \circ \pi$ and  $F |_{\pi^{-1} (s)} : \pi^{-1} (s) \to \pi^{-1} (f(s))$ are a smooth diffeomorphisms that depend continuously on $s$ in the smooth topology such that:
\begin{enumerate}[start=4]
\item $(M^s,g^s)$ has positive scalar curvature for all $s \in V$.
\item $f = \id$ and $F = \id$ over $(K \setminus V) \cup L$ and $\pi^{-1} ((K \setminus V) \cup L)$.
\item $f(U) \subset L$.
\end{enumerate}
Let $\eta : K \to [0,1]$ be a continuous map such that $1- \eta$ is supported in $U$ and $\eta \equiv 0$ on $L$.
Then
\[ h^{\prime\prime,s}_t := \ov{F}^* h^{f(s)}_{ \eta(s) t} \] 
satisfies Assertions~\ref{Thm_main_conf_flat_a}--\ref{Thm_main_conf_flat_d}, which finishes the proof.
\end{proof}

We can now prove Theorem~\ref{thm_psc_contractible} from Section~\ref{sec_introduction}.

\begin{proof}[Proof of Theorem~\ref{thm_psc_contractible}]
Suppose $\met_{PSC}(M)$ is nonempty.
So  $M$ is a connected sum of spherical space forms and copies of $S^2\times S^1$ by Perelman \cite{Perelman2}.   

Now consider a continuous map $\al:S^k\ra \met_{PSC}(M)$ for some $k\geq 0$.  Since $\met(M)$ is contractible, we may extend $\al$ to a continuous map  $\wh\al:D^{k+1}\ra\met(M)$.   Letting $K:=D^{k+1}$ and $\pi:E:=K\times M\ra K$ be the trivial bundle, the map $\wh\al$ defines a family of fiberwise Riemannian metrics as in Theorem~\ref{Thm_main_general_case}.  Applying the theorem, we may reinterpret the resulting family $(h^s_t)_{s\in K, t\in [0,1]}$ as defining a homotopy of pairs $(\wh\al_t:(D^{k+1},S^k)\ra (\met(M),\met_{PSC}(M)))_{t\in [0,1]}$ where $\wh\al_0=\wh\al$ and $\wh\al_1$ takes values in $\met_{PSC}(M)$.  Restricting this homotopy to $S^n$, we obtain a homotopy $(\al_t:S^n\ra \met_{PSC}(M))_{t\in [0,1]}$ where $\al_1$ is null-homotopic via $\wh\al_1:D^{k+1}\ra \met_{PSC}(M)$.  Thus $\al$ is null-homotopic.  Since $\al$ was arbitrary and $\met_{PSC}(M)$ has the homotopy type of a CW-complex \cite[Sec. 2.1]{rubinstein_et_al}, it follows that $\met_{PSC}(M)$ is contractible.
\end{proof}

\bigskip

The remaining theorems from Section~\ref{sec_introduction} will follow from Theorem~\ref{Thm_main_conf_flat}  using the following lemma.

\begin{lemma}  \label{lem_conf_flat_round}
Let $(M,g)$ be a Riemannian 3-manifold.
\begin{enumerate}[label=(\alph*)]
\item \label{lem_conf_flat_round_a} If $M \approx S^3/\Gamma$ and $g$ is a conformally flat metric, then there is a unique $\phi \in C^\infty (M)$ such that $(M, e^{2\phi} g)$ is isometric to the standard round sphere and such that $\int e^{-\phi} d\mu_g$ is minimal.
Moreover $\phi$ depends continuously on $g$ (in the smooth topology).
\item \label{lem_conf_flat_round_b} If $M$ is diffeomorphic to a quotient of the round cylinder and $g$ is a  conformally flat Riemannian metric on $M$, then there is a unique $\phi \in C^\infty (M)$ such that $e^{2\phi} g$ is isometric to a quotient of the standard round cylinder.
Moreover, $\phi$ depends continuously on $g$ (in the smooth topology).
\end{enumerate}

\end{lemma}

\begin{proof}
Assertion~\ref{lem_conf_flat_round_a}.
It suffices to consider the case $M \approx S^3$, since we may pass to the universal cover, and the uniqueness of the minimizer guarantees that it will descend to a minimizer on $M$.
Let $V_g \subset C^\infty (M)$ be the space of functions $\phi$ with the property that $(M, e^{2\phi} g)$ is isometric to the standard round sphere. 
By \cite{Kuiper-1949} this space is non-empty.
Pick some arbitrary $\phi' \in V_g$ and identify $(M,e^{2\phi'} g)$ with $(S^3, g_{S^3})$.
Then $g = e^{-2\phi'} g_{S^3}$.
So we need to show that the functional
\[ F_{\phi'} (\phi) := \int_{S^3} e^{-\phi} e^{-3\phi'} d\mu_{S^3} \]
restricted to $V_\phi$ has a unique minimum.
To see this note that $\phi \in V_g$ if and only if $(S^3, e^{2(\phi- \phi')} g_{S^3})$ is isometric to the standard round sphere, which holds if and only if  for some $\vec c \in \IR^4$
\[ \phi - \phi' = - \log \big(\sqrt{1+ |\vec{c}|^2} - \vec{c} \cdot \vec{x} \big). \]
In this case we obtain that
\[ F_{\phi'} (\phi) = \int_{S^3} \big(\sqrt{1+ |\vec{c}|^2} - \vec{c} \cdot \vec{x} \big) e^{-4\phi'}  d\mu_{S^3} =: \td{F}_{\phi'} (\vec{c}). \]
It can be verified easily that $\td{F}_{\phi'} : \IR^4 \to \IR$ is strictly convex.
Moreover along every ray $t \mapsto t \vec{c}$ we have
\[ \lim_{t \to \infty} \td{F}_{\phi'} (t \vec{c}) = \lim_{t \to \infty}  \int_{S^3} \big(\sqrt{1+ t^2 |\vec{c}|^2} - t \vec{c} \cdot \vec{x} \big)e^{-4\phi'}d\mu_{S^3} = \infty. \]
So $\td{F}_{\phi'}$ and therefore $F_{\phi'} |_{V_g}$ attains a unique minimum.  

For the continuous dependence claim, note that for any continuous family $(g_s)_{s \in X}$ of conformally flat metrics on $M$ we can find a continuous family of smooth maps $(\psi_s : M \to S^3)_{s \in X}$ such that $\psi^*_s g_{S^3} = e^{2\phi'_{s}} g_s$ for some continuous family $(\phi_s)_{s \in X}$ of smooth functions; such a family can be constructed via the developing map for example, compare with the methods of \cite{Kuiper-1949}.
So it suffices to show that the minimizer of the functional $F_{\phi'_{s}}$ depends continuously on $s$, which is clear since the Hessians of $\td{F}_{\phi'_{s}}$ are positive definite.

Assertion~\ref{lem_conf_flat_round_b}.  We first observe that $M$ is diffeomorphic to either $S^2\times S^1$ or to $\IR P^3\# \IR P^3$.  Let $dev:\td M\ra S^3$ be a developing map of the conformally flat structure and $\pi_1(M)\stackrel{\rho}{\acts}S^3$ be the holonomy action, so that $dev$ is $\rho$-equivariant for the deck group action $\pi_1(M)\acts\td M$.  

We identify the conformal group $\conf(S^3)$ with $\isom(\IH^4)$.  Since $\pi_1(M)$ has a cyclic subgroup of index at most $2$, we have the following three cases.

{\em Case 1: The action $\pi_1(M)\stackrel{\rho}{\acts} \IH^4$ is elliptic, i.e. fixes a point in $\IH^4$. \quad}  Then there exists a $\rho$-invariant metric $g_0$ with sectional curvature $\equiv 1$ in the conformal class of $S^3$.  The pullback $dev^*g_0$ is $\pi_1(M)$-invariant and complete, and hence must be isometric to $S^3$, contradicting the fact that $\td M$ is noncompact.

{\em Case 2: The action $\pi_1(M)\stackrel{\rho}{\acts} \IH^4$ is parabolic, i.e. it fixes precisely one point $p$ in the ideal boundary $\D \IH^4=S^3$, and has no fixed points in $\IH^4$. \quad}  Letting $S=dev^{-1}(p)$, then $S$ is a closed, discrete subset of $\td M$ because $dev$ is an immersion.   There is a $\rho$-invariant complete flat metric $\wh g$ on $S^3\setminus\{p\}$.  Letting $\td M':=\td M\setminus S$, the pullback $dev\big|_{\td M'}^*\wh g$ is a complete flat metric on the simply connected manifold $\td M'$, and must therefore be isometric to $\R^3$; this contradicts the fact that $\td M$ is diffeomorphic to $(S^2\times \IR)\setminus S$.

{\em Case 3: The action $\pi_1(M)\stackrel{\rho}{\acts} \IH^4$ is hyperbolic, i.e. it preserves an axial geodesic $\gamma\subset \IH^4$, and fixes precisely two points $p_1,p_2\in S^3=\D \IH^4$. \quad } Letting $S=dev^{-1}(p)$, then $S$ is a closed, discrete subset of $\td M$.   There is a $\rho$-invariant complete cylindrical metric $\wh g$ on $S^3\setminus\{p_1,p_2\}$. Letting $\td M':=\td M\setminus S$, the pullback $dev\big|_{\td M'}^*\wh g$ is a complete metric on the simply connected manifold $\td M'$ which is locally isometric to $S^2\times \R$.  By the splitting theorem, $\td M'$ can have at most two ends, and hence $S=\emptyset$, and $dev\big|_{\td M}:\td M\ra S^3\setminus\{p_1,p_2\}$, being a local isometry between simply connected complete manifolds,  must be an isometry.  (Alternatively, use the developing map for the cylindrical structure.) Now suppose $\wh g'$ is another cylindrical metric on $S^3\setminus\{p_1,p_2\}$ that is invariant under the action of $\pi_1(M)\stackrel{\rho}{\acts}S^3$ and conformal to $\wh g$.  Then there is an isometry $\phi:(S^3\setminus\{p_1,p_2\},\wh g)\ra (S^3\setminus\{p_1,p_2\},\lambda\wh g')$ for some $\lambda>0$; but then $\phi$ is a conformal diffeomorphism of $S^3\setminus\{p_1,p_2\}$.  Hence $\phi\in \isom(\td M,\wh g)$, and so $\wh g=\lambda \wh g'$.
\end{proof}

\begin{proof}[Proof of Theorem~\ref{thm_cc_contractible}.]
We can argue similarly as in the proof of Theorem~\ref{thm_psc_contractible}.
Suppose that $\Met_{CC} (M)$ is non-empty and thus $M$ is diffeomorphic to an isometric quotient of the round sphere or round cylinder.

Consider a continuous map $\alpha : S^k \to \Met_{CC} (M)$ for some $k \geq 0$ and let $\wh\al:D^{k+1}\ra\met(M)$ be an extension of $\alpha$.   
As in the proof of Theorem~\ref{thm_psc_contractible} we can apply Theorem~\ref{Thm_main_general_case} to the associated family of metrics on the trivial vector bundle over $K = D^{k+1}$, this time with $L = S^k = \partial D^{k+1}$, and obtain a homotopy of pairs $(\wh\al_t:(D^{k+1},S^k)\ra (\met(M),\met_{CC}(M)))_{t\in [0,1]}$, where $\wh\al_0=\wh\al$ and $\wh\al_1$ takes values in the space $\met_{CF}(M)$ of conformally flat metrics on $M$.

Let $(g_s)_{s \in D^{k+1}}$ be the family of conformally flat metrics on $M$ corresponding to the map $\wh\alpha_1$.
By Lemma~\ref{lem_conf_flat_round} there is a certain continuous family $(\phi_s \in C^\infty (M))_{s \in D^{k+1}}$ such that $g'_s := e^{2\phi_s} g_s \in \Met_{CC} (M)$.
By the uniqueness statements in Lemma~\ref{lem_conf_flat_round} we have $\phi_s \equiv 0$ for all $s \in \partial D^{k+1}$.
So the family of metrics $(g'_s)_{s \in D^{k+1}}$ defines a null-homotopy of $\wh\alpha_1$ in $\Met_{CC}(M)$.
So $\alpha$, which is freely homotopic to $\wh\alpha_1$, is null-homotopic.

Since $\al$ was arbitrary and $\met_{CC}(M)$ has the homotopy type of a CW-complex \cite[Sec. 2.1]{rubinstein_et_al}, it follows that $\met_{CC}(M)$ is contractible.
\end{proof}

\begin{proof}[Proof of Theorem~\ref{thm_gen_smal}.]
Theorem~\ref{thm_gen_smal} follows from Theorem~\ref{thm_cc_contractible} via a standard topological argument, see for example \cite[Lemma 2.2]{gsc}.
\end{proof}

\begin{proof}[Proof of Theorem~\ref{thm_RP3RP3}.]
Let $M = \IR P^3 \# \IR P^3$ and consider a metric $g \in \Met_{CC} (M)$.
It can be seen easily that $(M, g)$ is isometric to $(S^2 \times S^1(l_g) )/ \IZ_2$, for some $l_g > 0$, where $S^2$ is equipped with the standard round metric, $S^1(l_g)$ has total length $l_g$ and $\IZ_2$ acts by the antipodal map on $S^2$ and by a reflection with two fixed points on $S^1(l_g)$.
There are unique numbers $a_g, b_g \in \IR$, depending continuously on $g$, such that $a_g g + b_g \Ric_g$ is isometric to $(S^2 \times S^1(2\pi ) )/ \IZ_2$.
Denote by $\Met_{CC1}(M) \subset \Met_{CC} (M)$ the space of metrics that are isometric to $(S^2 \times S^1(2\pi ) )/ \IZ_2$.
It follows that $\Met_{CC} (M)$ deformation retracts onto $\Met_{CC1}(M)$ and thus, by Theorem~\ref{thm_cc_contractible}, the space $\Met_{CC1}(M)$ is contractible as well.

We can now argue as in the proof of Theorem~\ref{thm_gen_smal} that 
\[ O(1) \times O(3) \cong \Isom ((S^2 \times S^1(2\pi ) )/ \IZ_2) \longrightarrow \Diff(M) \]
is a homotopy equivalence.
\end{proof}

\begin{proof}[Proof of Theorem~\ref{thm_S2S1_diff}.]
Let $M = S^2 \times S^1$ and denote by $S(M)$ the space of spherical structures on $M$.
We can view $S(M)$ as a subspace of the space of 2-dimensional distributions equipped with Riemannian metrics, which carries a natural smooth topology; equip $S(M)$ with the subspace topology.
Note that continuity with respect to this topology is equivalent to transverse continuity of spherical structures in the sense of Definition~\ref{Def_spherical_struct_transverse_cont}.
The space $S(M)$ has the homotopy type of a CW-complex, since it is a Fr\'echet manifold (see \cite[Sec. 2.1]{rubinstein_et_al}).
Let $\SS_{S^2 \times S^1} \in S(M)$ be the standard spherical structure on $S^2 \times S^1$.
It can be seen as in the proof of \cite[Lemma 2.2]{gsc} that the map
\[ \Diff (M) \longrightarrow S(M), \qquad \phi \longmapsto \phi^* \SS_{S^2 \times S^1} \]
is a fibration and that the inclusion $\Diff (M, \SS_{S^2 \times S^1} ) \to \Diff (M)$ is a homotopy equivalence if and only if $S(M)$ is contractible, where 
\[ \Diff (M, \SS_{S^2 \times S^1} ) \cong \Diff (S^1) \times O(3) \times \Omega O(3) \cong O(2) \times O(3) \times \Omega O(3)  \]
denotes the space of diffeomorphisms fixing $\SS_{S^2 \times S^1}$.

So it remains to show that $S(M)$ is contractible.
To see this, consider a continuous family of spherical structures $(\SS^s)_{s \in S^k}$, $k \geq 0$, on $M$.
For every $s \in S^k$ the space of metrics in $\Met_{CC} (M)$ that are compatible with $\SS^s$ is convex and non-empty (see Lemma~\ref{lem_spherical_struct_classification}).
Therefore, we can find a continuous family of Riemannian metrics $(g^s \in \Met_{CC} (M))_{s \in S^k}$ compatible with $(\SS^s)_{s \in S^k}$.
By Theorem~\ref{thm_cc_contractible} we can extend this family to a continuous family $(\wh g^s \in \Met_{CC} (M))_{s \in D^{k+1}}$.
The corresponding family of spherical structures $(\wh\SS^s)_{s \in D^{k+1}}$ constitute a null-homotopy of $(\SS^s)_{s \in S^k}$.
\end{proof}

\bibliography{library}
\bibliographystyle{amsalpha}

\end{document}